\documentclass[11pt,a4paper]{article}
\pdfoutput=1
\usepackage[pdftex]{graphicx}
\usepackage{ifpdf,latexsym,amsfonts,amssymb,galois,setspace,amsmath,color,mathrsfs, amstext,bm,bookmark,lineno,enumerate,}
\usepackage{amsbsy, amsopn, amscd, amsxtra, amsthm,authblk,indentfirst,multirow}
\usepackage{subfigure,colortbl}

\usepackage{geometry}
\usepackage[authoryear,numbers,round,sort&compress,colon,super]{natbib}
\bibpunct{[}{]}{;}{n}{}{,} 
\usepackage{float}
\usepackage{yhmath}
\usepackage{microtype}
\usepackage{fancyhdr}
\usepackage{hyperref}
\usepackage{cleveref,upref}

\definecolor{mygray}{gray}{.9}

\numberwithin{equation}{section}              
\newtheorem{theorem}{Theorem}[section]
\newtheorem*{theorem*}{Theorem}
\newtheorem{lemma}{Lemma}[section]
\newtheorem*{lemma*}{Lemma}
\newtheorem{proposition}{Proposition}[section]
\newtheorem*{proposition*}{Proposition}

\newtheorem*{corollary*}{Corollary}

\newtheorem*{definitions*}{Definitions}
\newtheorem*{conjecture*}{\bf Conjecture}
\newtheorem{example}{\bf Example}[section]
\newtheorem*{example*}{\bf Example}
\theoremstyle{remark}
\newtheorem{remark}{\bf Remark}[section]

\title{\bf{Turing-Hopf bifurcation and spatiotemporal patterns in a ratio-dependent diffusive Holling-Tanner system with time delay}
\footnotetext{The authors are supported by the National Natural Science Foundation of China (No.11371112).}
}

\author[a]{Qi An}
\author[a]{Weihua Jiang\thanks{Corresponding author.}}

\affil[a]{Department of Mathematics, Harbin Institute of Technology\\Harbin, 150001, P.R. China.}
\date{}
\begin{document}	
\maketitle
\begin{abstract}
\noindent
The Turing-Hopf type spatiotemporal patterns in a diffusive Holling-Tanner model with discrete time delay is considered. A global Turing bifurcation theorem for $\tau=0$ and a local Turing bifurcation theorem for $\tau>0$ are given by the method of eigenvalue analysis and prior estimation. Further considering the degenerated situation, the existence of Bogdanov-Takens bifurcation and Turing-Hopf bifurcation are obtained. The normal form method is used to study the explicit dynamics near the Turing-Hopf singularity, and we establish the existence of various self-organized spatiotemporal patterns, such as two non-constant steady states (stripe patterns) coexist and two spatially inhomogeneous periodic solutions (spot patterns) coexist.
Moreover, the Turing-Turing-Hopf type spatiotemporal patterns, that is  a subharmonic phenomenon with two  spatial wave numbers and one  temporal frequency, are also found and theoretically explained, when there is another Turing bifurcation curve which is relatively closed to the studied one.


\end{abstract}

\section{Introduction}

The term \textit{pattern} usually refers to a discernible regularity. Such as the spiral, tree, stripe, spot and the oscillations with spatial structure that have been observed in nature, chemical reaction and ecology could all be view as one types of spatiotemporal patterns, see \cite{Maini1997Spatial,Castets1990Experimental,ouyang1991transition,gunaratne1994pattern,pearson1993complex}.
One of the seminal works to study the pattern formation was given by Turing \cite{Turing1952} in 1952, he demonstrated that a simple reaction-diffusion-activation-inhibition mechanism in early embryo could generate complex spatial patterns of the epidermis of animals. These patterns usually have the structure of strip or spot, and also commonly known as Turing patters. After that,  various reaction diffusion systems, including chemical reaction models, predator-prey models and ecological models,  are widely used to explain the formation of patterns, see \cite{Lengyel1991Modeling,Murray2003II,Segel1972Dissipative,Yuan2014Pattern,Yan2012Pattern}.

Study on spatiotemporal patterns of predator-prey system is conducive to understand the reasons of the spatial and temporal oscillation of species, see \cite{Shi2015Spatial,Zhang2014Spatio,yi2009bifurcation,Chen2017Stability,Ni2005Turing}.
In this work, we revisit the classical Holling-Tanner models which was first proposed by  May \cite{May1973Stability}. It incorporated the self-regulation of prey and a Holling type \uppercase\expandafter{\romannumeral2} functional response function of predator, and it is used to describe the real ecological systems such as mite and spider mite, lynx and hare, sparrow and sparrow hawk, etc., (see Tanner \cite{Tanner1975The}
and Wollkind \textit{et al.} \cite{Wollkind1988Metastability}). Taking spatial dispersal into account within a fixed bounded domain $\Omega\in \mathbb{R}^n$, after a scaling as in \cite{An2017J}, this system is rewritten as follows:
\begin{equation}\label{eqA3}
\left\{
\begin{aligned}
&\frac{\mathrm{d}}{\mathrm{d}t}u-d_{1}\Delta u=u(1-u)-\frac{auv}{u+b},&&x\in\Omega,~t>0,&\\
&\frac{\mathrm{d}}{\mathrm{d}t}v-d_{2}\Delta v=rv(1-\frac{v}{u}),& &x\in\Omega,~t>0,&\\
&\partial_{\eta}u=\partial_{\eta}v=0,& &x\in\partial\Omega,~t>0,&\\
&u(x,0)=u_0(x), v(x,0)=v_0(x),& &x\in\Omega.&\\
\end{aligned}\right.
\end{equation}
Here $u,v$ respectively, represent the species densities of the prey and predator, and all the parameters appearing in \eqref{eqA3} are assumed to be positive. For more detailed biological implications of the model, please refer to \cite{Holling1966The,May1973Stability,An2017J}. This model has been investigated widely at the pattern formation, such as the Hopf bifurcation and Turing instability in \cite{li2011hopf}, the Turing
and non-Turing patterns in \cite{banerjee2012turing}, the steady state bifurcation of simple and double eigenvalues in \cite{Ma2013Bifurcation}, the degenerated Turing-Hopf bifurcation in \cite{An2017J}. In particular, Chang, L. \textit{et al} \cite{Chang2015Rich} unveiled six types of patterns exist in \eqref{eqA3} with a time delay in the negative feedback of the predator density by the numerical method and the  eigenvalue analysis of the Turing and Hopf bifurcation.

It is worth noting that the Turing-Hopf bifurcation can be considered as the simplest mechanism to form the patterns that are periodic  oscillated in both spatial and temporal. Meanwhile, this interaction between Hopf bifurcation and Turing bifurcation can bring many interesting dynamics, such as the bistable structure and the space-time chaos. Much previous work has focus on Turing-Hopf bifurcations of reaction diffusion systems (see \cite{Cangelosi2015Nonlinear,Meixner1997Generic,song2017,Song2014Spatiotemporal,Shi2015Spatial}), but there are few related theoretical research for the time-delay systems to the best of our knowledge.

The normal form  method \cite{Faria2000Normal} and centre manifold theorem \cite{Lin1992Centre,Wu1996Theory,MR1950831} are the effective tools to study the bifurcation dynamics, see \cite{yi2009bifurcation,Peng2013Spatiotemporal}. The advantage here is that they can give a complete division to the parameter space according to the different dynamics of the original system. Meanwhile, the mathematical approximation of spatial or temporal patterns are clear by using this method, see \cite{An2017,Faria2000Normal,Guo2016Stability,JA2018,Song2009Stability}.
In order to simplify the calculation of the normal form,  Jiang \textit{et al.} \cite{An2017,JA2018} give a relevant formula for the delayed reaction diffusion system with a Turing-Hopf bifurcation singularity. This formula only involves some simple algebraic operations and can be applied to computer program.

 Motived by the work of \cite{Chang2015Rich,An2017J}, we investigate the following delayed system in a one-dimensional spatial region $\Omega=(0,l\pi)$:
\begin{equation}\label{eqA}
\left\{
\begin{aligned}
&\frac{\mathrm{d}}{\mathrm{d}t}u(t)-d_{1}\Delta u(t)=u(t)[1-u(t)]-\frac{au(t)v(t)}{u(t)+b},&&x\in(0,l\pi),t>0,&\\
&\frac{\mathrm{d}}{\mathrm{d}t}v(t)-d_{2}\Delta v(t)=rv(t)[1-\frac{v(t-\tau)}{u(t-\tau)}],&&x\in(0,l\pi),t>0,&\\
&u_x(0,t)=v_x(0,t)=0,  u_x(l\pi,t)=v_x(l\pi,t)=0, &&t>0,&\\
&u(x,t)=\varphi(x,t), v(x,t)=\psi(x,t),&&x\in(0,l\pi),t\in[-\tau,0]&\\
\end{aligned}\right.
\end{equation}
Different from the exiting work \cite{Chang2015Rich}, we will first establish the conditions for the existence of the Turing pattern, and then further consider the effect of the time delay on the Turing pattern. By applying the normal form method, the parameter region near the Turing-Hopf bifurcation is divided into several parts with different dynamics. Some interesting phenomena such as two non-constant steady state coexist and two spatially inhomogeneous periodic solutions coexist will be found. We will show that large delay may induce the oscillation of the Turing pattern.
Of particular interest here, is the discovery of the Turing-Turing-Hopf type spatiotemporal patterns, that is a subharmonic phenomenon with two spatial wave numbers and one temporal frequency, which usually exist when
 there is another Turing bifurcation close to the studied one. We claim that the second Turing bifurcation would contribute a cosine function $\cos(\frac{n_I}{l}x)$ to affect the spatial distribution of the Turing-Hopf type patterns, but it have no impact on the division of the parameter plane.




 This paper is organized as follows. In Section 2, we devote to the bifurcation analysis of the Holling-Tanner system \eqref{eqA} with and without delay. The conditions for the existence of Turing, Bogdanov-Takens and Turing-Hopf bifurcation are given. In Section 3, the normal form near the Turing-Hopf critical point up to the third order are given by using the method present in \cite{An2017}. Then in Section 4, by analyzing the normal forms, we show that the Holling-Tanner models exhibits various spatiotemporal patterns. Appropriate simulations are carried out to illustrate the theoretical results. Finally a brief conclusion completes the paper.

\section{Stability and bifurcation analysis}
First of all, we define the following real-value Hilbert space
$$X:=\{(u,v)\in{H^{2}(0,l\pi)\times H^{2}(0,l\pi)} : (u_{x},v_{x})|_{x=0,l\pi}=0\}.$$
and the corresponding complexification $X_{\mathbb{C}}:=\{ x_{1}+ix_{2} : x_{1},x_{2}\in X \}.$
Let
$\mathcal{C}_{\tau}:=C([-\tau,0],X_{\mathbb{C}})$ ($\tau>0$) denote the phase space
with the sup norm. We write $\varphi^{t}\in \mathcal{C}_{\tau}$ for  $\varphi^{t}(\theta)=\varphi(t+\theta),-\tau\leq \theta \leq0.$

It is easy to check that system \eqref{eqA} has a unique coexistence equilibrium point
$E=(u_{0},v_{0})$, with
$u_{0}=v_{0}=\dfrac{1}{2}[(1-a-b)+\sqrt{(a+b-1)^{2}+4b}]<1$.
Taking the transformation $u\rightarrow u+u_{0}$ and $v\rightarrow v+u_{0}$ in \eqref{eqA}, we obtain an equivalent abstract equation in phase space $\mathcal{C}_{\tau}$:
\begin{equation}\label{eqB}
\frac{\mathrm{d}}{\mathrm{d}t}U(t)=D\Delta U(t)+L(r,\tau)(U^{t})+F(r,\tau,U^{t}).
\end{equation}
Here $D=\mathrm{diag}~(d_{1},d_{2})$,~ $U=(u,v)^{\mathrm{T}}\in X_{\mathbb{C}}$,~
$U^{t}=(u^{t},v^{t})^{\mathrm{T}}\in \mathcal{C}.$ And $L(r,\tau)(\cdot):\mathcal{C}_{\tau}\rightarrow X_{\mathbb{C}} $ is a bounded linear operator given by
\begin{equation}\label{eqL}
L(r,\tau)(\phi)={\left(\begin{array}{cc}
	A_{0} &B_{0}\\
	0  &0
	\end{array}\right)} \phi(0)+{\left(\begin{array}{cc}
	0&0\\
	r&-r
	\end{array}\right)} \phi(-\tau),
\end{equation}
with $A_0=\dfrac{u_{0}}{b+u_{0}}(1-b-2u_{0})$, $B_0=u_{0}-1$.
And $F(r,\tau,\cdot):\mathcal{C}_{\tau}\rightarrow X $ is a $C^{k}$ $(k\geq3)$ function and  given by
\begin{equation}\label{eqF}
F(r,\tau,\phi)=\begin{pmatrix}
f_{1}(\phi+E^\mathrm{T},r,\tau)-A_{0}\phi_{1}(0)-B_{0}\phi_{2}(0)\\
f_{2}(\phi+E^\mathrm{T},r,\tau)-r\phi_{1}(-\tau)+r\phi_{2}(-\tau)
\end{pmatrix}
\end{equation}
with $f_{1}(\varphi,r,\tau)=\varphi_{1}(0)[1-\varphi_1(0)]-
\dfrac{a\varphi_{1}(0)\varphi_{2}(0)}{\varphi_{1}(0)+b}$, $f_{2}(\varphi,r,\tau)=r\varphi_{2}(0)[1-\dfrac{\varphi_{2}(-\tau)}{\varphi_{1}(-\tau)}]$,
for $\varphi=(\varphi_{1},\varphi_{2})\in\mathcal{C}_{\tau}$, and satisfies $F(r,\tau,0)=0,~D_{\phi}F(r,\tau,0)=0$.

The corresponding characteristic equation of the linearized system of \eqref{eqB} is
\begin{equation}\label{eqC}
\mathbf{\Delta}(\lambda,r,\tau)y=\lambda y-D\Delta y-L(r,\tau)(e^{\lambda\cdot}y)=0,
\end{equation}
for some $y\in\mathrm{dom}(\Delta)\backslash\{0\},$ which is equivalent to the sequence of characteristic equations
\begin{equation}\label{characteristic}
\begin{aligned}
&G_n(\lambda,r,\tau):=\lambda^{2}-{T_{n}}(\lambda,r,\tau)\lambda+{D_{n}}(\lambda,r,\tau)=0,&&n=0,1,2,\cdots&
\end{aligned}
\end{equation}
with $n$ is identified as the wave number, and
\begin{equation}\label{eqTD}
\begin{aligned}
&{T_{n}}(\lambda,r,\tau)=A_{0}-(d_{1}+d_{2})\dfrac{n^{2}}{l^{2}}-re^{-\lambda\tau},&\\
&{D_{n}}(\lambda,r,\tau)=d_{2}\dfrac{n^{2}}{l^{2}}(d_{1}\dfrac{n^{2}}{l^{2}}-A_{0})+re^{-\lambda\tau}(d_{1}\dfrac{n^{2}}{l^{2}}-A_{0}-B_{0}).&
\end{aligned}
\end{equation}
\subsection{Turing bifurcation and Bogdanov-Takens bifurcation}
We, firstly, investigate the dynamics of \eqref{eqA} when $\tau=0$ and always assume that $a>\frac{(b+1)^{2}}{2(1-b)}$, if not, the constant steady state $(u_0,v_0)$ is  locally asymptotically stable (see \cite{An2017J}). Define
\begin{equation*}
\begin{aligned}
& r_n^H:=A_{0}-(d_{1}+d_{2})\frac{n^{2}}{l^{2}},&
&r_n^{T}:=-\frac{{d_{2}\frac{n^{2}}{l^{2}}(d_{1}\frac{n^{2}}{l^{2}}-A_{0})}}{{(d_{1}\frac{n^{2}}{l^{2}}-A_{0}-B_{0})}},&&\quad n\in \mathbb{N}_0,&
\end{aligned}
\end{equation*}
which satisfy $T_n(\lambda,r_n^H,0)=0$ and $D_n(\lambda,r_n^{T},0)=0$, respectively.
 The following results was proved by An and Jiang in \cite{An2017J}.

\begin{lemma}\label{lemma 1}
For system \eqref{eqA}, we assume that $d_1,d_2,r,l>0$, $1>b>0$, $a>\dfrac{(b+1)^{2}}{2(1-b)}$. Define $r_*:=\max\limits_{n\geq 0}r_n^{T}=r_{n_T}^{T}\; ( n_T\in\mathbb{N}_0)$
and
\begin{equation*}
\begin{aligned}
&b_*:=\dfrac{[(d_{1}+d_{2})^{2}-\sqrt{(d_{1}+d_{2})^{4}-(d_{1}-d_{2})^{4}}]^{2}}{(d_{1}-d_{4})^{4}}<1,&\\
&a_{\pm}:=\dfrac{(1\!-\!b)(d_{1}\!+\!d_{2})^{4}\!\pm\!(d_{1}\!+\!d_{2})^{2}\sqrt{(b\!+\!1)^{2}(d_{1}\!-\!d_{2})^{4}\!-\!4b(d_{1}\!+\!d_{2})^{4}}}{(d_{1}\!+\!d_{2})^{4}\!\!-\!\!(d_{1}\!-\!d_{2})^{4}},&\\
&x_{\pm}:=\frac{1}{2d_{1}d_{2}}[(d_{2}-d_{1})A_{0}\pm\sqrt{(d_{1}+d_{2})^{2}A_{0}^{2}+4d_{1}d_{2}A_{0}B_{0}}].
\end{aligned}
\end{equation*}
 Let
\begin{equation}
l_n^- := n\sqrt{\frac{1}{x_-}},  \hspace{1cm}
l_n^+ := n\sqrt{\frac{1}{x_+}},  \hspace{1cm}
\forall n\in\mathbb{N}_0.
\end{equation}
And $M_1(l),M_2(l)\in\mathbb{N}$ are the two non-negative integers, such that
$l_{M_{1}-1}^{-}\leq l< l_{M_{1}}^{-},\;\; l_{M_{2}}^{+}<l\leq l_{M_{2}+1}^{+}.$ Then we have
			$\max\limits_{n\geq 0}r_n^H=A_{0}<r_{*}$ if and only if
			\begin{enumerate}[$({{\mathbf{A}}}6^{''})$]
				\item $d_{2}>d_{1}$, $0<b<b_*$,  $a_{-}<a<a_{+}$, $M_1(l)\leq M_2(l)$.  Moreover, $A_{0}< r_{n}^{T}$ only when $M_1\leq n\leq M_2.$
			\end{enumerate}
\end{lemma}
Applying the Lemma \ref{lemma 1} when system parameters of \eqref{eqA} satisfy $({{\mathbf{A}}}6^{''})$, we obtain $D_{n_T}(\lambda,r_*,0)=0$, $T_{n_T}(\lambda,r_*,0)<0$ and $D_{n}(\lambda,r_*,0)>0$, $T_n(\lambda,r_*,0)<0$ for $n\neq {n_T}$. Furthermore, it means that the characteristic equation \eqref{eqC} has just one simple zero eigenvalue when $r=r_*$, and other eigenvalues have strictly negative real parts.
We claim the following global Turing bifurcation theorem.
\begin{theorem}\label{th1}
		Suppose that $a\leq\dfrac{(b+1)^{2}}{2(1-b)}$ and the condition $({{\mathbf{A}}}6^{''})$ hold in the Holling-Tanner system \eqref{eqA} without time delay. When the time delay $\tau=0$, we have:
		\begin{enumerate}[(1)]
			\item The coexistence equilibrium $E=(u_0,v_0)$ is  locally asymptotically stable when $r>r_*$, and unstable when $r<r_*$.
			\item The system \eqref{eqA} undergoes a Turing bifurcation at $r=r_*$. Moreover, there is a smooth curve $\varGamma$ of positive steady state of \eqref{eqA} bifurcating from $(r_*,u_0,v_0)$, with $\varGamma$ contained in a global branch $\mathcal{B}\in \mathbb{R}_{+}\times X$ of the positive steady state of \eqref{eqA}. Either $\mathcal{B}$ contains another bifurcation point $(r_n^T,u_0,v_0)$, or the projection of $\mathcal{B}$ onto $r$-axis contains the interval $(0,r_*)$ or $(r_*,\infty)$. If $d_1>l^2[1-2\mathcal{M}-\frac{ab\mathcal{M}}{(b+1)^2}]$ with $\mathcal{M}$ given by \eqref{M},  then the projection of $\mathcal{B}$ onto $r$-axis  can not contain the entire interval $(0,r_*)$.
		\end{enumerate}	
\end{theorem}
In order to prove the Theorem \eqref{th1}, we need to introduce two lemmas at first.
\begin{lemma}\label{lem2}
Suppose that $a,r,d_1,d_2>0$, $1>b>0$. Then there exists a positive constants $\mathcal{M}$ depending on $a$, $b$ and $\Omega$, such that any positive solution $(u(x),v(x))\in X$ of \eqref{eqA} satisfies
\begin{equation}\label{M}
\mathcal{M}\leq u(x),v(x) \leq 1 \quad \mathit{for\; any\;} x\in\overline{\Omega}.
\end{equation}
\end{lemma}
\begin{proof}
Set $u(\overline{x})=\max\limits_{\overline{\Omega}}u(x)$, $u(\underline{x})=\min\limits_{\overline{\Omega}}u(x)$, $v(\overline{y})=\max\limits_{\overline{\Omega}}v(x)$, $v(\underline{y})=\min\limits_{\overline{\Omega}}v(x)$.
From the maximum principle for weak solutions (see \cite{peng2008stationary}), we obtain
\begin{equation*}
\begin{aligned}
&1-u(\overline{x})-\frac{av(\overline{x})}{u(\overline{x})+b}\geq 0,& &1-u(\underline{x})-\frac{av(\underline{x})}{u(\underline{x})+b}\leq 0,&
&1-\frac{v(\overline{y})}{u(\overline{y})}\geq 0,& &1-\frac{v(\underline{y})}{u(\underline{y})}\leq 0.&
\end{aligned}
\end{equation*}
Which implies that $\frac{b}{a+b}\leq v(\overline{y})\leq u(\overline{x})\leq 1.$

Let $c_1(x):=1-u({x})-\frac{av({x})}{u({x})+b}$ and then we have $|c_1(x)|\leq 2+\frac{a}{b}$.
Using the Harnack inequality for weak solutions (see \cite{peng2008stationary}), there exists a positive constant $M_1$, depending on $a$, $b$, $\Omega$ such that
\begin{equation*}
\min\limits_{\overline{\Omega}} u(x) \geq M_1 \max\limits_{\overline{\Omega}} u(x) \geq \frac{M_1 b}{a+b}.
\end{equation*}
A similar method deal with $v(x)$, we can obtain the uniformly lower bound $\mathcal{M}=\mathcal{M}(a,b,\Omega)$ and complete the proof.
\end{proof}

\begin{lemma}\label{lem3}
Suppose that $a,r,d_2>0$, $1>b>0$ and $d_1\geq \mathcal{D}_1:=\frac{1}{\lambda_1}[1-2\mathcal{M}-\frac{ab\mathcal{M}}{(b+1)^2}]$, with $\lambda_1$ is the positive eigenvalue of the operator $-\Delta$ on $\Omega$ with the homogeneous Neumann boundary condition. Then there exists a small positive constant $R={R}(d_1,d_2,a,b,\Omega)$, such that the system \eqref{eqA} has no non-constant positive solution provided $r\leq R$.
\end{lemma}
\begin{proof}
Let $(u, v)$ be a positive steady state solution of \eqref{eqA}, and denote
\begin{equation*}
\bar{u}=\frac{1}{|\Omega|}\int_{\Omega} u\; \mathrm{d}x, \qquad \bar{v}=\frac{1}{|\Omega|}\int_{\Omega} v\; \mathrm{d}x.
\end{equation*}
Multiplying the equation of $u$ in \eqref{eqA} by $(u-\bar{u})$ and integrating over $\Omega$. According to  Lemma \ref{lem2}, we have
\begin{equation*}
\begin{aligned}
d_1\int_{\Omega}|\nabla(u\!-\!\bar{u})|^2 \mathrm{d}x&=\int_{\Omega}[1\!-\!(u+\bar{u})\!-\!\frac{abv}{(u+b)(\bar{u}+b)}](u\!-\!\bar{u})^2\! -\! \frac{a\bar{u}}{\bar{u}+b}(u\!-\!\bar{u})(v\!-\!\bar{v}) \mathrm{d}x\\
&\leq [1-2\mathcal{M}-\frac{ab\mathcal{M}}{(b+1)^2}+\varepsilon] \int_{\Omega}(u-\bar{u})^2 \mathrm{d}x + C(\varepsilon)\int_{\Omega}(v-\bar{v})^2 \mathrm{d}x
\end{aligned}
\end{equation*}
for any $\varepsilon=\varepsilon(a,b,\Omega)>0$.
Similarly to deal with the equation of $v$ in \eqref{eqA}, we obtain
\begin{equation*}
\begin{aligned}
\frac{d_2}{r}\int_{\Omega}|\nabla(v-\bar{v})|^2 \mathrm{d}x&=\int_{\Omega}[1-\frac{(v+\bar{v})\bar{u}}{u\bar{u}}](v-\bar{v})^2+\frac{\bar{v}^2}{u\bar{u}}(u-\bar{u})(v-\bar{v}) \mathrm{d}x\\
&\leq\varepsilon \int_{\Omega}(u-\bar{u})^2 \mathrm{d}x + [1+C(\varepsilon)]\int_{\Omega}(v-\bar{v})^2 \mathrm{d}x
\end{aligned}
\end{equation*}
Summing up the above two estimates and using the Poinc\'{a}re inequality, we get
 \begin{equation*}
 \begin{aligned}
d_1\int_{\Omega}|\nabla(u\!-\!\bar{u})|^2 \mathrm{d}x &+ \frac{d_2}{r}\int_{\Omega}|\nabla(v-\bar{v})|^2 \mathrm{d}x&\\&\leq  (\mathcal{D}_1+\varepsilon)\int_{\Omega}|\nabla(u\!-\!\bar{u})|^2 \mathrm{d}x + [\frac{1}{\lambda_1}+C(\varepsilon)]\int_{\Omega}|\nabla(v-\bar{v})|^2 \mathrm{d}x.
 \end{aligned}
 \end{equation*}

 It follows that, if $d_1>\mathcal{D}_1$, then there exists a ${R}={R}(d_1,d_2,a,b,\Omega)$ such that \eqref{eqA} has only the positive constant solution $(u,v)=(u_0,v_0)$ when $r<{R}$. The proof is completed.
\end{proof}
Based on the above lemmas, we prove the Theorem \ref{th1} as follows.
\begin{proof}[\bf Proof of Theorem \ref{th1}]
	Since $D_n(\lambda,r,0)$ and $T_n(\lambda,r,0)$ about $r$  are the  increasing and decreasing function respectively, we have $D_n(\lambda,r,0)>0$, $T_{n}(\lambda,r,0)<0$ for $r>r_*, n\in  \mathbb{N}_0$ and $D_{n_T}(\lambda,r,0)<0$ for $r<r_*.$ That proved the result of (1).

Now, we prove (2).  Assume that $\gamma(r)$ is the eigenvalue of the characteristic equation \eqref{eqC} with $\tau=0$, which satisfies $\gamma(r_*)=0$ and $G_{n_T}(\gamma(r),r,0)=0$ with $r$ close to $r_*$. Due to the fact that
\begin{equation*}
\frac{\partial}{\partial r}\gamma(r_*)= \frac{d_1\frac{n_T^2}{l^2}-A_0-B_0}{A_0-(d_1+d_2)^2\frac{n_T^2}{l^2}-r_*}\neq 0,
\end{equation*}
we conclude that the system \eqref{eqA} undergoes a Turing bifurcation at $r=r_*$. For the remainder part, we use the global bifurcation theory (see \cite{yi2009bifurcation}) and take $I=(0,\infty)$. Form the Lemma \eqref{lem2} and the elliptic regularity theory, any positive steady state of \eqref{eqA} are uniformly bounded in $X$. Hence the global branch $\mathcal{B}$ is bounded in $X$ and $\mathcal{B}\cap \mathbb{R}_+\times \{0\}=\emptyset$. If, in addition $ d_1>l^2[1-2\mathcal{M}-\frac{ab\mathcal{M}}{(b+1)^2}]$, then from Lemma \ref{lem3} we have the projection of $\mathcal{B}$ onto $r$-axis cannot contains the interval $(0,r_*)$, which completes the proof of Theorem \ref{th1}.
\end{proof}

When $\tau>0$, since $G_{n}(0,r,\tau)= D_{n}(\lambda,r,0)=0$, we obtain that zero is always a eigenvalue of the characteristic equation \eqref{eqC} with $r=r_*$.
We can get a deeper conclusion in the following theorem.
\begin{theorem}\label{theorem2}
Suppose that $a\leq\dfrac{(b+1)^{2}}{2(1-b)}$ and the condition $({{\mathbf{A}}}6^{''})$ hold in the Holling-Tanner system \eqref{eqA}.
Let \begin{equation}
\tau_0:=\frac{r_*+(d_1+d_2)\frac{n_T^2}{l^2}-A_0}{r_*(d_1\frac{n_T^2}{l^2}-A_0-B_0)}>0.\quad
\end{equation}
We have the following results.
\begin{enumerate}[(1)]
	\item The characteristic equation \eqref{eqC} has a simple zero eigenvalue when $r=r_*$ and $0\leq\tau\neq\tau_0.$ Moreover, if other eigenvalues have non-zero real part, then the system \eqref{eqA} undergoes a Turing bifurcation at $(r_*,\tau)$.
	\item The characteristic equation \eqref{eqC} has a double zero eigenvalues when $(r,\tau)=(r_*,\tau_0)$. Moreover, if other eigenvalues have non-zero real part, then
	the system \eqref{eqA} undergoes a {Bogdanov-Takens} bifurcation at $(r_*,\tau_0)$.
\end{enumerate}
\end{theorem}
\begin{proof}  When $\tau=\tau_0$, since
	\begin{equation}
	\begin{aligned}
	&\frac{\partial}{\partial \lambda}G_{n_T}(0,r_*,\tau_0)=[r_*+(d_1+d_2)\frac{n_T^2}{l^2}-A_0]- r_*(d_1\frac{n_T^2}{l^2}-A_0-B_0)\tau_0=0,\\
	&\frac{\partial^2}{\partial \lambda^2}G_{n_T}(0,r_*,\tau)=2+2r_*\tau+r_*(d_1\frac{n_T^2}{l^2}-A_0-B_0)\tau^2>0,
	\end{aligned}
	\end{equation}
we have zero is a eigenvalue of (algebraic) multiplicity two of the characteristic equation \eqref{eqC}.

When $\tau\neq\tau_0$ and $r$ close to $r_*$, apply the implicit function theorem to $G_{n_T}(\lambda,r,\tau)=0$, we can assume that $\gamma(r,\tau)$ is the eigenvalue of the characteristic equation \eqref{eqC} and  satisfies $\gamma(r_*,\tau)=0$ and $G_{n_T}(\gamma(r,\tau),r,\tau)=0$. Then we have
\begin{equation}
\begin{aligned}
\frac{\partial}{\partial r}\gamma(r_*,\tau)=\frac{d_1\frac{n_T^2}{l^2}-A_0-B_0}{A_0-(d_1+d_2)^2\frac{n_T^2}{l^2}-r_*+r_*\tau(d_1\frac{n_T^2}{l^2}-A_0-B_0)}\neq 0,\\
\end{aligned}
\end{equation}
 which satisfies the transversal condition. We complete the proof.
\end{proof}

\begin{example}\label{exa1}
Let $d_1=0.5$, $d_2=8.0$, $a=1$, $b=0.1$ and $l=5.0$, which are satisfy the condition $({{\mathbf{A}}}6^{''})$. From calculation, we have ${n_T}=2$, $r_*=0.4268>0.2625=A_0$ and $\tau_0=6.5248$.
\end{example}
\subsection{Turing-Hopf bifurcation}
In the following, we further investigate the impact of time delay on dynamics of system  \eqref{eqA}.
Assume that $\mathrm{i}\omega$$(\omega>0)$ is a pure imaginary eigenvalue of \eqref{eqC} and substitute it into \eqref{characteristic}. After the separation of the real and imaginary parts, we obtain that $\omega$ satisfies the following equations for some $n\in\mathbb{N}_0$,
\begin{equation}\label{eqw}
\left\{
\begin{aligned}
&\omega^{4}-P_{n}(r)\omega^{2}+Q_{n}(r)=0,\\
&\cos\omega\tau=C_n(\omega,r),\quad
\sin\omega\tau=S_n(\omega,r),
\end{aligned}
\right.
\end{equation}
with
\begin{equation}\label{eqPnQn}
\begin{aligned}
&P_{n}(r)
=-(d_{1}^{2}+d_{2}^{2})\frac{n^{4}}{l^{4}}+2d_{1}A_{0}\frac{n^{2}}{l^{2}}+r^{2}-A_{0}^{2}\\
&Q_{n}(r)=-(r-r_n^T)(r+r_n^T)(d_{1}\frac{n^{2}}{l^{2}}-A_{0}-B_{0})^{2},\\
&C_n(\omega,r)=\frac{-(B_{0}+d_{2}\frac{n^{2}}{l^{2}})\omega^{2}-d_{2}\frac{n^{2}}{l^{2}}(d_{1}\frac{n^{2}}{l^{2}}-A_{0})(d_{1}\frac{n^{2}}{l^{2}}-A_{0}-B_{0})}{r[\omega^{2}+(d_{1}\frac{n^{2}}{l^{2}}-A_{0}-B_{0})^{2}]},\\
&S_n(\omega,r)=\frac{\omega\{\omega^{2}+(d_{1}\frac{n^{2}}{l^{2}}-A_{0})^{2}+B_{0}[A_{0}-(d_{1}+d_{2})\frac{n^{2}}{l^{2}}]\}}{r[\omega^{2}+(d_{1}\frac{n^{2}}{l^{2}}-A_{0}-B_{0})^{2}]}.
\end{aligned}
\end{equation}
The existence of positive roots of \eqref{eqw} can be characterized as: (1) if one of the conditions  ${\bf(C1)}\, Q_{n}(r)<0$, ${\bf(C2)}\, Q_{n}(r)=0,\, P_{n}(r)>0$ or ${\bf(C3)}\, Q_{n}(r)>0,\, P_{n}(r)=2\sqrt{Q_n}(r) $ is satisfied, then  \eqref{eqw} has one positive root;
(2) if ${\bf(C4)}\, Q_{n}(r)>0,\, P_{n}(r)>2\sqrt{Q_n(r)}$ is hold, then \eqref{eqw} has two positive roots.

Firstly, we define two auxiliary functions
\begin{equation}
\begin{aligned}
&P(x,r):=-(d_{1}^{2}+d_{2}^{2})x^2+2d_{1}A_{0}x+r^{2}-A_{0}^{2}, &\quad x\geq 0,\;r>0,\\
&Q(x,r)
:=d_{1}d_{2}x^2-(d_{1}r+\!d_{2}A_{0})x+\!r(A_{0}+B_{0}),&\quad x\geq 0,\;r>0.\\
\end{aligned}
\end{equation}
They are satisfy $P(\frac{n^2}{l^2},r)=P_n(r)$ and $Q(\frac{n^2}{l^2},r)=-(r\!+\!r_{n}^{T})(d_{1}\frac{n^{2}}{l^{2}}\!-\!A_{0}\!-\!B_{0})$.
The zero roots of the equations $P(x,r)=0$ and $Q(x,r)=0$ are
\begin{equation}
\begin{aligned}
&x_{P}(r):=\frac{d_1A_0+\sqrt {r^2(d_1^2+d_2^2)-d_2^2A_0^2}}{d_1^2+d_2^2},\\
&x_Q(r):=\frac{d_1r+d_2A_0+\sqrt {(d_1r-d_2A_0)^2-4d_1d_2rB_0}}{2d_1d_2},
\end{aligned}
\end{equation}
respectively.
We can get the following result about the existence of the pure imaginary eigenvalues.
\begin{lemma}\label{lem_w} Assume that $a>\dfrac{(b+1)^2}{2(1-b)}$ and $({{\mathbf{A}}}6^{''})$ hold in the Holling-Tanner system \eqref{eqA}.
Then there exist two  positive integers $N_1\leq N_Q(r)$ such that $$ {N_1}\sqrt{\frac{d_1}{A_0}}<l\leq {(N_1+1)}\sqrt{\frac{d_1}{A_0}} \quad \textit{and}\quad {N_Q(r)}\sqrt{\frac{1}{x_Q(r)}}<l\leq {(N_Q(r)+1)}\sqrt{\frac{1}{x_Q(r)}} .$$
The characteristic equation \eqref{eqC} has $N_Q(r)+1$ pairs of  of pure imaginary eigenvalues when $r>r_*$, and at most $N_1+N_Q(r)+1$ pairs of  of pure imaginary eigenvalues when $r<r_*$. More precisely, the characteristic equation \eqref{eqC} has $card(S_0)$ pairs of pure imaginary eigenvalues
$\pm\mathrm{i}\omega_{n}(r_*)$, when $(r,\tau)=(r_*,\tau_n^{(k)}(r_*))$ and $n\in S_0$.
Here $\omega_{n}(r)=\left[\frac{Pn(r)+\sqrt{Pn(r)^2-4Q_n(r)}}{2}\right]^{\frac{1}{2}}$, and
	\begin{equation}
	\tau_{n}^{(k)}(r)=\left\{
	\begin{aligned}
	&\frac{1}{\omega_{n}}[\arccos C_n(\omega_{n},r)+2k\pi],&  &S_n(\omega_{n})>0,&\\         &\frac{1}{\omega_{n}}[-\arccos C_n(\omega_{n},r)+2(k+1)\pi],&  &S_n(\omega_{n})<0,&\\
	\end{aligned}
	\right.
	\end{equation}
	and $S_0$ is a set defined by
	\begin{equation}
	S_0 = \left\{
	\begin{aligned}
	&\{n\in\mathbb{N}_0: 0\leq n\leq N_Q(r_*) \; \mathit{and}\; n\neq {n_T}\},& &\mathit{when}\; l \leq {n_T}\sqrt{\frac{1}{x_P(r_*)}},&\\
	&\{n\in\mathbb{N}_0: 0\leq n\leq N_Q(r_*)\},& &\mathit{when}\; l > {n_T}\sqrt{\frac{1}{x_P(r_*)}}.&\\
	\end{aligned}
	\right.
	\end{equation}
\end{lemma}
\begin{proof} Our proof based on the existence of the positive roots of \eqref{eqw}, one can refer to Table \ref{tab1} to get a more intuitive understanding. First of all, we declare $N_1\leq N_Q(r)$ since $x_Q(r)-\frac{A_0}{d_1}>0$.
\begin{table}[h]
	\caption{The existence of the pure imaginary eigenvalues}\label{tab1}
	\vskip 0.2cm
	\renewcommand\arraystretch{1.1}
	\centering
\begin{tabular}{|c|c|c|c|}
\hline
\rowcolor{mygray}
\hline
{$n=0$} &{$1\leq n\leq N_1$}&{ $N_1+1\leq n\leq  N_Q(r)$} & {$n\geq N_Q(r)+1$}\\
\hline
\multirow{4}{*}{$Q_n(r)<0$}&\multirow{2}{*}{$r_n^T>0$,}&\multirow{2}{*}{$r_n^T\leq 0$, $r>-r_n^T$}&  \multirow{2}{*}{$r_n^T<0<r\leq-r_n^T$}\\
&&       &  \\
&If $r>r_n^T$: $Q_n(r)<0$&\multirow{2}{*}{$Q_n(r)<0$}&\multirow{2}{*}{$Q_n(r)\geq 0, P_n(r)< 0$}\\
&If $r\leq r_n^T$: $Q_n(r)\geq 0$&&\\
\cline{1-3}
\hline
\end{tabular}
\end{table}

When $n\geq N_Q(r)+1$, we have $Q(\frac{n^2}{l^2},r)>0$, $r_n^T<0$ and then $Q_n(r)\geq 0$. In addition, due to the fact that $P_n(r)\leq -(d_{1}^{2}+d_{2}^{2})\frac{n^{4}}{l^{4}}+2d_{1}A_{0}\frac{n^{2}}{l^{2}}+(r_n^T)^{2}-A_{0}^{2}<0$, we claim that \eqref{eqw} has no positive root. Adopting the same method for $0\leq n\leq N_Q(r)$, we obtain that there exist $N_Q(r)+1$ pairs of of pure imaginary eigenvalues when $r>r_*$, since $Q_n(r)<0$. Meanwhile, there exist at most $N_1+N_Q(r)+1$ pairs of of pure imaginary eigenvalues when $r< r_*$, since $Q_n(r)\geq 0$ only when $1\leq n\leq N_1$.

If $r=r_*$, more accurately, we have that $Q_{n_T}(r_*)=0$, $P_{n_T}(r_*)>0$ when $l > {n_T}\sqrt{\frac{1}{x_P(r_*)}}$, and $Q_{n_T}(r_*)=0$, $P_{n_T}(r_*)\leq 0$ when $l\leq {n_T}\sqrt{\frac{1}{x_P(r_*)}}$. Thus the characteristic equation \eqref{eqC} has $card(S_0)$ pairs of pure imaginary eigenvalues $\pm\mathrm{i}\omega_n$ when $(r,\tau)=(r_*,\tau_n^{(k)}(r_*))$, with $\omega_n(r)$ and $\tau_n^{(k)}(r)$  are obtained by \eqref{eqw}. The proof is completed.
\end{proof}
In fact, we can learn from the proof of the Lemma \ref{lem_w} that $\pm\mathrm{i}\omega_{n}(r)$, $n\leq N_Q(r)$ are the entire pure imaginary eigenvalues of the characteristic equation \eqref{eqw}, when $r>r_*$. After a few straightforward calculations, we have the following transversality conditions.
\begin{lemma}
	Assume $a>\dfrac{(b+1)^2}{2(1-b)}$ and $({{\mathbf{A}}}6^{''})$ hold in the Holling-Tanner system \eqref{eqA}. Let $\lambda_{n,\pm}(r,\tau)=\alpha_{n}(r,\tau)\pm\mathrm{i}\omega_{n}(r,\tau)$ are the eigenvalues of \eqref{eqC} that satisfy $\alpha_{n}(r,\tau_n^{(k)}(r))=0$, $\omega_{n}(r,\tau_n^{(k)}(r))=\omega_{n}(r)>0$  $(k\in \mathbb{N}_0)$.
Then we have
$$\frac{\partial}{\partial \tau}\alpha_{n}(r,\tau_n^{(k)}(r))>0.$$
\end{lemma}
\begin{proof}
Taking $\lambda_{n,\pm}(r,\tau)$ into \eqref{eqC} and doing the partial derivative about $\tau$, we have
\begin{align*}
\left[\frac{\partial}{\partial \tau}\lambda_{n,\pm}(r,\tau)\right]^{-1}=&-\frac{2\lambda_{n,\pm}+(d_1+d_2)\frac{n^2}{l^2}-A_0}{\lambda_{n,\pm}^3-[A_0-(d_1+d_2)\frac{n^2}{l^2}]\lambda_{n,\pm}^2+d_2\frac{n^2}{l^2}(d_1\frac{n^2}{l^2}-A_0)\lambda_{n,\pm}}\\
&+\frac{1}{\lambda_{n,\pm}[\lambda_{n,\pm}+(d_1\frac{n^2}{l^2}-A_0-B_0)]}-\frac{\tau}{\lambda_{n,\pm}},
\end{align*}
and
\begin{align*}
\mathrm{Re}\left[\frac{\partial}{\partial \tau}\lambda_{n,\pm}(r,\tau_n^{(k)}(r))\right]^{-1}
=&\frac{2\omega_n^2+r^2-P_n}{r^2[\omega_n^2+(d_1\frac{n^2}{l^2}-A_0-B_0)^2]}\!-\!\frac{1}{\omega_n^2+(d_1\frac{n^2}{l^2}-A_0-B_0)^2}\\
=&\frac{\sqrt{P_n^2-4Q_n^2}}{r^2[\omega_n^2+(d_1\frac{n^2}{l^2}-A_0-B_0)^2]}>0
\end{align*}
Due to the fact that
$$\mathrm{Sign}\left\{\frac{\partial}{\partial \tau}\alpha_{n}{(r,\tau_n^{(k)}(r))}\right\}=\mathrm{Sign}\left\{\mathrm{Re}\left[\frac{\partial}{\partial \tau}\lambda_{n,\pm}(r,\tau_n^{(k)}(r))\right]^{-1}\right\},$$ we complete the proof.
\end{proof}
For convenience, we denote
\begin{equation}\label{eqt}
\tau_*:=\min\limits_{n\in S_0}\tau_n^{(0)}(r_*)=\tau_{n_H}^{(0)}(r_*),\qquad \omega_*:=\omega_{n_H}(r_*),\qquad n_H\in\mathbb{N}_0,
\end{equation}  in the remainder of this article.
Based on the analysis above, we obtain the following Turing-Hopf bifurcation theorem.
\begin{theorem}\label{theorem3} For the  Holling-Tanner system \eqref{eqA}, assume that $a>\dfrac{(b+1)^2}{2(1-b)}$ and $({{\mathbf{A}}}6^{''})$ are satisfied.
	Then the constant steady state $(u_0,v_0)$ is  locally asymptotically stable when $r>r_*$ and $\tau<\min\limits_{n\leq N_Q(r)}\{\tau_n^{(0)}(r)\}$. Moreover, the system \eqref{eqA} undergoes a Turing-Hopf bifurcation at $(r_*,\tau_n^{(k)}(r_*))$, with $n\in S_0$ and $k\in \mathbb{N}_0$.
	The stable  bifurcating solutions can only bifurcated from $(r,\tau)=(r_*,\tau_*)$.
\end{theorem}
\begin{example}\label{exa2}
	Let $d_1=0.5$, $d_2=8.0$, $a=1$, $b=0.1$ and $l=5.0$ as in Example \ref{exa1}. Further calculation, we have $N_Q(r_*)=4$ and $l<9.5258={n_T}\sqrt{\frac{1}{x_P}}$. Then
	$S_0=\{n\in \mathbb{N}_0: 0\leq n\leq 4\;\mathit{and}\;n\neq 2\},$ and
	\begin{equation*}
	\begin{aligned}
	&\omega_0(r_*) = 0.5138,&   &\omega_1(r_*) = 0.4514,&   &\omega_3(r_*) = 0.0495 ,&   &\omega_4(r_*) = 0.0318,& \\
	&\tau_0^{(0)}(r_*) = 0.7014,&  &\tau_1^{(0)}(r_*) = 1.9291,&  &\tau_3^{(0)}(r_*) = 12.1121,& &\tau_4^{(0)}(r_*) = 84.0058.&
	\end{aligned}
	\end{equation*}
	Moreover, we get $n_H=0$ and $\tau_*=\tau_0^{(0)}(r_*)=0.7014<\tau_0.$ Therefore, the system undergoes a Turing-Hopf bifurcation at $(r_*,\tau_*)$ and a Bogdanov-Takens bifurcation at $(r_*,\tau_0)$.
\end{example}

\section{Normal form of Turing-Hopf bifurcation}
In order to further study the detailed dynamics properties of the Holling-Tanner system with $(r,\tau)$ near the Turing-Hopf singularity $(r_*,\tau_*),$ we adopt the framework of \cite{An2017} to get the normal forms at $(r_*,\tau_*)$ up to three orders in this section. We assume that $n_H=0$ and $n_T\neq 0$, which is the most common case.

Taking the time scale $t\rightarrow {t}/{\tau}$ and
rewriting \eqref{eqB} into an equivalent system in $ \mathcal{C}:=C([-1,0],X_{\mathbb{C}})$,
\begin{equation}\label{eqE}
\frac{\mathrm{d}}{\mathrm{d}t}U(t) =D(r,\tau) \Delta U(t) +A(r,\tau)U(t)+B(r,\tau)U(t-1)+F_0(r,\tau,U^t),
\end{equation}
with $F_0(r,\tau,\phi)=\tau F(r,1,\phi)$
for $\phi\in \mathcal{C}$, and
\begin{equation}\label{eqDAB}
\begin{aligned}
D(r,\tau)\!=\! \left(
\begin{array}{cc}
\tau d_1&0\\0&\tau d_2
\end{array}\right)\!,
A(r,\tau)\!=\!& \left(
\begin{array}{cc}
\tau A_0&\tau B_0\\0&0
\end{array}\right)\!,
B(r,\tau)\!=\!\left(
\begin{array}{cc}
0&0\\r\tau&-r\tau
\end{array}\right)\!.
\end{aligned}
\end{equation}
The characteristic equation of the linearized system of \eqref{eqE} is
\begin{equation}\label{eqD}
\mathbf{\Delta}_0(\lambda,r,\tau)y=\lambda I y-D(r,\tau)\Delta y-A(r,\tau)y-B(r,\tau)e^{-\lambda}y=0,
\end{equation}
for some $y\in\mathrm{dom}(\Delta)\backslash\{0\},$ which is equivalent to the sequence of characteristic equations
\begin{equation*}
\begin{aligned}
&G_n^0(\lambda,r,\tau):=\lambda^{2}-\tau{T_{n}}(\lambda,r,1)\lambda+\tau^2{D_{n}}(\lambda,r,1)=0,&&n=0,1,2,\cdots&,
\end{aligned}
\end{equation*}
with $T_n,D_n$ are given by \eqref{eqTD}. According to Theorem \ref{theorem2} - Theorem \ref{theorem3}, we have the following result.
\begin{theorem}
Assume that $a>\dfrac{(b+1)^2}{2(1-b)}$ and $({{\mathbf{A}}}6^{''})$ are satisfied in the equivalent Holling-Tanner system \eqref{eqE}. Then the system \eqref{eqE} undergoes a Turing-Hopf bifurcation at $(r_*,\tau_*)$. In addition,  if $\tau_*<\tau_0$, then except the simple  zero eigenvalue and a pair of pure imaginary eigenvalues $\pm\mathrm{i}\omega_{*}\tau_*$, the rest eigenvalues of the characteristic equation \eqref{eqD} with $(r,\tau)=(r_*,\tau_*)$ have strictly negative real parts.
\end{theorem}
\begin{remark}
We have done a great deal of numerical experiments, but the case when $\tau_*\geq \tau_0$ has not been found yet. It implies that $\tau_*<\tau_0$ is a usual case in this Holling-Tanner system.
\end{remark}

Expanding the phase space $\mathcal{C}$ into
$$\mathcal{BC}:=\{\psi:[-1,0]\rightarrow X_{\mathbb{C}} : \psi~\mathrm{is~ continuous~ on}~[-1,0),~\exists \lim_{\theta\rightarrow0^{-}}\psi(\theta)\in X_{\mathbb{C}}\}.$$
Then after the  parameters transformation $(\alpha_1,\alpha_2)=(r-r_*,\tau-\tau_*)$, the system \eqref{eqB} can be written as an abstract ordinary system in $\mathcal{BC}$ ,
\begin{equation}
\frac{\mathrm{d}}{\mathrm{d}t}U^t =AU^t+ X_0 \mathcal{F}(\alpha_1,\alpha_2,U^t).
\end{equation}
Here $$X_{0}(\theta)=\left\{\begin{array}{cc}
0,&~-1\leq\theta<0,\\
I,&~\theta=0.
\end{array}\right.$$
$A:\mathcal{C}_{0}^{1}\subset~ \mathcal{BC}\rightarrow\mathcal{BC},$ is defined by
$$ A\varphi=\dot{{\varphi}}+X_{0}[\tau_*D\Delta\varphi(0)+\tau_*L(r_*,1)(\varphi)-\dot{{\varphi}}(0)],$$
with $~\mathcal{C}_{0}^{1}=\{\varphi\in\mathcal{C}:\dot{\varphi}\in\mathcal{C},~\varphi(0)\in dom(\Delta)\}$.
And $\mathcal{F}:\mathbb{R}^2\times\mathcal{C}\rightarrow X_{\mathbb{C}} $ is defined by
$$\mathcal{F}(r,\tau,\varphi) = (\tau_*+\alpha_2)\{D\Delta\varphi(0)+L(r_*+\alpha_1,1)(\varphi)+F(r_*+\alpha_1,1,\varphi)\}-A\varphi(0).$$

We adopt the notations in \cite{An2017} with $(\mu_1,\mu_2)=(r,\tau)$, $\mu_0=(r_*,\tau_*)$  and $n_1=n_H$, $n_2=n_T$.
 Assume that $\{\phi_1(\theta)\beta_{n_{H}},\phi_2(\theta)\beta_{n_{T}}\}$ and $\{\psi_1(s)\beta_{n_{H}},\psi_2(s)\beta_{n_{T}}\}$ are the eigenfunctions of $A$ and its dual $A^*$ corresponding to the eigenvalues $\{ \mathrm{i}\omega_{*}\tau_*,0\}$, respectively. According to the Theorem 3.1 in \cite{An2017}, we obtain that
\begin{equation}\label{phi-psi}
\begin{aligned}
\phi_1(\theta)&= e^{\mathrm{i}\omega_{*}\tau_*\theta}(1\,,\,k_1)^{\mathrm{T}},\,\qquad
\phi_2(\theta)=(1\,,\,k_3)^{\mathrm{T}},\\
\psi_1(s)&= e^{\mathrm{-i}\omega_{*}\tau_* s}\,T_1(1\,,\,k_2),
\quad\psi_2(s)=T_2(1\,,\,k_4),
\end{aligned}
\end{equation}
with
\begin{equation*}
\begin{aligned}
&k_{1} = -(A_0 - \mathrm{i}\omega_{*} )/B_0,\qquad\quad
k_{2} = -(A_0 -\mathrm{i}\omega_{*} )e^{\mathrm{i}\omega_{*}\tau_*}/ r_*,\\
&k_{3} = -(A_{0} - d_{1} \frac{n_{T}^2}{l^2})/ B_{0},\quad\quad
k_{4} = -(A_{0} - d_{1} \frac{n_{T}^2}{l^2})/ r_*,\\
&T_{1} = [k_{1} k_{2} + e^{\mathrm{-i}\omega_{*}\tau_*}r_*\tau_* k_{2} ( 1 -k_1) + 1]^{-1},\\		
&T_{2} = [k_3 k_{4}+r_*\tau_* k_{4})(1- k_{3}) + 1)]^{-1}.
\end{aligned}
\end{equation*}
Decomposing $\mathcal{BC}$  into the direct sum of the center subspace $\mathcal{P}$ which is spanned by  $\{\phi_1(\theta)\beta_{n_{H}},\bar{\phi}_1(\theta)\beta_{n_{H}},\phi_2(\theta)\beta_{n_{T}}\}$, and its complement space,
$$\mathcal{BC}=\mathcal{P}\oplus \ker \pi.$$
Here $\pi:\mathcal{BC}\rightarrow\mathcal{P}$ is the projection operator.
Then $U^{t}\in\mathcal{C}_{0}^{1}\subset\mathcal{BC}$ can be divided as
\begin{equation}\label{U{t}}
\begin{split}
U^t(\theta)=\phi_{1}(\theta)z_1\beta_{n_{H}}+\bar{\phi}_{1}(\theta)\bar{z}_1\beta_{n_{H}}+\phi_{2}(\theta)z_2\beta_{n_{T}}+y(\theta),
\end{split}
\end{equation}
with $ y\in {\mathcal{C}_{0}^{1}}\cap\mathrm{ker}\pi :={\mathcal{Q}^{1}}$.
After a series of coordinate transformations $(z,y)\rightarrow (z+\frac{1}{j!}U_2^1(z),y+\frac{1}{j!}U_2^2(z))$ as shown in \cite{An2017},  the normal forms for \eqref{eqA} (or \eqref{eqB}) relative to $\Lambda=\{\pm\mathrm{i}\omega_{*}\tau_*,0 \}$ up to three order are obtained
\begin{equation}\label{eqNF}
\begin{aligned}
\dot{z_{1}}=&\quad \mathrm{i}\omega_*\tau_*z_{1}\!+\!\frac{1}{2}f_{\alpha_{1}z_1}^{11}\alpha_{1}z_{1}\!+\!\frac{1}{2}f_{\alpha_{2}z_1}^{11}\alpha_{2}z_{1}\!+\!\frac{1}{6}g_{210}^{11} z_1^2{\bar{z}_1} \!+\!\frac{1}{6} g_{102}^{11}z_1z_{2}^2+O(4),\\
\dot{\bar{z}}_1=&-\mathrm{i}\omega_*\tau_*{\bar{z}_1}\!+\!\frac{1}{2}\overline{f_{\alpha_{1}z_1}^{11}}\alpha_{1}{\bar{z}_1}\!+\!\frac{1}{2}\overline{f_{\alpha_{2}z_1}^{11}}\alpha_{2}{\bar{z}_1}\!+\!\frac{1}{6}\overline{g_{210}^{11}}z_1{\bar{z}_1^2}\!+\!\frac{1}{6}\overline{g_{102}^{11}}{\bar{z}_1}z_{2}^2+O(4) ,\\
\dot{z_{2}}=&\hspace{2cm}\frac{1}{2}f_{\alpha_{1}z_{2}}^{13}\alpha_{1}z_{2}\!+\!\frac{1}{2}f_{\alpha_{2}z_{2}}^{13}\alpha_{2}z_{2}\!+\!\frac{1}{6}g_{111}^{13}z_1{\bar{z}_1}z_{2}\!+\!\frac{1}{6}g_{003}^{13}z_{2}^3+O(4).
\end{aligned}
\end{equation}
 The formulas of $f_{\alpha_{1}z_1}^{11}$, $f_{\alpha_{2}z_1}^{11}$, $f_{\alpha_{1}z_{2}}^{13}$, $f_{\alpha_{2}z_{2}}^{13}$, $g_{210}^{11}$, $g_{102}^{11}$, $g_{111}^{13}$, $g_{003}^{13}$  can be accurately given by the help of {\bf{Matlab}}. The specific calculation process are based on the formulas in \cite{An2017,JA2018} and will be shown in the following.
 \vskip 0.25cm\noindent

{\bf Step 1.} First of all, we need calculate the second and third derivatives of $\mathcal{F}(\alpha_1,\alpha_2,U^t)$ with respect to $u(t),v(t),u(t-1),v(t-1)$ at $(\alpha_1,\alpha_2,U^t)=(0,0,0)$.
Denote $u,v,u_{\tau},v_{\tau}$ as the simplified form of $u(t),v(t),u(t-1),v(t-1)$, respective. From calculation,  the non-zero partial derivatives are listed as follows.
\begin{equation}
\begin{aligned}
&F_{uu} = ({2\tau_*abu_0}/{(b + u_0)^3} -1,\,0)^{\mathrm{T}},\qquad\\
&F_{uv} = (-\tau_*ab/(b + u_0)^2 ,\,0)^{\mathrm{T}},\\
&2F_{vv_\tau} = -2F_{vu_\tau} = -2F_{u_\tau v_\tau} = F_{u_\tau u_\tau} = (0,\,{-2r_*\tau_*}/{u_0})^{\mathrm{T}},\\
&F_{uuu} = (6\tau_*abu_0/(b + u_0)^4 ,\,0)^{\mathrm{T}},\qquad\qquad\\
&F_{uuv} = (2\tau_*ab/(b + u_0)^3 ,\,0)^{\mathrm{T}},\\
&-3 F_{vu_\tau u_\tau} =-3 F_{u_\tau u_\tau v_\tau}= 6F_{vu_\tau v_\tau}=F_{u_\tau u_\tau u_\tau}=
(0,\,6r_*\tau_*/u_0^2)^{\mathrm{T}} ,
\end{aligned}
\end{equation}
 with $F_{uu}=\frac{\partial}{\partial u}\mathcal{F}(0,0,0)$, and so forth.\
 \vskip 0.25cm\noindent

{\bf Step 2.} Secondly, the coefficient vectors $F_{mnk}$ of the terms $z_1^m\bar{z}_1^nz_2^k\beta_{n_{H}}^{m+n}\beta_{n_{T}}^k$ in $\mathcal{F}(\alpha_1,\alpha_2,U^t)$ after the decomposition \eqref{U{t}} are worked out.
\begin{equation*}
\begin{aligned}
F_{200} =&
 F_{uu}\! +\! 2F_{uv}k_1\! +\! (F_{u_\tau u_\tau } \!+\! 2F_{u_\tau v_\tau } k_1) e^{-2\mathrm{i}\omega_*\tau_*} \!+\! 2 k_1 ( F_{vu_\tau } \!+\! F_{vv_\tau } k_1) e^{-\mathrm{i}\omega_*\tau_*},\\
F_{110} =&
 2[ F_{uu} \!+\!  F_{u_{\tau}u_{\tau}} \!+\! \! (F_{uv}+ F_{u_{\tau}v_{\tau}}) (k_{1}+\bar{k}_1) \!+
\!+\! F_{vu_{\tau}} (\bar{k}_{1}e^{\mathrm{-i}\omega_*\tau_*}+ k_{1} e^{\mathrm{i}\omega_*\tau_*})\\&+ F_{vv_{\tau}} k_{1}\bar{k}_{1}( e^{\mathrm{-i}\omega_*\tau_*}\!+\!e^{\mathrm{i}\omega_*\tau_*}),\\
F_{101}  =& 2 [F_{uu}\!\! +\!\! F_{u_{\tau}u_{\tau}} e^{\mathrm{-i}\omega_*\tau_*} \!\!+\!\!  F_{uv} (k_{1}\! \!+\!\!k_{3})
\!\!+\!\!  F_{vu_{\tau}}( k_{1}\! \!+\!\! k_{3} e^{\mathrm{-i}\omega_*\tau_*})
\!+\!F_{vv_{\tau}} k_{1} k_{3}(1\!+\! e^{\mathrm{-i}\omega_*\tau_*}) \\&\!+\! F_{u_{\tau}v_{\tau}} (k_{1}\!+\! k_{3}) e^{\mathrm{-i}\omega_*\tau_*}],\\
F_{002}  =& F_{uu} +F_{u_{\tau}u_{\tau}}+2( F_{uv} k_{3} + F_{vu_{\tau}} k_{3}+F_{vv_{\tau}} k_{3}^2+ F_{u_{\tau}v_{\tau}} k_{3}),\\
F_{020} =& \overline{F_{200}},\hspace{1cm}
F_{011}  = \overline{F_{101}},\\
\end{aligned}
\end{equation*}
\begin{equation*}
\begin{aligned}
	F_{210} =& 3[F_{uuu}+F_{uuv}(2 k_{1}+\bar{k}_1)+  F_{vu_{\tau}u_{\tau}} (2 k_{1}+\bar{k}_{1}e^{\mathrm{-2i}\omega_*\tau_*} )+2  F_{vu_{\tau}v_{\tau}} k_1(\bar{k}_{1}\\&+\bar{k}_{1}e^{\mathrm{-2i}\omega_*\tau_*}+k_{1})+ F_{u_{\tau}u_{\tau}u_{\tau}} e^{\mathrm{-i}\omega_*\tau_*}+  F_{u_{\tau}u_{\tau}v_{\tau}} (2 k_{1}+ \bar{k}_{1})e^{\mathrm{-i}\omega_*\tau_*}],\\
	F_{102} =& 3 [F_{uuu} \!+\! F_{uuv} ( k_{1} \!+\! 2  k_{3})\!+\! F_{vu_{\tau}u_{\tau}} ( k_{1} \!+\! 2  k_{3} e^{\mathrm{-i}\omega_*\tau_*}) \!+\! 2F_{vu_{\tau}v_{\tau}}  k_{3}( k_{1}  \!+\!  k_{3} e^{\mathrm{-i}\omega_*\tau_*} \\&\!+\!  k_{1} e^{\mathrm{-i}\omega_*\tau_*}) \!+\! F_{u_{\tau}u_{\tau}u_{\tau}} e^{\mathrm{-i}\omega_*\tau_*} \!+\! F_{u_{\tau}u_{\tau}v_{\tau}} ( k_{1} e^{\mathrm{-i}\omega_*\tau_*} \!+\! 2  k_{3} e^{\mathrm{-i}\omega_*\tau_*})],\\
	F_{111} =&  6\{ F_{uuu}  +  F_{uuv} ( k_{1} + k_{3} +  \bar{k}_{1})  + F_{vu_{\tau}u_{\tau}} ( k_{3} +  k_{1} e^{\mathrm{i}\omega_*\tau_*} +  e^{\mathrm{-i}\omega_*\tau_*} \bar{k}_{1})  \\&+  F_{vu_{\tau}v_{\tau}} [k_{1} k_{3}(1+ e^{\mathrm{i}\omega_*\tau_*}) +  k_{1}\bar{k}_{1}(e^{\mathrm{-i}\omega_*\tau_*} + e^{\mathrm{i}\omega_*\tau_*})+\bar{k}_{1} k_{3} (1+ e^{\mathrm{-i}\omega_*\tau_*} ) ] \\& +  F_{u_{\tau}u_{\tau}u_{\tau}} +   F_{u_{\tau}u_{\tau}v_{\tau}} ( k_{1} +  k_{3} +  \bar{k}_{1})\},\\
	F_{003} =& F_{uuu} \!+\! 3F_{uuu_{\tau}} \!+\! F_{u_{\tau}u_{\tau}u_{\tau}} \!+\!3 F_{uuv} k_{3}\!+\!3 F_{u_{\tau}u_{\tau}v_{\tau}} k_{3} \!+\! 3F_{vu_{\tau}u_{\tau}} k_{3} \!+\! 6F_{vu_{\tau}v_{\tau}} k_{3}^2,
\end{aligned}
\end{equation*}
Therefore, according to $(2.23_2)$, $(2.23_3)$ in \cite{An2017}, we have
\begin{equation}\label{f_2^1}
\begin{aligned}
&f_{200}^{11} = \frac{1}{\sqrt{l\pi}}\psi_1(0)F_{200},\quad
f_{110}^{11} =  \frac{1}{\sqrt{l\pi}}\psi_1(0)F_{110},\quad
f_{020}^{11} =\frac{1}{\sqrt{l\pi}}\psi_1(0)F_{020},\\
&f_{002}^{11} = \frac{1}{\sqrt{l\pi}}\psi_1(0)F_{002},\quad
f_{101}^{13} =  \frac{1}{\sqrt{l\pi}}\psi_2(0) F_{101},\quad
f_{011}^{13} = \frac{1}{\sqrt{l\pi}}\psi_2(0)F_{011},~~~\\
&f_{200}^{12} = \overline{f_{020}^{11}},\qquad \quad\;\,
f_{110}^{12} = \overline{f_{110}^{11}},\qquad \quad\;\,
f_{020}^{12} = \overline{f_{200}^{11}},\qquad \quad\;\,
f_{002}^{12} = \overline{f_{002}^{11}}, \\
\end{aligned}
\end{equation}
and
\begin{equation}\label{f_3^1}
\begin{aligned}
f_{210}^{11}=&\frac{1}{l\pi}\psi_{1}(0)F_{210},\qquad
f_{102}^{11}=\frac{1}{l\pi}\psi_{1}(0)F_{102},\\
f_{111}^{13}=&\frac{1}{l\pi}\psi_{2}(0)F_{111},\qquad
f_{003}^{13}=\frac{3}{2l\pi}\psi_{2}(0)F_{003}.
\end{aligned}
\end{equation}
\vskip 0.25cm\noindent

{\bf Step 3.} Further more, the linear operators $S_{yz_i}(i=1,2) ,S_{y\bar{z}_1}: \mathcal{Q}_{1}\rightarrow X_{\mathbb{C}}$ are defined by
\begin{equation}\label{S}
\begin{aligned}
S_{yz_i}(\varphi) =& (F_{y_1(0)z_i},\;F_{y_2(0)z_i})\varphi(0)+(F_{y_1(-1)z_i},\; F_{y_2(-1)z_i})\varphi(-1),\\
S_{yz_2}(\varphi) =& (F_{y_1(0)z_2},\;F_{y_2(0)z_2})\varphi(0)+(F_{y_1(-1)z_2},\; F_{y_2(-1)z_2})\varphi(-1),\\
\end{aligned}
\end{equation}
with
\begin{equation*}\label{Eyz}
\begin{aligned}
&F_{y_1(0)z_1} \;\;= 2( F_{uu} \!+\! F_{uv} k_{1} ),\hspace{2cm}\\
&F_{y_1(-1)z_1} = 2( F_{vu_{\tau}} k_{1} \!+\!  F_{u_{\tau}u_{\tau}} e^{\mathrm{-i}\omega_*\tau_*} \!+\! F_{u_{\tau}v_{\tau}} k_{1} e^{\mathrm{-i}\omega_*\tau_*}),\\
&F_{y_2(0)z_1} \;\;= 2( F_{uv} \!+\! F_{vu_{\tau}} e^{\mathrm{-i}\omega_*\tau_*} \!+\! F_{vv_{\tau}} k_{1} e^{\mathrm{-i}\omega_*\tau_*}),\\
&F_{y_2(-1)z_1} =  2( F_{vv_{\tau}} k_{1} \!+\! F_{u_{\tau}v_{\tau}} e^{\mathrm{-i}\omega_*\tau_*}),\\
&F_{y_1(0)z_{2}} \;\;= 2( F_{uu} \!+\! F_{uv} k_{3} ),\\
&F_{y_1(-1)z_{2}} = 2( F_{u_{\tau}u_{\tau}} \!+\! F_{u_{\tau}v_{\tau}} k_{3} \!+\! F_{vu_{\tau}} k_{3}),\\
&F_{y_2(0)z_{2}} \;\;=2( F_{uv} \!+\! F_{vu_{\tau}} +\! F_{vv_{\tau}} k_{3} ),\\
&F_{y_2(-1)z_{2}} =2( F_{u_{\tau}v_{\tau}} \!+\! F_{vv_{\tau}} k_{3}).
\end{aligned}
\end{equation*}
\vskip 0.25cm \noindent

{\bf Step 4.} Next, we will calculate $U_2^2(\cdot)\in V_{2}^{3}(\mathcal{Q}_{1})$.  In fact, it is enough to get the following formulas.
 \begin{equation*}
\begin{aligned}
&{\langle h_{200}(\theta)\beta_{n_H},\beta_{n_H} \rangle}
\!\!=\!\!~\frac{1}{l\pi}e^{2\mathrm{i}\omega_0\theta}[2\mathrm{i} \omega_0\!\!-\!\!L_0(e^{\mathrm{2i}\omega_0\cdot}I_d)]^{-1}F_{200} \!\!-\!\!\frac{1}{\mathrm{i}\omega_0}\frac{1}{\sqrt{l\pi}}[f_{200}^{11}\phi_1(\theta)\!\!+\!\!\frac{1}{3}f_{200}^{12}\bar{\phi}_1(\theta)],\\
&{\langle h_{110}(\theta)\beta_{n_H},\beta_{n_H} \rangle} \!\!=\!\! -\frac{1}{l\pi}[L_0(I_d)]^{-1}F_{110}+\frac{1}{\mathrm{i}\omega_0}\frac{1}{\sqrt{l\pi}}[f_{110}^{11}\phi_1(\theta)\!-\!f_{110}^{12}\bar{\phi}_1(\theta)],\\
&{\langle h_{110} (\theta)\beta_{n_T},\beta_{n_T} \rangle}=~ {\langle h_{110}(\theta)\beta_{n_H},\beta_{n_H} \rangle},\hspace{4cm}\\
\end{aligned}
\end{equation*}
\begin{equation*}
\begin{aligned}
&{\langle h_{101} (\theta)\beta_{n_T},\beta_{n_H} \rangle} \!=\!~\frac{1}{l\pi}e^{\mathrm{i}\omega_0\theta}[\mathrm{i} \omega_0\!+\!\frac{n_T^2}{l^2}D_0\!-\!L_0(e^{\mathrm{i}\omega_0\cdot} I_d)]^{-1}F_{101}
\!-\!\frac{1}{\mathrm{i}\omega_0}\frac{1}{\sqrt{l\pi}}f_{101}^{13}\phi_2(0),\\
&{\langle h_{011} (\theta)\beta_{n_H},\beta_{n_T} \rangle}\!\!=\!\!~\frac{1}{l\pi}e^{\mathrm{\!-\!i}\omega_0\theta}[-\mathrm{i} \omega_0+\frac{n_T^2}{l^2}D_0\!-\!L_0(e^{\mathrm{-i}\omega_0\cdot} I_d)]^{-1}F_{011}
\!+\!\frac{1}{\mathrm{i}\omega_0}\frac{1}{\sqrt{l\pi}}f_{011}^{13}\phi_2(0),\\
&{\langle h_{002} (\theta)\beta_{n_H},\beta_{n_H} \rangle}=\!-\!\frac{1}{l\pi}[L_0(I_d)]^{-1}F_{002}\!+\!\frac{1}{\mathrm{i}\omega_0}\frac{1}{\sqrt{l\pi}}[f_{002}^{11}\phi_1(\theta)-f_{002}^{12}\bar{\phi}_1(\theta)],\\
&{\langle h_{002} (\theta)\beta_{n_T},\beta_{n_T} \rangle}=\frac{1}{2l\pi}[\frac{(2n_T)^2}{l^2}D_0\!-\!L_0(I_d)]^{-1}F_{002}\!+\!{\langle h_{002} (\theta)\beta_{n_H},\beta_{n_H} \rangle}.\\
\end{aligned}
\end{equation*}
Here $\omega_0=\omega_*\tau_*$. $L_0:\mathcal{C}\rightarrow X_{\mathbb{C}}$ is a linear operator and  given by
$L_0(\phi)=\tau_*L(r_*,1)(\phi)$.
\vskip 0.25cm \noindent

{\bf Step 5.} In finally, according to Theorem 3.2-Theorem 3.3 in
\cite{An2017}, we get the formulas of the coefficients in the normal forms \eqref{eqNF}.
\begin{theorem}\label{theNF}
For the Holling-Tanner system \eqref{eqA}, assume that $a>\dfrac{(b+1)^2}{2(1-b)}$ and $({{\mathbf{A}}}6^{''})$ are hold. If $n_H=0$ and $n_T\neq0$,  then the  quadratic and cubic terms in the normal forms \eqref{eqNF} of the system \eqref{eqA} relative to $(r_*,\tau_*)$ are
\begin{equation*}\label{alphaz}
\begin{aligned}
&f_{\alpha_{1}z_1}^{11} =
2\psi_1(0)\,\left[\frac{\partial}{\partial r}A(r_*,\tau_*)\phi_1(0)+\frac{\partial}{\partial r}B(r_*,\tau_*)\phi_1(-1)\right],\\
&f_{\alpha_{2}z_1}^{11}=
2\psi_1(0)\,\left[\frac{\partial}{\partial \tau}A(r_*,\tau_*)\phi_1(0)+\frac{\partial}{\partial \tau}B(r_*,\tau_*)\phi_1(-1)\right],\\
&f_{\alpha_{1}z_{2}}^{13} =
2\psi_2(0)\,\left[-\frac{n_T^2}{l^2}\frac{\partial}{\partial r}D(r_*,\tau_*)\phi_2(0)+\frac{\partial}{\partial r}A(r_*,\tau_*)\phi_2(0)+\frac{\partial}{\partial r}B(r_*,\tau_*)\phi_2(-1)\right],\\
&f_{\alpha_{2}z_{2}}^{13} =
2\psi_2(0)\,\left[-\frac{n_T^2}{l^2}\frac{\partial}{\partial \tau}D(r_*,\tau_*)\phi_2(0)+\frac{\partial}{\partial \tau}A(r_*,\tau_*)\phi_2(0)+\frac{\partial}{\partial \tau}B(r_*,\tau_*)\phi_2(-1)\right],\\
&g_{210}^{11} = f_{210}^{11}+\frac{3}{\mathrm{2i}\omega_0}(-f_{110}^{11} f_{200}^{11} +f_{110}^{11} f_{110}^{12} + \frac{2}{3}f_{020}^{11} f_{200}^{12})+\\
&\hspace{1.2cm}\frac{3}{2}\psi_1(0)\left[S_{yz_1}(\langle h_{110}(\theta)\beta_{n_{H}},\beta_{n_{H}}\rangle)  +  S_{y{\bar{z}_1}}(\langle h_{200}(\theta)\beta_{n_{H}},\beta_{n_{H}}\rangle)\right],\\
&g_{102}^{11} =f_{102}^{11}+ \frac{3}{\mathrm{2i}\omega_0}(-2f_{002}^{11} f_{200}^{11} +f_{002}^{12} f_{110}^{11} +2f_{002}^{11} f_{101}^{13} )+\\
&\hspace{1.2cm}\frac{3}{2}\psi_1(0)[S_{yz_1}(\langle h_{002}(\theta)\beta_{n_{H}},\beta_{n_{H}}\rangle)  +  S_{yz_{2}}(\langle h_{101}(\theta)\beta_{n_{T}},\beta_{n_{H}}\rangle)],\\
&g_{111}^{13} =f_{111}^{13}+\frac{3}{\mathrm{2i}\omega_0}(-f_{101}^{13} f_{110}^{11} +f_{011}^{13} f_{110}^{12})+\frac{3}{2}\psi_2(0)[S_{yz_1}(\langle h_{011}(\theta)\beta_{n_{H}},\beta_{n_{T}}\rangle)  +\\& \hspace{1.2cm} S_{y{\bar{z}_1}}(\langle h_{101}(\theta)\beta_{n_{H}},\beta_{n_{T}}\rangle)+S_{yz_{2}}(\langle h_{110}(\theta)\beta_{n_{T}},\beta_{n_{T}}\rangle)],\\
&g_{003}^{13} = f_{003}^{13}+\frac{3}{\mathrm{2i}\omega_0}(-f_{002}^{11} f_{101}^{13} +f_{002}^{12} f_{011}^{13})+\frac{3}{2}\psi_2(0)[S_{yz_{2}}(\langle h_{002}(\theta)\beta_{n_{T}},\beta_{n_{T}}\rangle)].
\end{aligned}
\end{equation*}
\end{theorem}

\section{Turing-Hopf and Turing-Turing-Hopf type spatiotemporal patterns}

Through a large number of numerical experiments, we have observed the widespread existence of two types of spatiotemporal patterns in the vicinity of the Turing-Hopf bifurcation point. In more detail, these two types of patterns follow two different spatial distribution laws.  One of them can be portrayed as
$$\rho\left[\phi_{1}(0)e^{\mathrm{i}\omega_*\tau_* t}+\bar{\phi}_{1}(0)e^{\mathrm{-i}\omega_*\tau_* t}\right]+h\cos(\frac{n_T}{l}x),$$
where $\rho$, $h$ are constants (see {\bf Group 1} below), and the formation mechanism of it can be  completely interpreted by Turing-Hopf bifrucation.
The other can be  characterized as
$$\rho\left[\phi_{1}(0)e^{\mathrm{i}\omega_*\tau_* t}+\bar{\phi}_{1}(0)e^{\mathrm{-i}\omega_*\tau_* t}\right]+h_1\cos(\frac{n_T}{l}x)+h_2\cos(\frac{n_I}{l}x)$$
where $\rho$, $h_1$, $h_2$ are constants and $h_1,\;h_2$ are zero or non-zero at the same time (see {\bf Group 2} below), and we believe that its formation is due to the further impact of the Turing-Turing-Hopf bifurcation,  although the parameters are valued near Turing-Hopf singularity.

For detailed explanations and intuitive understanding, please refer to the following two groups of numerical experiments.
\subsection{Turing-Hopf type spatiotemporal patterns}
{\bf Group 1.} Taking $d_1 = 2.53,$ $d_2 = 9.87,$ $a = 0.90,$ $b = 0.001,$ and $l = 5.45$, which are satisfy the condition $({{\mathbf{A}}}6^{''})$. The unique coexistence equilibrium point now is
$(u_0, v_0) = (0.1082, 0.1082)$. From the  eigenvalue analysis, we can get $r_1^T = 1.1377$, $r_2^T = 1.2639$, $r_3^T = 0.0297$, $r_n^T <0\,(n\geq 4)$, $\tau_0^{(0)} = 0.7937$, $\tau_1^{(0)} = 1.0805$, $\tau_3^{(0)} = 4.5350$, $\omega_0 = 1.0087$, $\omega_1 = 1.0112$, $\omega_3 = 0.4070$.
The important informations can be summed up as
$$n_T = 2, \quad n_H=0, \quad r_* = 1.2639,  \quad \tau_*= 0.7937, \quad \omega_* = 1.0087.$$
The bifurcation diagram in $r-\tau$ plane has been shown in Figure \ref{figTH}.  The intersection of the colored Hopf curve and the dotted Turing curve, which is marked by TH1, TH2, TH3, are the Turing-Hopf bifurcation points.
A  stable region (i.e., $r>r_*,\tau<\tau_0^{0}(r)$) of the equilibrium $(u_0,v_0)$ is painted pale green and  labeled as "Stable Region" in Figure \ref{figTH}.
\begin{figure}[htb]
	\centering
	\subfigure[]{\begin{minipage}{0.49\linewidth}
	\centering\includegraphics[height = 0.78\linewidth,width=1.05\linewidth]{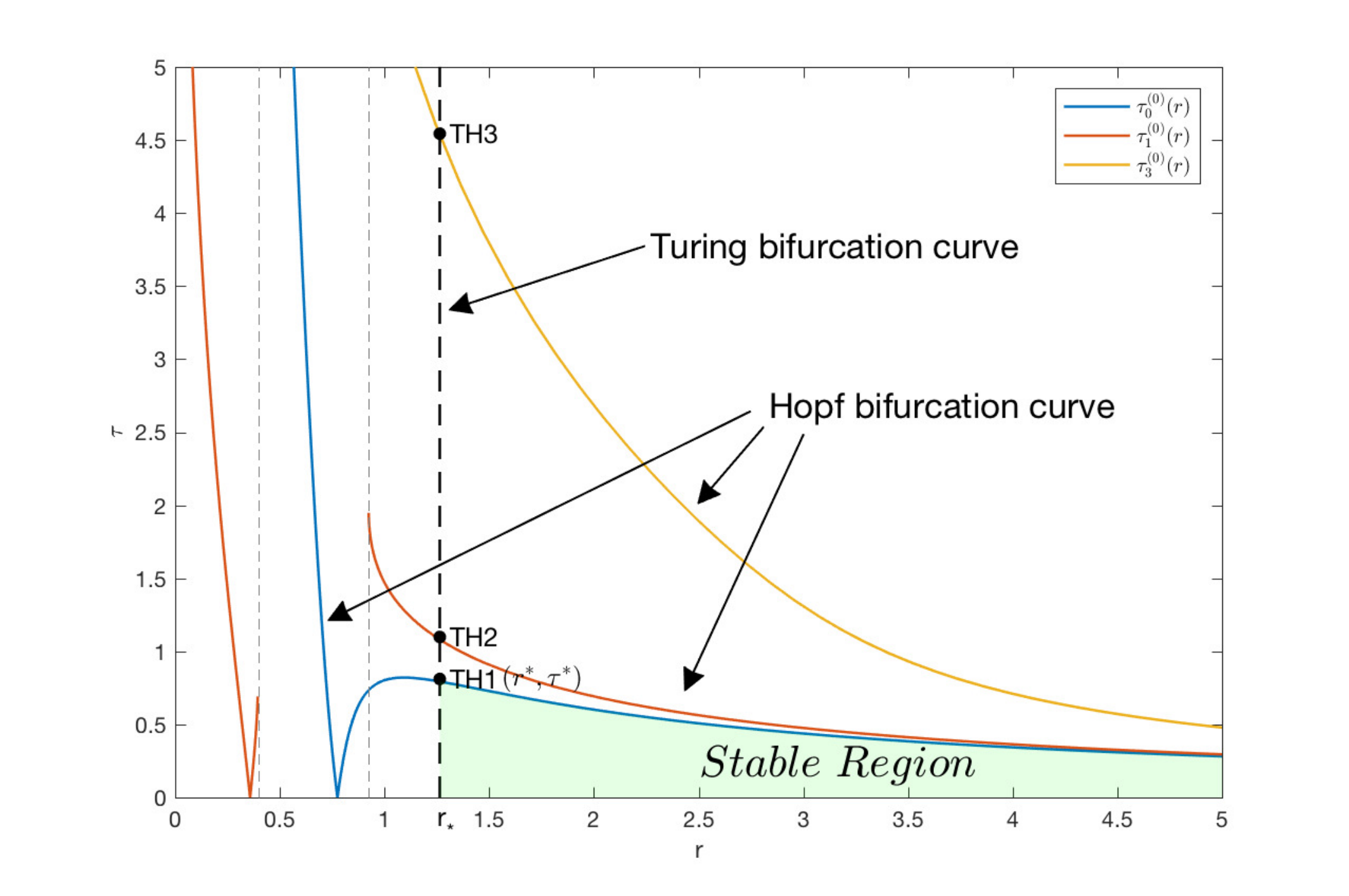}\label{figTH}
	\end{minipage}}
     \subfigure[]{\begin{minipage}{0.49\linewidth}
	\centering\includegraphics[scale=0.35]{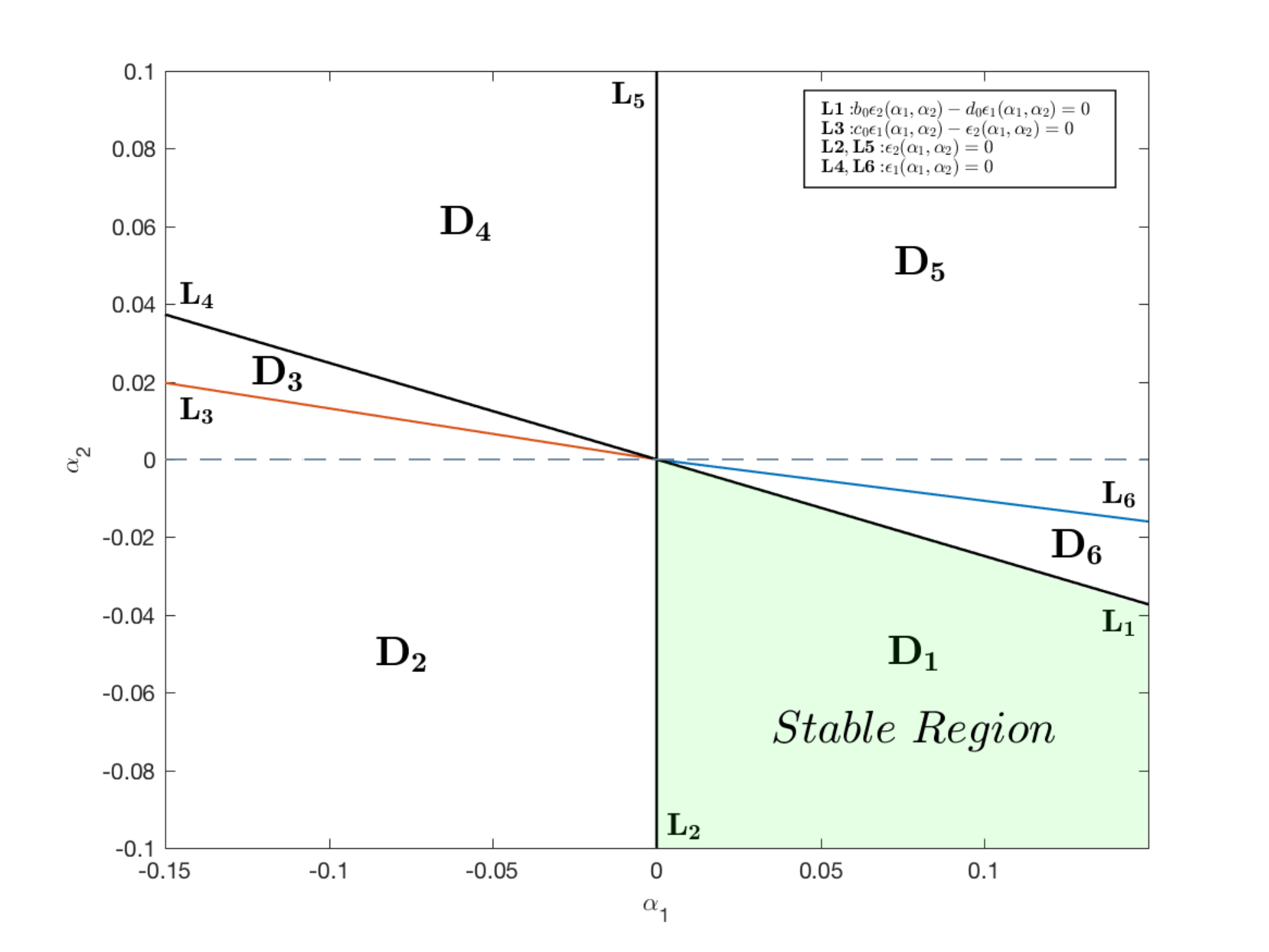}\label{Detail1}
    \end{minipage}}

\subfigure[]{\begin{minipage}{0.85\linewidth}\centering\includegraphics[height=0.45\linewidth,width=1\linewidth]{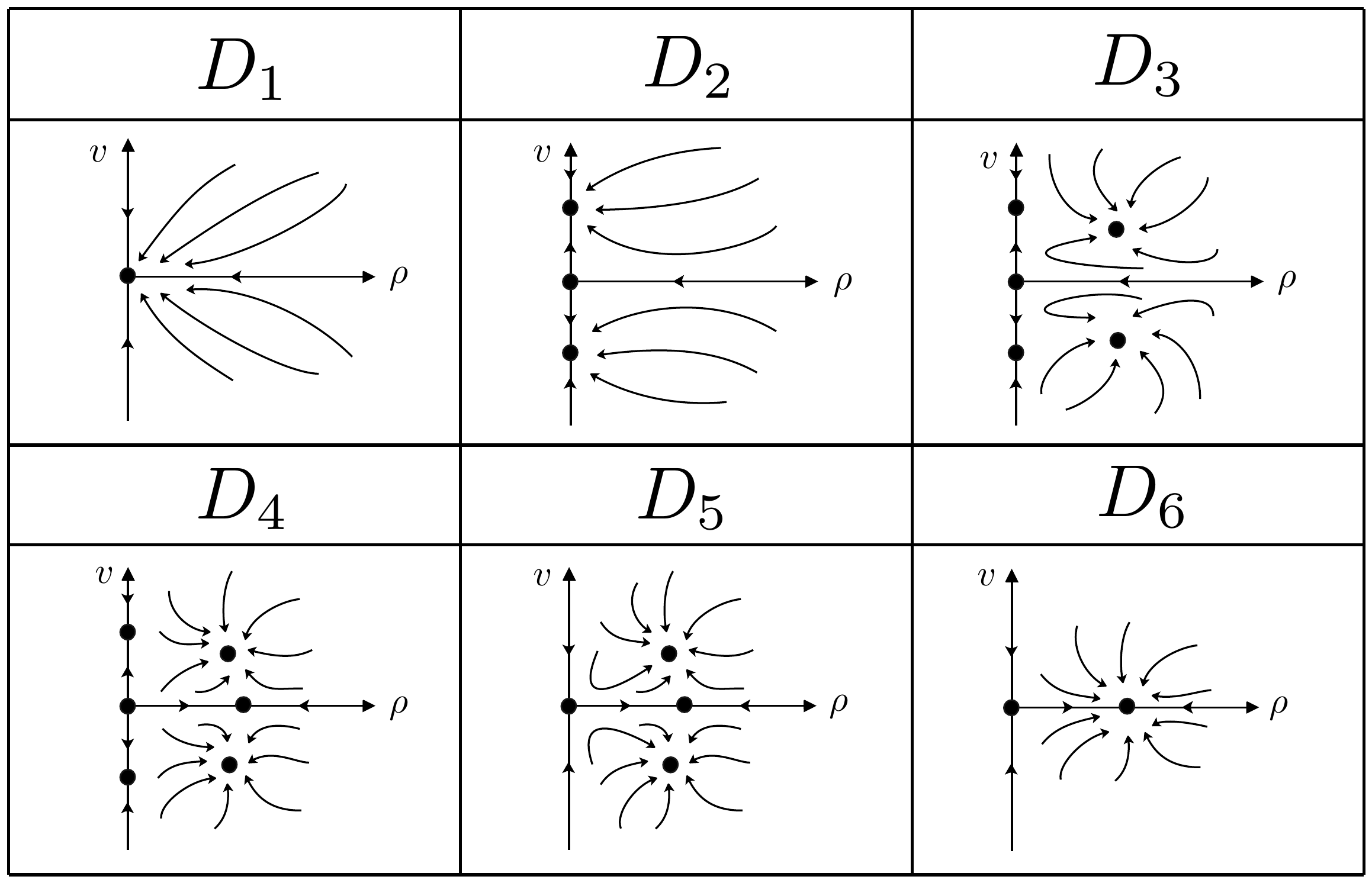}\label{phase}\end{minipage}}
\caption{(a) Bifurcation sets in $r-\tau$ plane. (b) Detailed bifurcation diagram near $(r_*,\tau_*)$ in $\alpha_1-\alpha_2$ plane. (c) Phase portraits in $D_1-D_6$}
\end{figure}

Using the  Theorem \ref{theNF}, we obtain the second and third order coefficients in the normal forms \eqref{eqNF} for system \eqref{eqA} near $(r_*,\tau_*)$:
\begin{equation*}
\begin{aligned}
&f_{\alpha_1z_1}^{11} = 0.5688 \!+\! 1.7938\mathrm{i}, \quad f_{\alpha_2z_1}^{11} = 2.2867 \!+\! 1.2923\mathrm{i},
\quad f_{\alpha_1z_2}^{13}= -0.4269,\quad  f_{\alpha_2z_2}^{13} = 0,\\
&g_{210}^{11} = -4.1739 \!-\!30.7512\mathrm{i},\; g_{102}^{11} = 1.3212 \!+\! 2.7793\mathrm{i},
\; g_{111}^{13} = 5.4813, \; g_{003}^{13} = -10.5144.
\end{aligned}
\end{equation*}
After the cylindrical coordinate transformation
\begin{equation}\label{eqcct}
\begin{aligned}
z_{1}=\mathcal{R}\cos\theta+\mathrm{i}\mathcal{R}\sin\theta,\quad
{\bar{z}_1}=\mathcal{R}\cos\theta-\mathrm{i}\mathcal{R}\sin\theta,\quad
z_{2}=\mathcal{V},
\end{aligned}
\end{equation}
and a  re-scaling ${\rho}=\sqrt{\frac{|\mathrm{Re}(g_{210}^{11})|}{6}}\mathcal{R}$, ${v}=\sqrt{\frac{|g_{003}^{13}|}{6}}\mathcal{V}$,
we get an equivalent planner system
\begin{equation}\label{eqrv2}
\begin{aligned}
&\frac{\mathrm{d}\rho}{\mathrm{d}{t}}=-\rho\left[\epsilon_{1}(\alpha_1,\alpha_2) + \rho^{2}+b_0 v^{2}\right],\\
&\frac{\mathrm{d}v}{\mathrm{d}{t}}=-v\left[\epsilon_{2}(\alpha_1,\alpha_2) + c_0 \rho^{2} + d_0 v^{2}\right],
\end{aligned}
\end{equation}
with  $(\alpha_1,\alpha_2)=(r-r_*,\tau-\tau_*)$ and $\epsilon_{1}(\alpha)=~-0.2273\alpha_1-1.3130\alpha_2$, $\epsilon_{2}(\alpha)=0.0839\alpha_1$, $b_0=18.8564$, $c_0=-0.6218$, $d_0=-1$, $d_0-b_0c_0= 10.7252>0$.
We claim that the  Case \uppercase\expandafter{\romannumeral4}a in \cite[\S 7.5]{Phillp1988} occurs, the detailed bifurcation diagram near $(r_*,\tau_*)$ in $(\alpha_1,\alpha_2)$ plane is given in Figure \ref{Detail1} and it is a larger image of Figure \ref{figTH} near the critical value. The detailed dynamics in $D_1-D_6$ can be summed up as the following proposition. For a more intuitive understanding, please refer to Figure \ref{phase}.
\begin{proposition}
When $d_1 = 2.53,$ $d_2 = 9.87,$ $a = 0.90,$ $b = 0.001,$ and $l = 5.45$, the system \eqref{eqA} undergoes a Turing-Hopf bifurcation at $(r_*,\tau_*)$ with $n_T=2$ and $n_H=0$. The parameter plane near the critical value is divided into six regions (see Figure \ref{Detail1}). The dynamics of each region $D_1-D_6$ are:
\begin{itemize}
	\item In $D_1$, the constant steady state $(u_0,v_0)$ is locally asymptotically stable (see Figure \ref{figD1}), but it lost its stability when the parameters passing the Turing bifurcation line $L_2$.
	\item In $D_2$, two stable non-constant steady states are coexistence (see Figure \ref{figD2_1}-Figure \ref{figD2_2}, the spatial distribution follows to the function: $h\cos(\frac{n_T}{l}x)$). Moreover, they undergo a Turing bifurcation at $L_3$, and lost their stability in $D_3$.
	\item Two stable spatially non-homogeneous periodic orbits are generated in $D_3$ (see Figure \ref{figD3_1}-Figure \ref{figD3_2}, the spatial distribution follows to the function: $h\cos(\frac{n_T}{l}x)$).
	\item In $D_4$, a unstable spatially homogeneous periodic orbits bifurcating form $(u_0,v_0)$, since a Hopf bifurcation occurs at $L_4$.
	\item $L_5$ is another Turing bifurcation curve of the constant steady state $(u_0,v_0)$, it eliminate the two non-constant steady states in $D_5$.
	\item In $D_6$, the spatially non-homogeneous periodic orbits becomes stable through the Turing curve $L_6$ and two spatially non-homogeneous periodic orbits are disappeared (see Figure \ref{figD6}). In addition, it lost its stability in $D_1$ since the existence of the Hopf bifurcation line $L_1$.
\end{itemize}

\end{proposition}
\begin{figure}[htbp]
	\centering
\subfigure[$u(x,t)$]{\begin{minipage}{0.25\linewidth}
		\centering\includegraphics[scale=0.2]{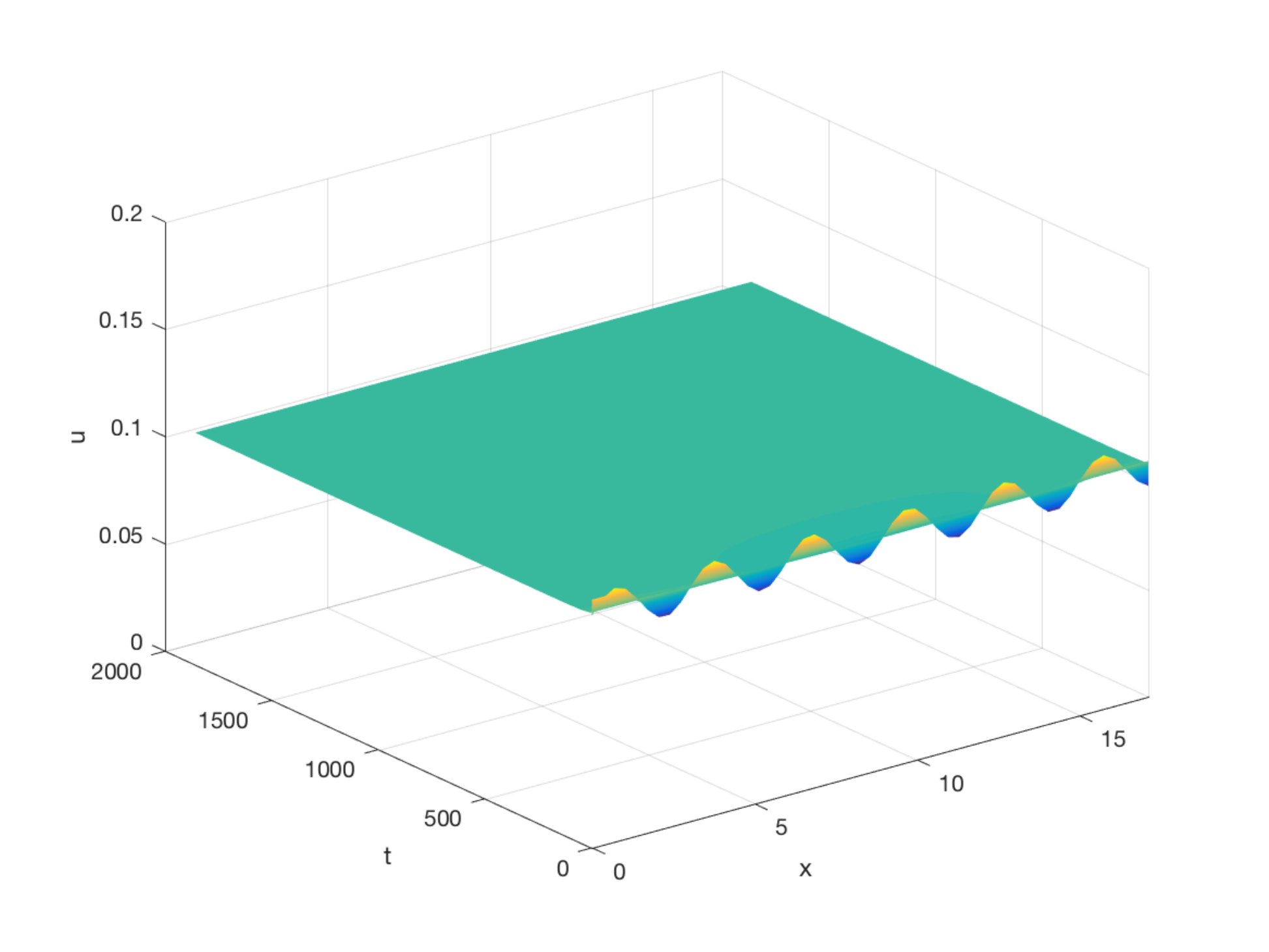}
\end{minipage}}
\subfigure[prey pattern]{\begin{minipage}{0.23\linewidth}
		\centering\includegraphics[height=0.93\linewidth,width=1.1\linewidth]{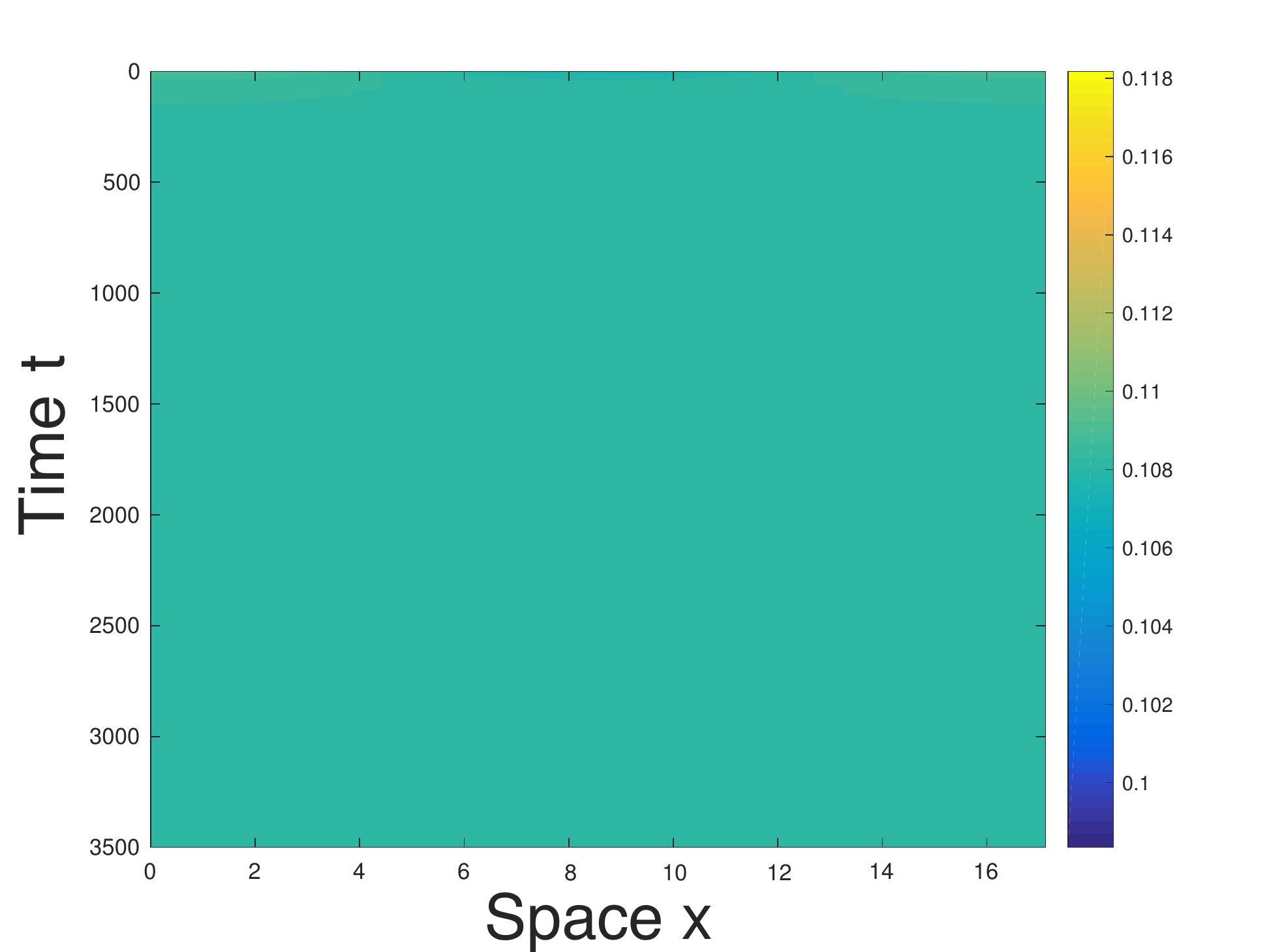}
\end{minipage}}
\subfigure[$v(x,t)$]{\begin{minipage}{0.25\linewidth}
		\centering\includegraphics[scale=0.2]{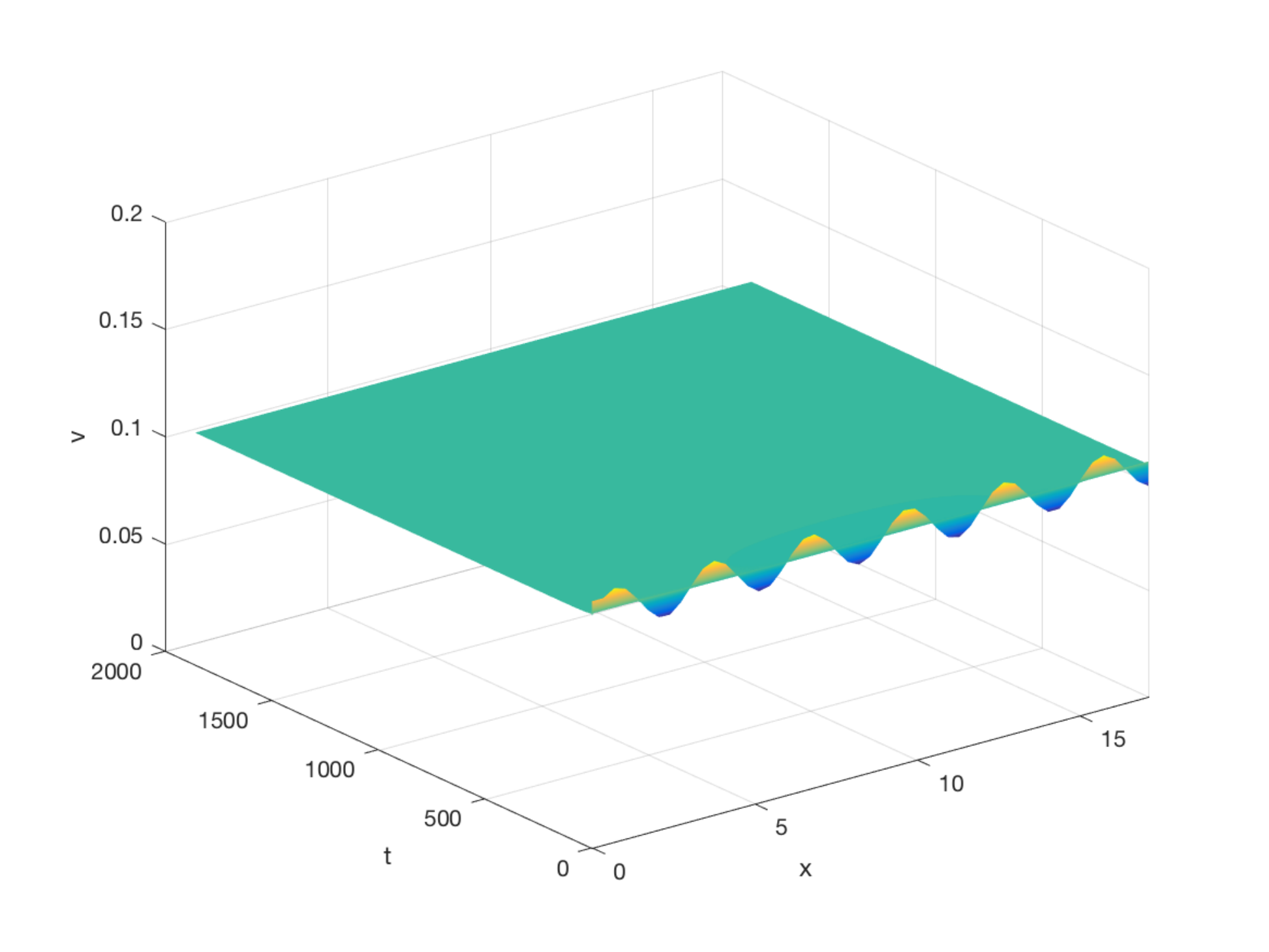}
\end{minipage}}
\subfigure[predator pattern]{\begin{minipage}{0.23\linewidth}
		\centering\includegraphics[height=0.93\linewidth,width=1.1\linewidth]{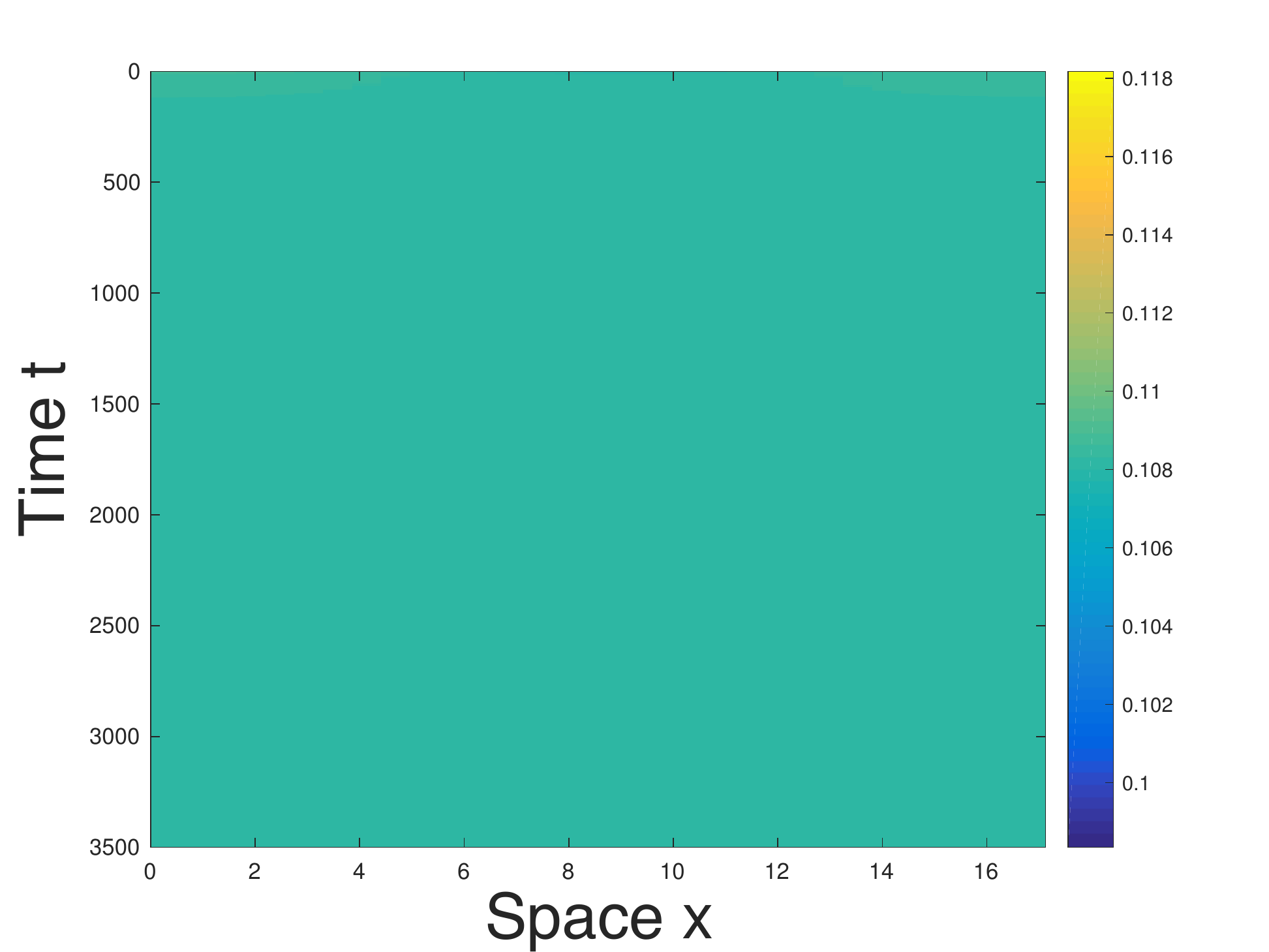}
\end{minipage}}
\caption{Constant steady state in $D_1$, with $(\alpha_1,\alpha_2)=(0.05,-0.05)$  and initial functions are $(u_0+0.01\sin 2x,u_0+0.01\sin 2x)$. }\label{figD1}
\end{figure}
\begin{figure}[htbp]
	\centering
	\subfigure[$u(x,t)$]{\begin{minipage}{0.25\linewidth}
			\centering\includegraphics[scale=0.2]{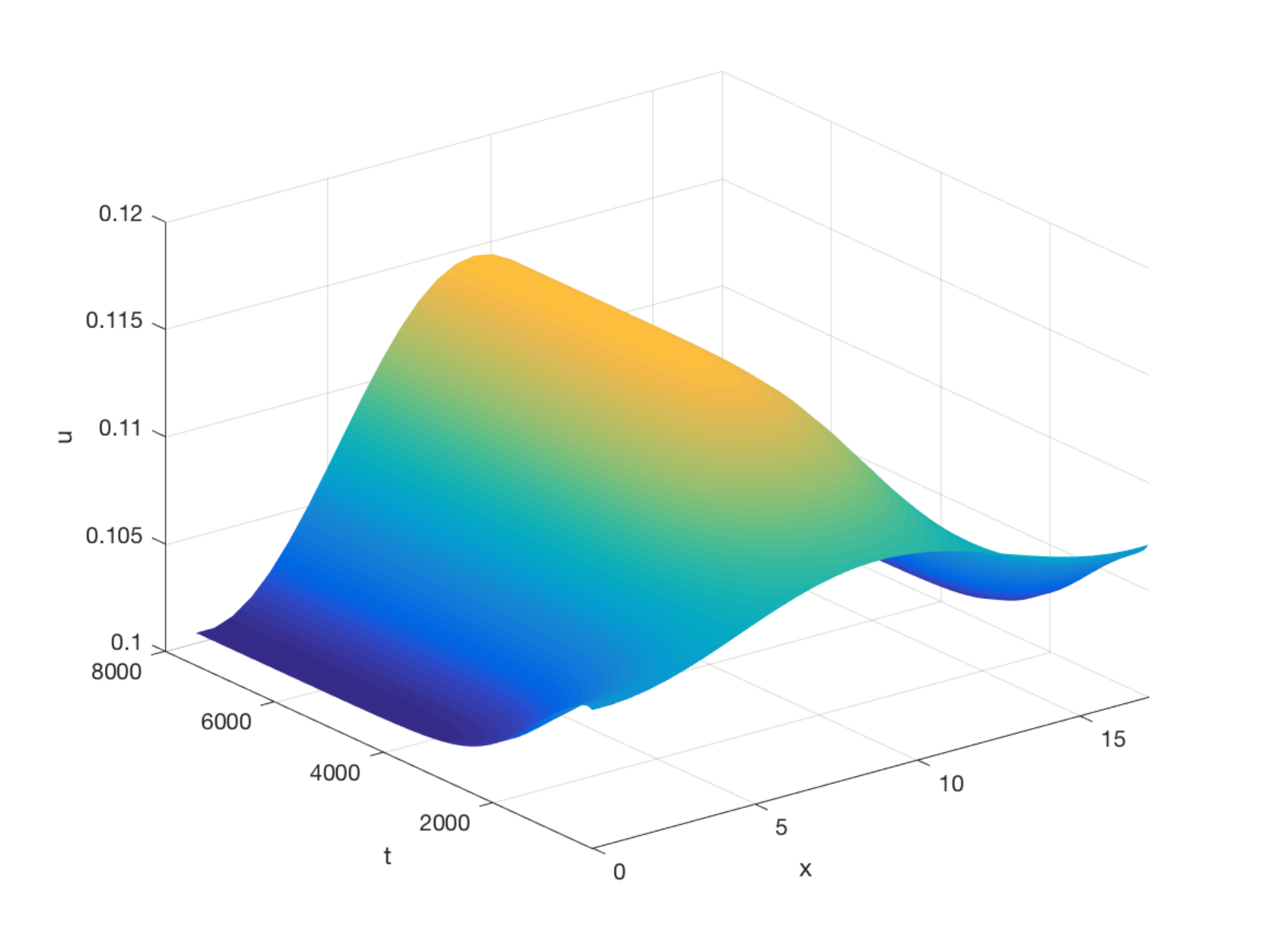}
	\end{minipage}}
    \subfigure[prey pattern]{\begin{minipage}{0.23\linewidth}
	\centering\includegraphics[height=0.93\linewidth,width=1.1\linewidth]{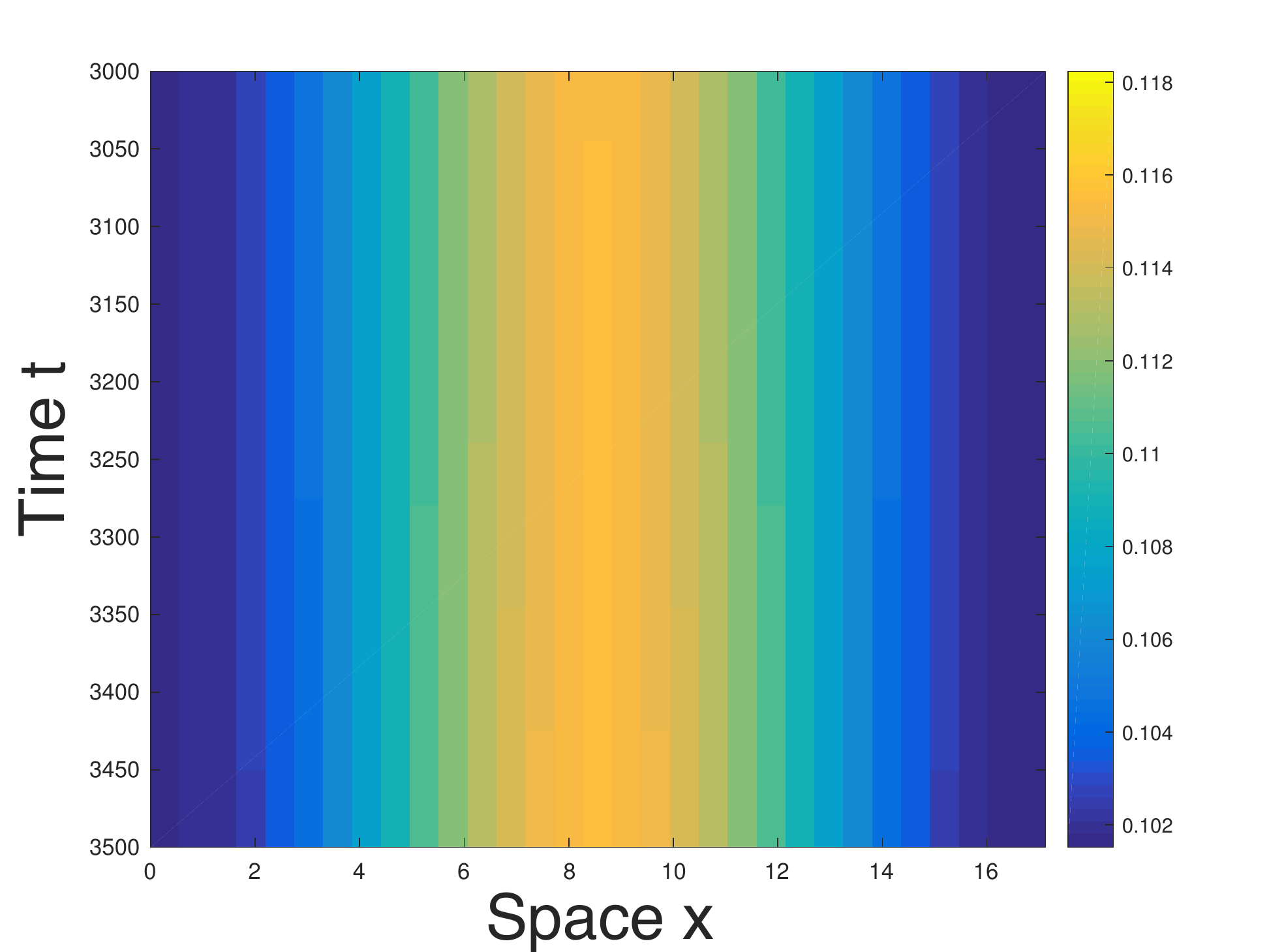}
    \end{minipage}}
    \subfigure[$v(x,t)$]{\begin{minipage}{0.25\linewidth}
		\centering\includegraphics[scale=0.2]{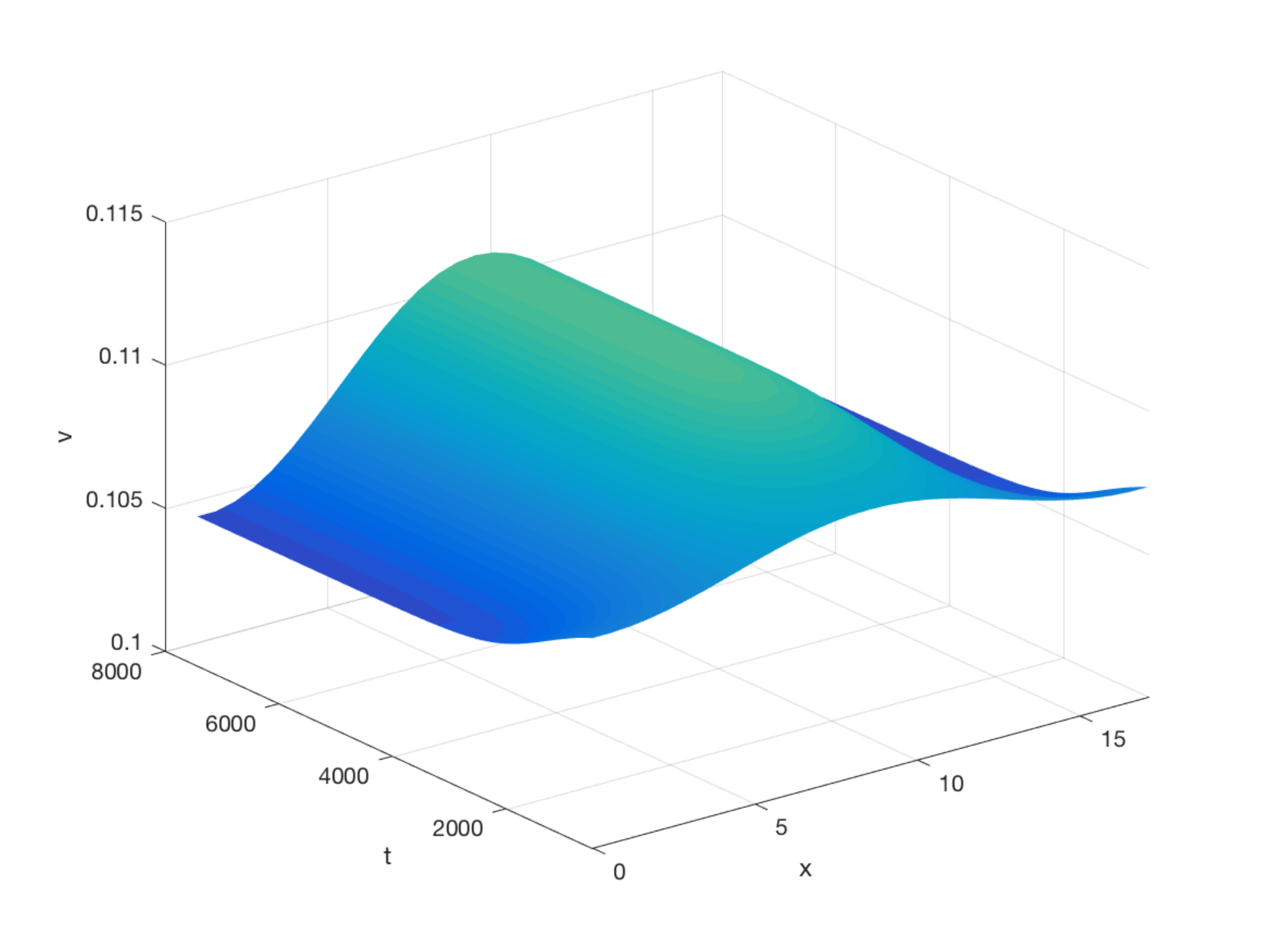}
    \end{minipage}}
    \subfigure[predator pattern]{\begin{minipage}{0.23\linewidth}
    		\centering\includegraphics[height=0.93\linewidth,width=1.1\linewidth]{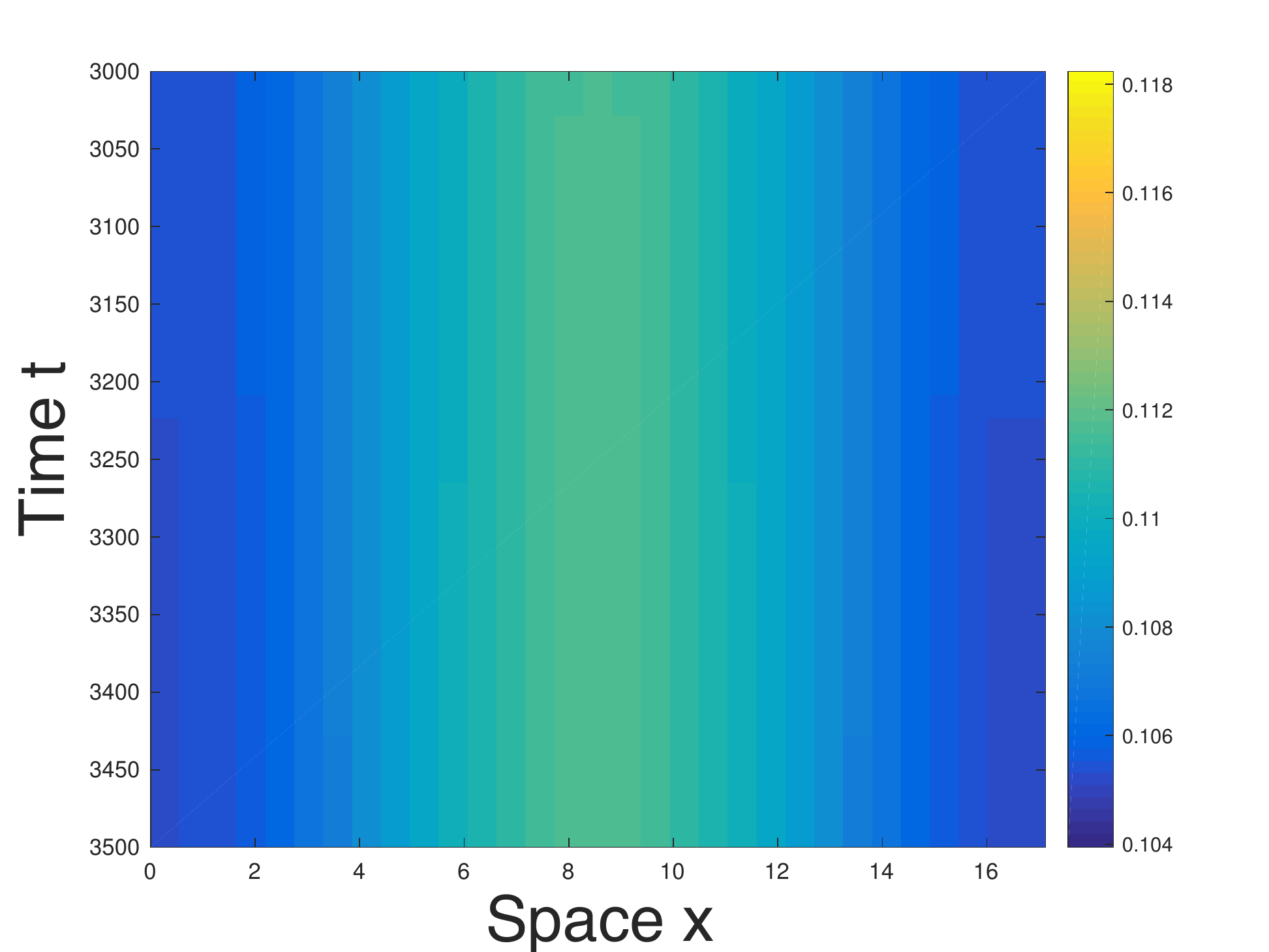}
    \end{minipage}}
\caption{Non-constant steady state in $D_2$, with $(\alpha_1,\alpha_2)=(-0.05,-0.05)$ and initial functions are $(u_0+0.01\sin 0.1x,u_0+0.01\sin 0.1x)$.}\label{figD2_1}
\end{figure}
\begin{figure}[htbp]
    \centering
    \subfigure[$u(x,t)$]{\begin{minipage}{0.25\linewidth}
    		\centering\includegraphics[scale=0.2]{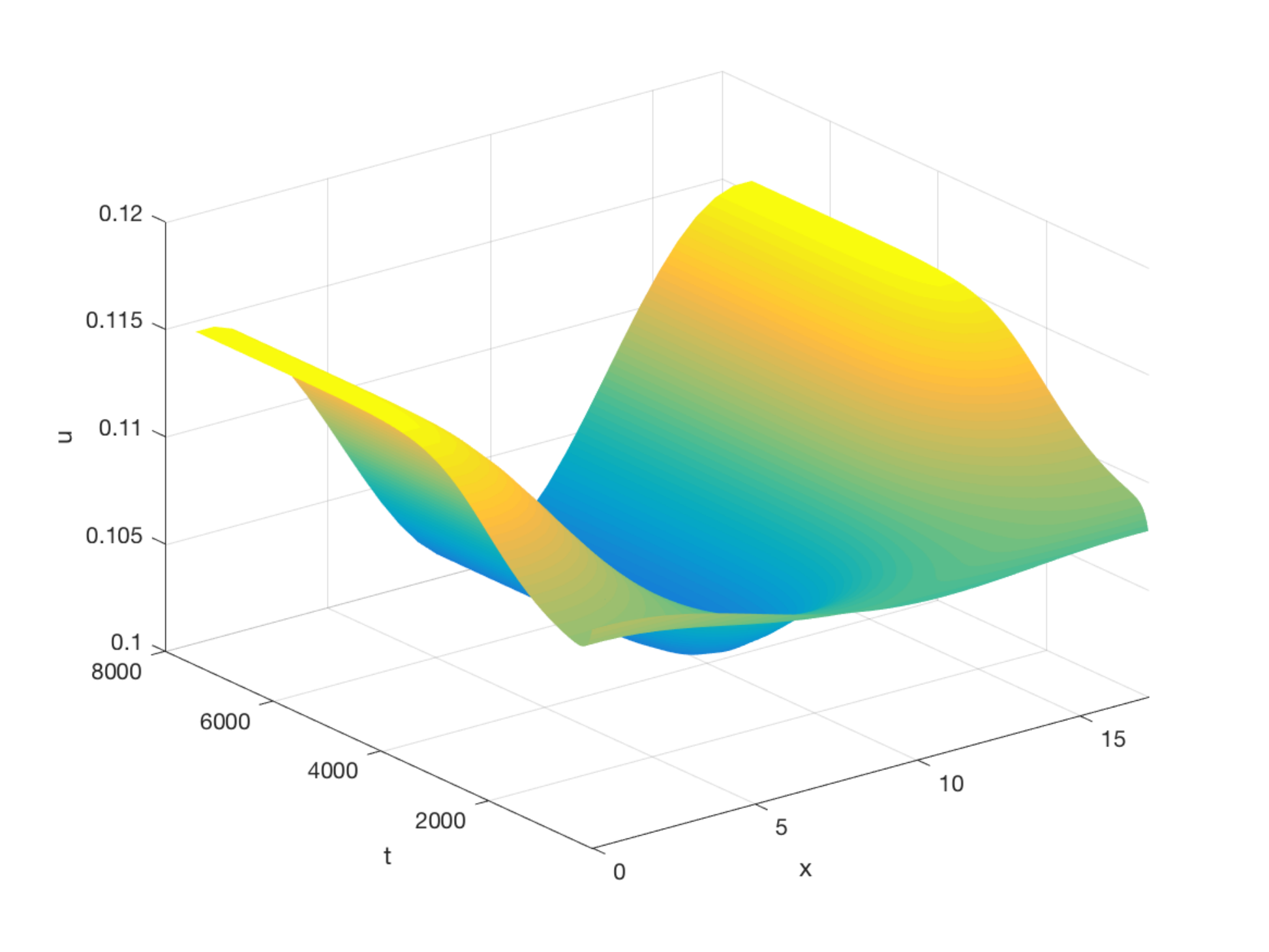}
    \end{minipage}}
    \subfigure[prey pattern]{\begin{minipage}{0.23\linewidth}
    		\centering\includegraphics[height=0.93\linewidth,width=1.1\linewidth]{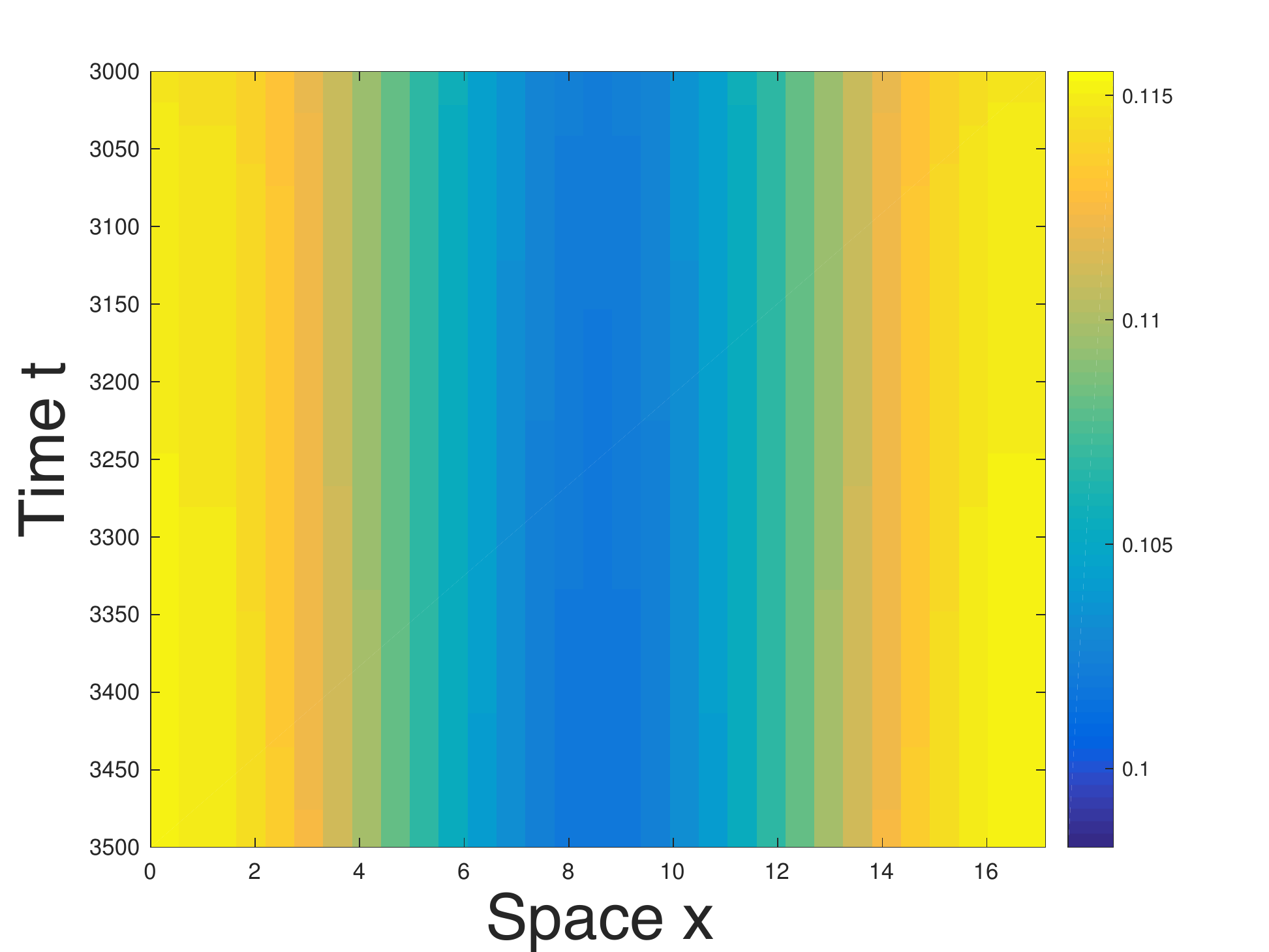}
    \end{minipage}}
    \subfigure[$v(x,t)$]{\begin{minipage}{0.25\linewidth}
    		\centering\includegraphics[scale=0.2]{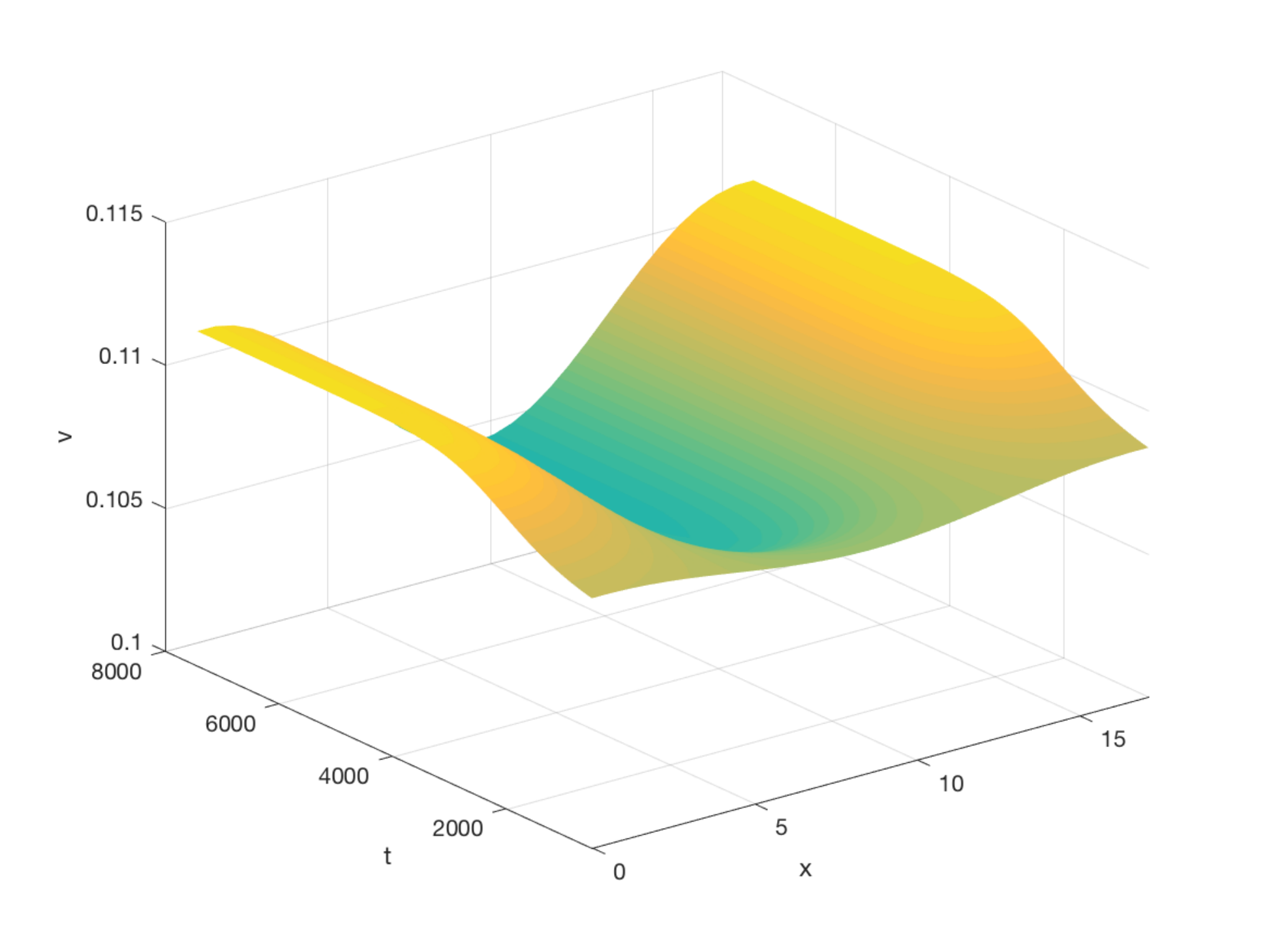}
    \end{minipage}}
    \subfigure[predator pattern]{\begin{minipage}{0.23\linewidth}
    		\centering\includegraphics[height=0.93\linewidth,width=1.1\linewidth]{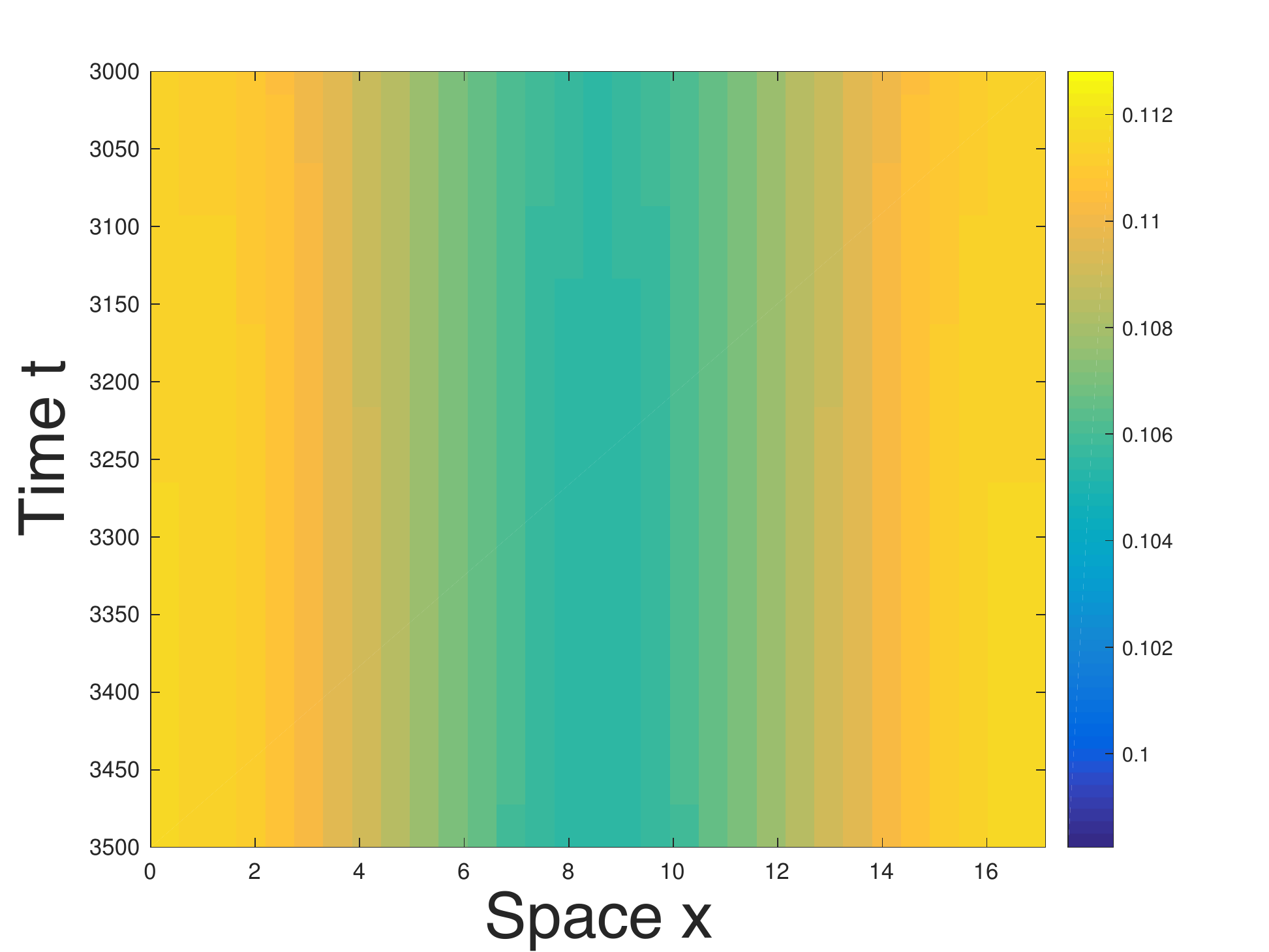}
    \end{minipage}}
\caption{ Non-constant steady state in $D_2$, with $(\alpha_1,\alpha_2)=(-0.05,-0.05)$ and initial functions are $(u_0-0.01\sin 0.1x,u_0-0.01\sin 0.1x)$. }\label{figD2_2}
\end{figure}
\begin{figure}[htbp]
	\subfigure[$u(x,t)$]{\begin{minipage}{0.25\linewidth}
			\centering\includegraphics[scale=0.2]{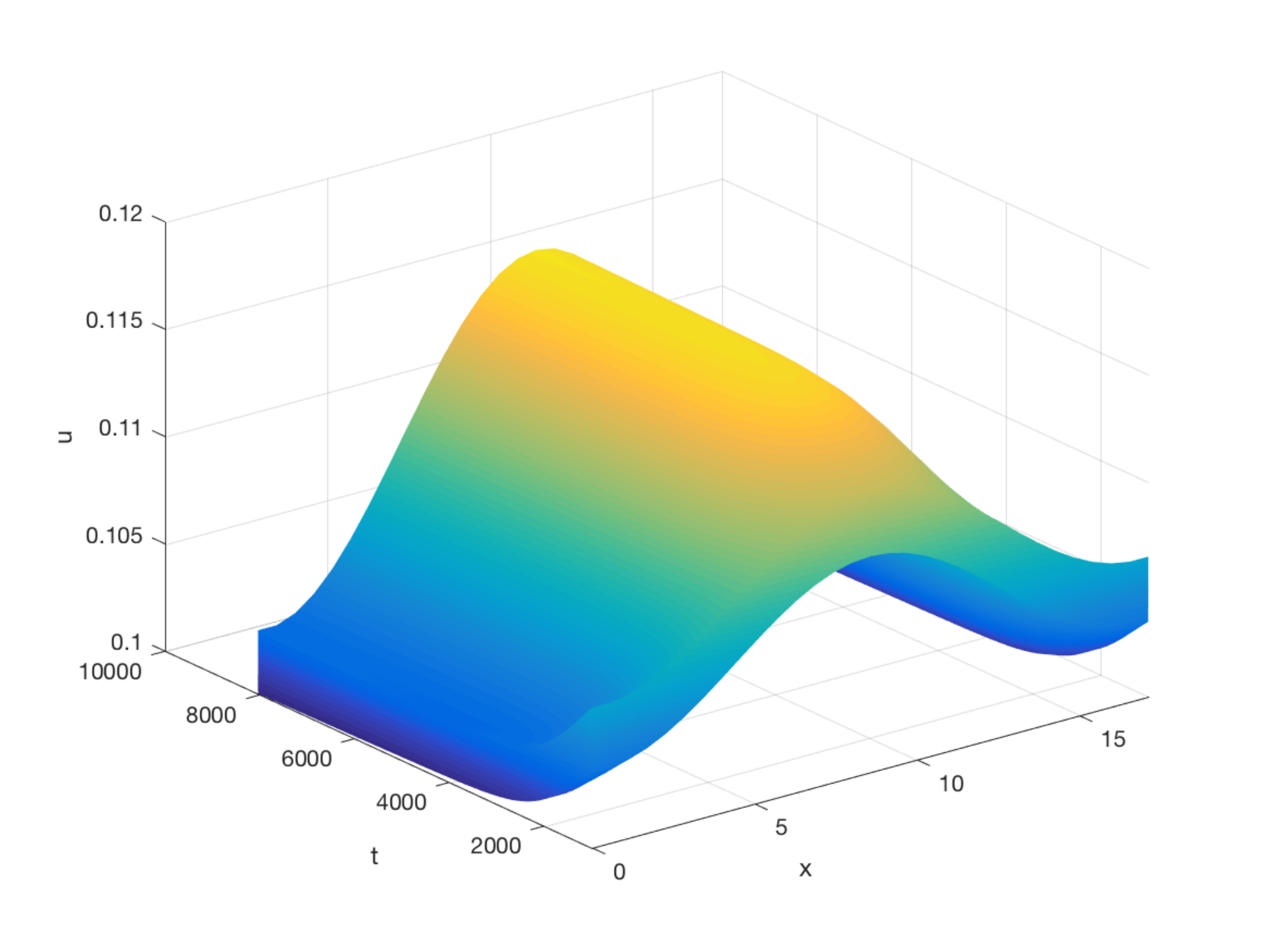}
	\end{minipage}}
	\subfigure[target pattern]{\begin{minipage}{0.23\linewidth}
			\centering\includegraphics[height=0.93\linewidth,width=1.1\linewidth]{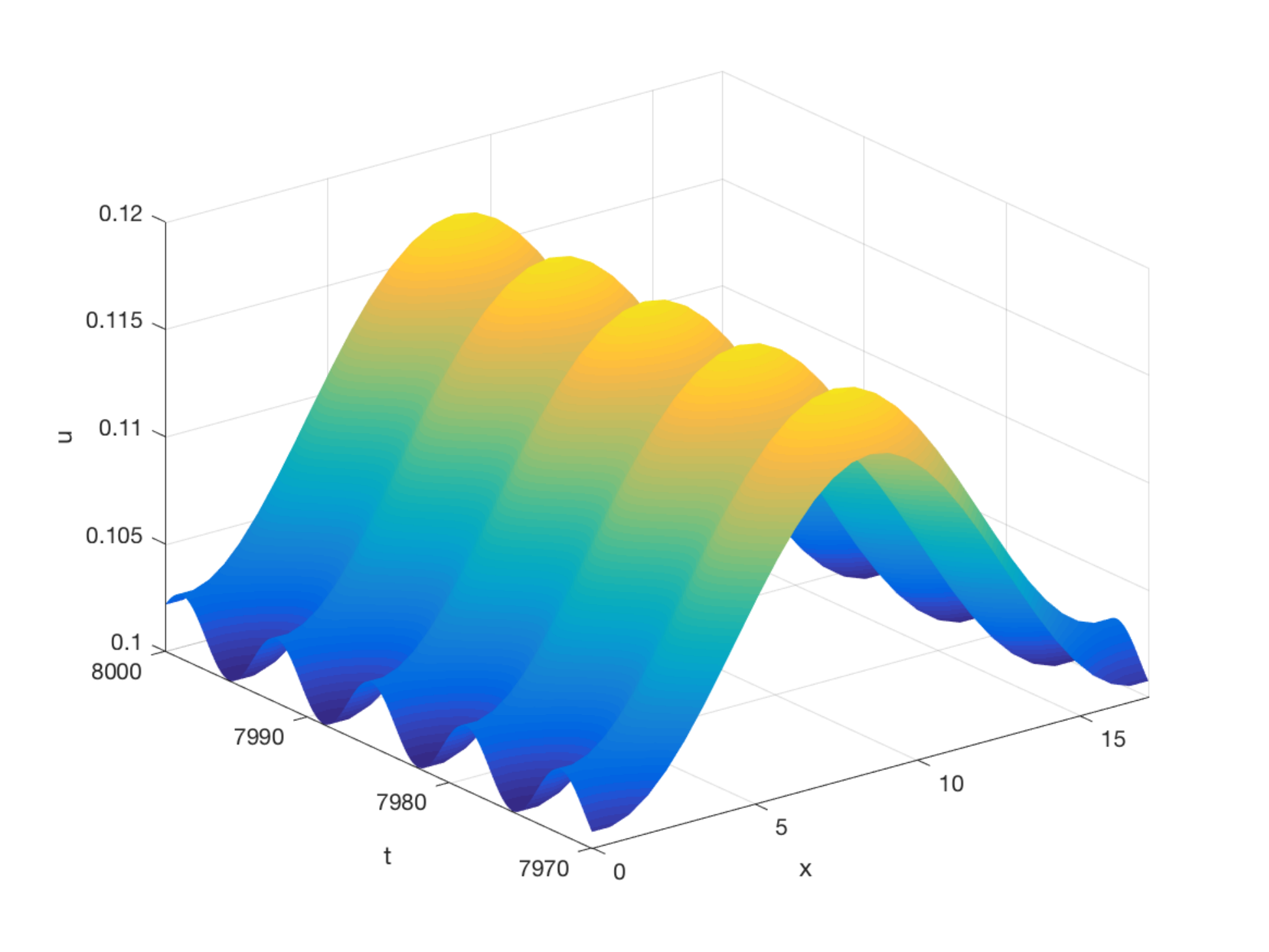}
	\end{minipage}}
	\subfigure[prey pattern]{\begin{minipage}{0.25\linewidth}
		\centering\includegraphics[scale=0.2]{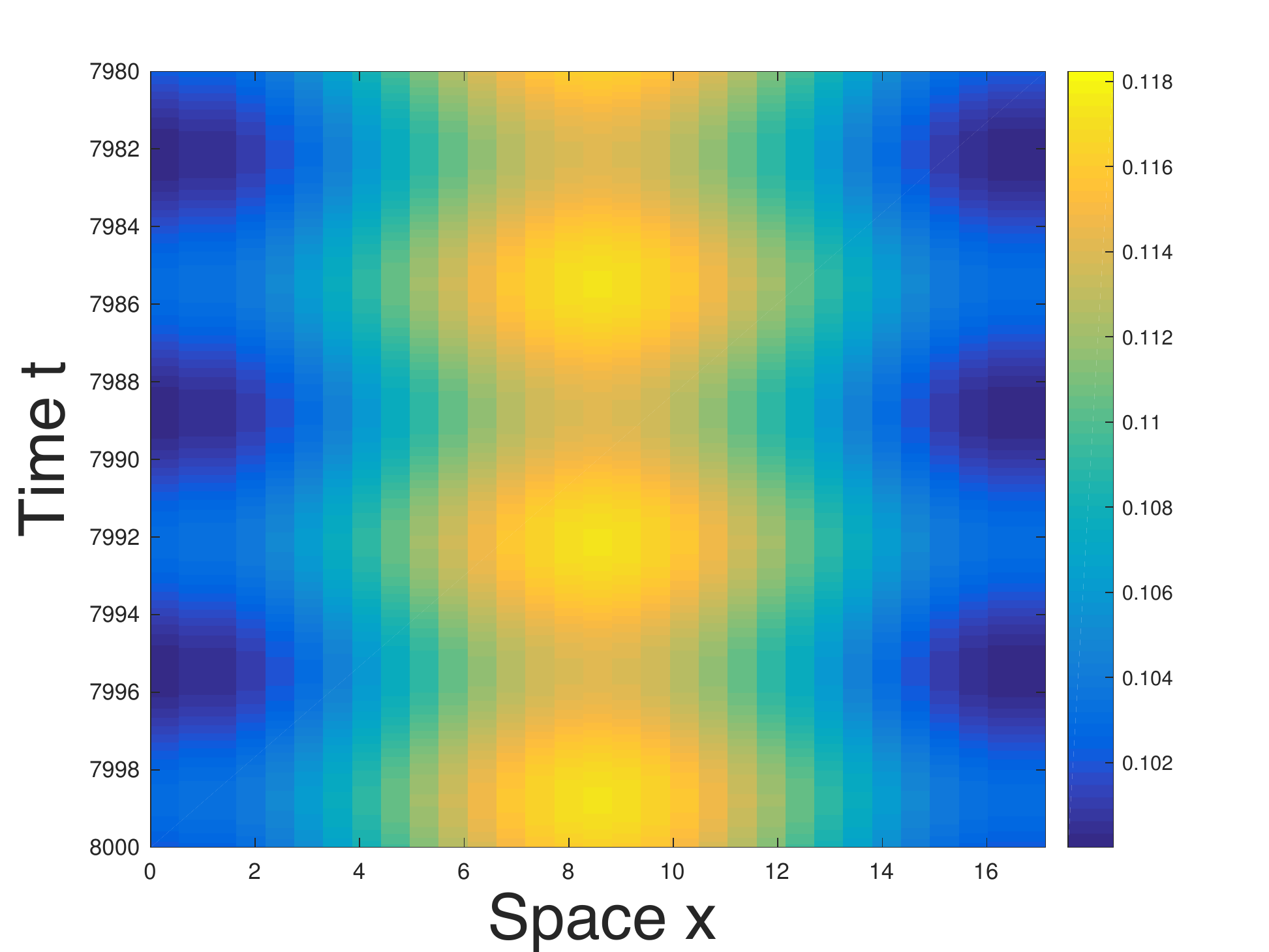}
    \end{minipage}}
	\subfigure[$u(x,7998)$]{\begin{minipage}{0.23\linewidth}
			\centering\includegraphics[height=0.93\linewidth,width=1.1\linewidth2]{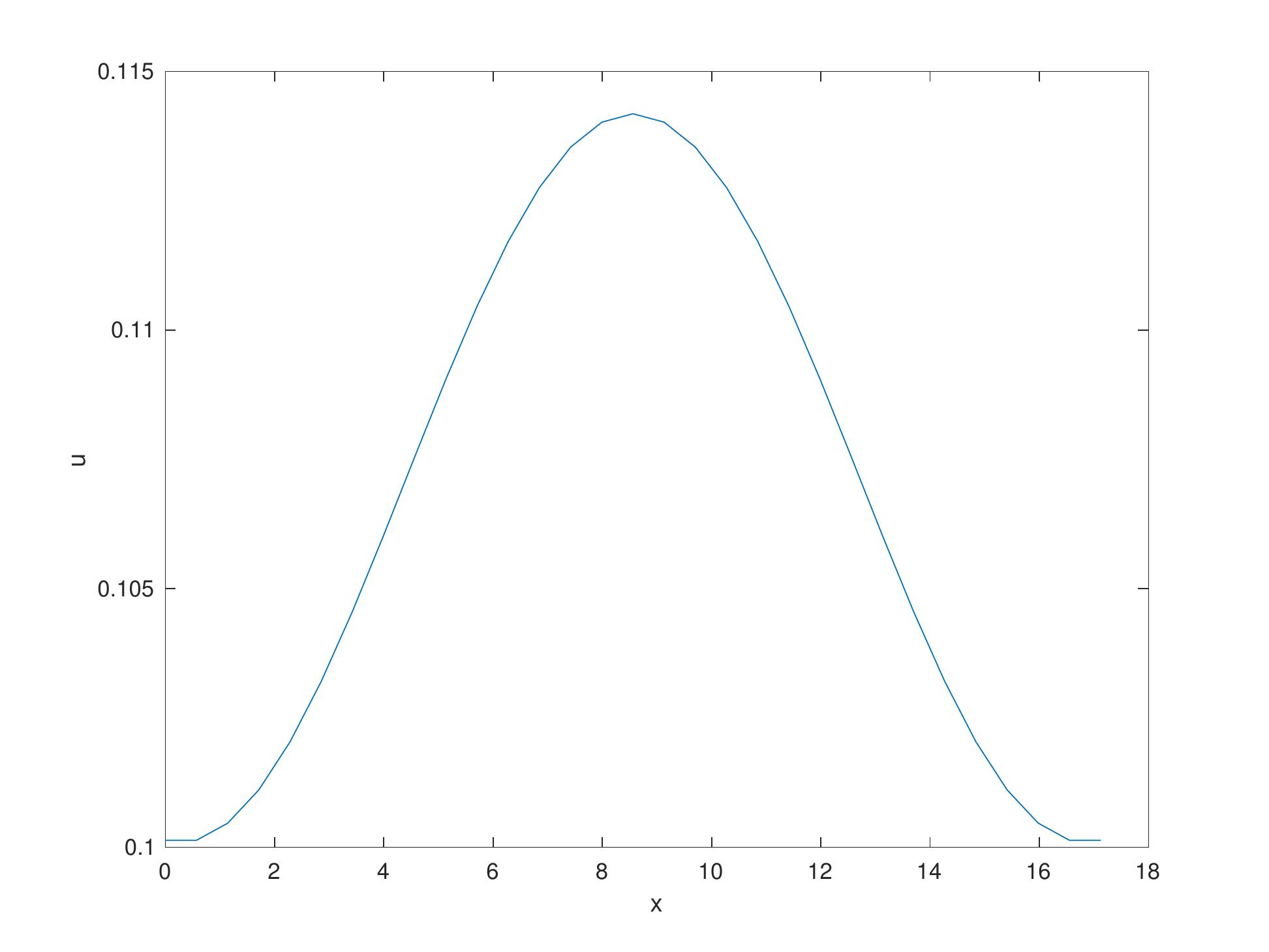}
	\end{minipage}}
	
	\subfigure[$v(x,t)$]{\begin{minipage}{0.25\linewidth}
			\centering\includegraphics[scale=0.2]{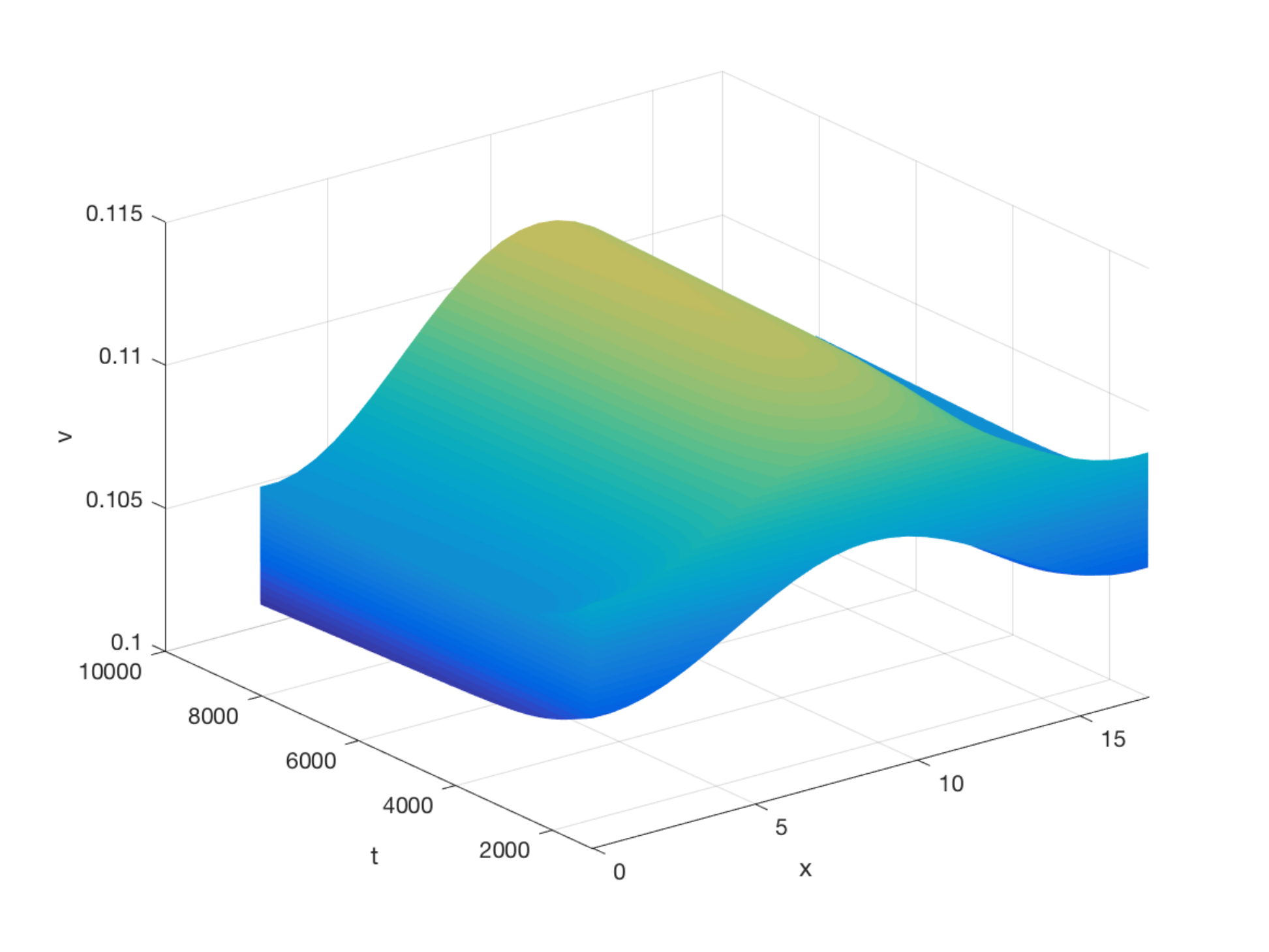}
	\end{minipage}}
	\subfigure[target pattern]{\begin{minipage}{0.23\linewidth}
			\centering\includegraphics[height=0.93\linewidth,width=1.1\linewidth2]{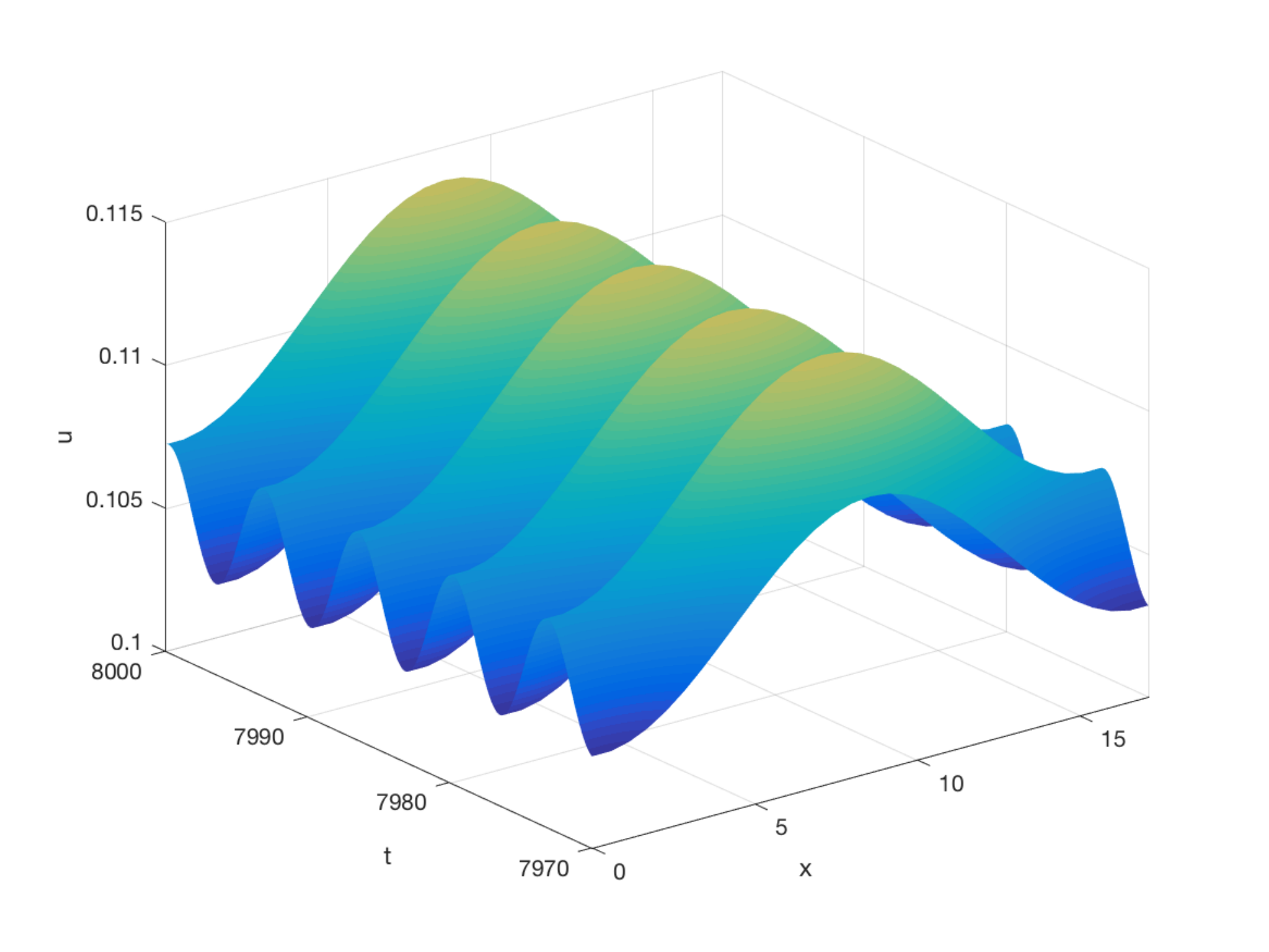}
	\end{minipage}}
	\subfigure[predator pattern]{\begin{minipage}{0.25\linewidth}
			\centering\includegraphics[scale=0.2]{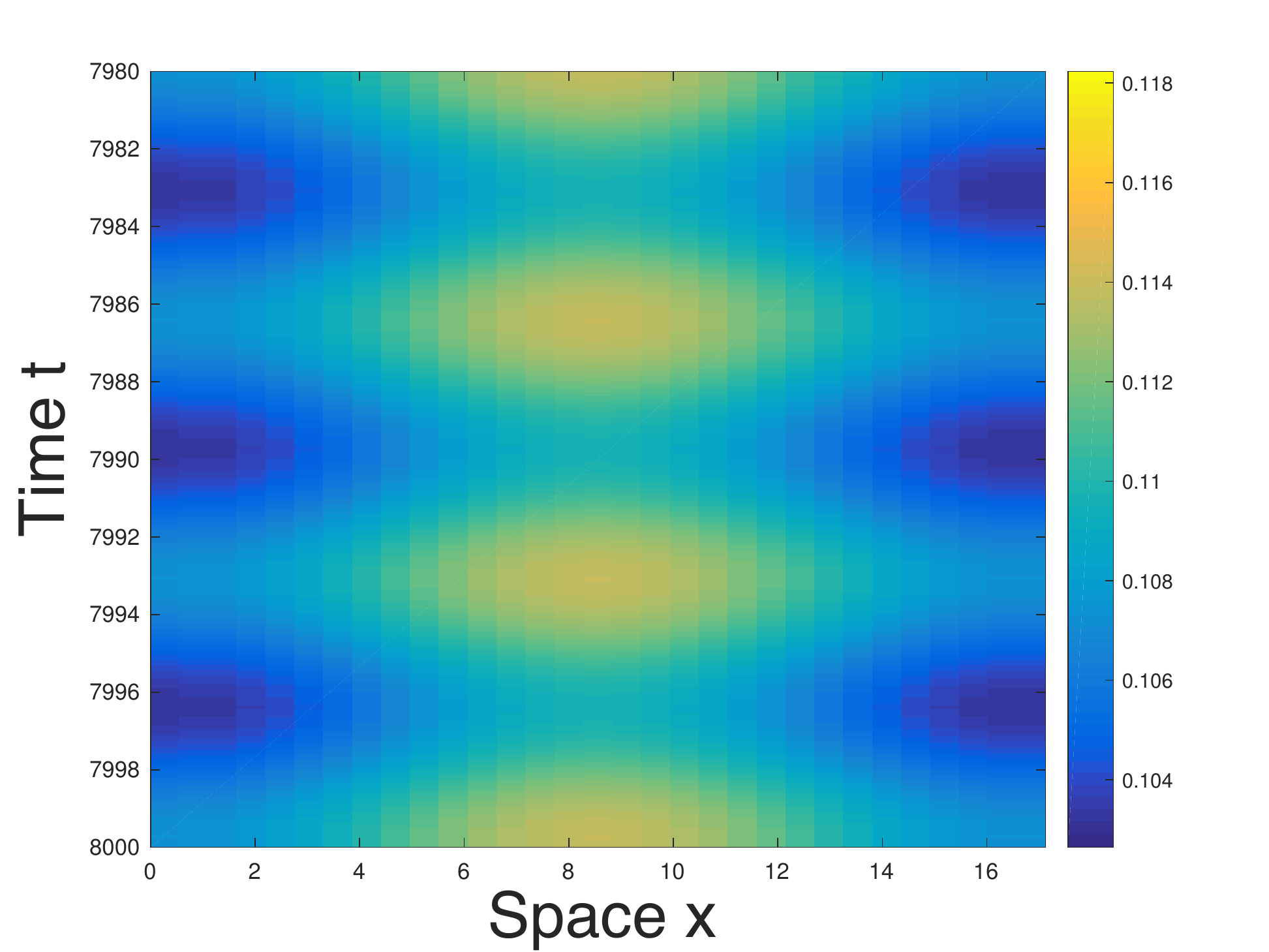}
	\end{minipage}}
	\subfigure[$v(x,7998)$]{\begin{minipage}{0.23\linewidth}
			\centering\includegraphics[height=0.93\linewidth,width=1.1\linewidth]{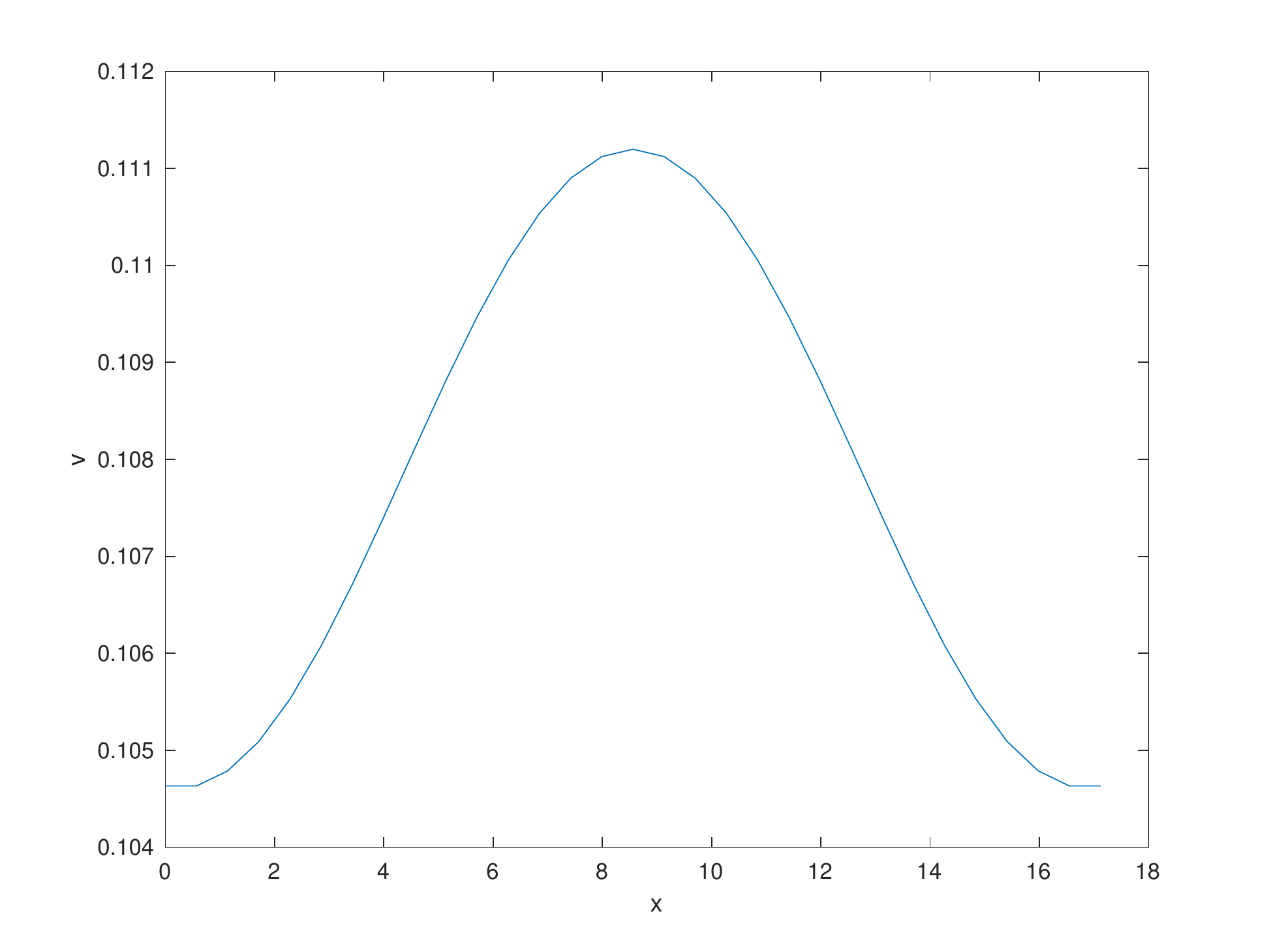}
	\end{minipage}}
\caption{Spatially non-homogeneous periodic solution in $D_3$, with $(\alpha_1,\alpha_2)=(-0.05,0.0105)$ and initial functions are $(u_0+0.01\sin 0.1x,u_0+0.01\sin 0.1x)$.}\label{figD3_1}
\end{figure}

\begin{figure}[htbp]
	\subfigure[$u(x,t)$]{\begin{minipage}{0.25\linewidth}
			\centering\includegraphics[scale=0.2]{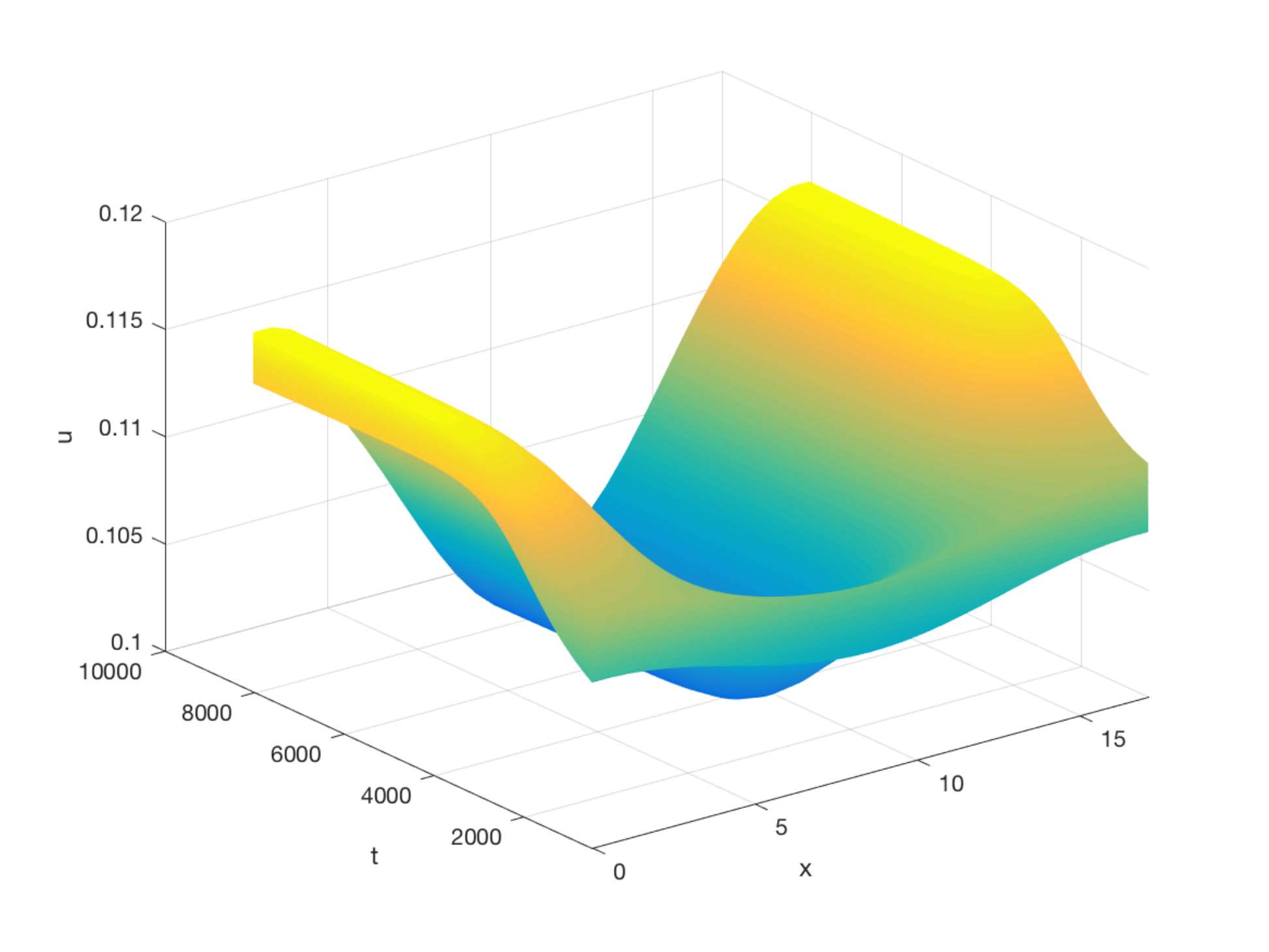}
	\end{minipage}}
	\subfigure[target pattern]{\begin{minipage}{0.23\linewidth}
			\centering\includegraphics[height=0.93\linewidth,width=1.1\linewidth]{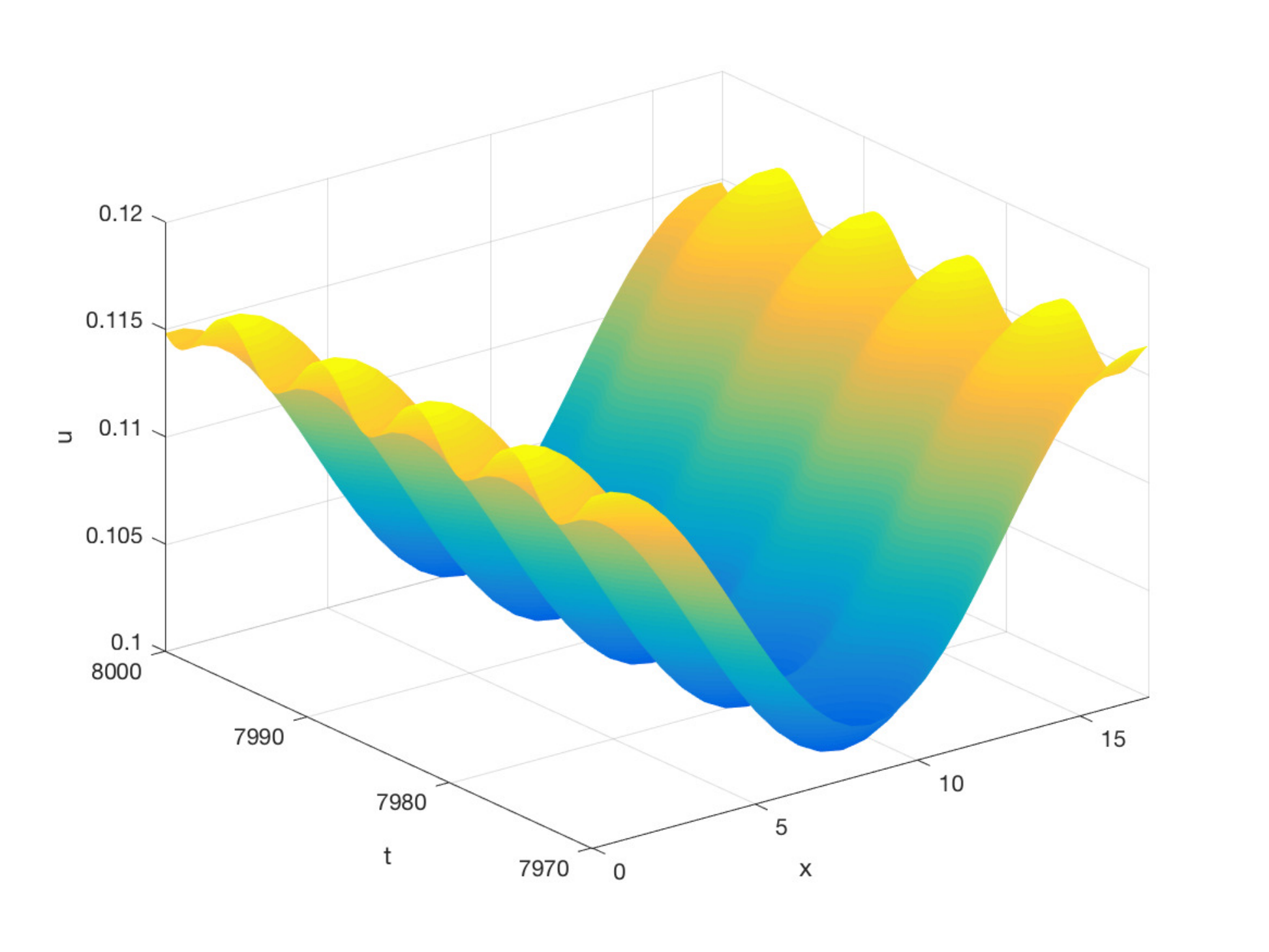}
	\end{minipage}}
	\subfigure[prey pattern]{\begin{minipage}{0.25\linewidth}
			\centering\includegraphics[scale=0.2]{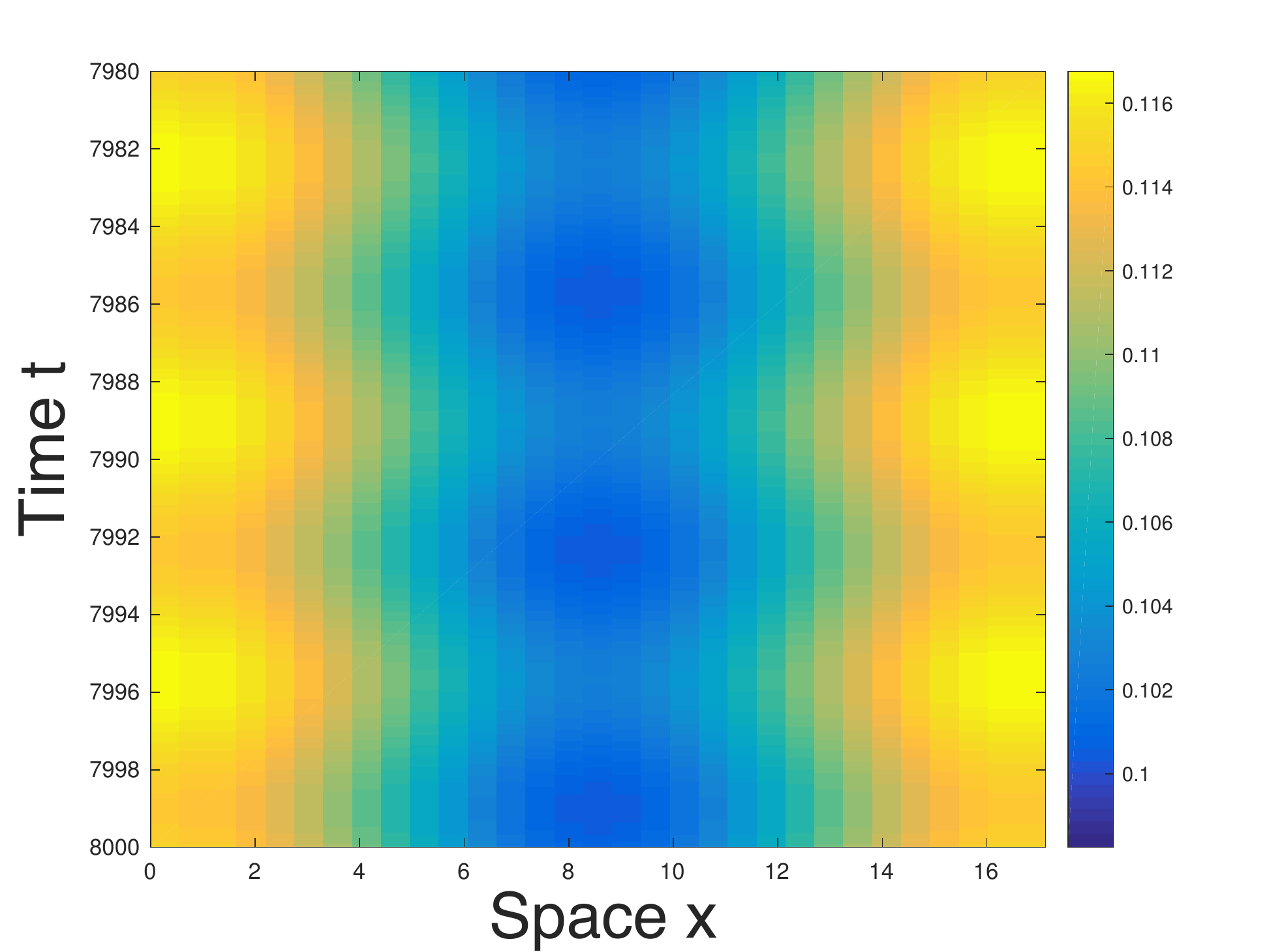}
	\end{minipage}}
	\subfigure[$u(x,7998)$]{\begin{minipage}{0.23\linewidth}
			\centering\includegraphics[height=0.93\linewidth,width=1.1\linewidth]{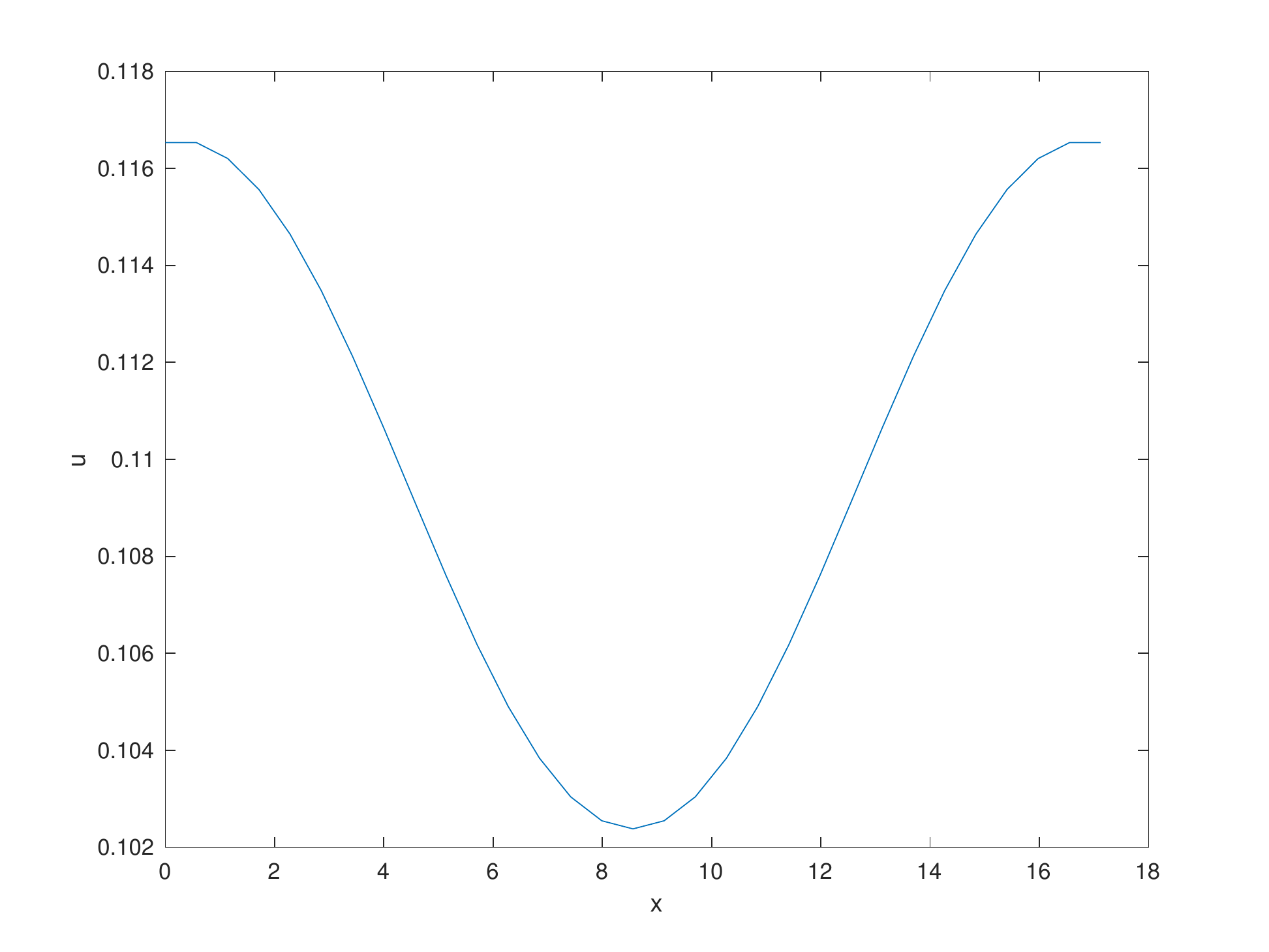}
	\end{minipage}}
	
	\subfigure[$v(x,t)$]{\begin{minipage}{0.25\linewidth}
			\centering\includegraphics[scale=0.2]{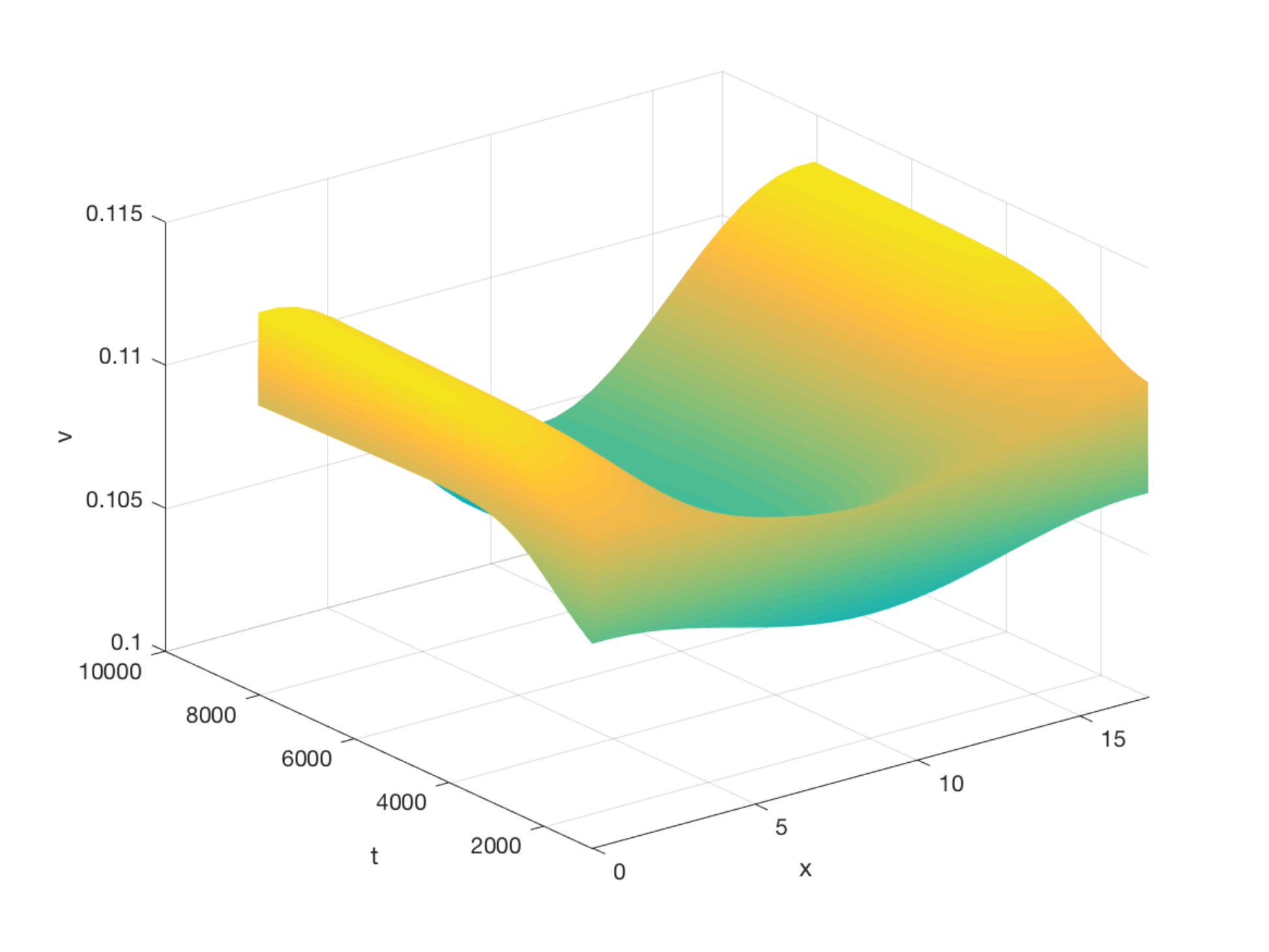}
	\end{minipage}}
	\subfigure[target pattern]{\begin{minipage}{0.23\linewidth}
			\centering\includegraphics[height=0.93\linewidth,width=1.1\linewidth]{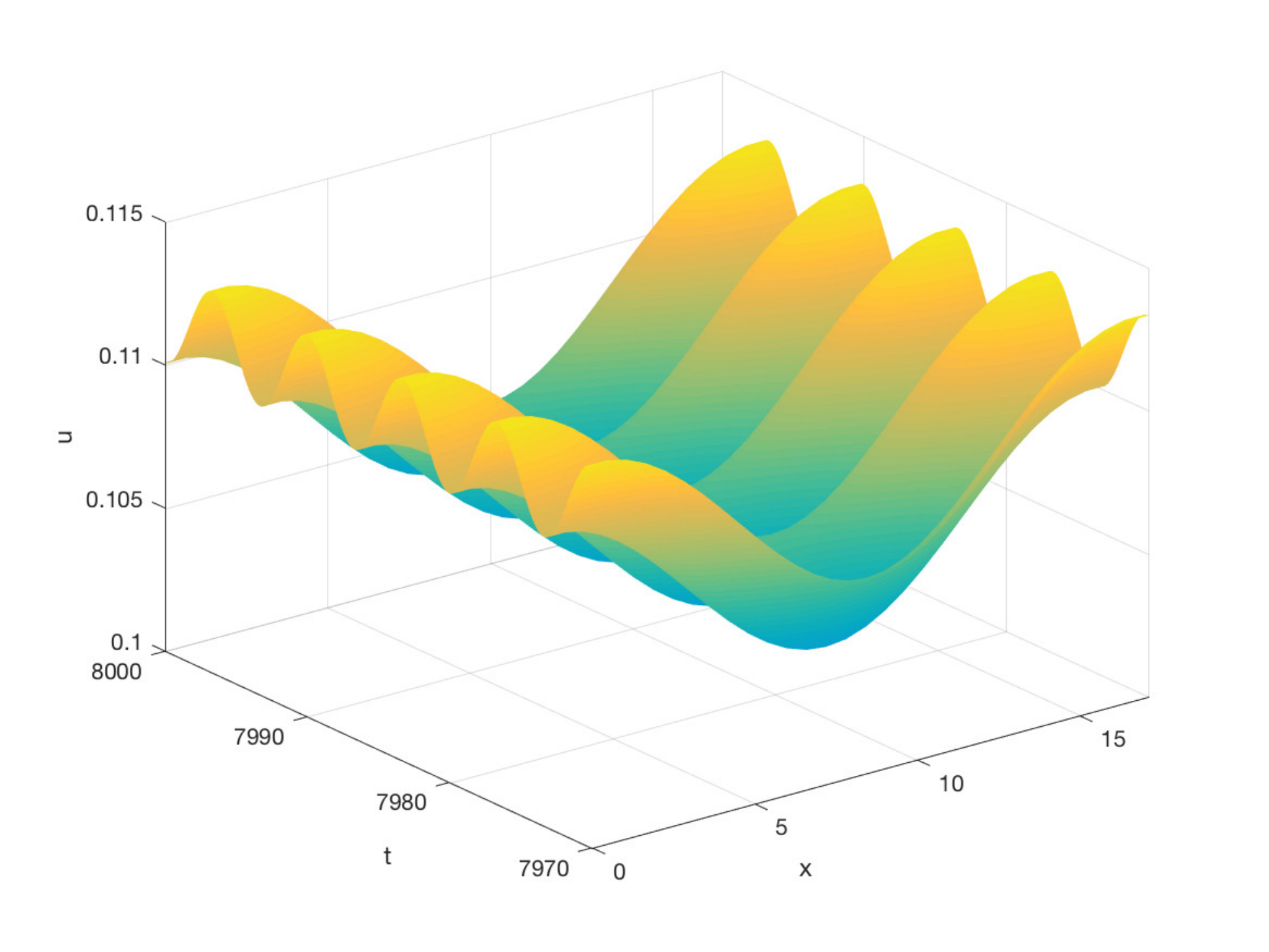}
	\end{minipage}}
	\subfigure[predator pattern]{\begin{minipage}{0.25\linewidth}
			\centering\includegraphics[scale=0.2]{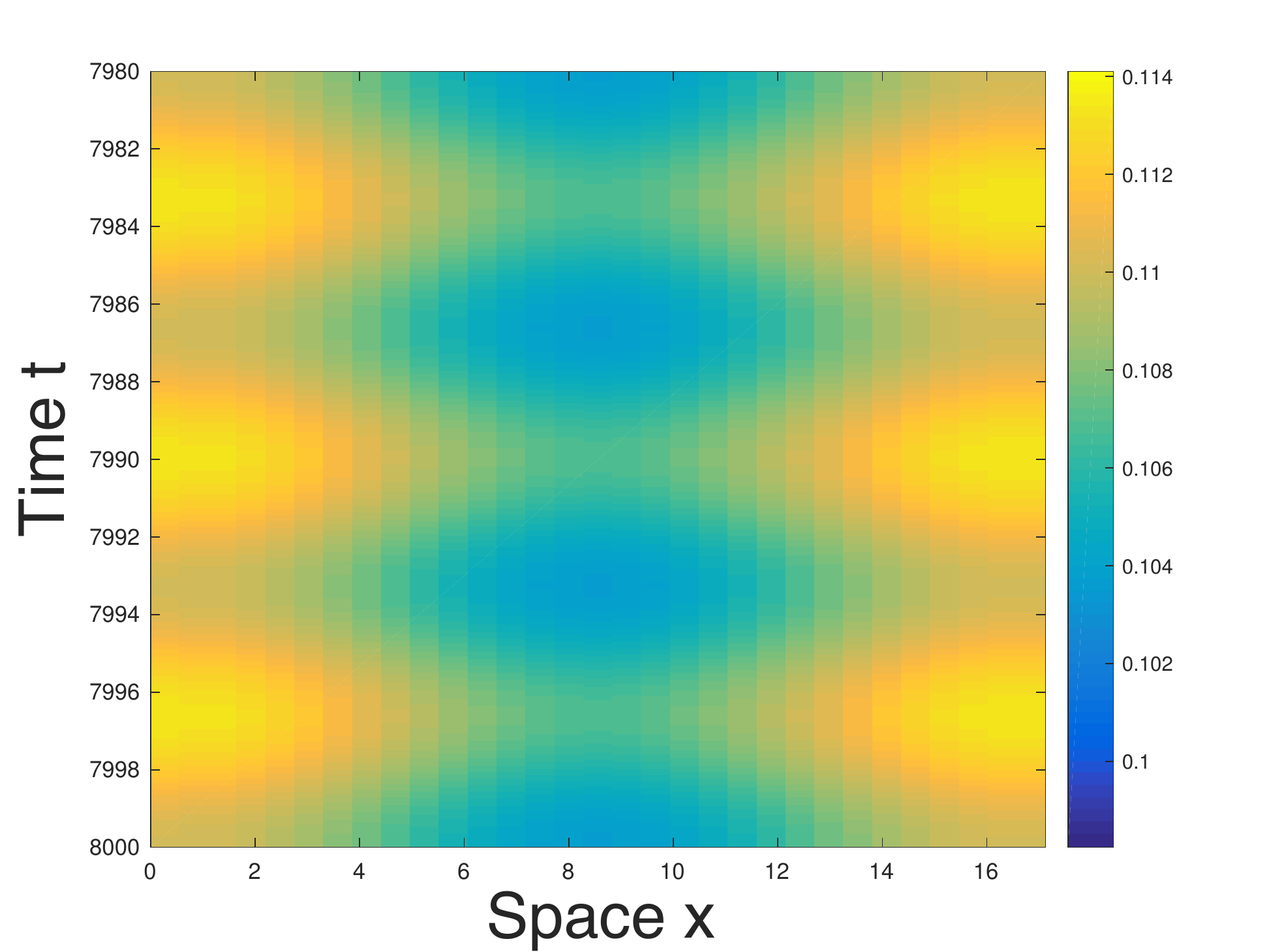}
	\end{minipage}}
	\subfigure[$v(x,7998)$]{\begin{minipage}{0.23\linewidth}
			\centering\includegraphics[height=0.93\linewidth,width=1.1\linewidth]{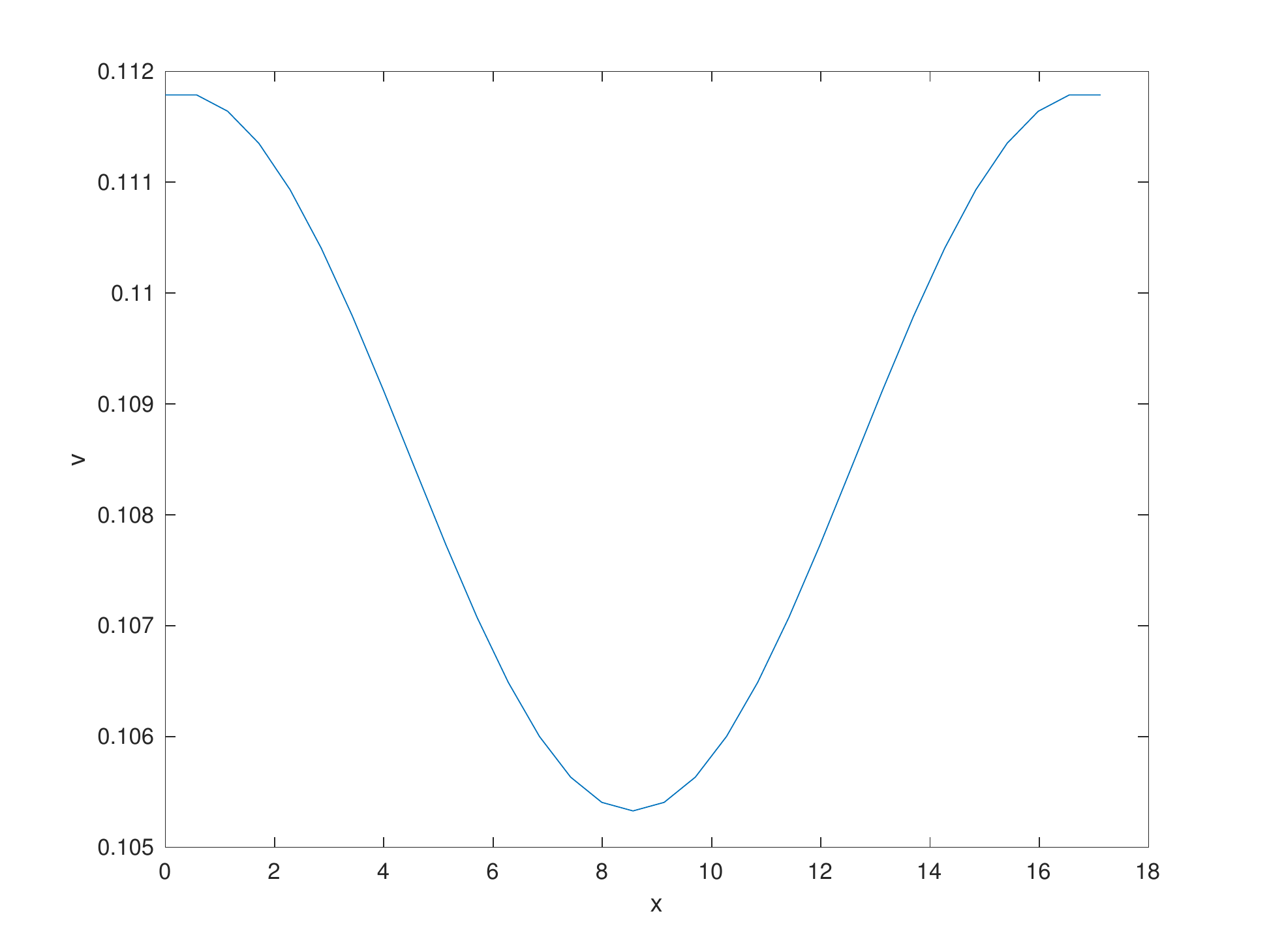}
	\end{minipage}}
\caption{Spatially non-homogeneous periodic solution in $D_3$, with $(\alpha_1,\alpha_2)=(-0.05,0.0105)$ and initial functions are $(u_0-0.01\sin 0.1x,u_0-0.01\sin 0.1x)$.}\label{figD3_2}
\end{figure}

\begin{figure}[htb]
\subfigure[$u(x,t)$]{\begin{minipage}{0.25\linewidth}
		\centering\includegraphics[scale=0.2]{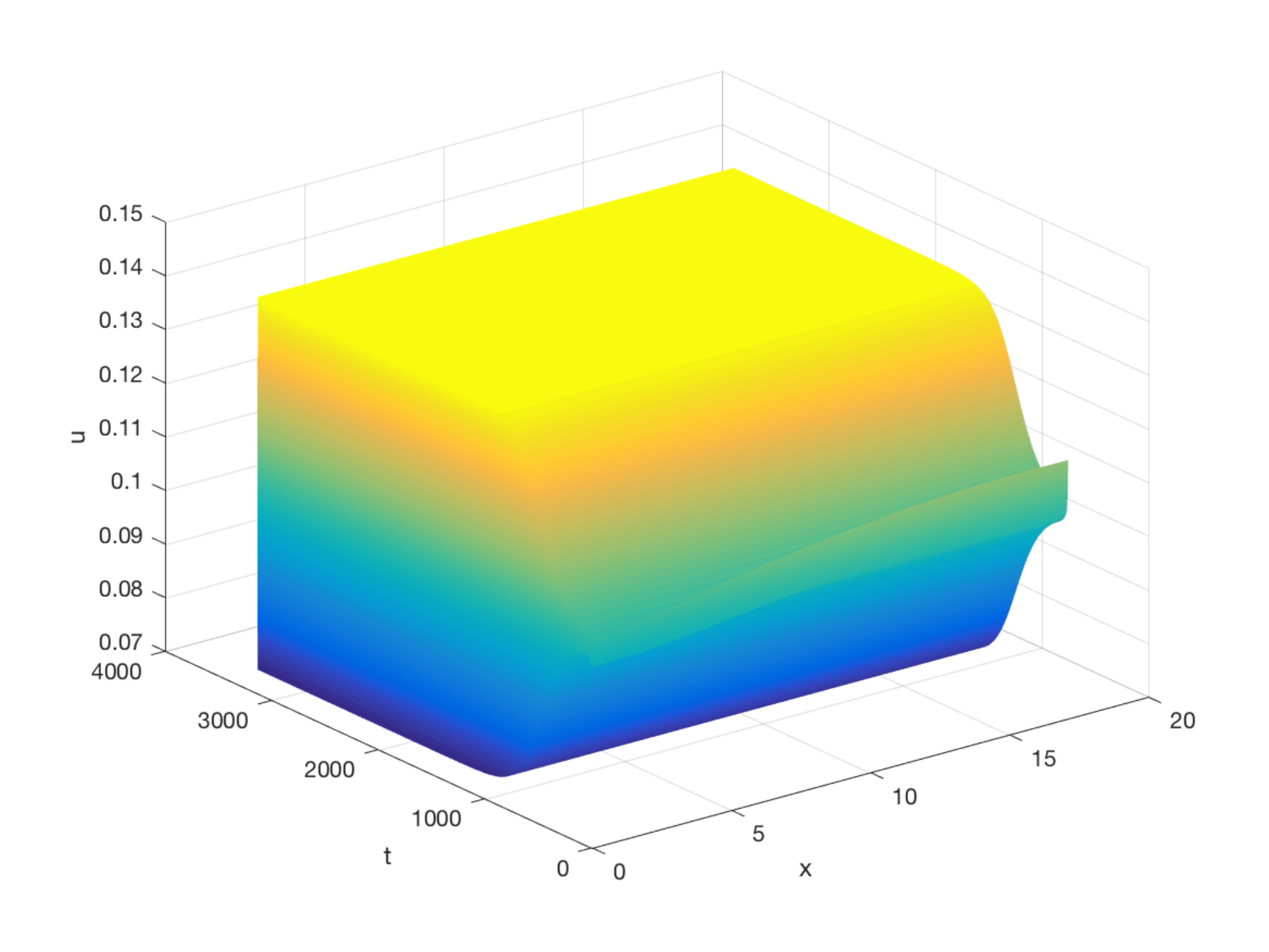}
\end{minipage}}
\subfigure[prey pattern]{\begin{minipage}{0.23\linewidth}
		\centering\includegraphics[height=0.93\linewidth,width=1.1\linewidth]{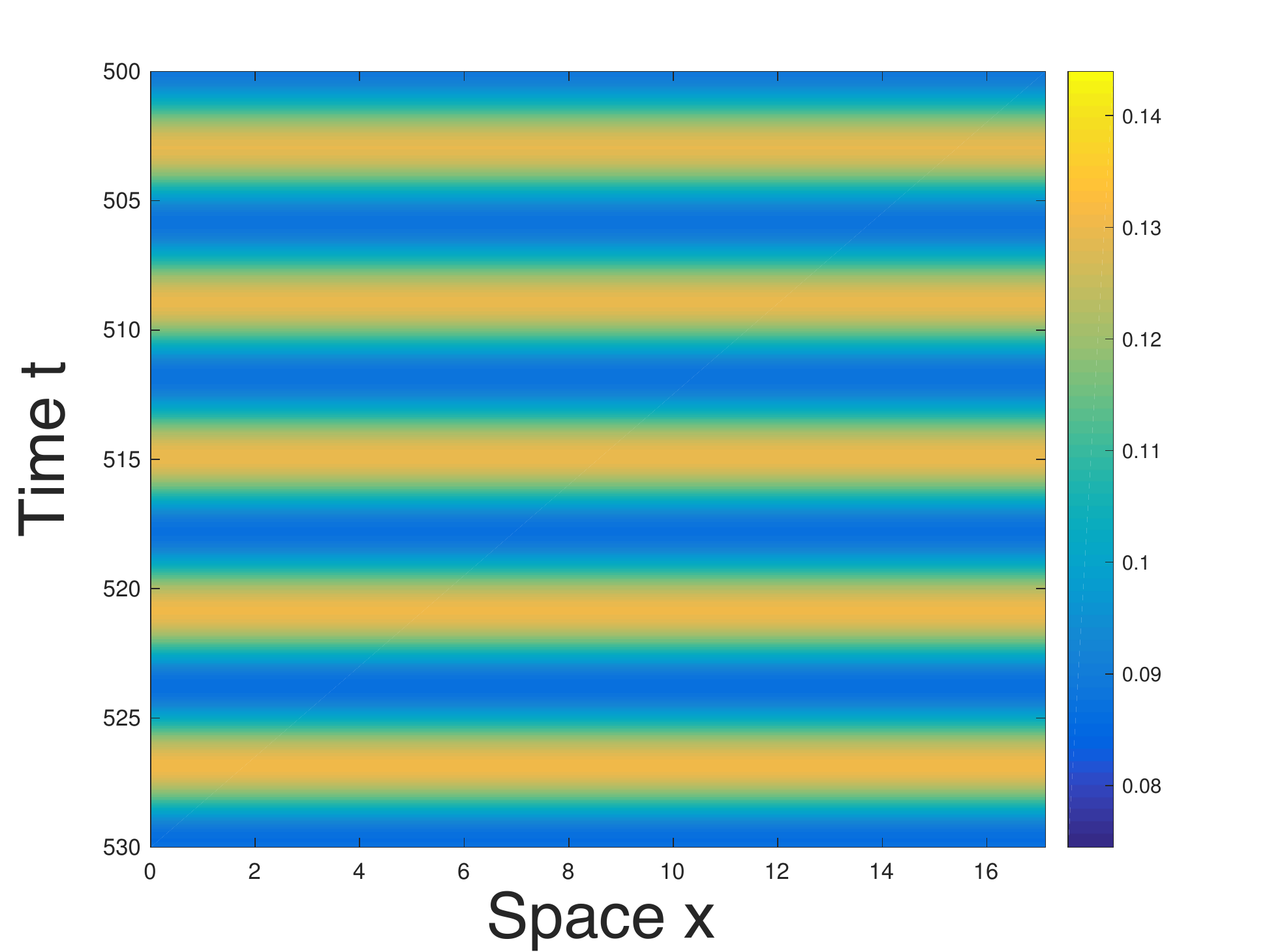}
\end{minipage}}
\subfigure[$v(x,t)$]{\begin{minipage}{0.25\linewidth}
		\centering\includegraphics[scale=0.2]{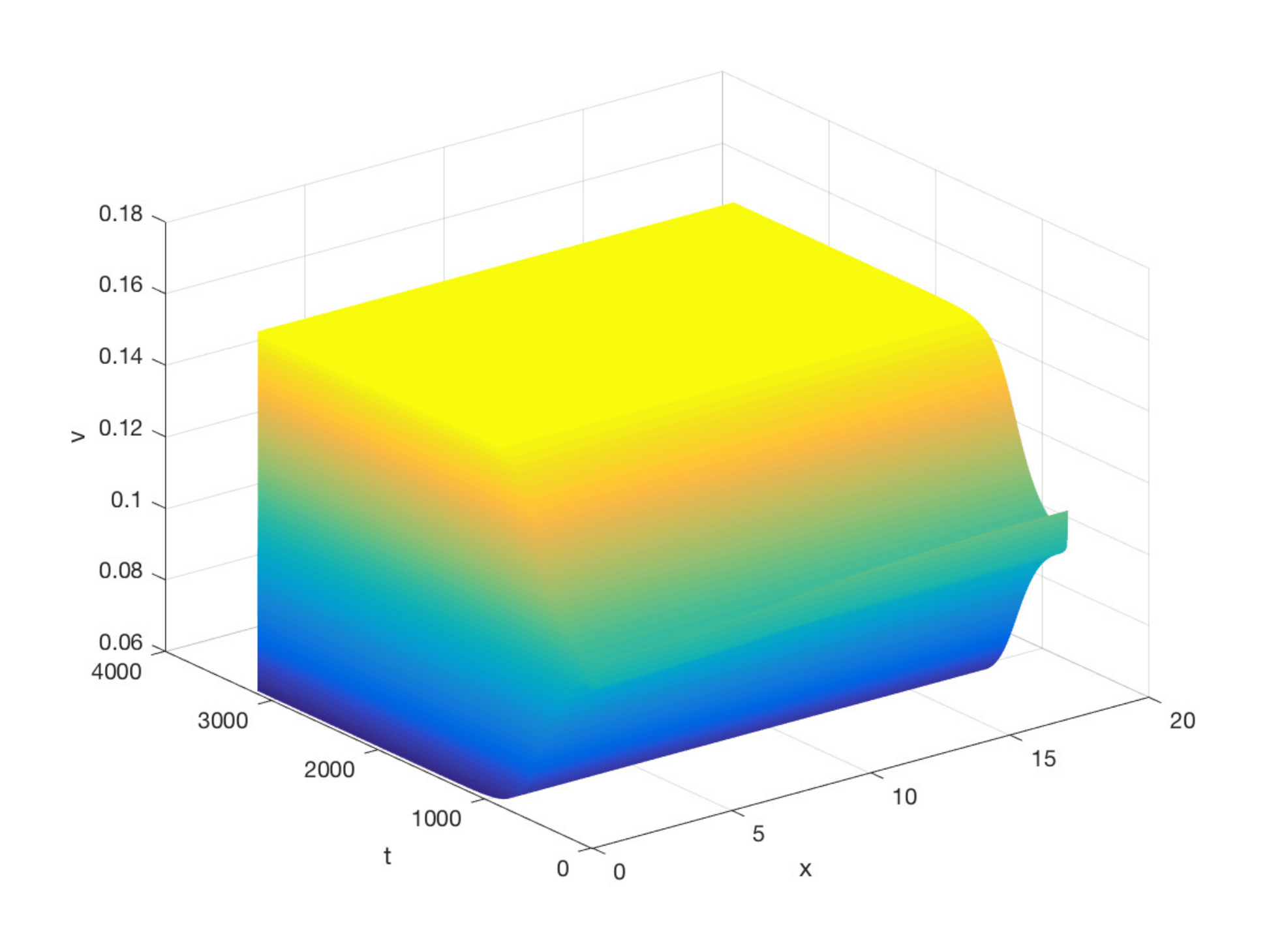}
\end{minipage}}
\subfigure[predator pattern]{\begin{minipage}{0.23\linewidth}
		\centering\includegraphics[height=0.93\linewidth,width=1.1\linewidth]{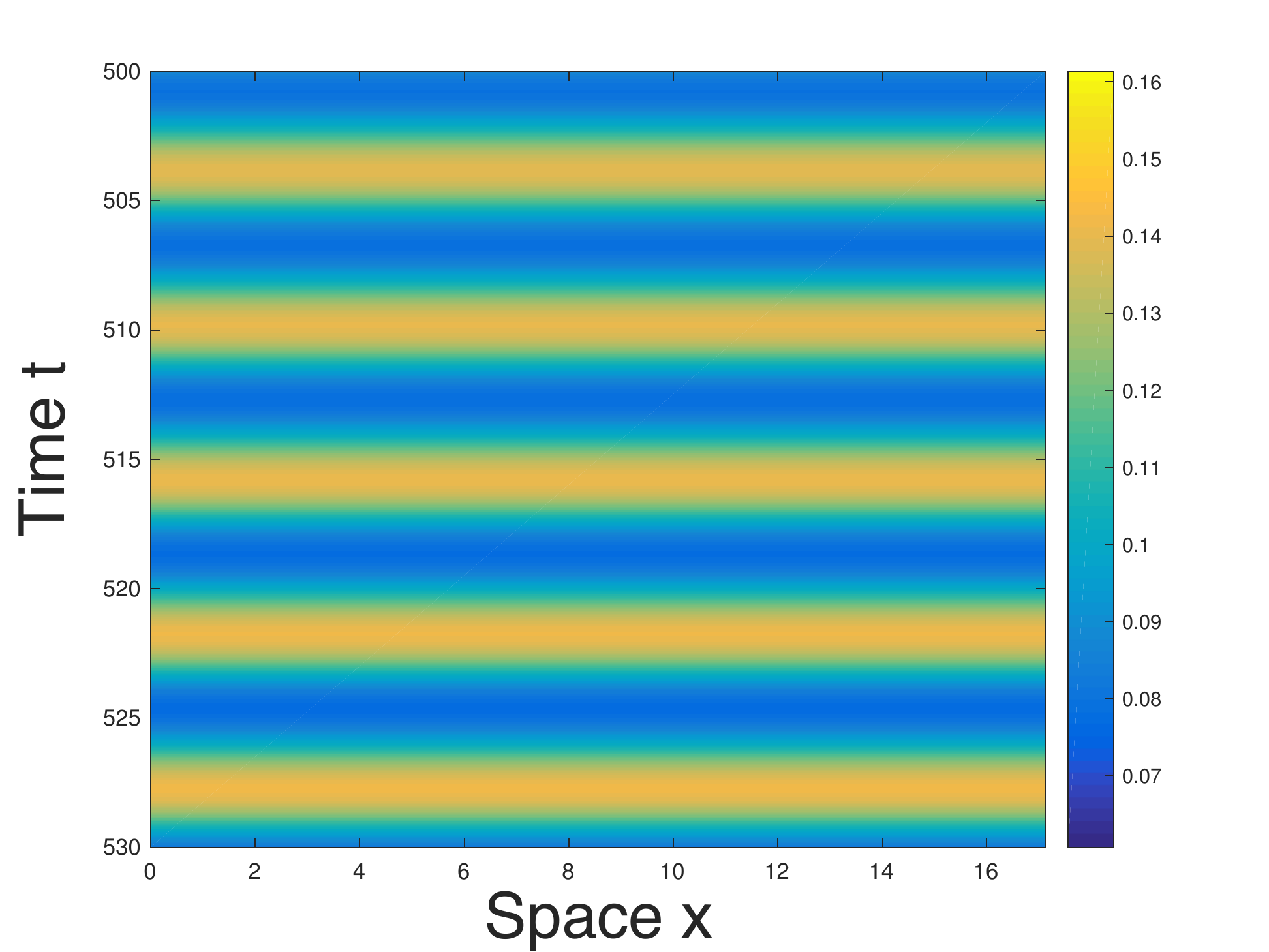}
\end{minipage}}
\caption{Spatially homogeneous periodic solution in $D_6$, with $(\alpha_1,\alpha_2)=(0.05,-0.01)$ and initial functions are $(u_0+0.01\sin 0.1x,u_0+0.01\sin 0.1x)$.}\label{figD6}
\end{figure}
\subsection{Turing-Turing-Hopf type spatiotemporal patterns}
{\bf Group 2.} Taking $d_1 = 1.68,$ $d_2 = 16.0,$ $a = 0.80,$ $b = 0.0004,$ and $l = 4.37$, which are satisfy the condition $({{\mathbf{A}}}6^{''})$. The unique coexistence equilibrium point now is
$(u_0, v_0) = (0.2016, 0.2016)$. From the  eigenvalue analysis, we can get $r_1^T = 1.4598$, $r_2^T = 1.4694$, $r_n^T <0\,(n\geq 3)$, $\tau_0^{(0)} = 0.7423$, $\tau_1^{(0)} = 1.3960$, $\omega_0 = 1.3612$, $\omega_1 = 1.0962$.
The important information can be summed up as
$$n_T = 2, \quad n_H=0, \quad r_* = 1.2639,  \quad \tau_*= 0.7937, \quad \omega_* = 1.0087.$$
Comparing with Group 1, the values of $r_*$ and  the second largest Turing point $r_1^T$ are relatively close in this group. The bifurcation diagram in $r-\tau$ plane has been shown in Figure \ref{figTH2}.

Deal with the same method as in {\bf Group 1}, we obtain the coefficients in the equivalent plane system \eqref{eqrv2} are  $\epsilon_{1}(\alpha)=~0.2844\alpha_1 + 1.1434 \alpha_2$, $\epsilon_{2}(\alpha)=-0.2134\alpha_1$, and $b_0=-0.1257$, $c_0=-1.3132$, $d_0=1$, $d_0-b_0c_0= 0.8350>0$.
The  Case \uppercase\expandafter{\romannumeral4}a in \cite[\S 7.5]{Phillp1988} occurs in {\bf Group 2} and the parameters plane $(\alpha_1,\alpha_2)$ can also  be divided into six parts as shown in \ref{Detail2}. Through a series of the numerical experiments, we summarize the dynamics of each region into the following proposition.

\begin{figure}[htb]
	\centering
	\subfigure[]{\begin{minipage}{0.48\linewidth}
			\centering\includegraphics[height = 0.78\linewidth,width=1.05\linewidth]{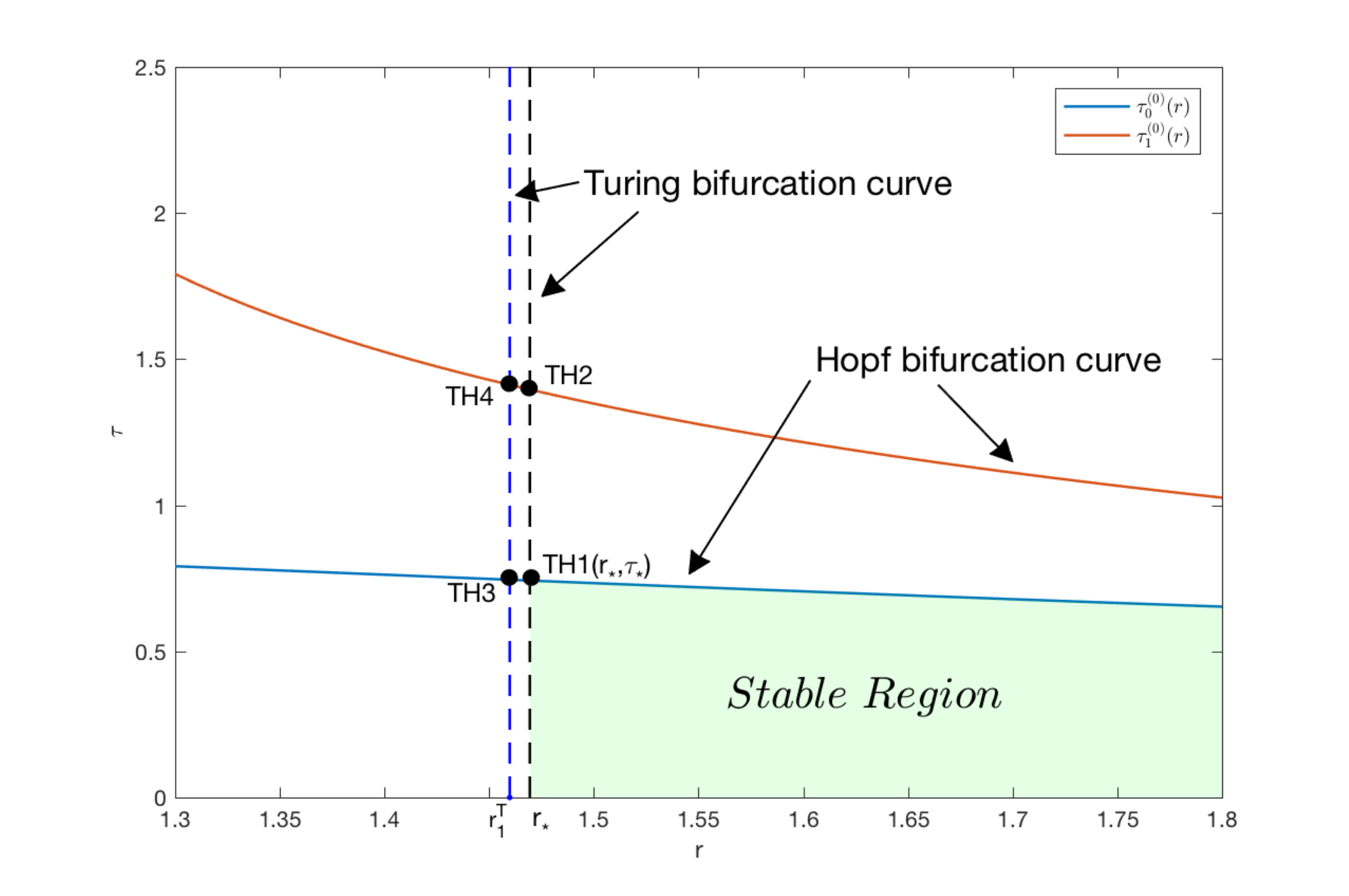}\label{figTH2}
	\end{minipage}}
	\subfigure[]{\begin{minipage}{0.48\linewidth}
			\centering\includegraphics[scale=0.35]{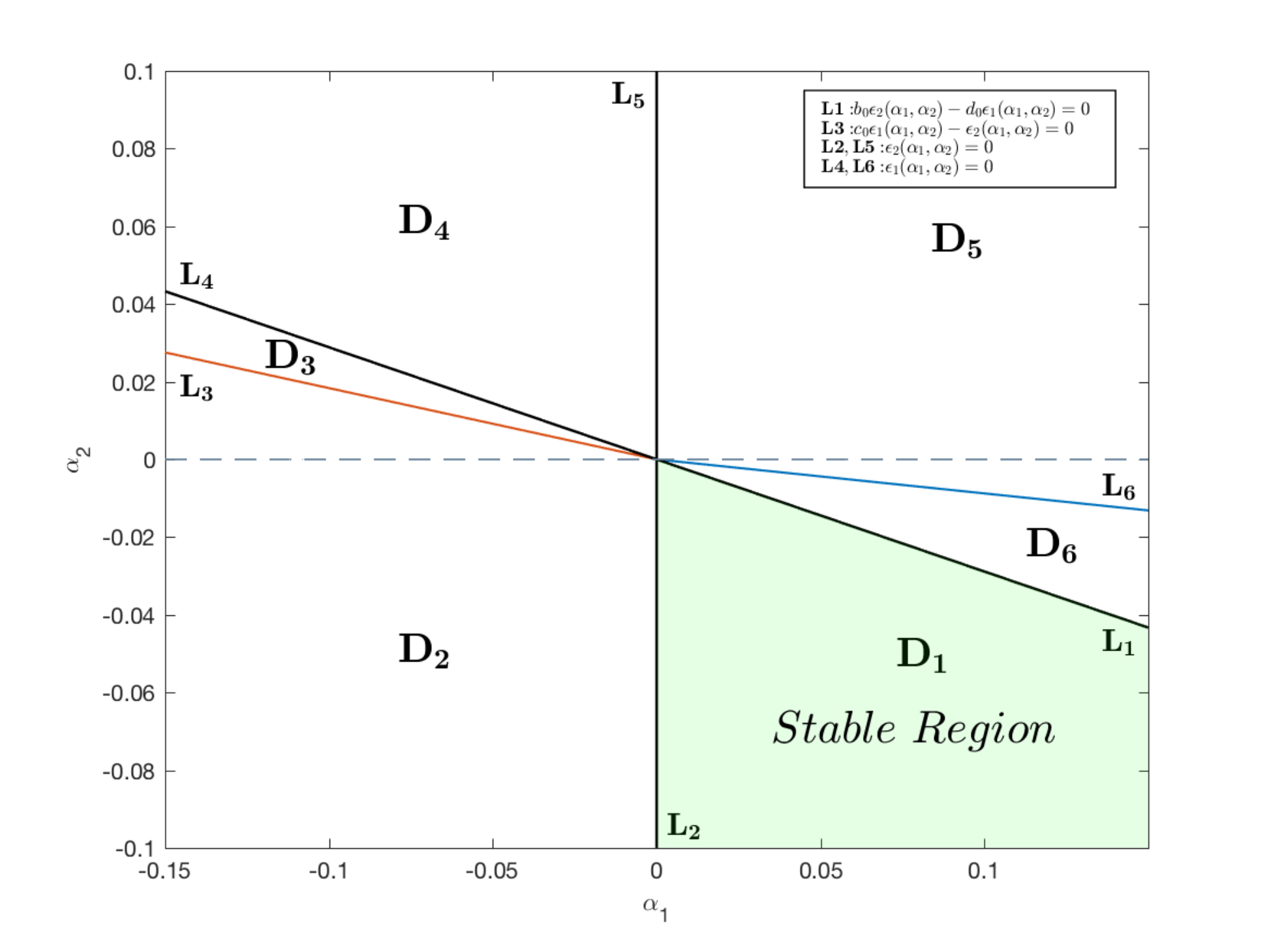}\label{Detail2}
	\end{minipage}}
\caption{(a) Bifurcation sets in $r-\tau$ plane. (b) Detailed bifurcation diagram near $(r_*,\tau_*)$ in $(\alpha_1,\alpha_2)$ plane.}
\end{figure}

\begin{proposition}
	When $d_1 = 1.68,$ $d_2 = 16.0,$ $a = 0.80,$ $b = 0.0004,$ and $l = 4.37$, the system \eqref{eqA} undergoes a Turing-Hopf bifurcation at $(r_*,\tau_*)$ with $n_T=2$ and $n_H=0$. The parameter plane near the critical value is divided into six regions (see Figure \ref{Detail2}). The dynamics of each region $D_1-D_6$ are:
	\begin{itemize}
		\item In $D_1$, the constant steady state $(u_0,v_0)$ is locally asymptotically stable.
		\item In $D_2$, two stable non-constant steady states are coexistence (see Figure \ref{fig2D2_1}-Figure \ref{fig2D2_2}), the spatial distribution follows to the function: $h_1\cos(\frac{2}{l}x)+h_2\cos(\frac{1}{l}x)$.
		\item In $D_3$, two stable spatially non-homogeneous periodic orbits are coexistence (see Figure \ref{fig2D3_1}-Figure \ref{fig2D3_2}).
		\item In $D_4$, two stable spatially non-homogeneous periodic orbits are coexistence.
		\item In $D_5$, two stable spatially non-homogeneous periodic orbits are coexistence (see Figure \ref{fig2D5_1}-Figure \ref{fig2D5_2}), the spatial distribution follows to the function: $h_1\cos(\frac{2}{l}x)+h_2\cos(\frac{1}{l}x)$.  In contrast to $D_3-D_4$, the solutions in this region have been oscillation for a long time before reaching the target patterns.
		\item In $D_6$, a spatially non-homogeneous periodic orbits is stable.
	\end{itemize}
\end{proposition}

\begin{figure}[htbp]
	\subfigure[$u(x,t)$]{\begin{minipage}{0.25\linewidth}
			\centering\includegraphics[scale=0.2]{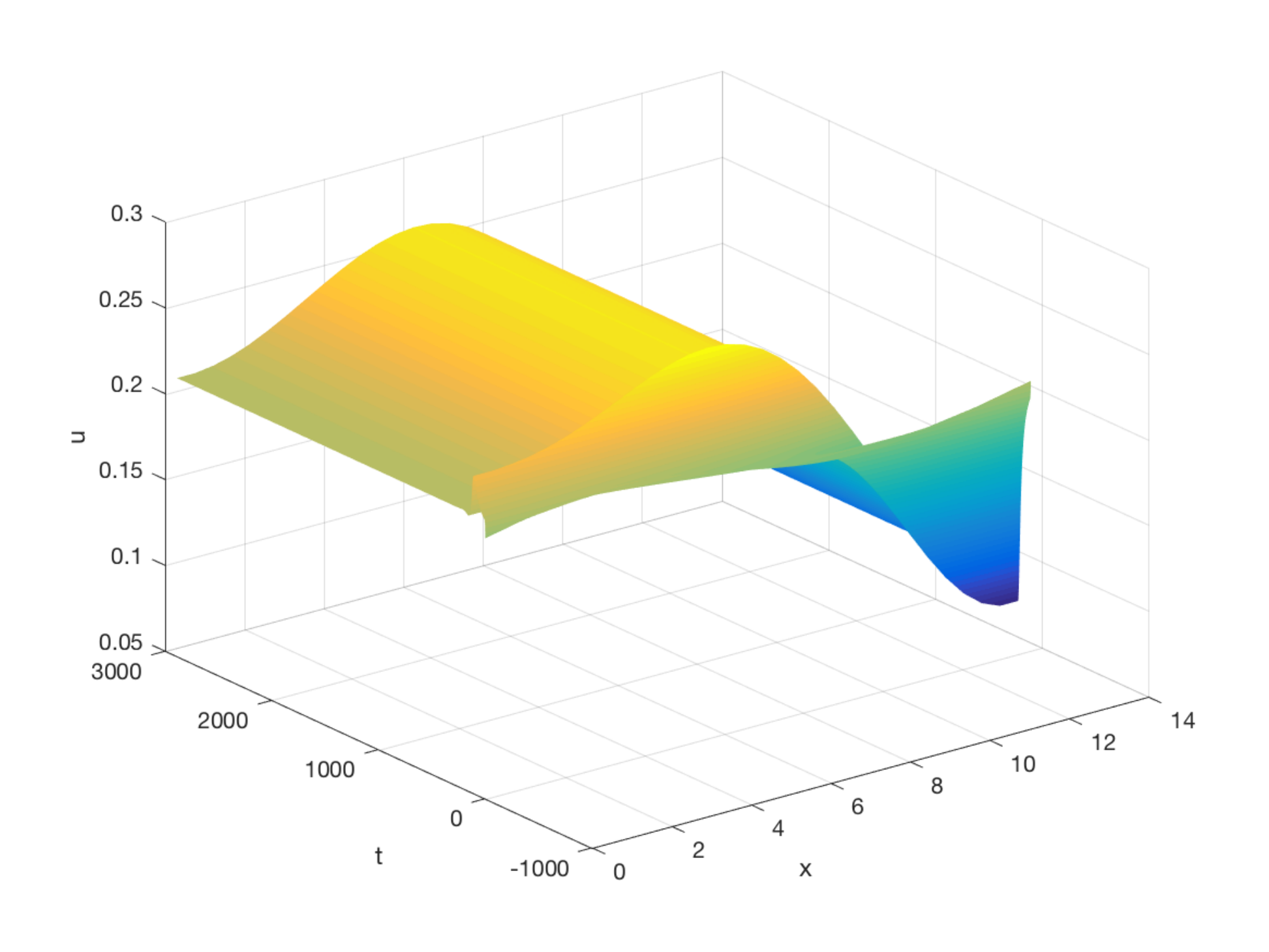}
	\end{minipage}}
	\subfigure[prey pattern]{\begin{minipage}{0.23\linewidth}
			\centering\includegraphics[height=0.93\linewidth,width=1.1\linewidth]{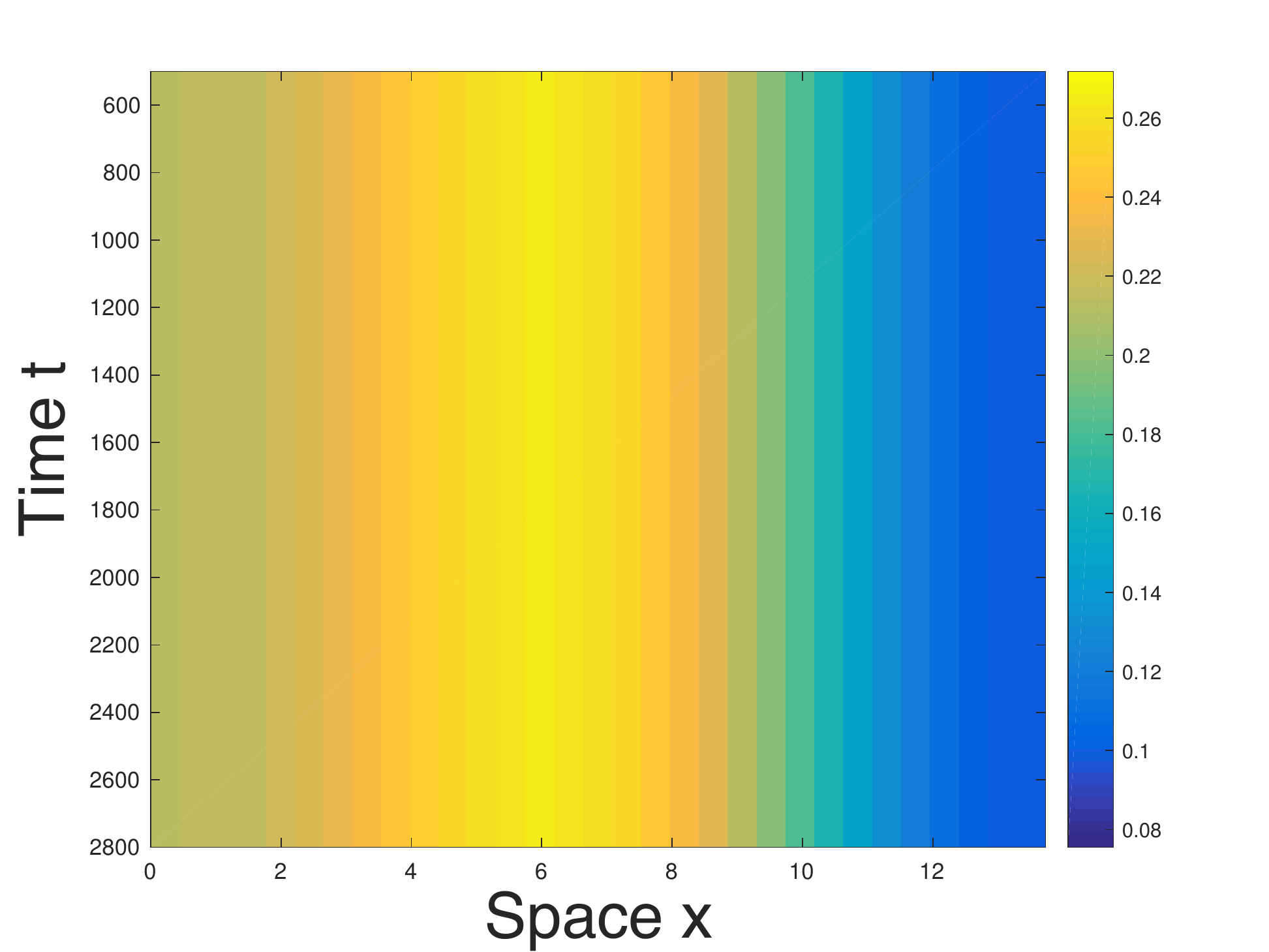}
	\end{minipage}}
	\subfigure[$v(x,t)$]{\begin{minipage}{0.25\linewidth}
			\centering\includegraphics[scale=0.2]{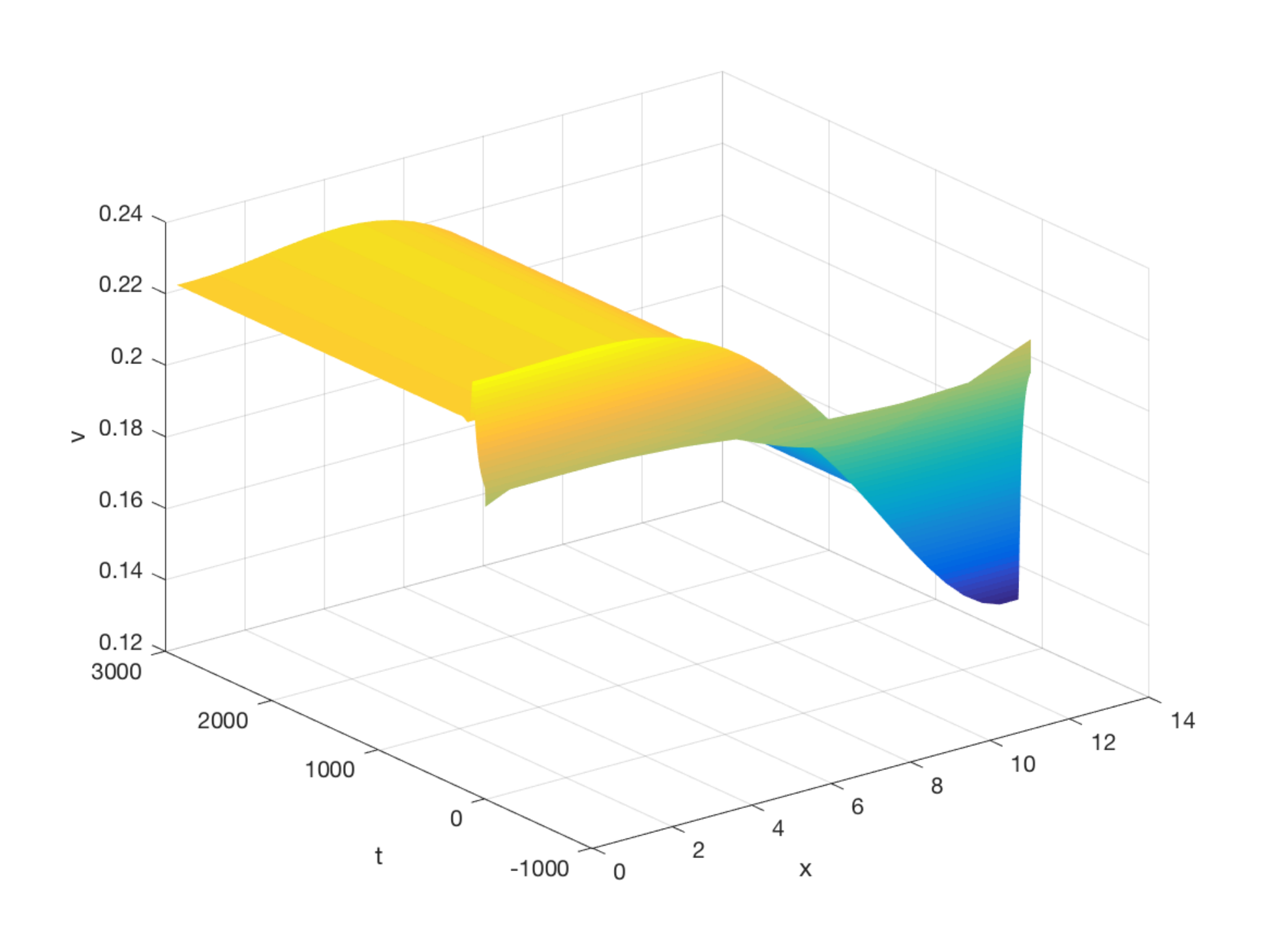}
	\end{minipage}}
	\subfigure[predator pattern]{\begin{minipage}{0.23\linewidth}
			\centering\includegraphics[height=0.93\linewidth,width=1.1\linewidth]{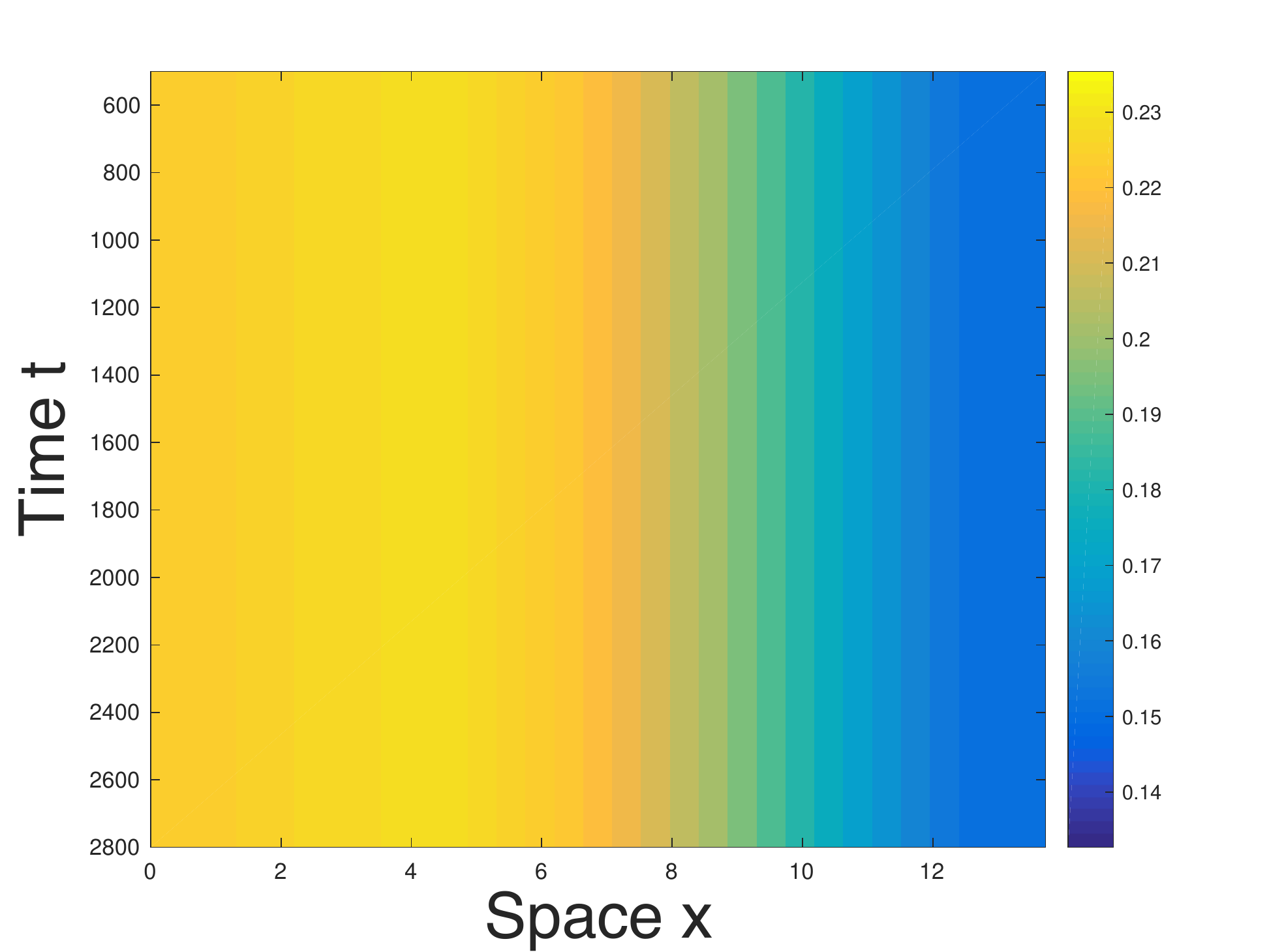}
	\end{minipage}}
\caption{Non-constant steady state in $D_2$, with $(\alpha_1,\alpha_2)=(-0.05,-0.02)$ and initial functions are $(u_0+0.01\sin 0.5x,u_0+0.01\sin 0.5x)$.}\label{fig2D2_1}
\end{figure}
\begin{figure}[htbp]
	\subfigure[$u(x,t)$]{\begin{minipage}{0.25\linewidth}
		\centering\includegraphics[scale=0.2]{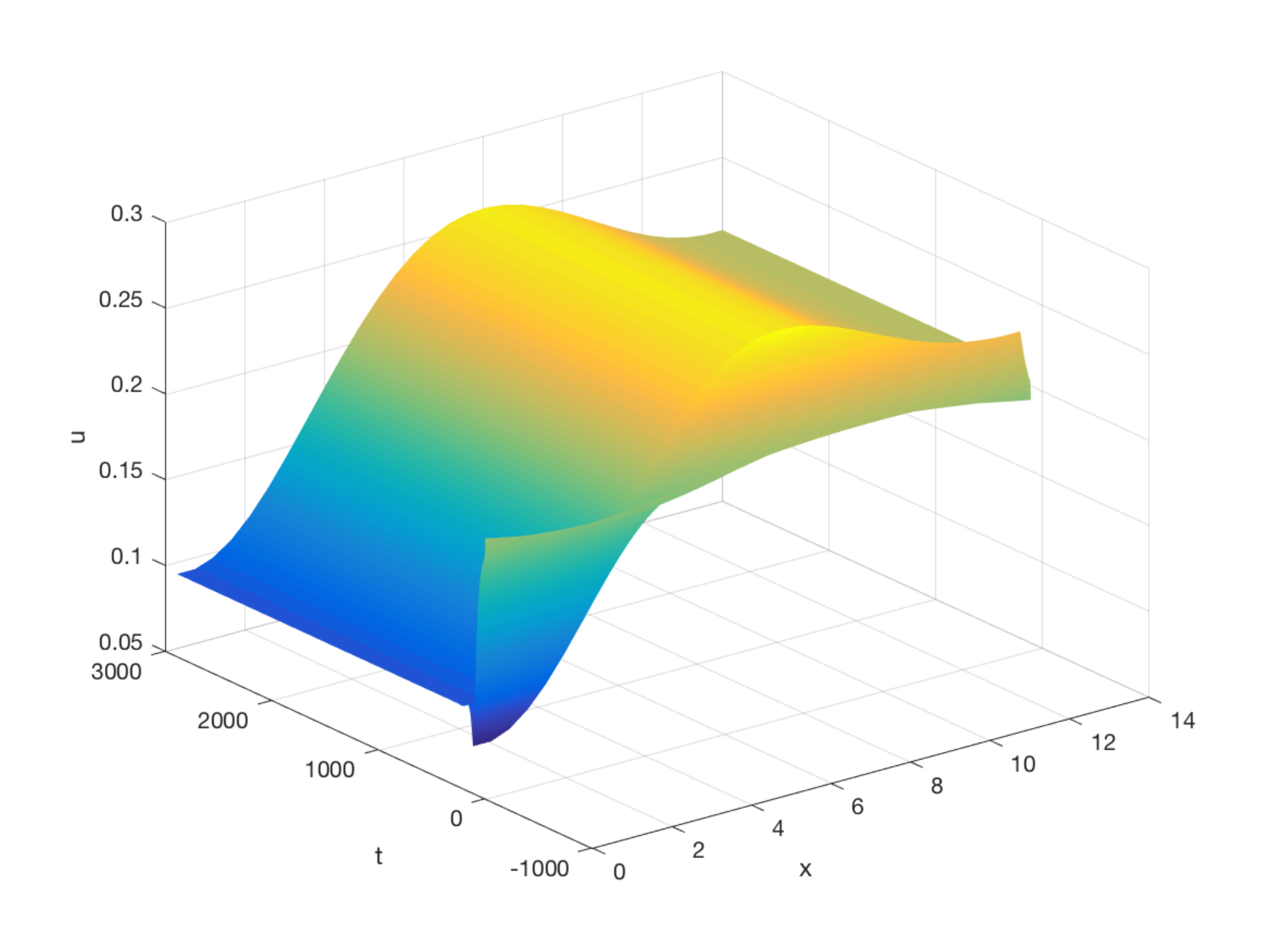}
\end{minipage}}
\subfigure[prey pattern]{\begin{minipage}{0.23\linewidth}
		\centering\includegraphics[height=0.93\linewidth,width=1.1\linewidth]{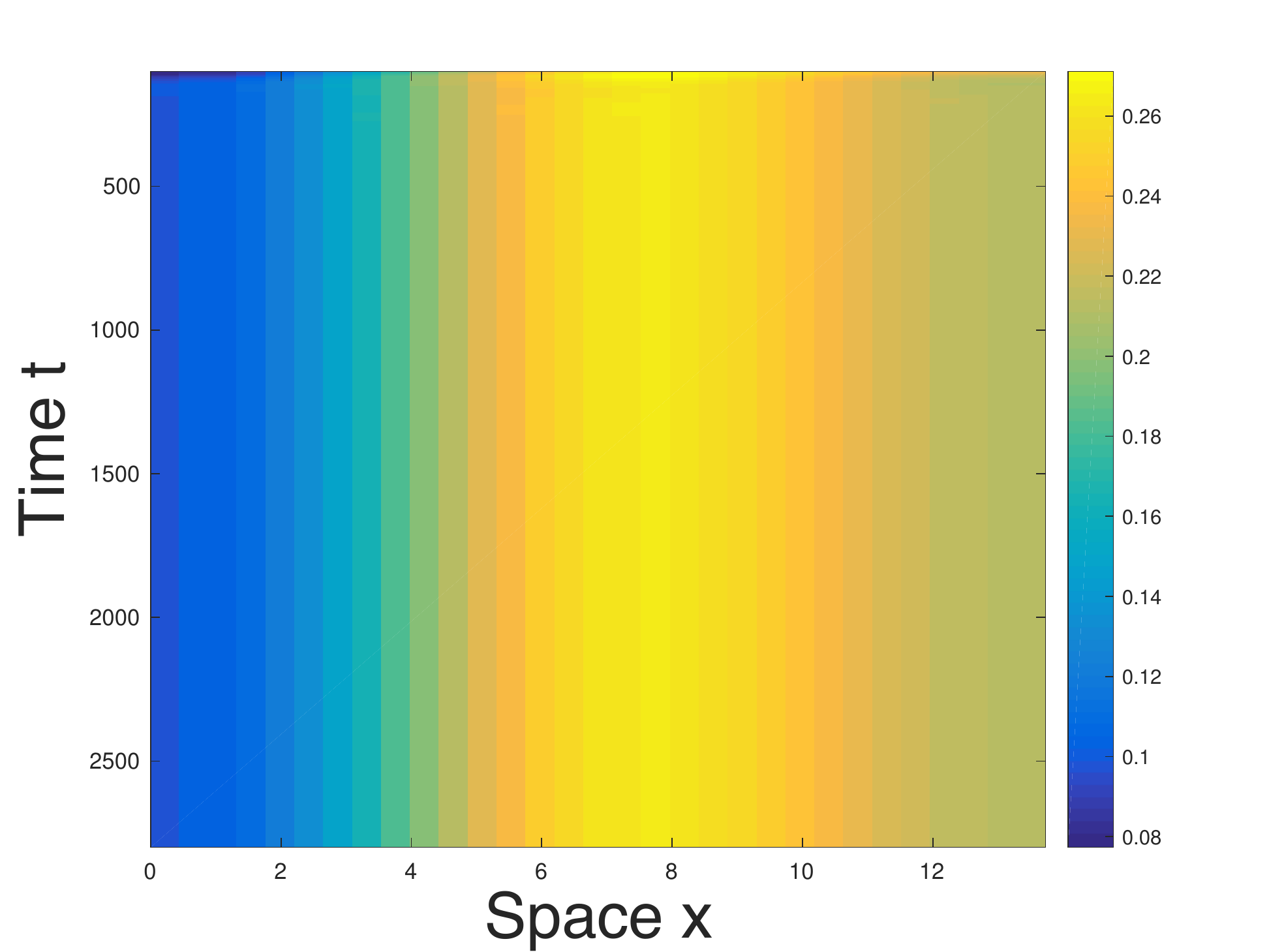}
\end{minipage}}
\subfigure[$v(x,t)$]{\begin{minipage}{0.25\linewidth}
		\centering\includegraphics[scale=0.2]{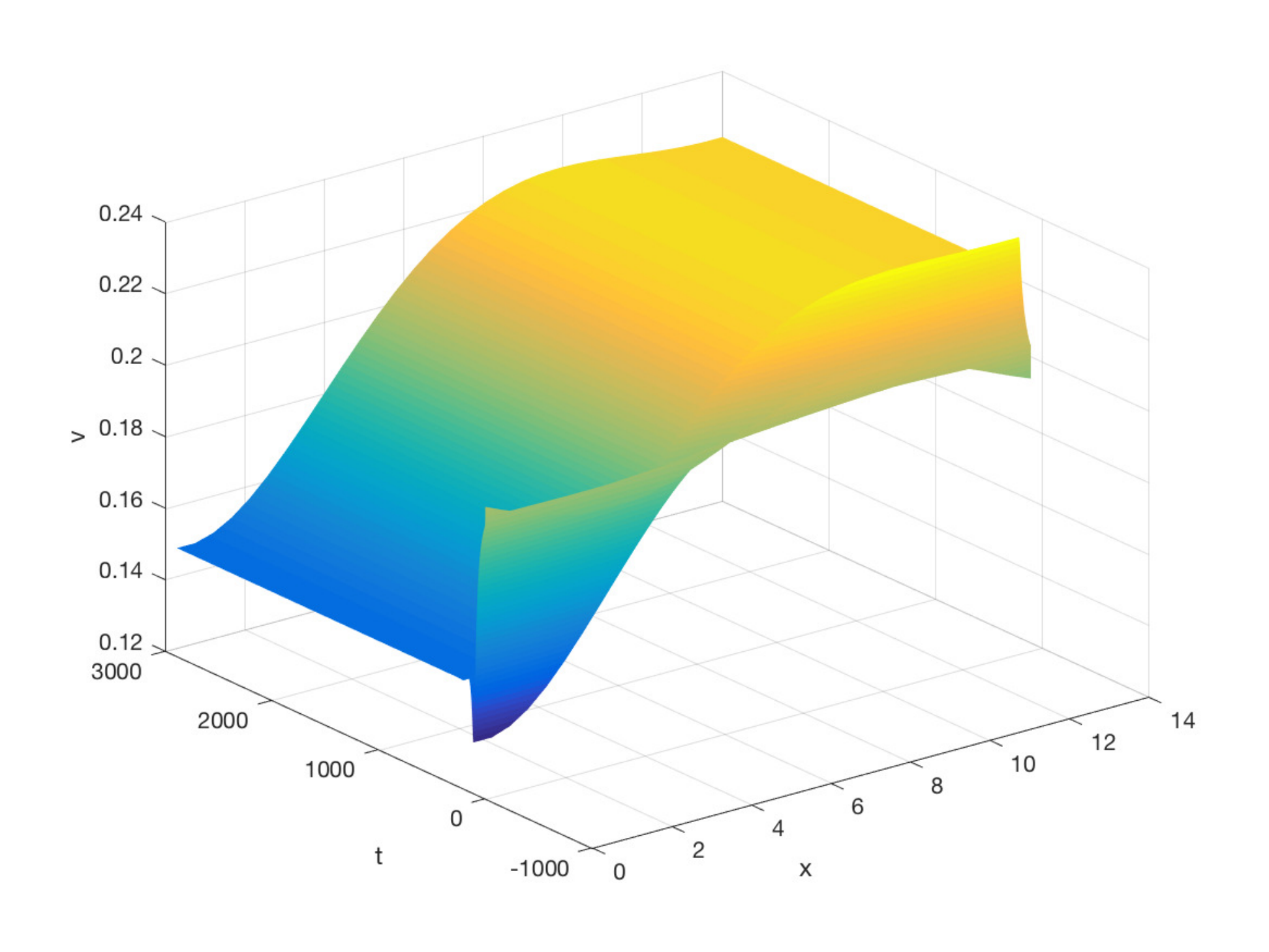}
\end{minipage}}
\subfigure[predator pattern]{\begin{minipage}{0.23\linewidth}
		\centering\includegraphics[height=0.93\linewidth,width=1.1\linewidth]{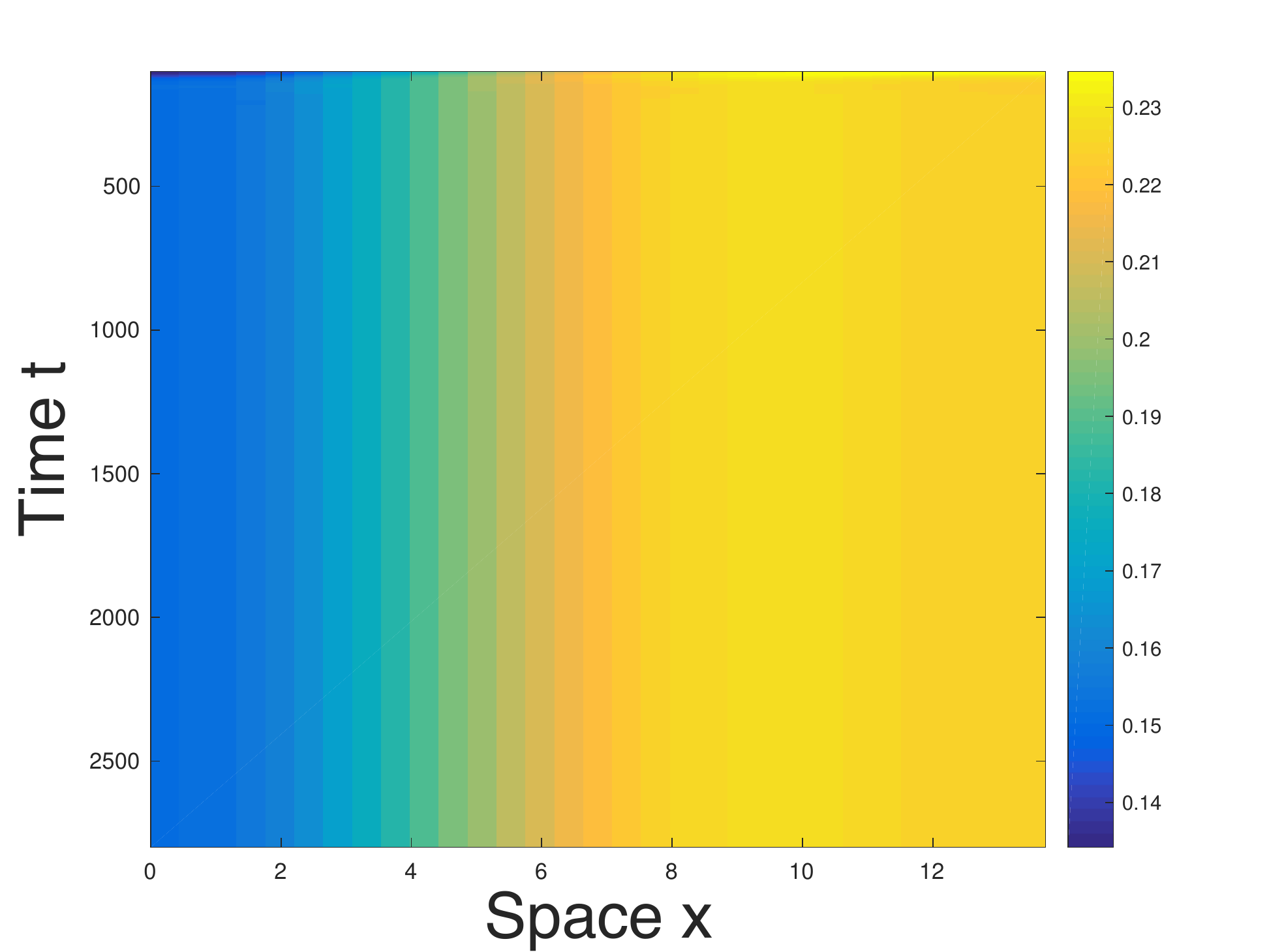}
\end{minipage}}
\caption{Non-constant steady state in $D_2$, with $(\alpha_1,\alpha_2)=(-0.05,-0.02)$ and initial functions are $(u_0-0.01\sin 0.5x,u_0-0.01\sin 0.5x)$.}\label{fig2D2_2}
\end{figure}

\begin{figure}[htbp]
\subfigure[$u(x,t)$]{\begin{minipage}{0.2\linewidth}
		\centering\includegraphics[scale=0.15]{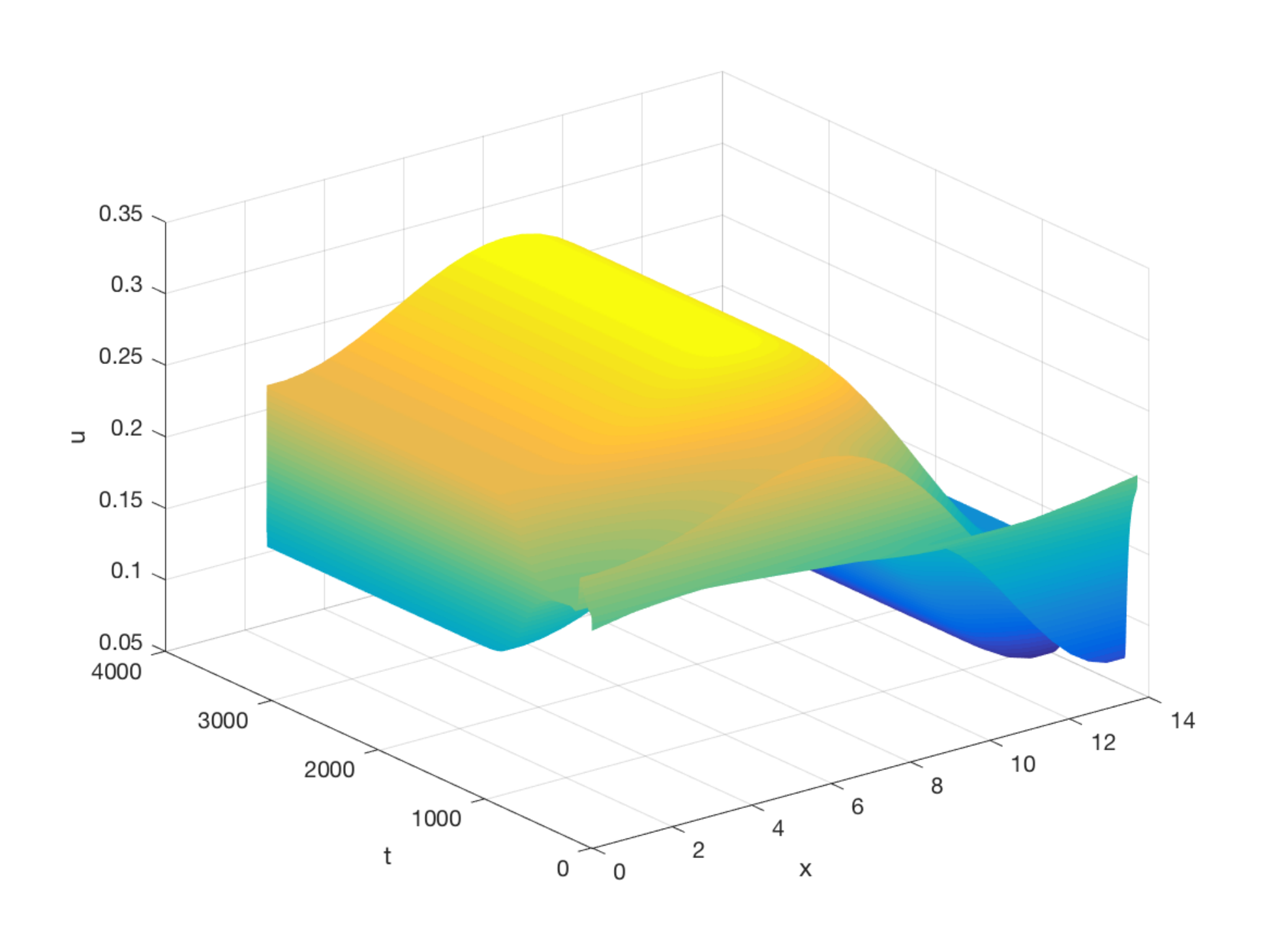}
\end{minipage}}
\subfigure[target pattern]{\begin{minipage}{0.2\linewidth}
		\centering\includegraphics[scale=0.15]{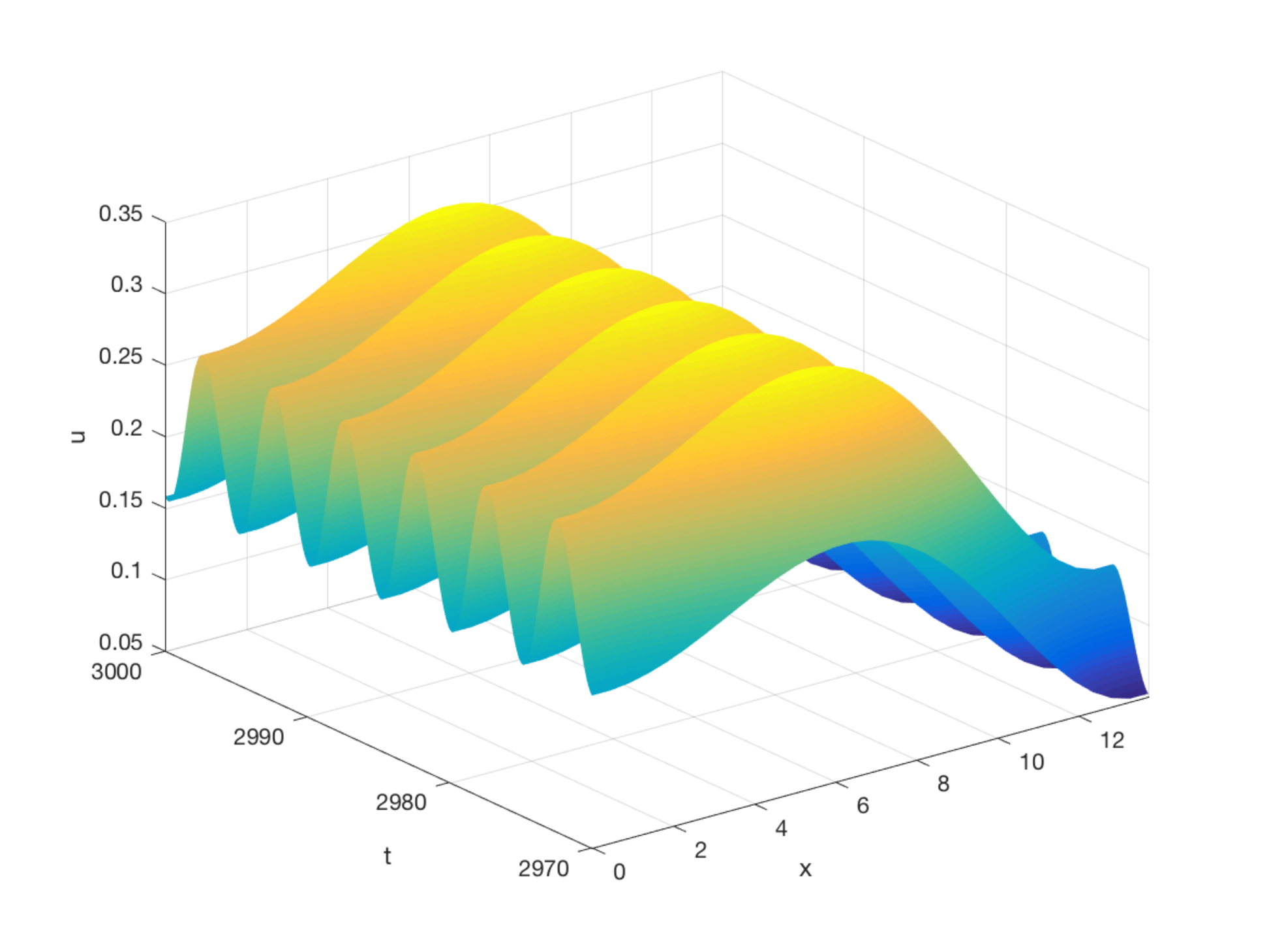}
\end{minipage}}
\subfigure[prey pattern]{\begin{minipage}{0.2\linewidth}
		\centering\includegraphics[height=0.8\linewidth,width=1.\linewidth]{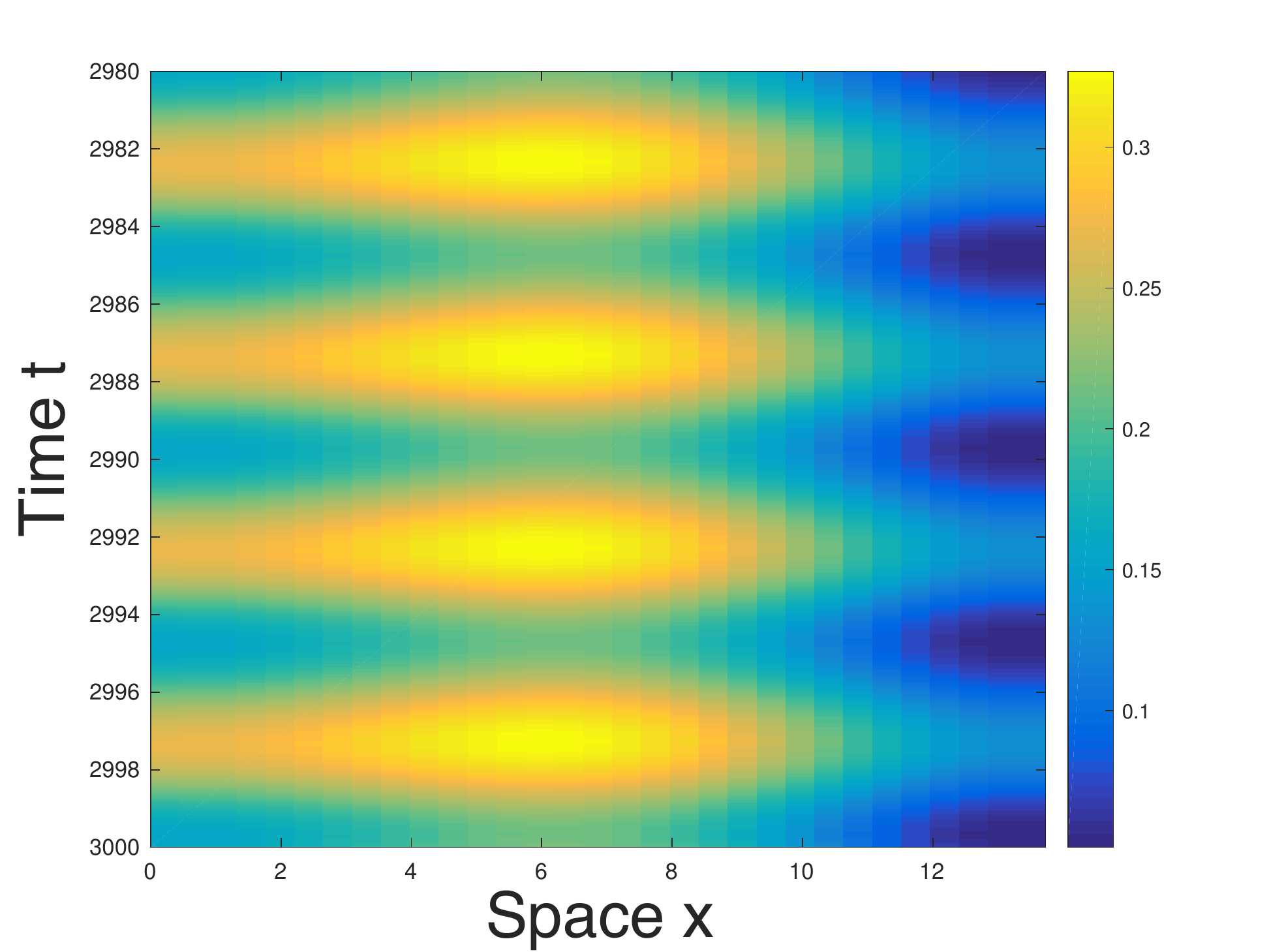}
\end{minipage}}
\subfigure[$u(x,3050)$]{\begin{minipage}{0.18\linewidth}
		\centering\includegraphics[height=0.85\linewidth,width=1.\linewidth]{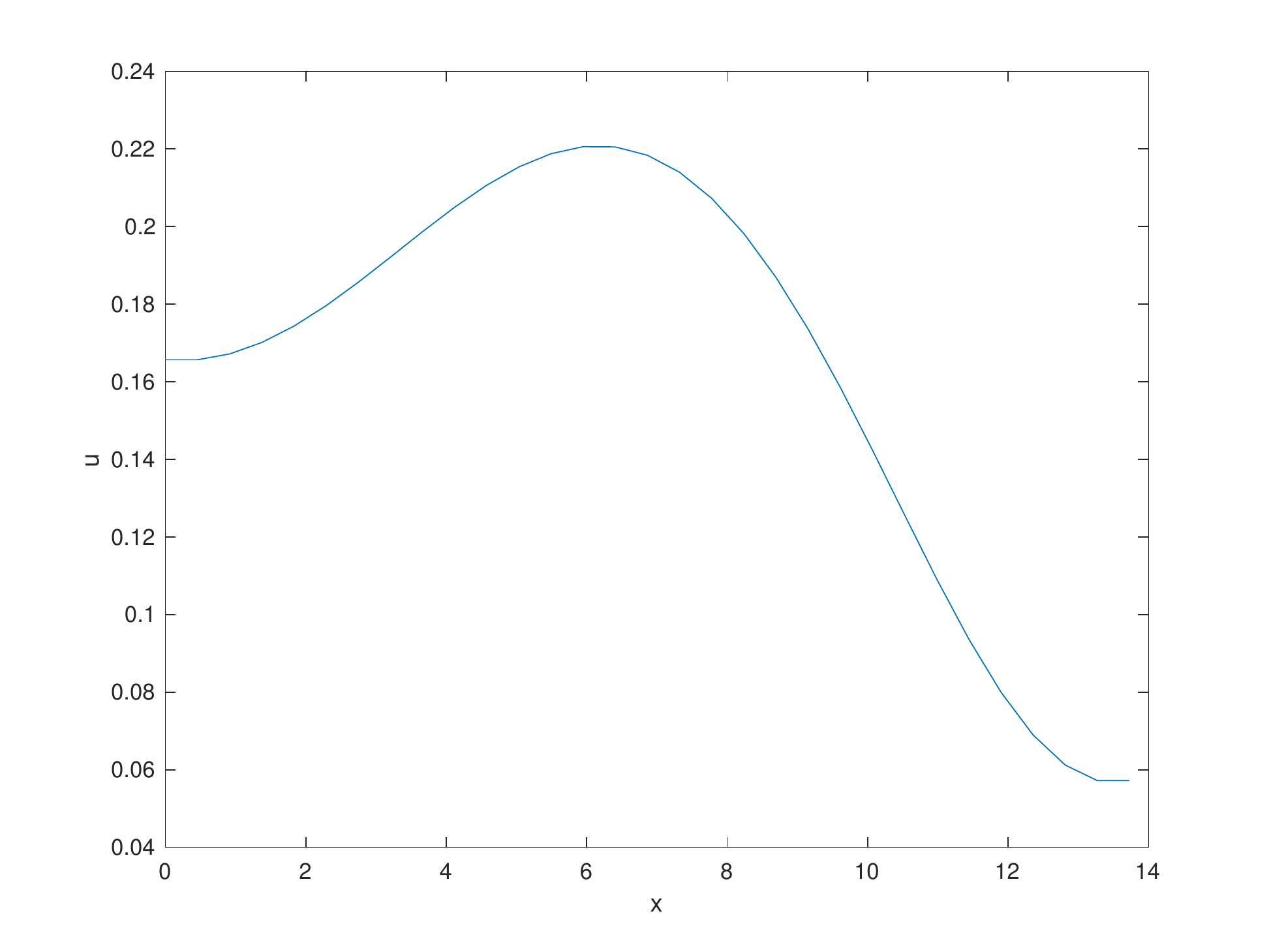}
\end{minipage}}
\subfigure[\tiny$\!\!cos(\frac{x}{l})\!-\!cos(\frac{2x}{l})$]{\begin{minipage}{0.18\linewidth}			 \centering\includegraphics[height=0.85\linewidth,width=1.\linewidth]{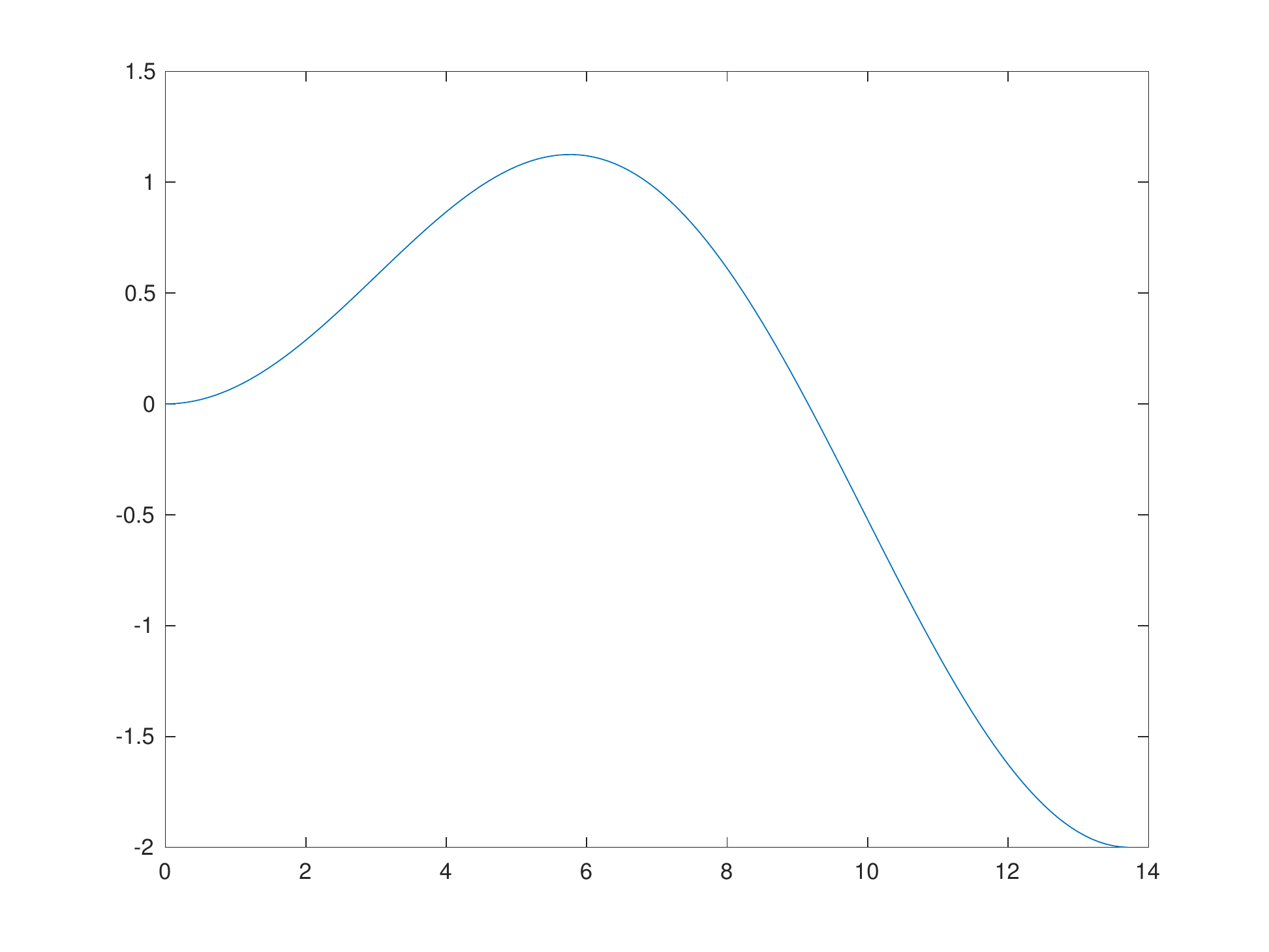}
\end{minipage}}
	
\subfigure[$v(x,t)$]{\begin{minipage}{0.2\linewidth}
		\centering\includegraphics[scale=0.15]{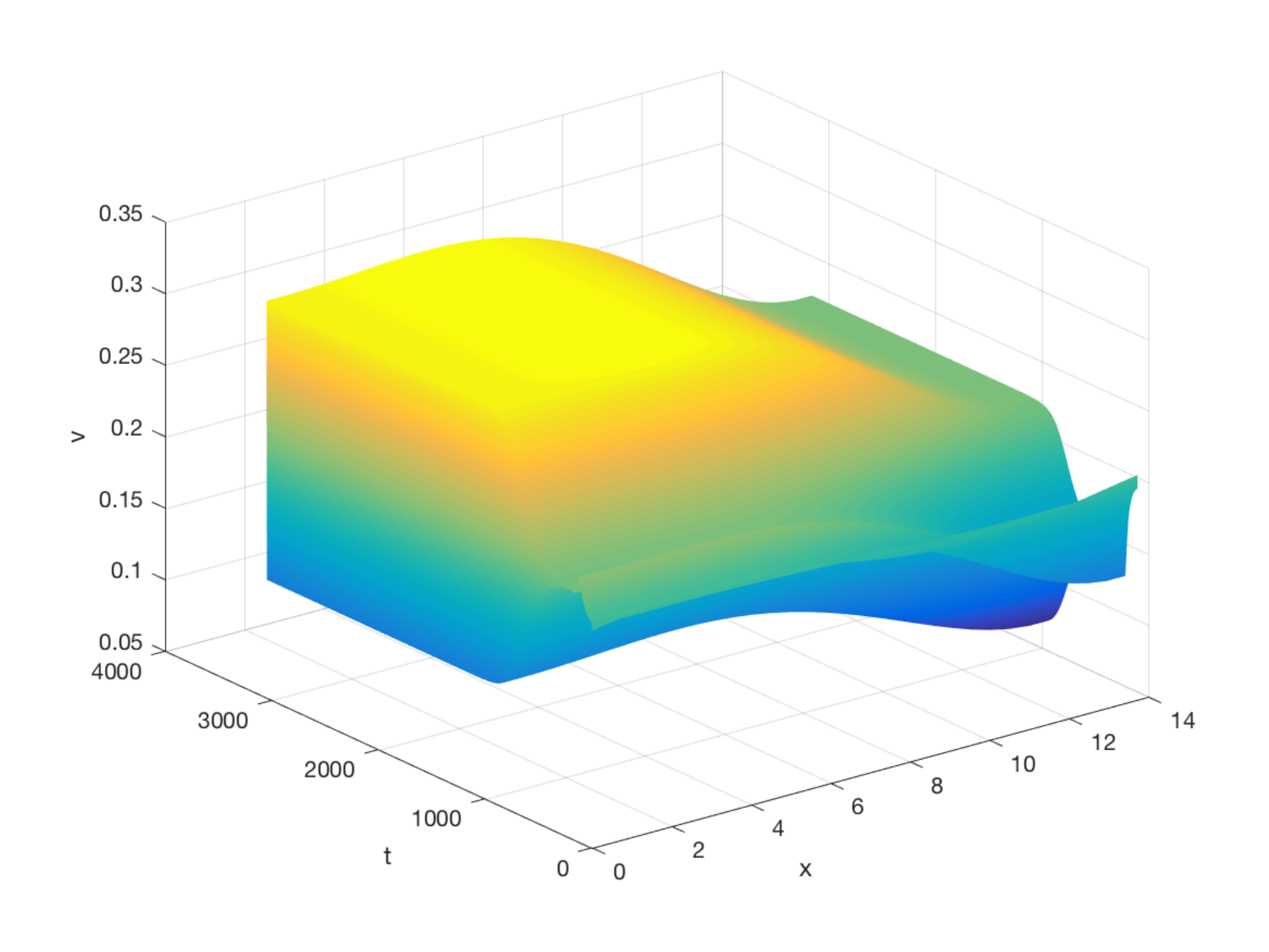}
\end{minipage}}
\subfigure[target pattern]{\begin{minipage}{0.2\linewidth}
		\centering\includegraphics[scale=0.15]{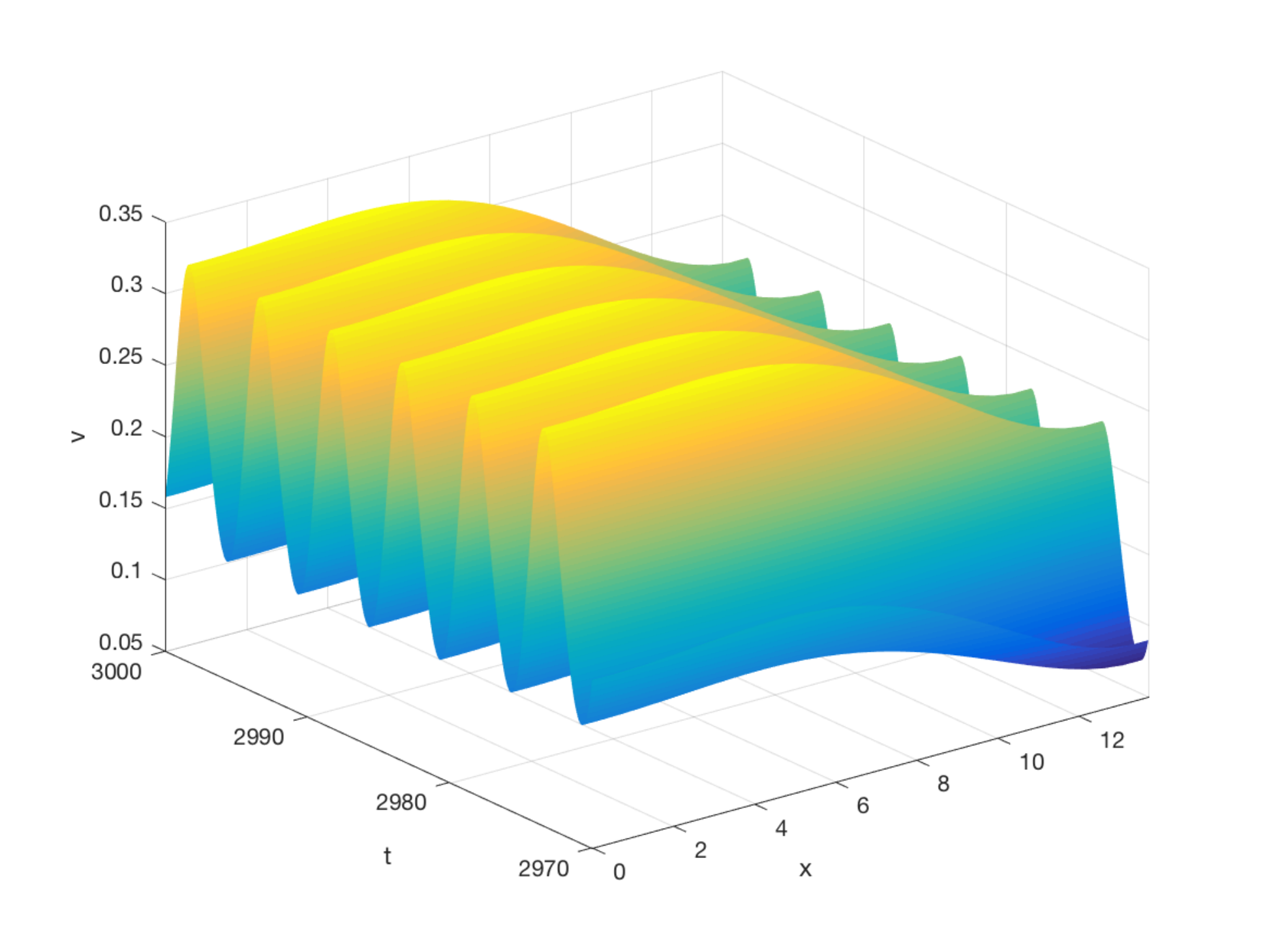}
\end{minipage}}
\subfigure[predator pattern]{\begin{minipage}{0.2\linewidth}
		\centering\includegraphics[height=0.8\linewidth,width=1.\linewidth]{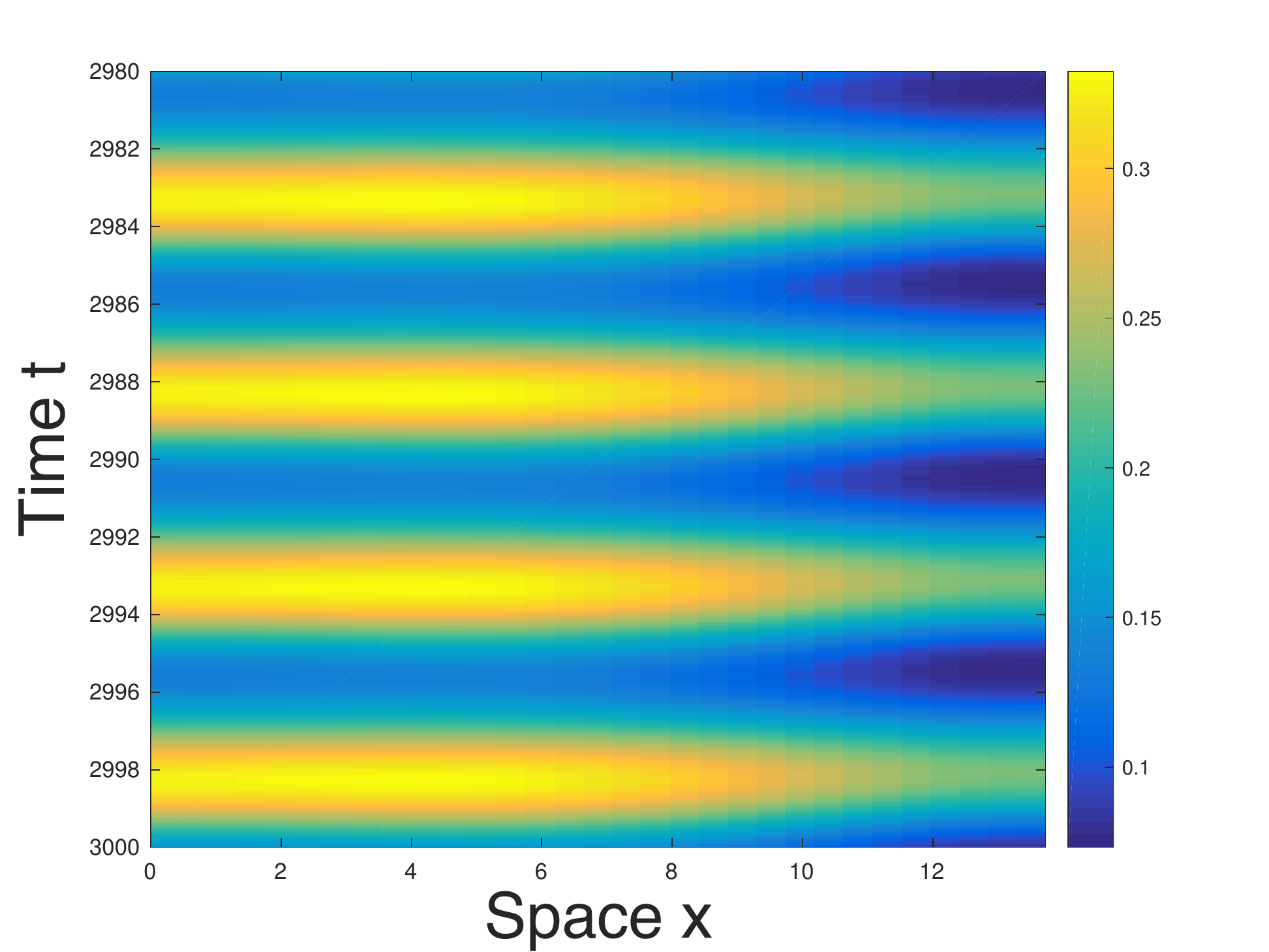}
\end{minipage}}
\subfigure[$v(x,3050)$]{\begin{minipage}{0.18\linewidth}
		\centering\includegraphics[height=0.85\linewidth,width=1.\linewidth]{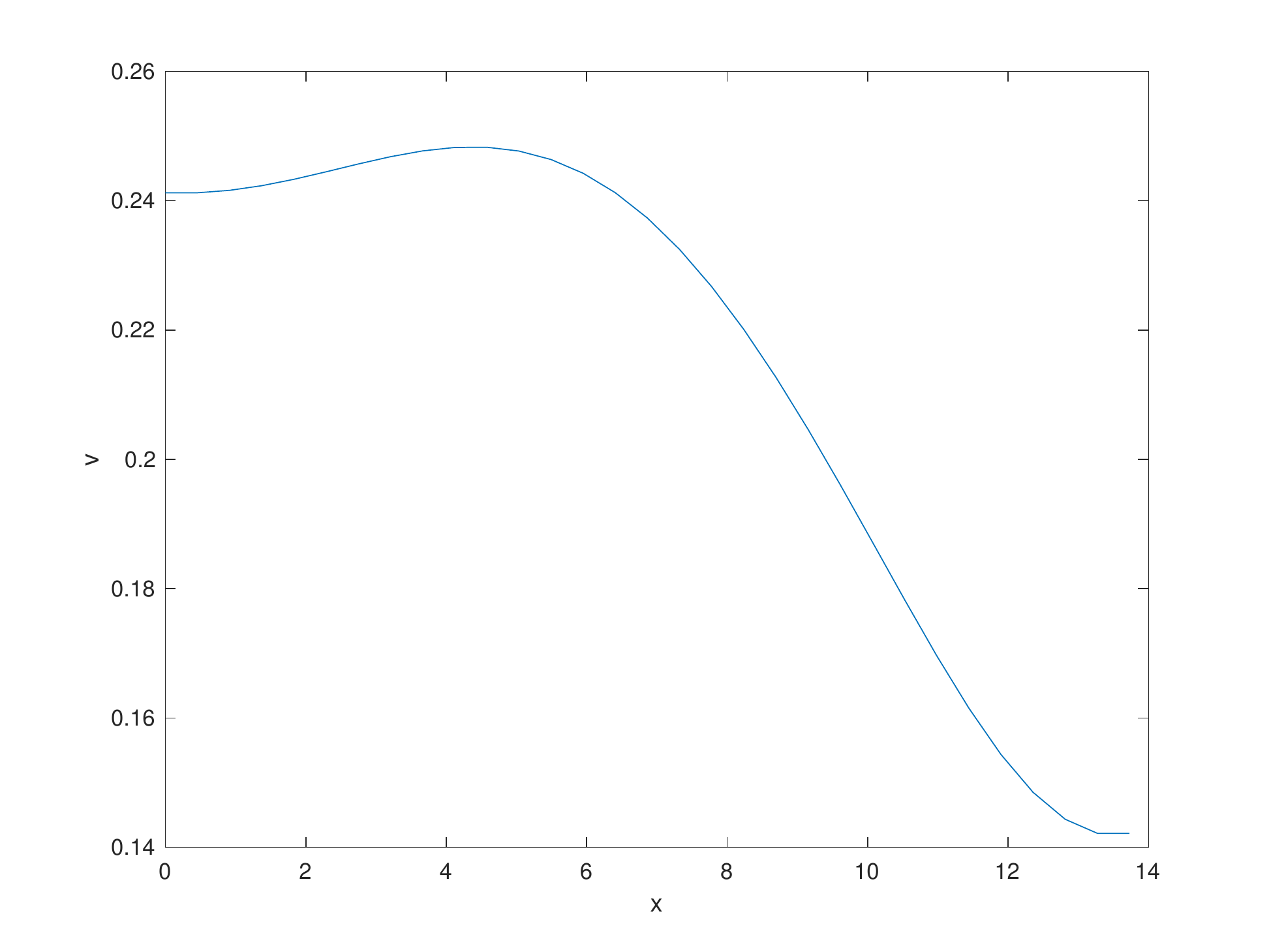}
\end{minipage}}
\subfigure[\tiny$2cos(\frac{x}{l})\!-\!cos(\frac{2x}{l})$]{\begin{minipage}{0.18\linewidth}			\centering\includegraphics[height=0.85\linewidth,width=1.\linewidth]{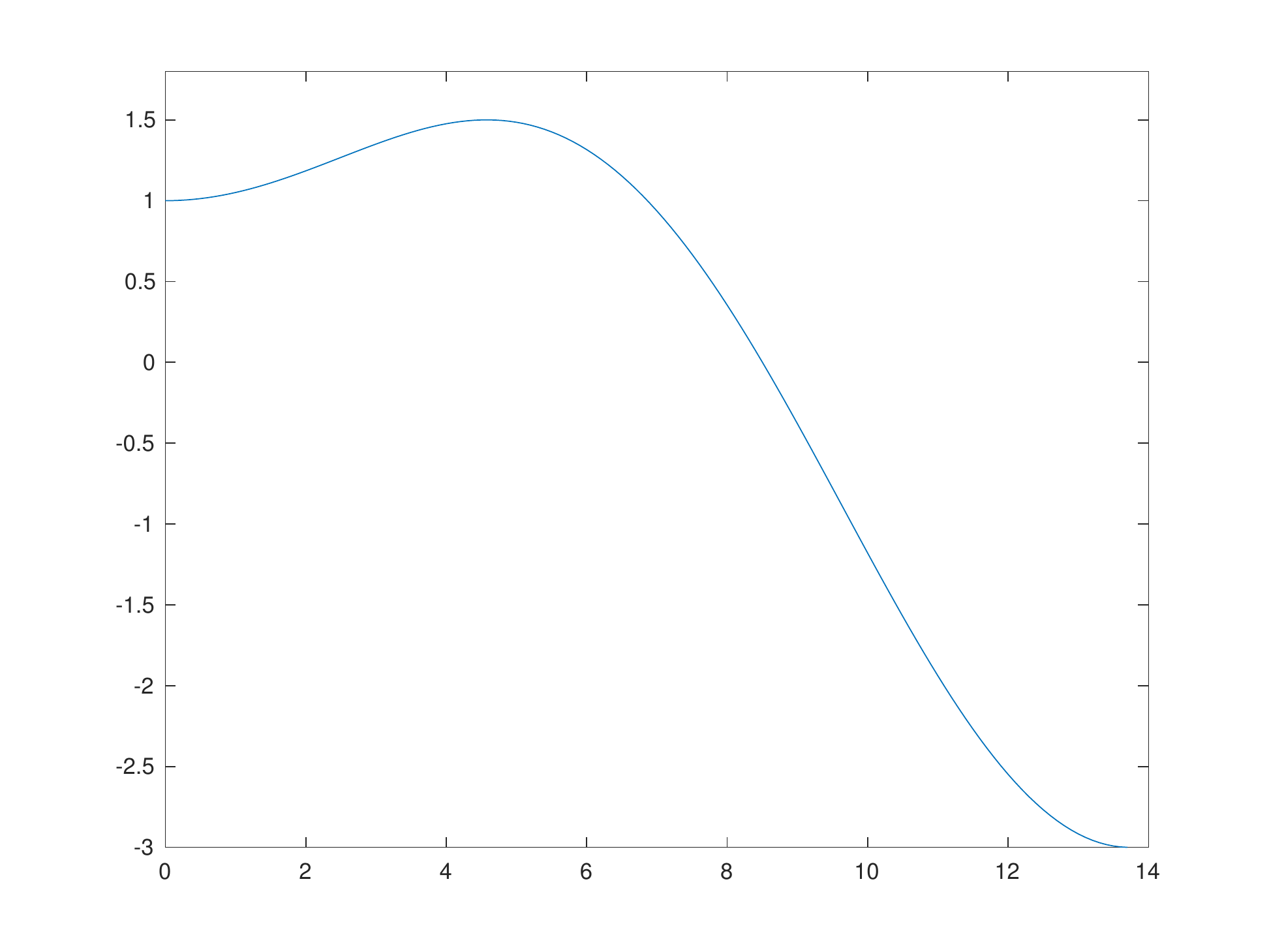}
\end{minipage}}
\caption{Spatially non-homogeneous periodic solution in $D_3$, with $(\alpha_1,\alpha_2)=(-0.05,0.02)$ and initial functions are $(u_0+0.01\sin 0.5x,u_0+0.01\sin 0.5x)$}\label{fig2D3_1}
\end{figure}

\begin{figure}[htbp]
	\subfigure[$u(x,t)$]{\begin{minipage}{0.2\linewidth}
			\centering\includegraphics[scale=0.15]{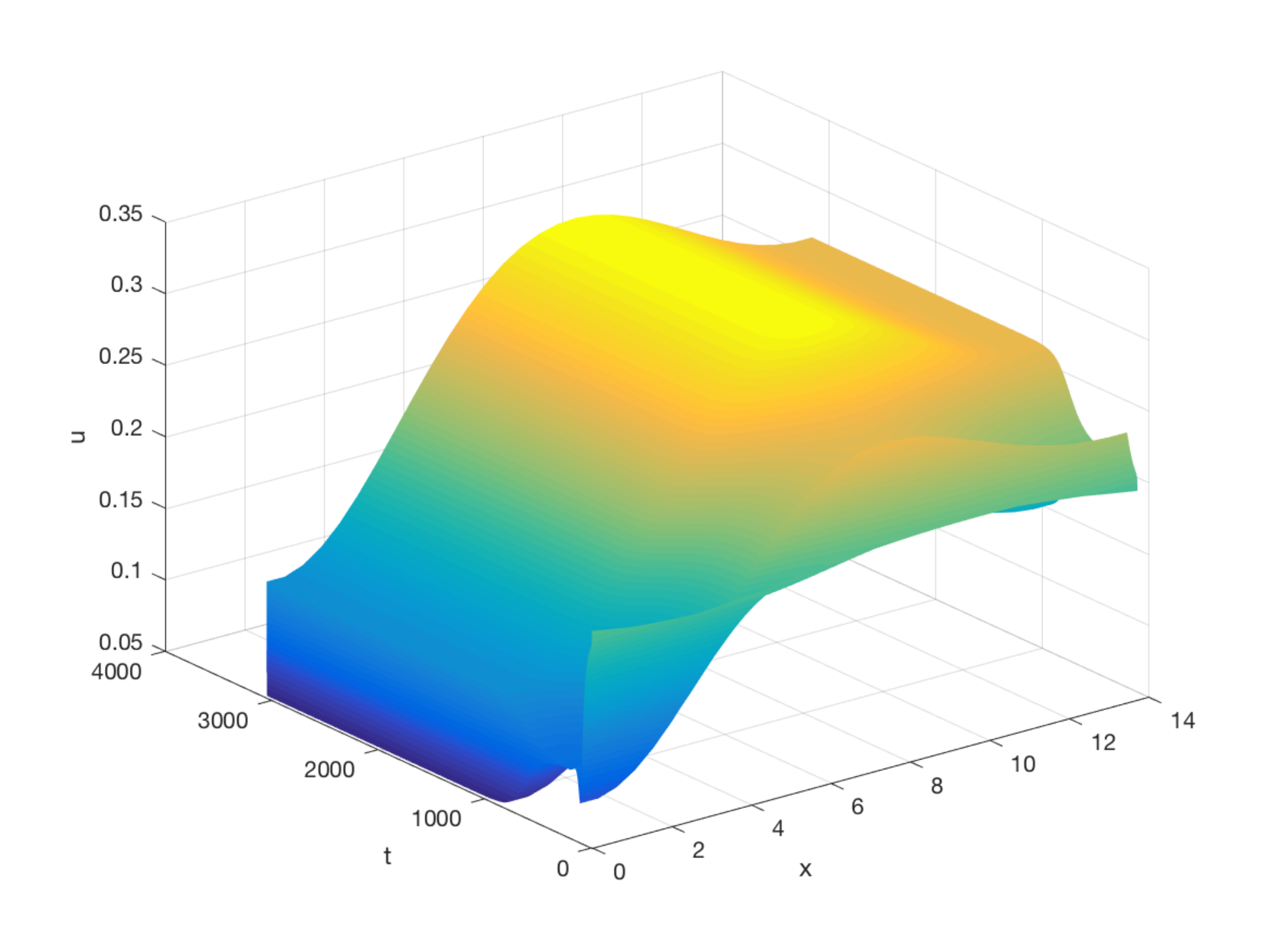}
	\end{minipage}}
	\subfigure[target pattern]{\begin{minipage}{0.2\linewidth}
			\centering\includegraphics[scale=0.15]{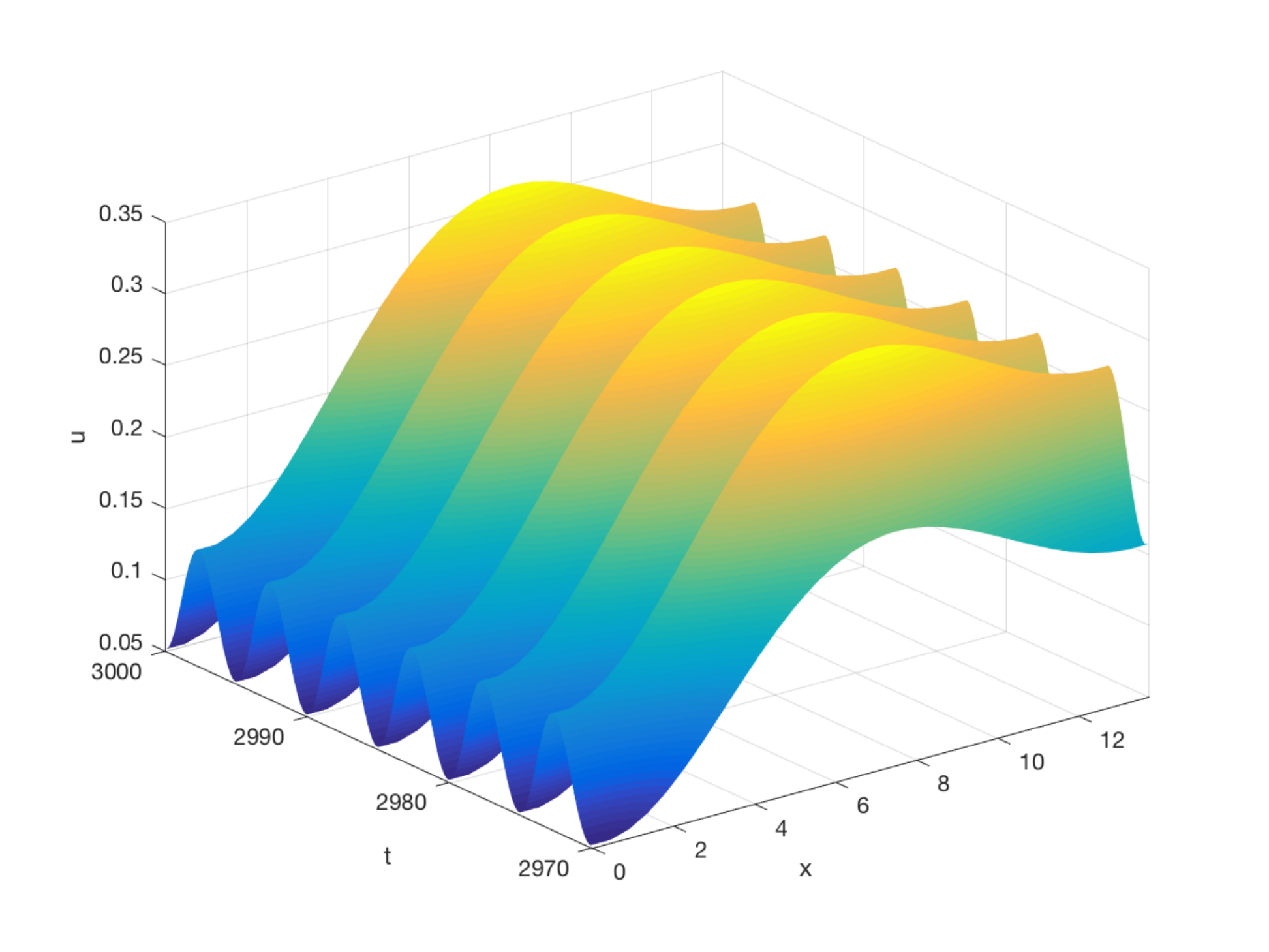}
	\end{minipage}}
    \subfigure[prey pattern]{\begin{minipage}{0.2\linewidth}
		\centering\includegraphics[height=0.8\linewidth,width=1.\linewidth]{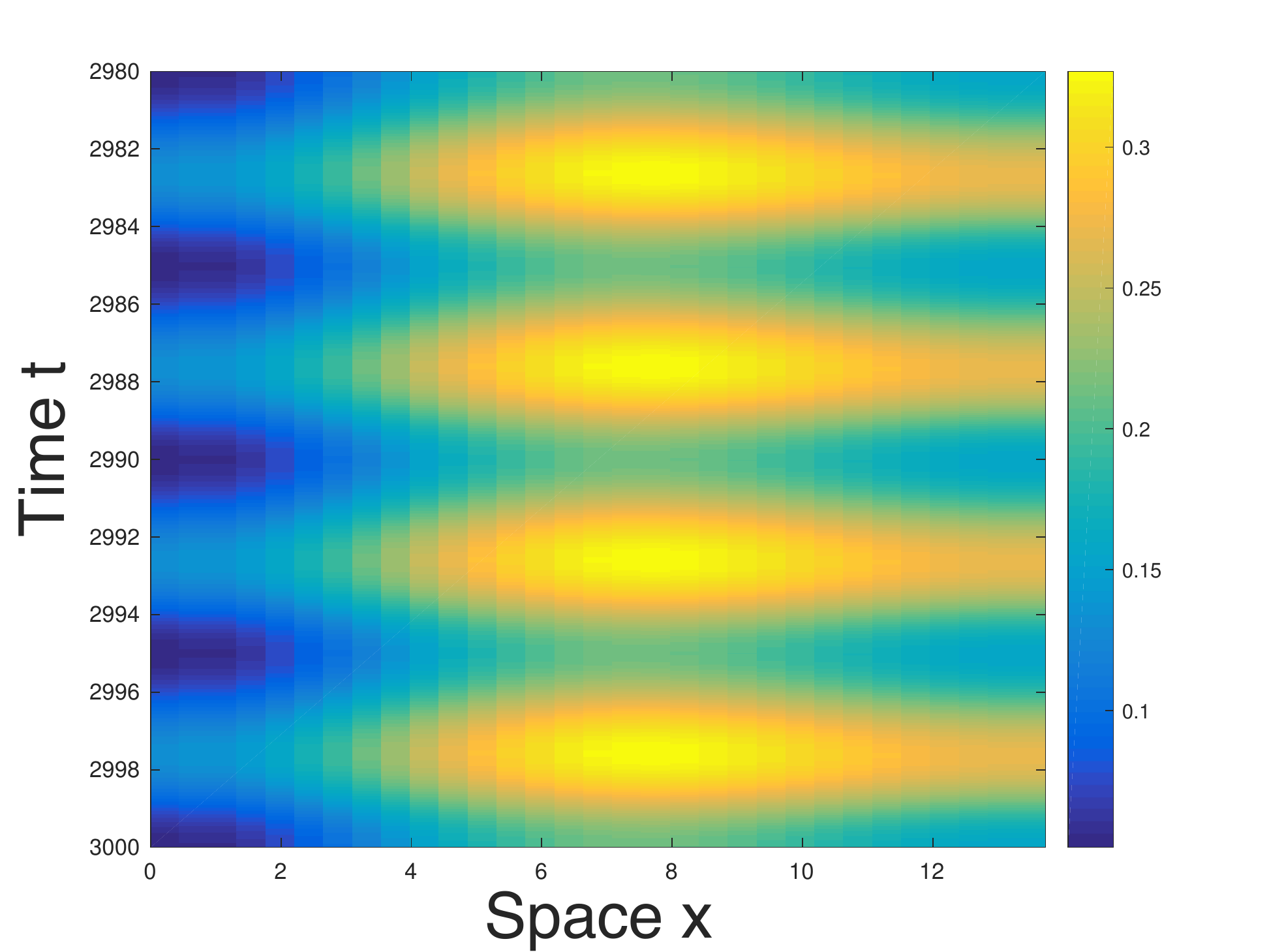}
    \end{minipage}}
	\subfigure[$u(x,3050)$]{\begin{minipage}{0.18\linewidth}
			\centering\includegraphics[height=0.85\linewidth,width=1.\linewidth]{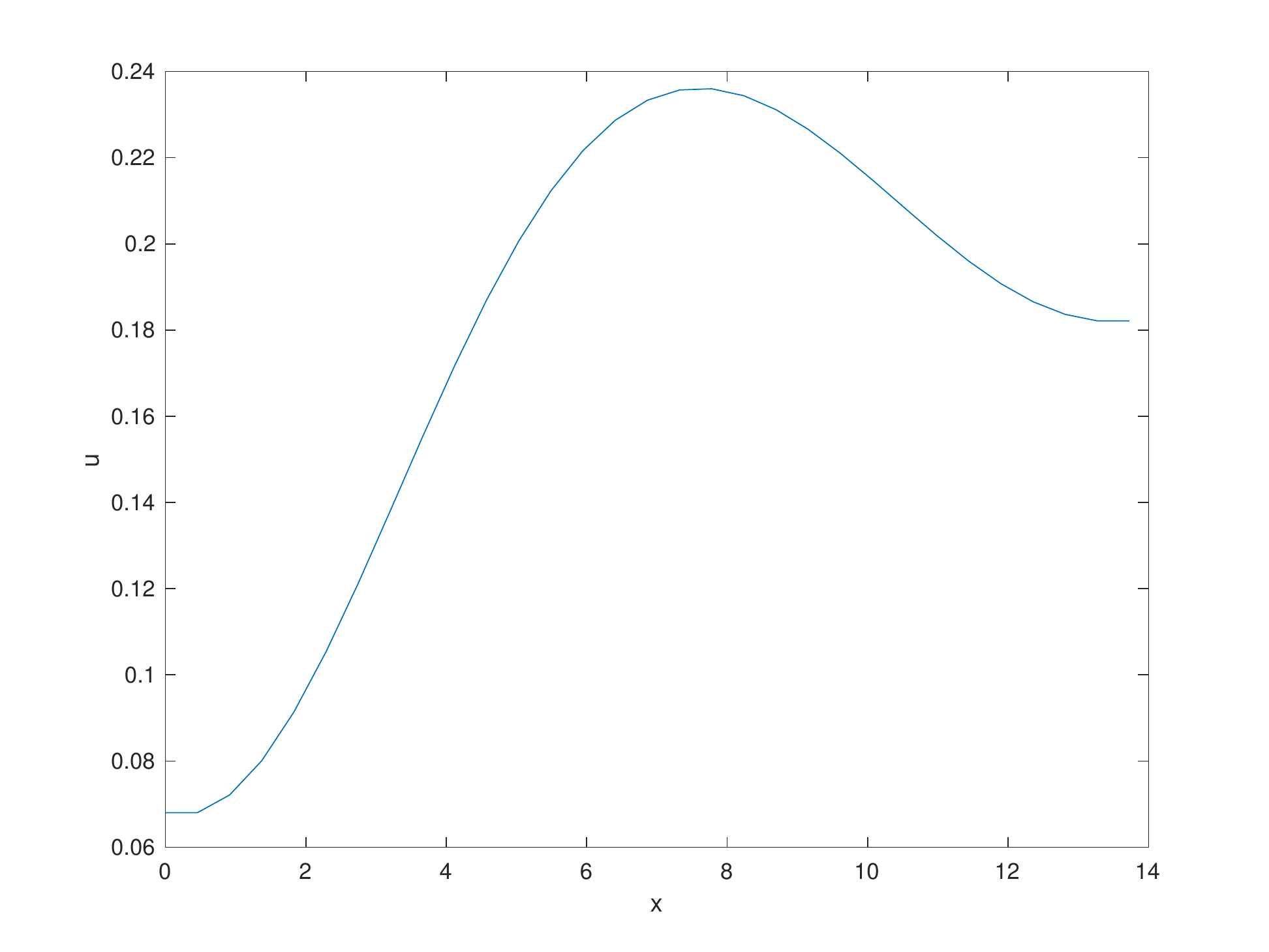}
	\end{minipage}}
	\subfigure[\tiny$\!-\!cos(\frac{x}{l})\!-\!cos(\frac{2x}{l})$]{\begin{minipage}{0.18\linewidth}			 \centering\includegraphics[height=0.85\linewidth,width=1.\linewidth]{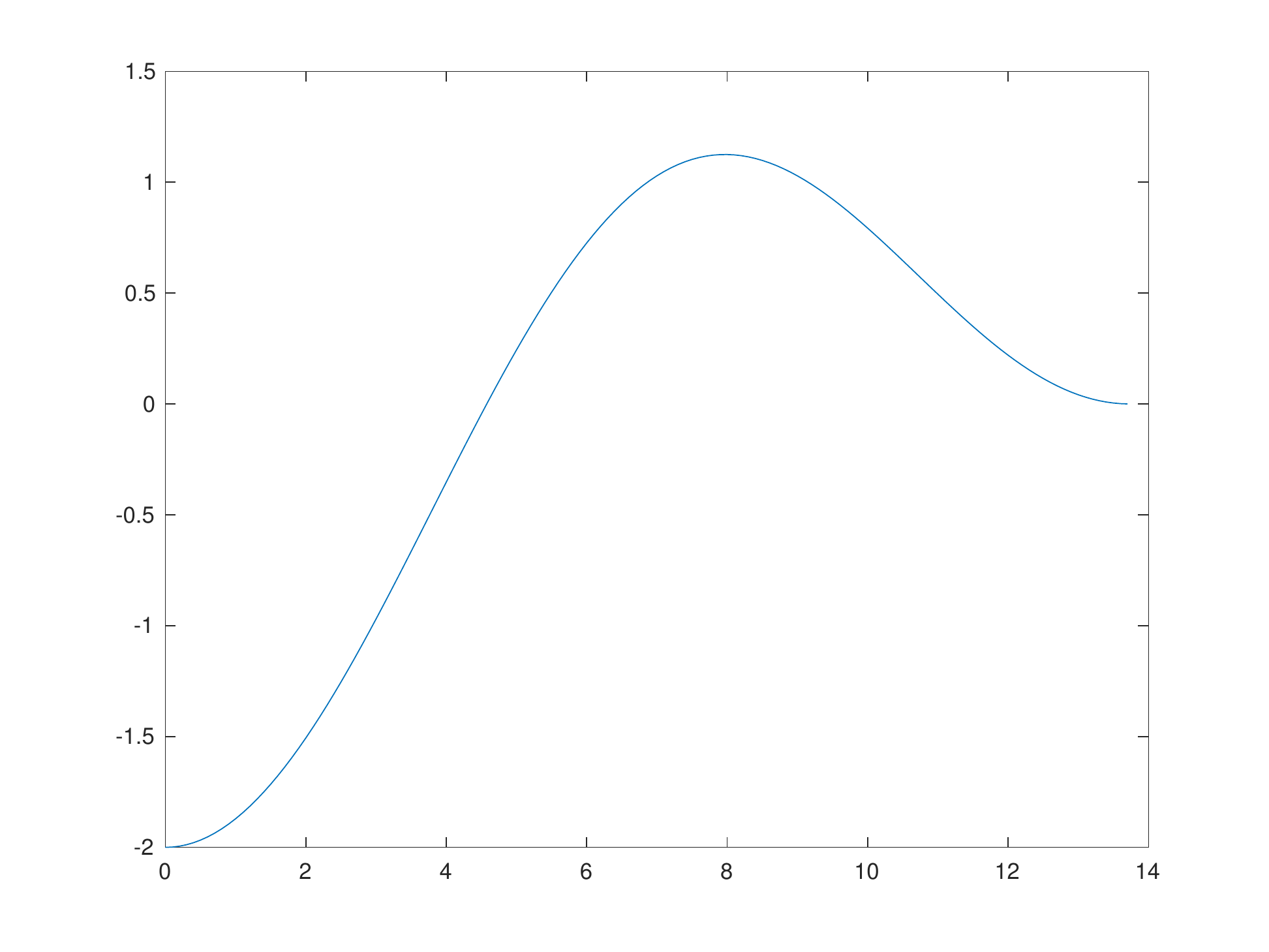}
	\end{minipage}}
	
\subfigure[$v(x,t)$]{\begin{minipage}{0.2\linewidth}
		\centering\includegraphics[scale=0.15]{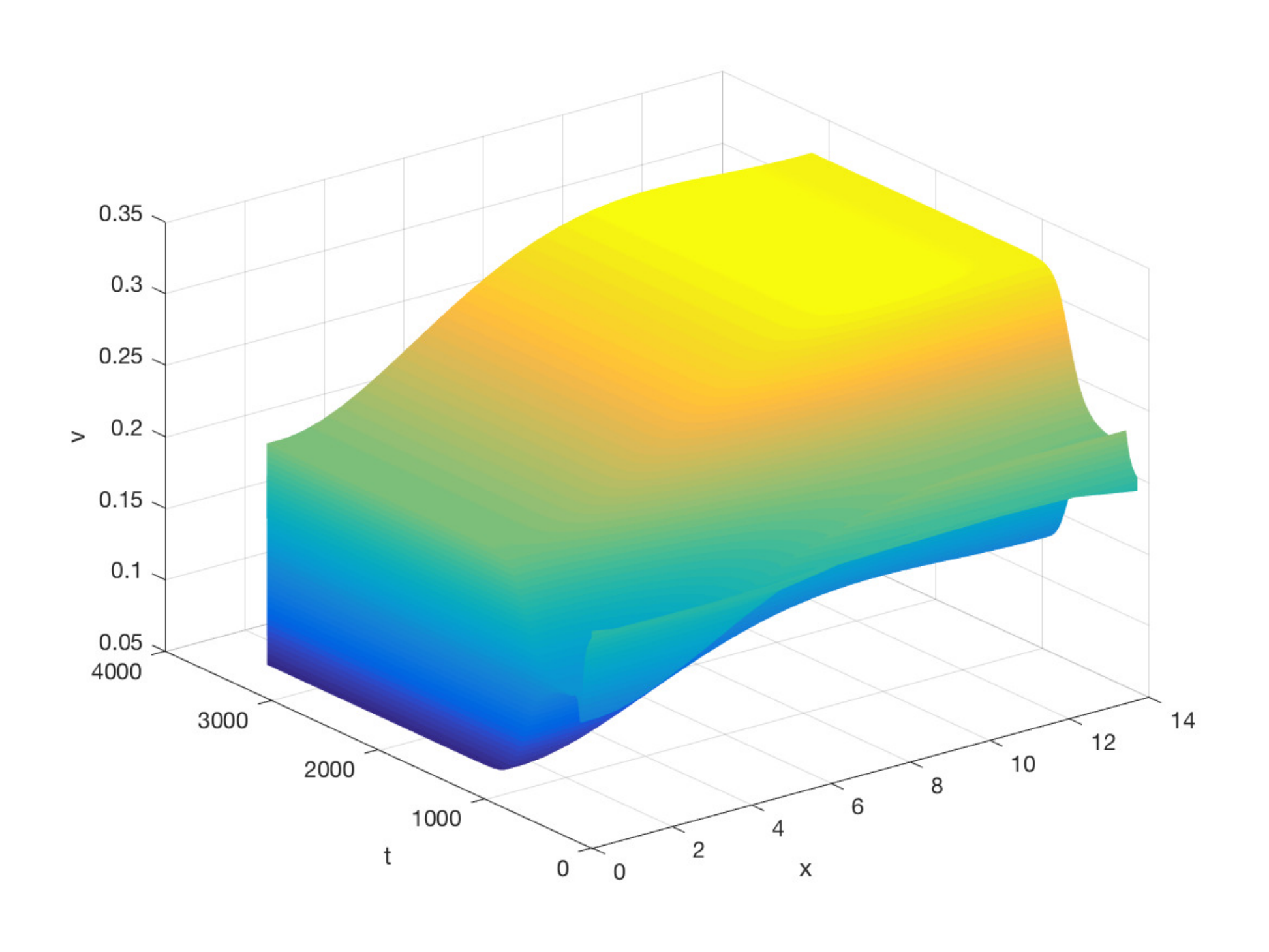}
\end{minipage}}
\subfigure[target pattern]{\begin{minipage}{0.2\linewidth}
		\centering\includegraphics[scale=0.15]{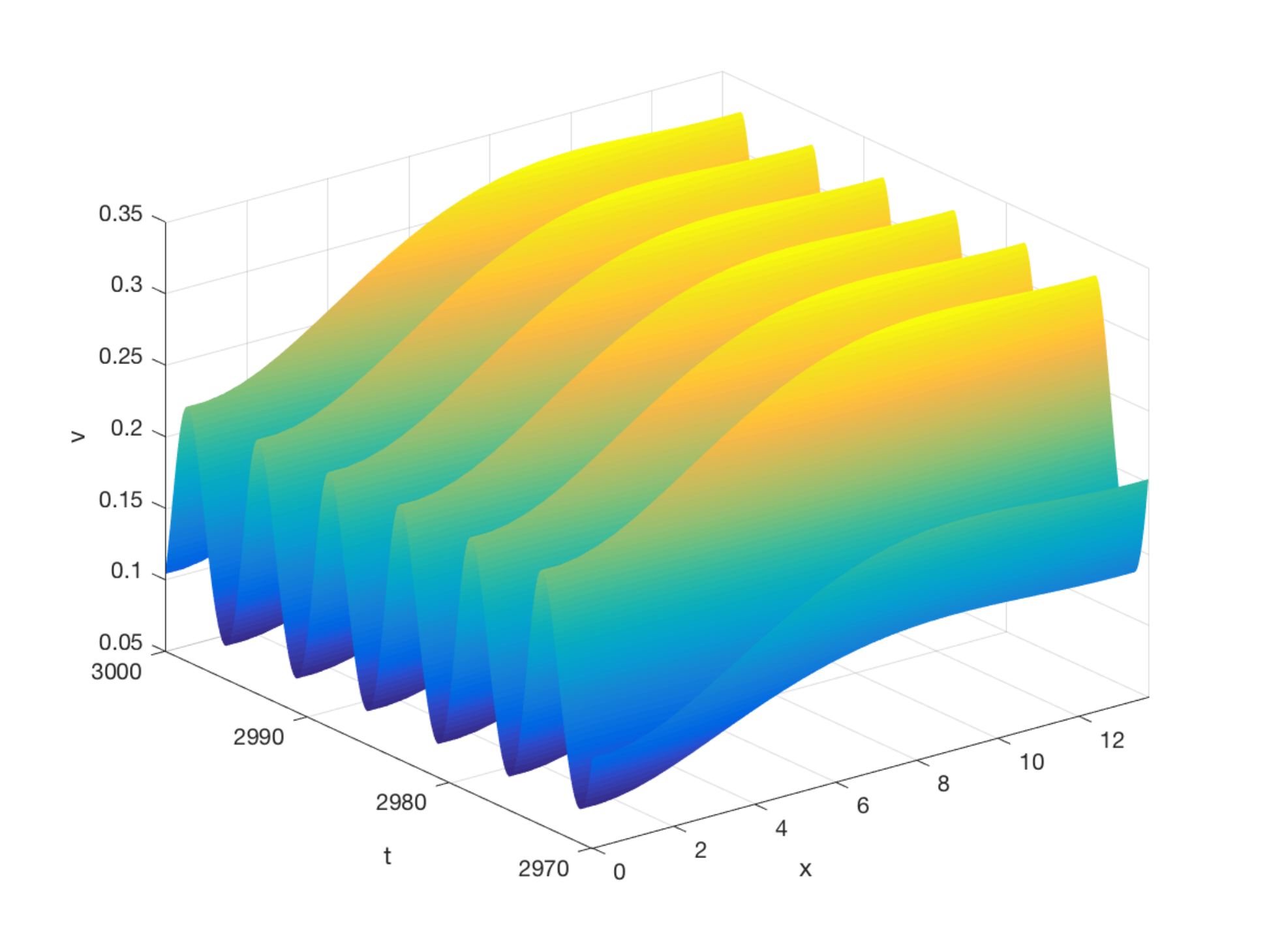}
\end{minipage}}
\subfigure[predator pattern]{\begin{minipage}{0.2\linewidth}
		\centering\includegraphics[height=0.8\linewidth,width=1.\linewidth]{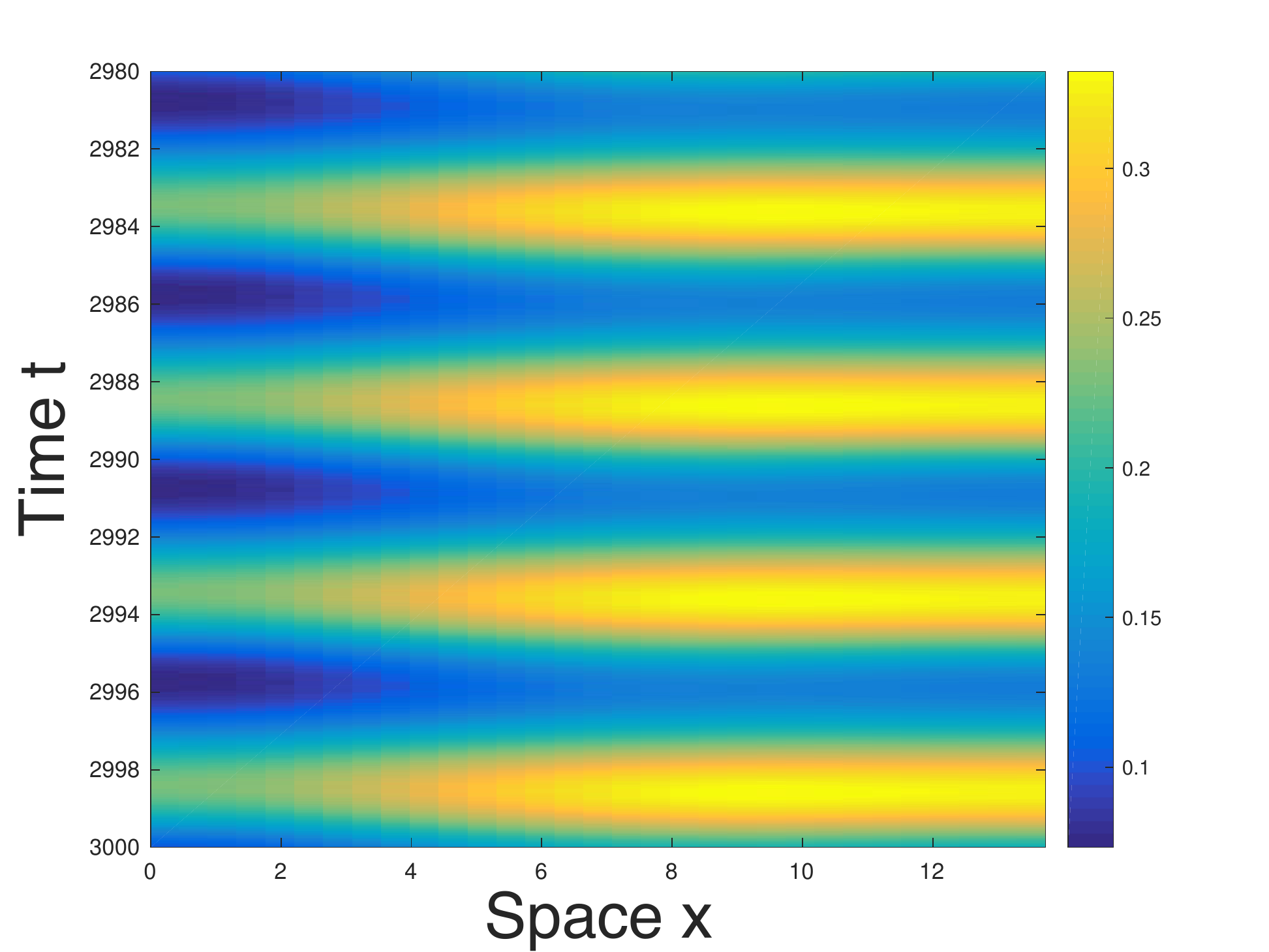}
\end{minipage}}
\subfigure[$v(x,3050)$]{\begin{minipage}{0.18\linewidth}
		\centering\includegraphics[height=0.85\linewidth,width=1.\linewidth]{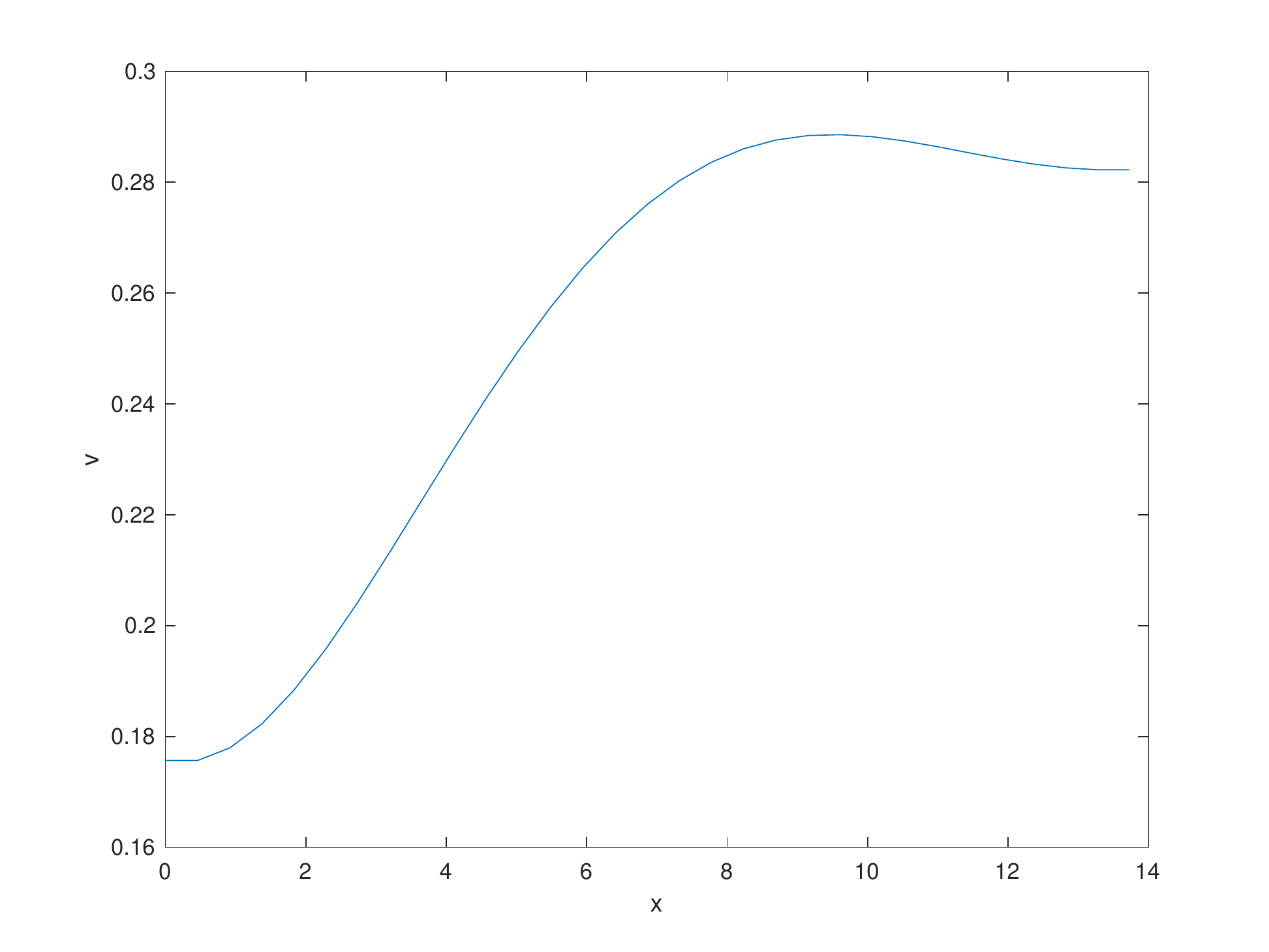}
\end{minipage}}
\subfigure[\tiny$\!-2\!cos(\frac{x}{l})\!-\!cos(\frac{2x}{l})$]{\begin{minipage}{0.18\linewidth}			 \centering\includegraphics[height=0.85\linewidth,width=1.\linewidth]{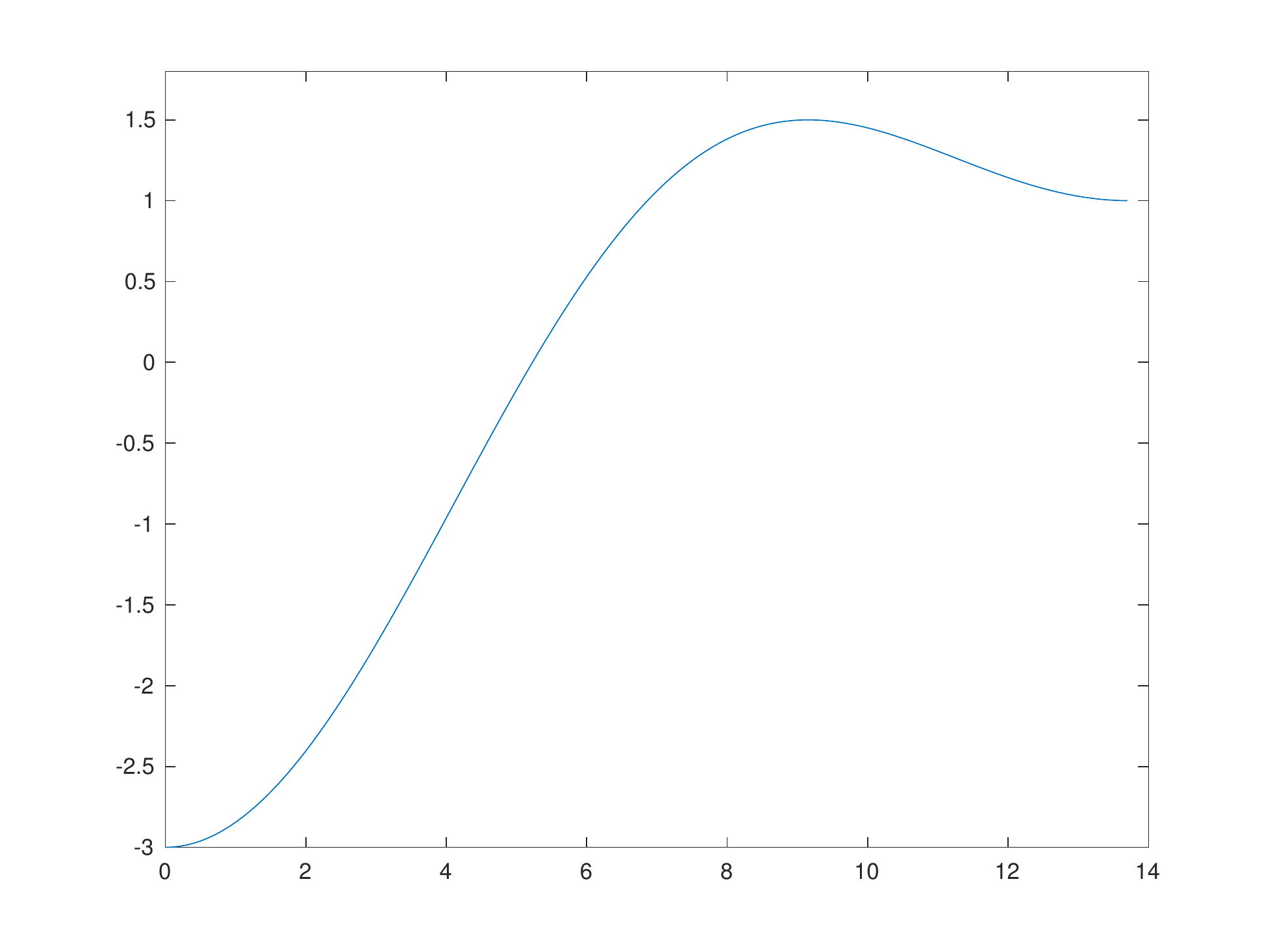}
\end{minipage}}
	\caption{Spatially non-homogeneous periodic solution in $D_3$, with $(\alpha_1,\alpha_2)=(-0.05,0.02)$ and initial functions are $(u_0-0.01\sin 0.5x,u_0-0.01\sin 0.5x)$}\label{fig2D3_2}
\end{figure}
\begin{figure}[htbp]
	\subfigure[\tiny$u(x,t)$]{\begin{minipage}{0.16\linewidth}
		\centering\includegraphics[scale=0.13]{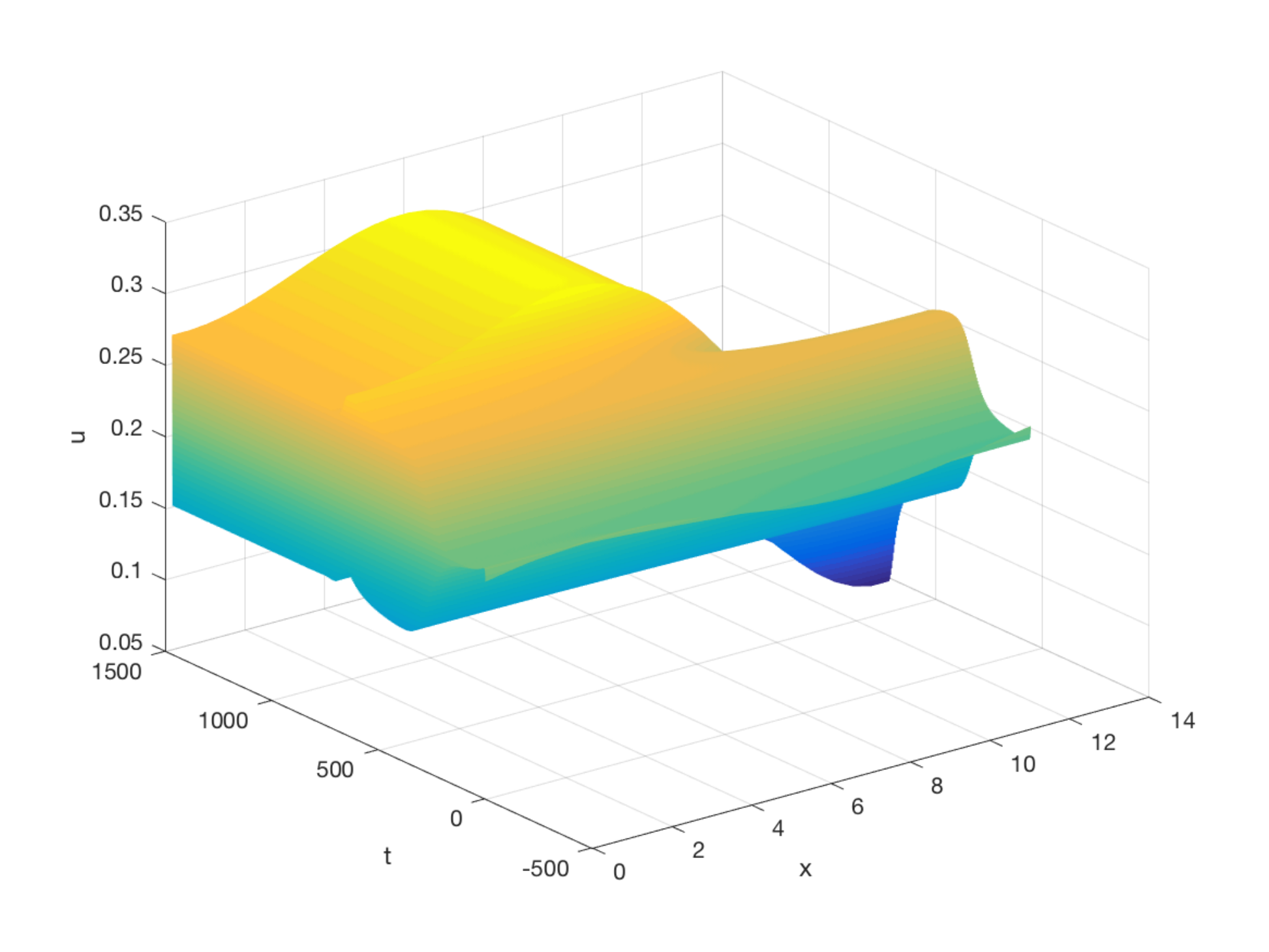}
\end{minipage}}
\subfigure[\tiny prey pattern]{\begin{minipage}{0.16\linewidth}
		\centering\includegraphics[scale=0.13]{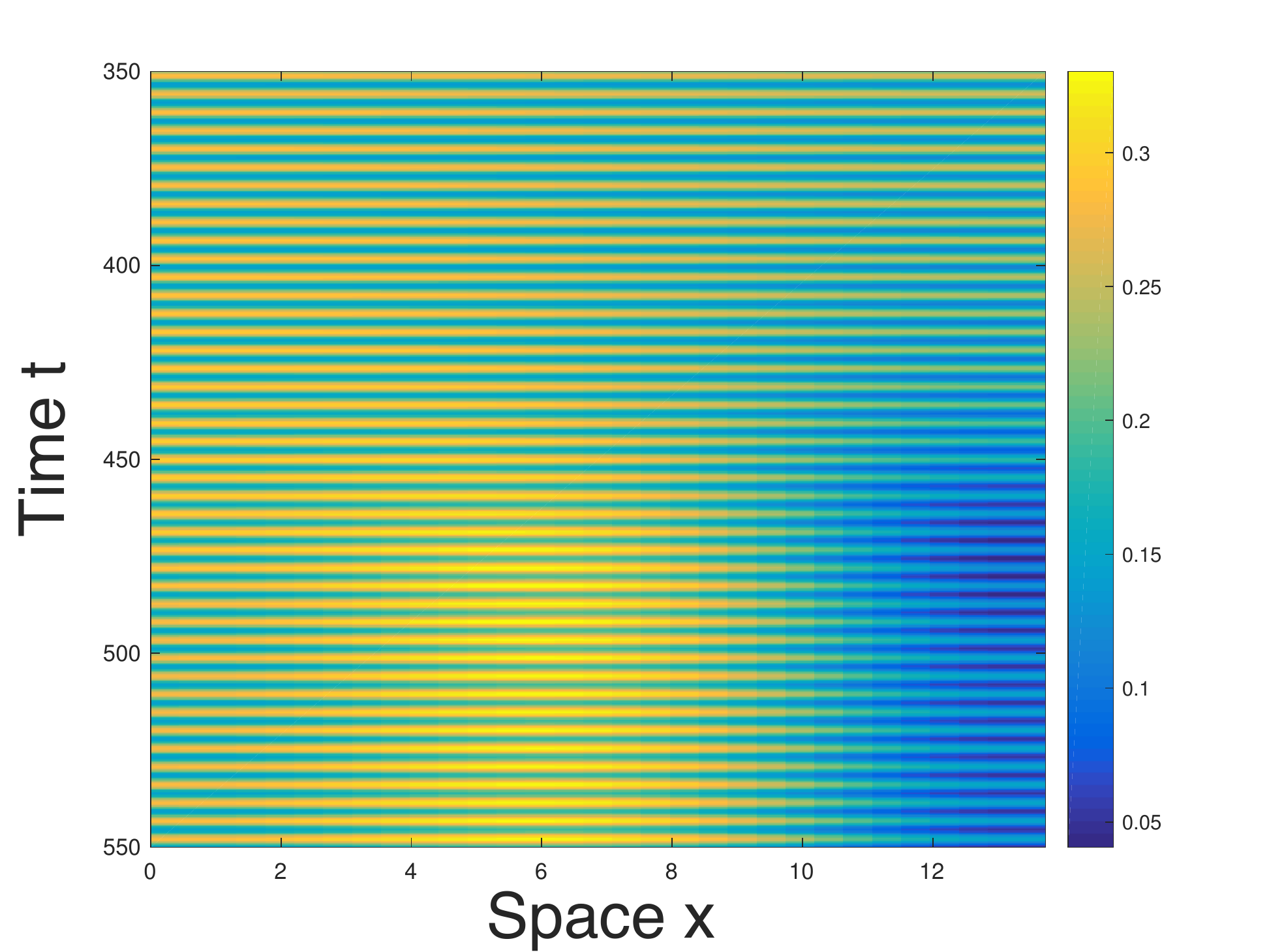}
\end{minipage}}
\subfigure[\tiny transition]{\begin{minipage}{0.16\linewidth}
		\centering\includegraphics[scale=0.13]{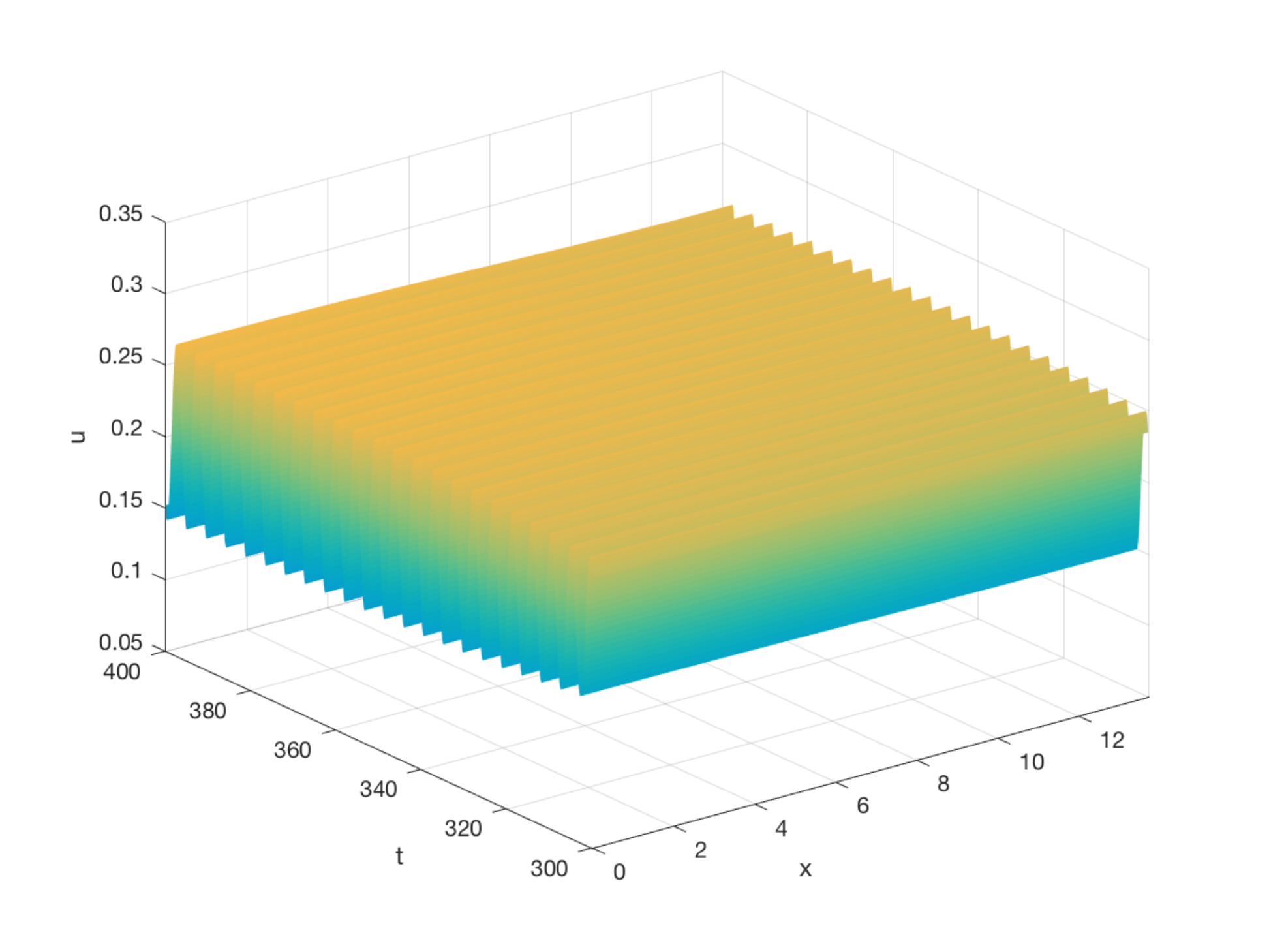}
\end{minipage}}
\subfigure[\tiny prey pattern]{\begin{minipage}{0.16\linewidth}
		\centering\includegraphics[scale=0.13]{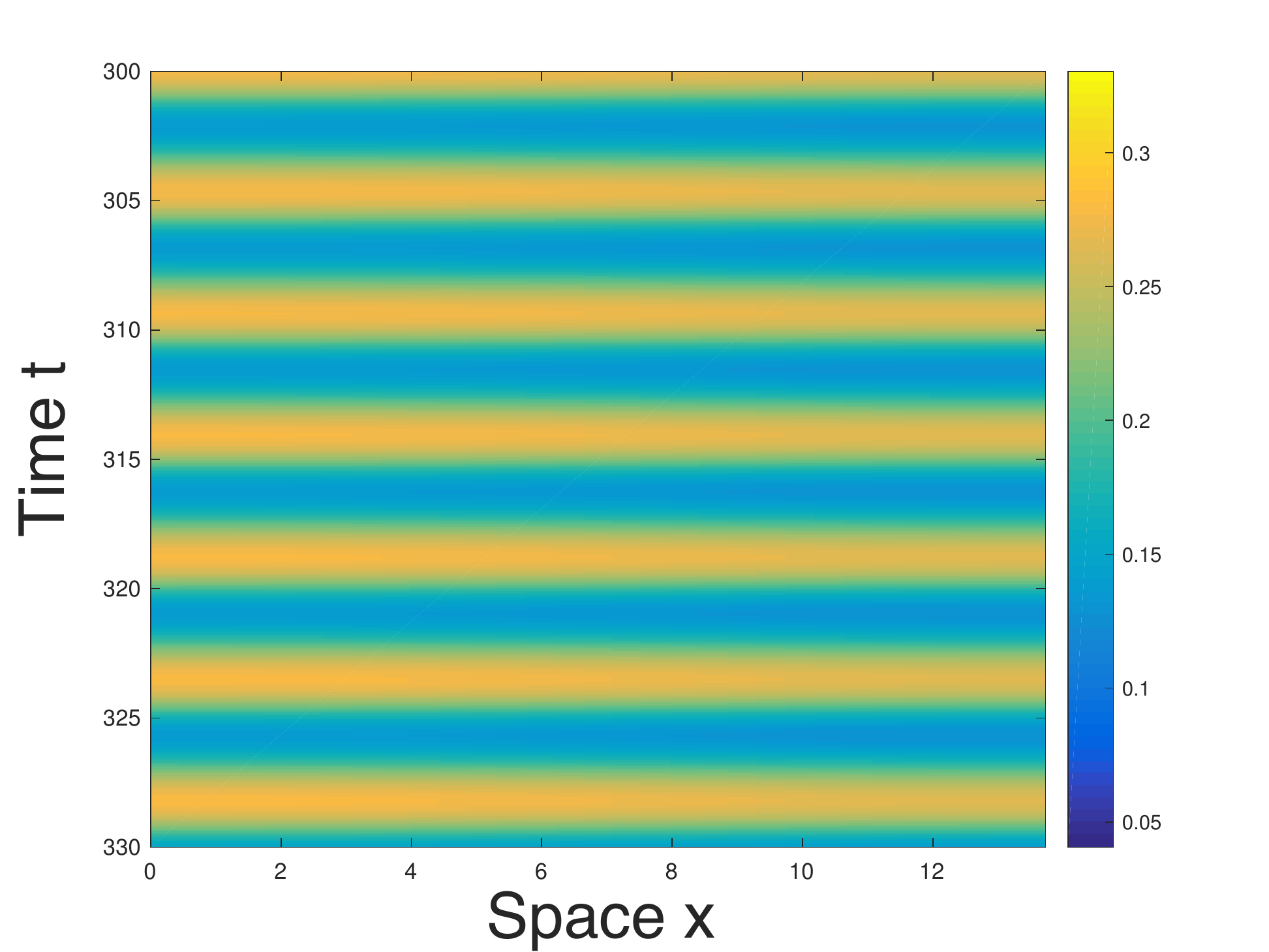}
\end{minipage}}
\subfigure[\tiny target]{\begin{minipage}{0.16\linewidth}
		\centering\includegraphics[scale=0.13]{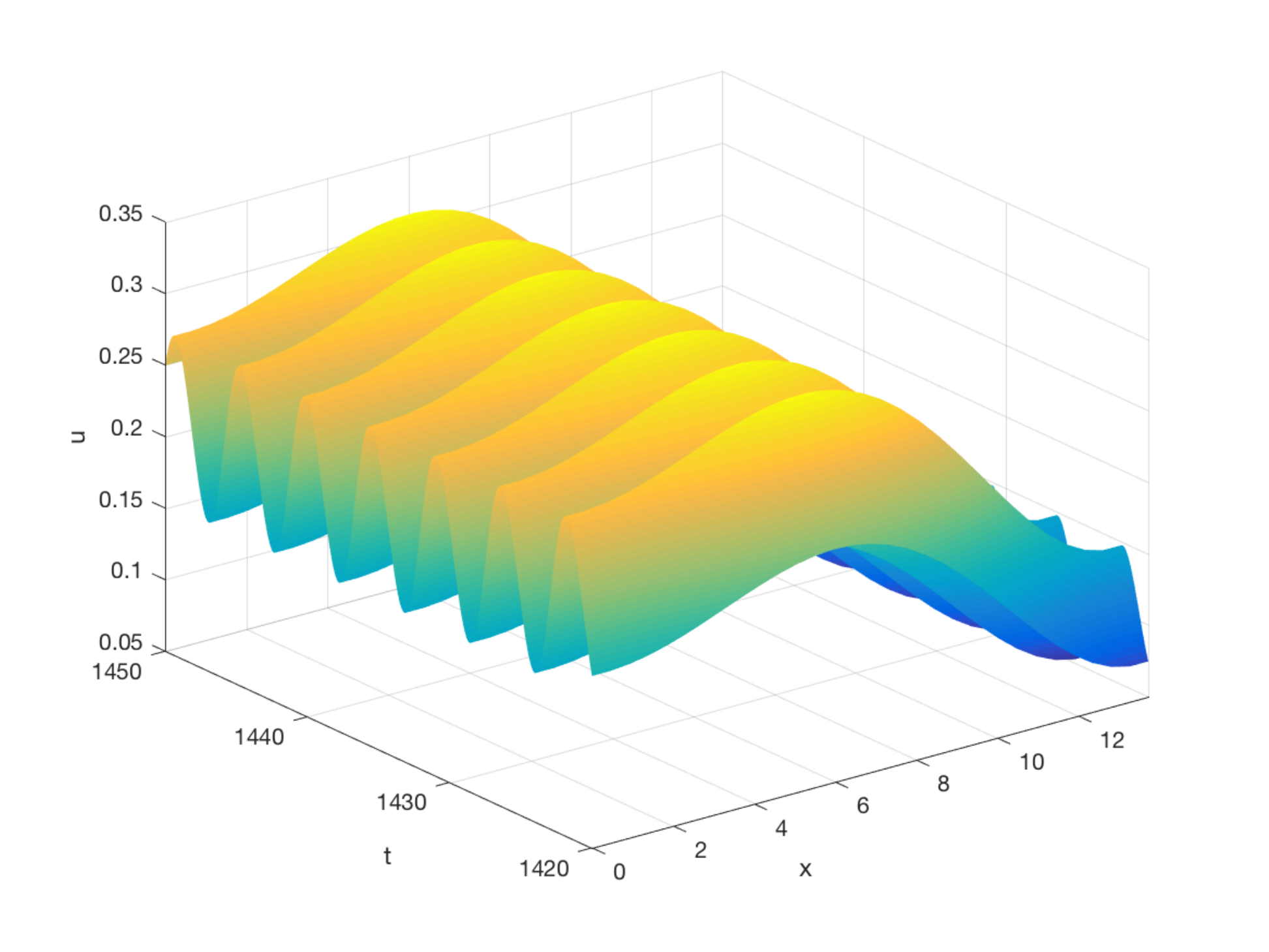}
\end{minipage}}
\subfigure[\tiny prey pattern]{\begin{minipage}{0.16\linewidth}
		\centering\includegraphics[scale=0.13]{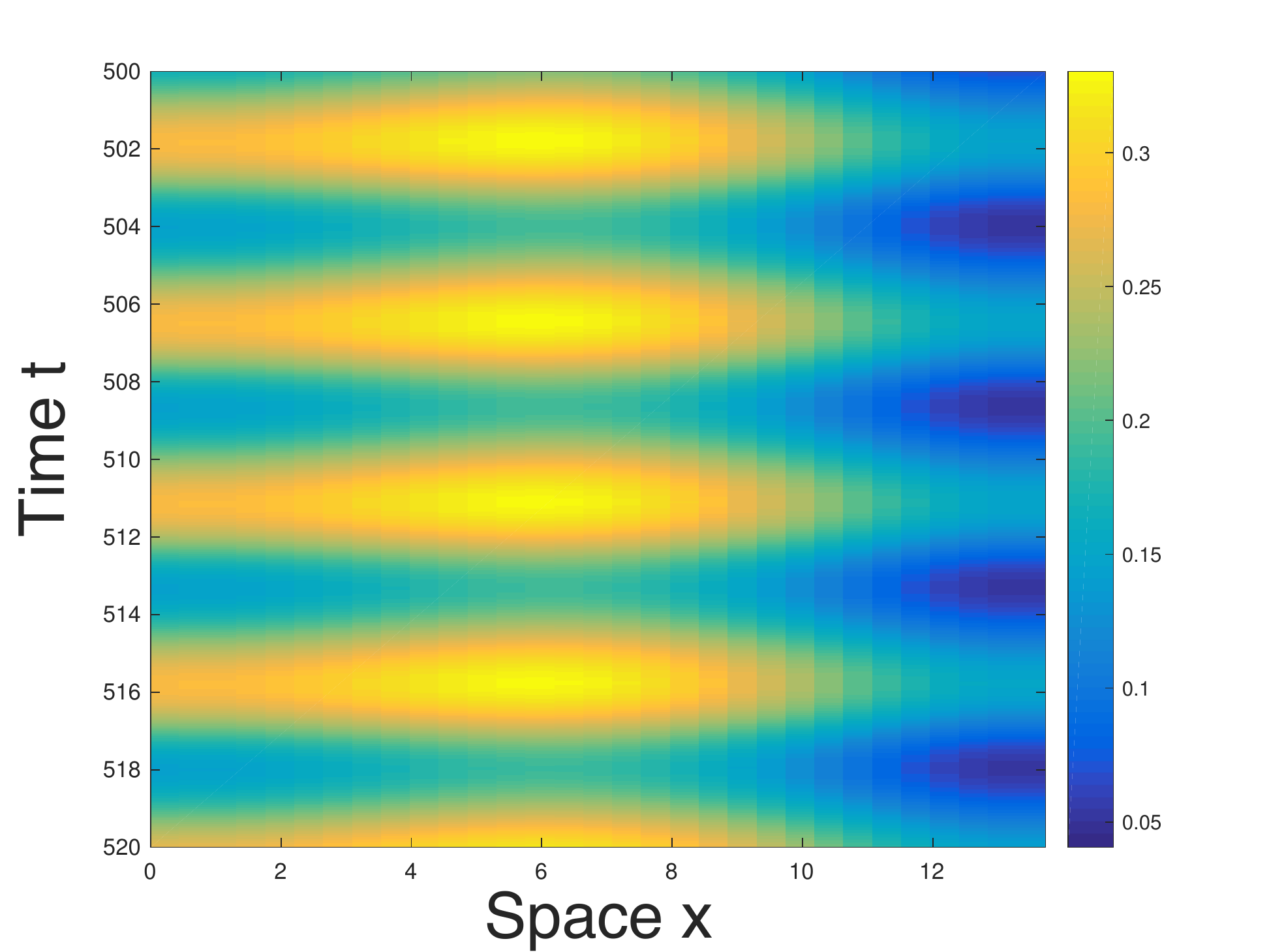}
\end{minipage}}

\subfigure[\tiny$v(x,t)$]{\begin{minipage}{0.16\linewidth}
		\centering\includegraphics[scale=0.13]{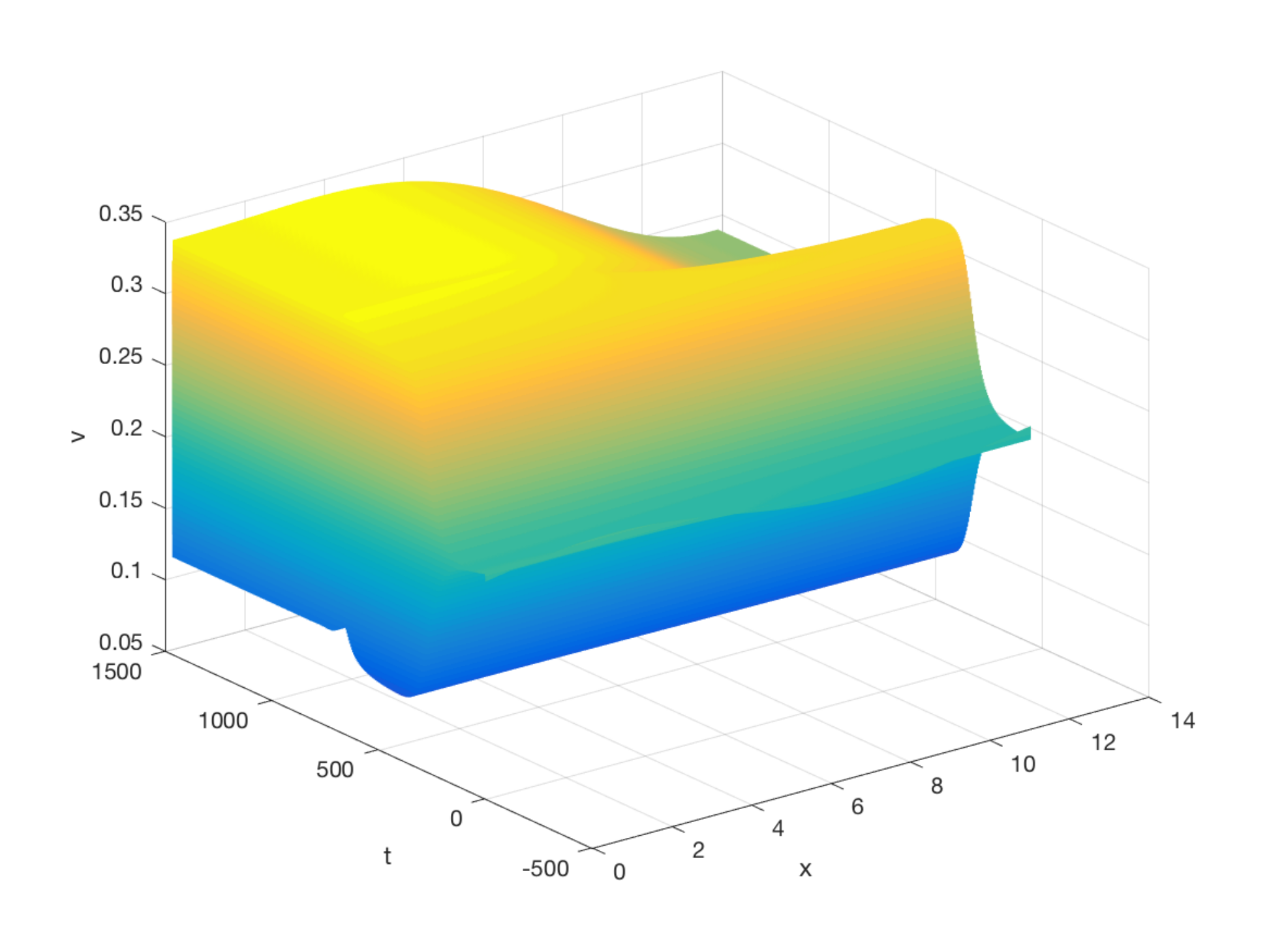}
\end{minipage}}
\subfigure[\tiny predator ]{\begin{minipage}{0.16\linewidth}
		\centering\includegraphics[scale=0.13]{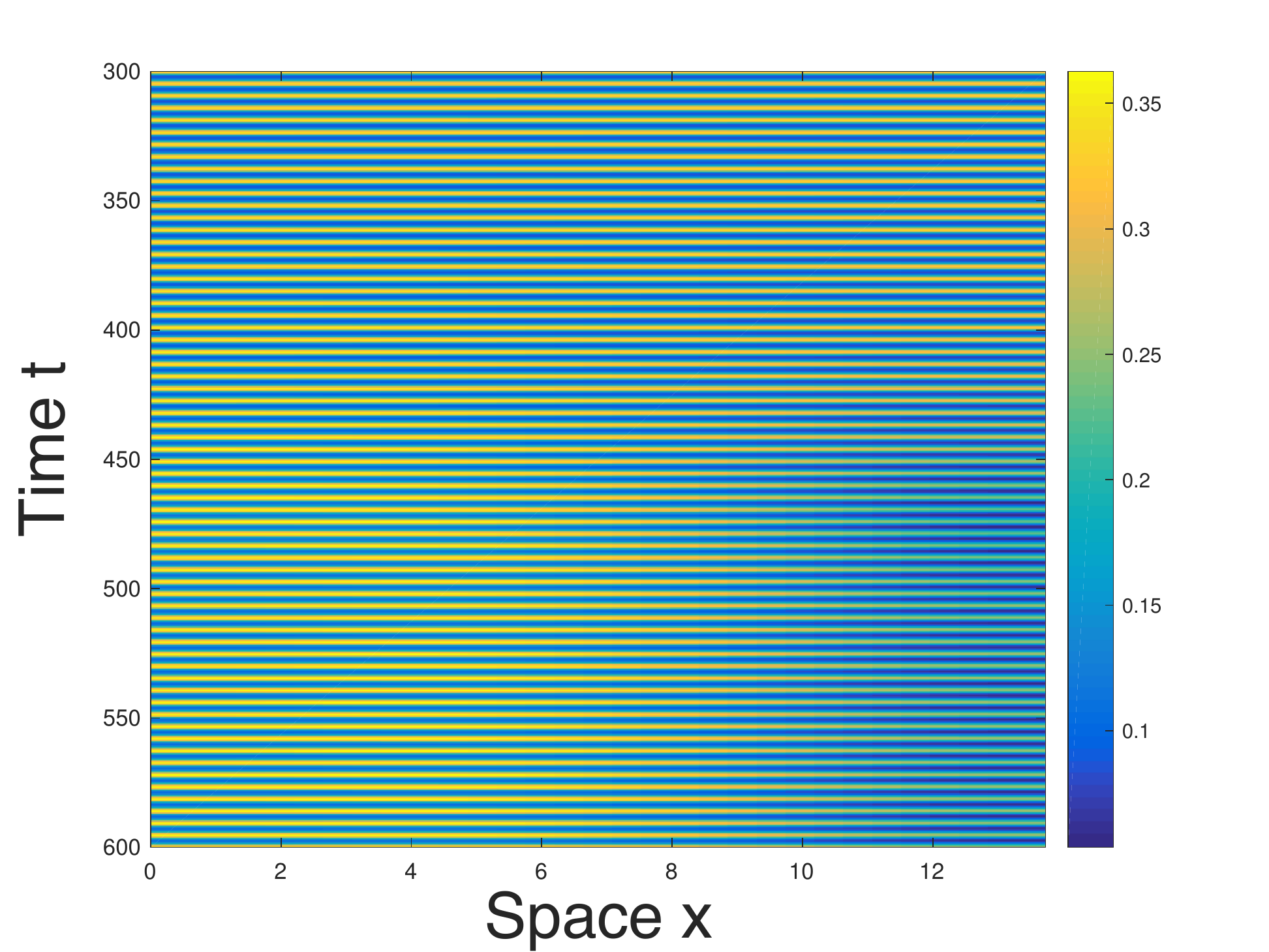}
\end{minipage}}
\subfigure[\tiny transition]{\begin{minipage}{0.16\linewidth}
		\centering\includegraphics[scale=0.13]{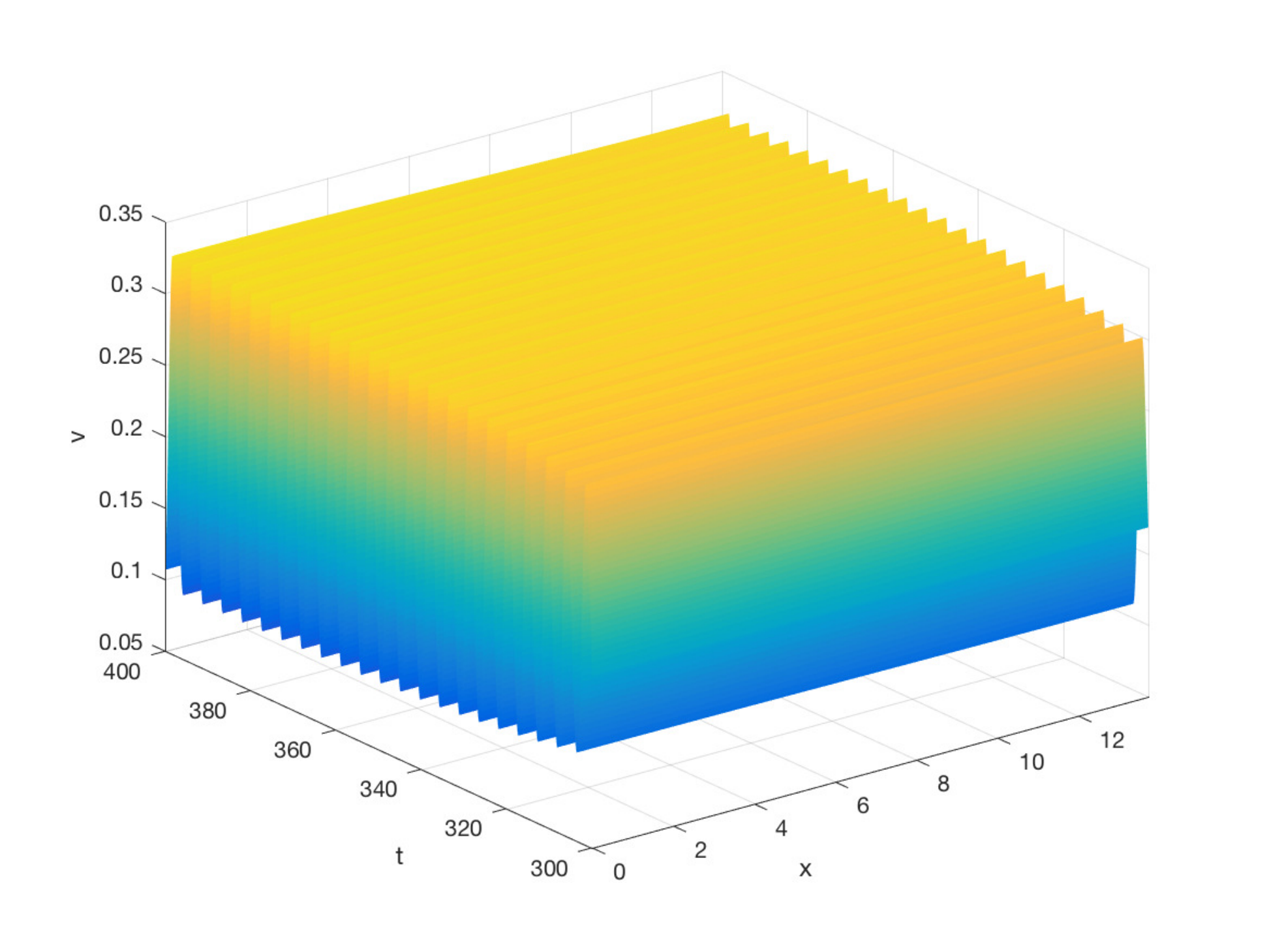}
\end{minipage}}
\subfigure[\tiny \tiny predator]{\begin{minipage}{0.16\linewidth}
		\centering\includegraphics[scale=0.13]{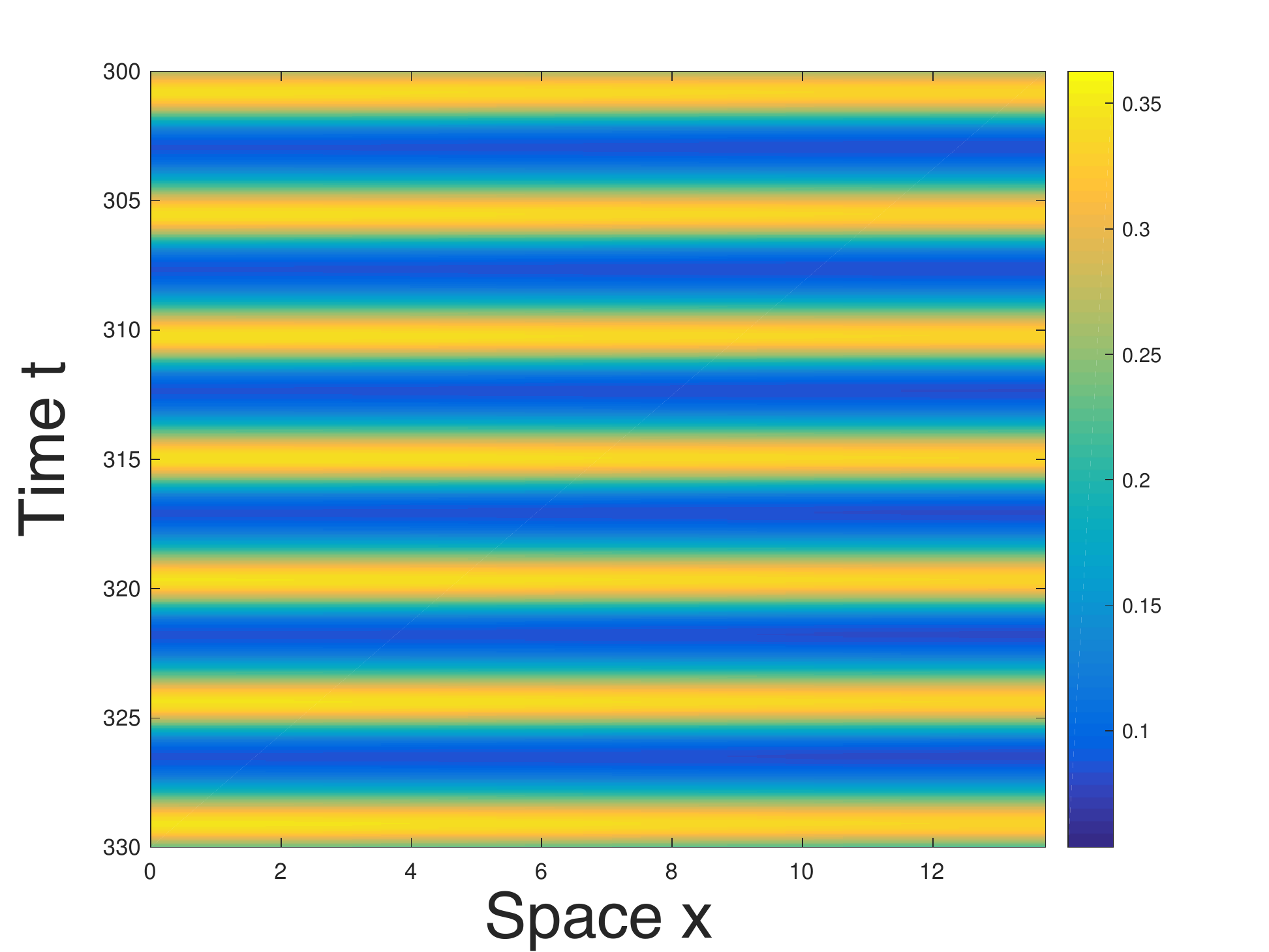}
\end{minipage}}
\subfigure[\tiny target pattern]{\begin{minipage}{0.16\linewidth}
		\centering\includegraphics[scale=0.13]{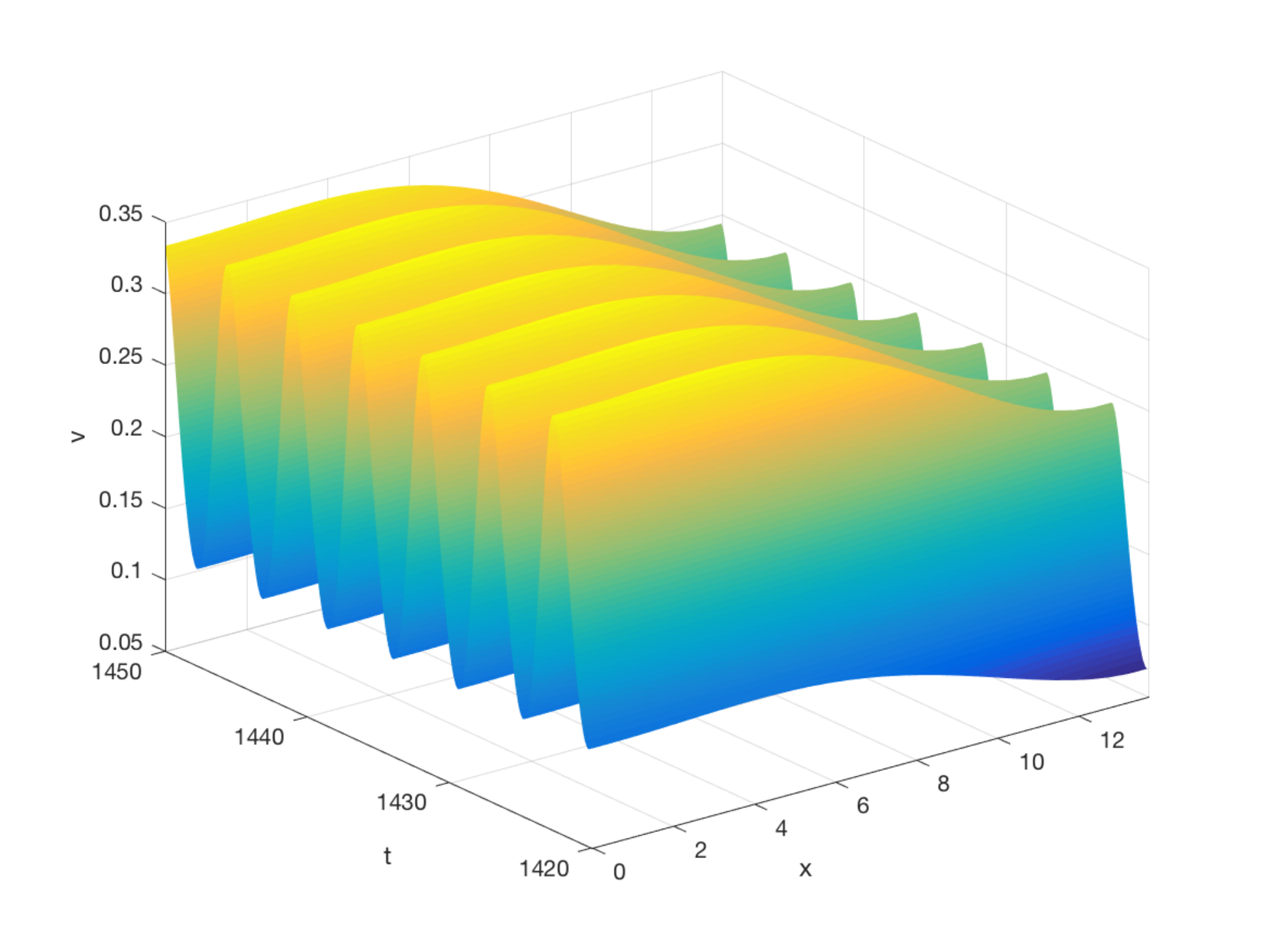}
\end{minipage}}
\subfigure[\tiny predator]{\begin{minipage}{0.16\linewidth}
		\centering\includegraphics[scale=0.13]{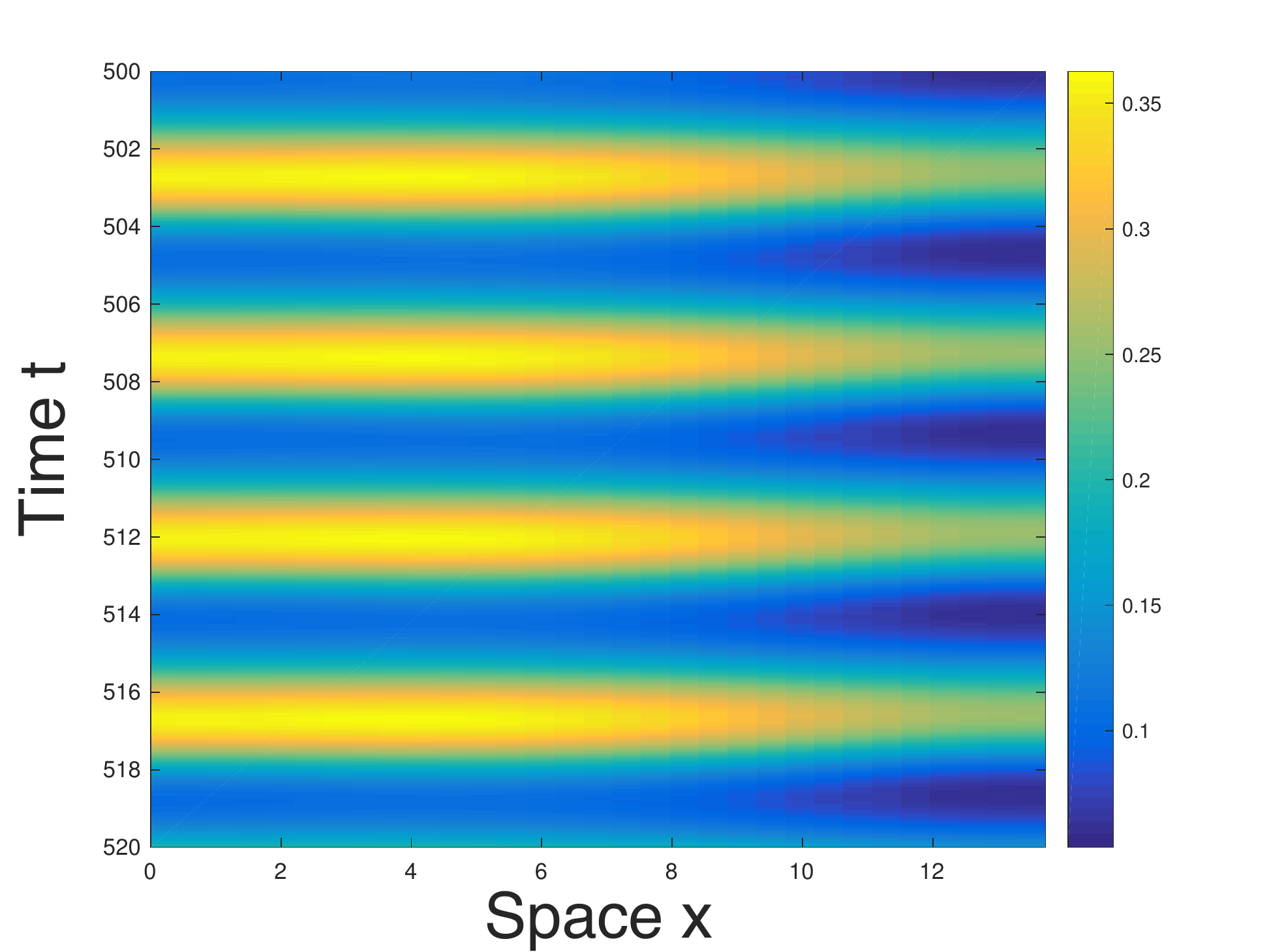}
\end{minipage}}
   \caption{Spatially non-homogeneous periodic solution in $D_5$, with $(\alpha_1,\alpha_2)=(-0.07,-0.007)$ and initial functions are $(u_0+0.01\sin 0.5x,u_0+0.01\sin 0.5x)$}\label{fig2D5_1}
\end{figure}

\begin{figure}[htb]
	\subfigure[\tiny$u(x,t)$]{\begin{minipage}{0.16\linewidth}
			\centering\includegraphics[scale=0.13]{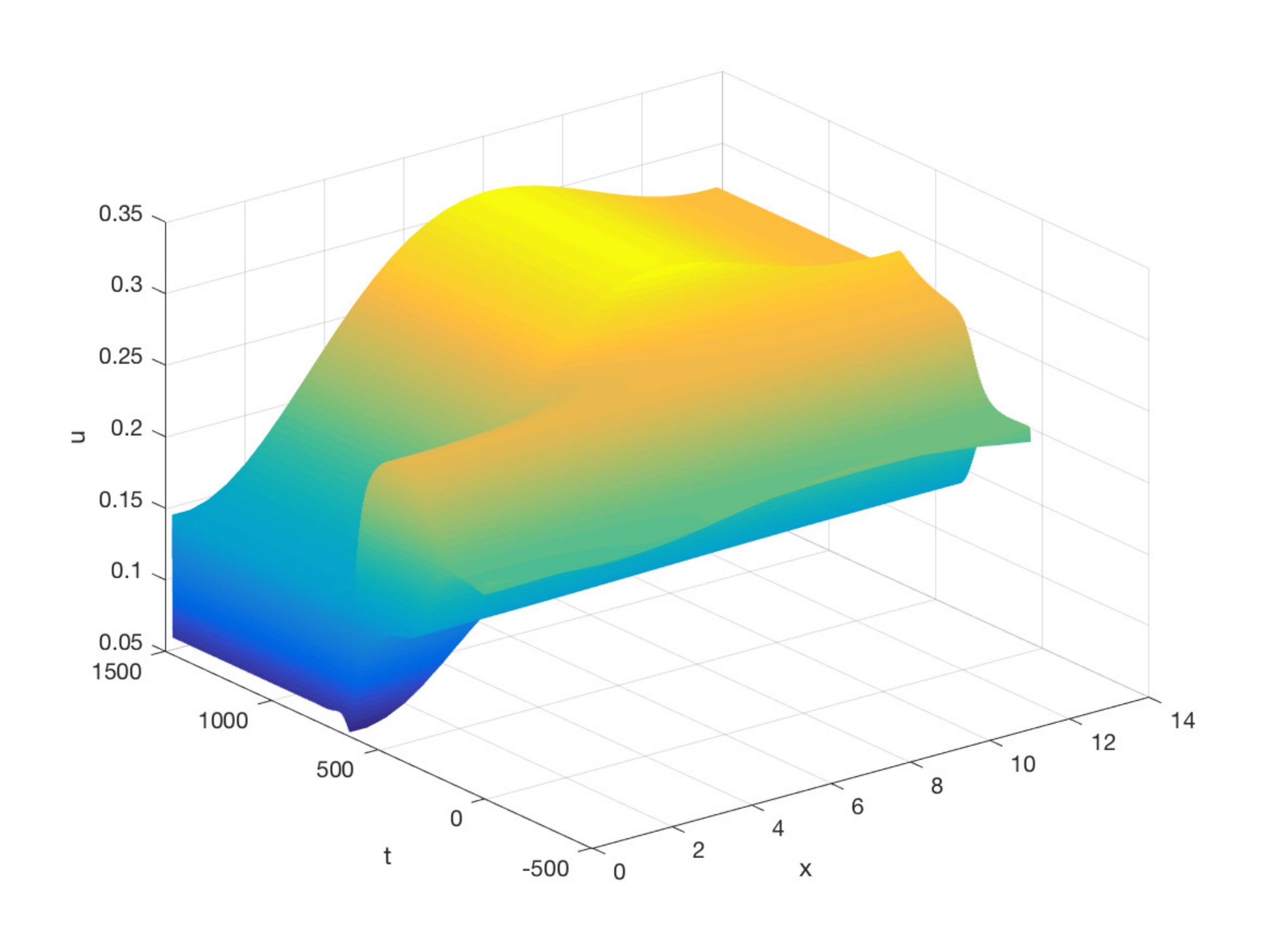}
	\end{minipage}}
	\subfigure[\tiny prey pattern]{\begin{minipage}{0.16\linewidth}
			\centering\includegraphics[scale=0.13]{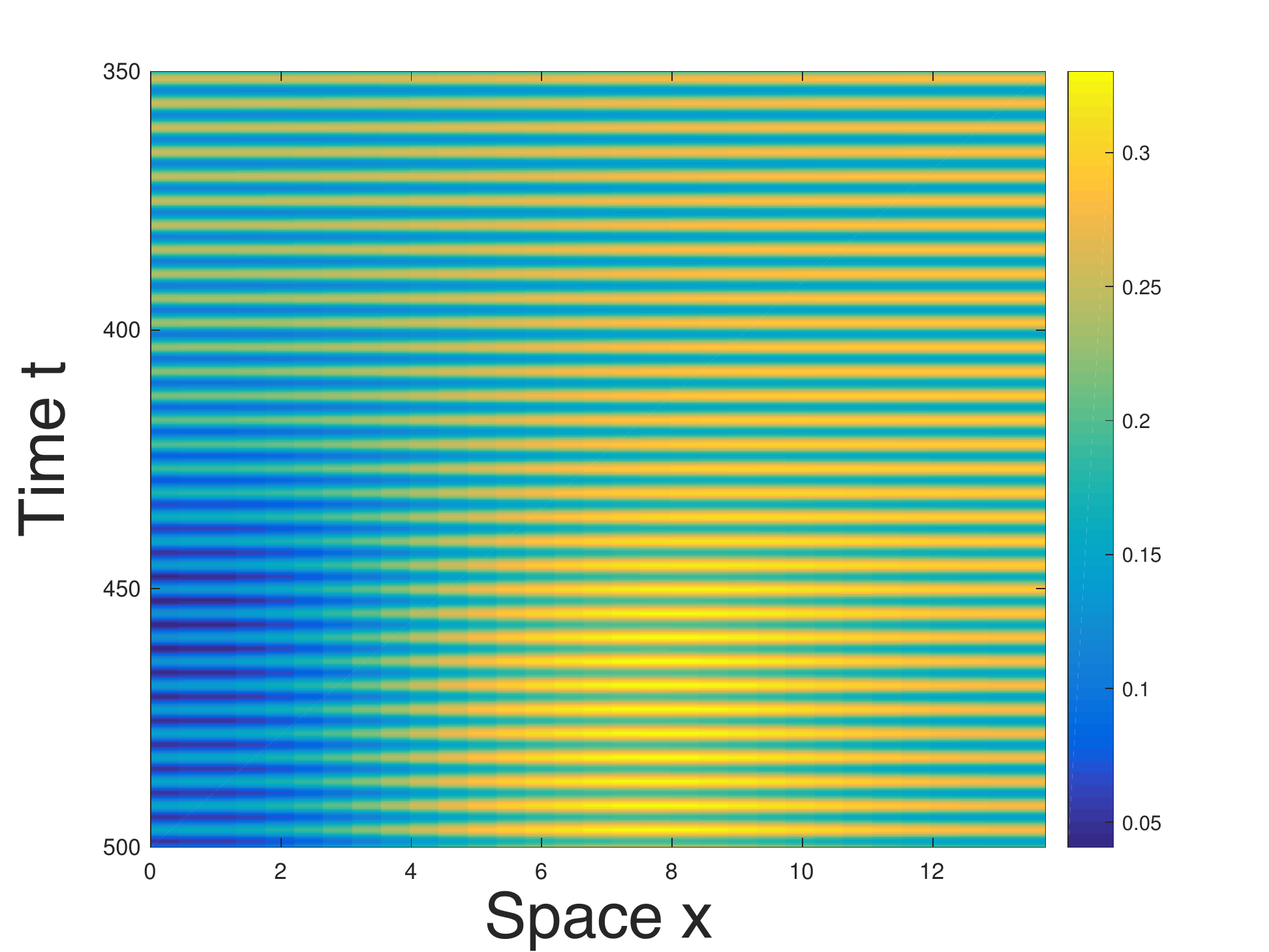}
	\end{minipage}}
	\subfigure[\tiny transition]{\begin{minipage}{0.16\linewidth}
			\centering\includegraphics[scale=0.13]{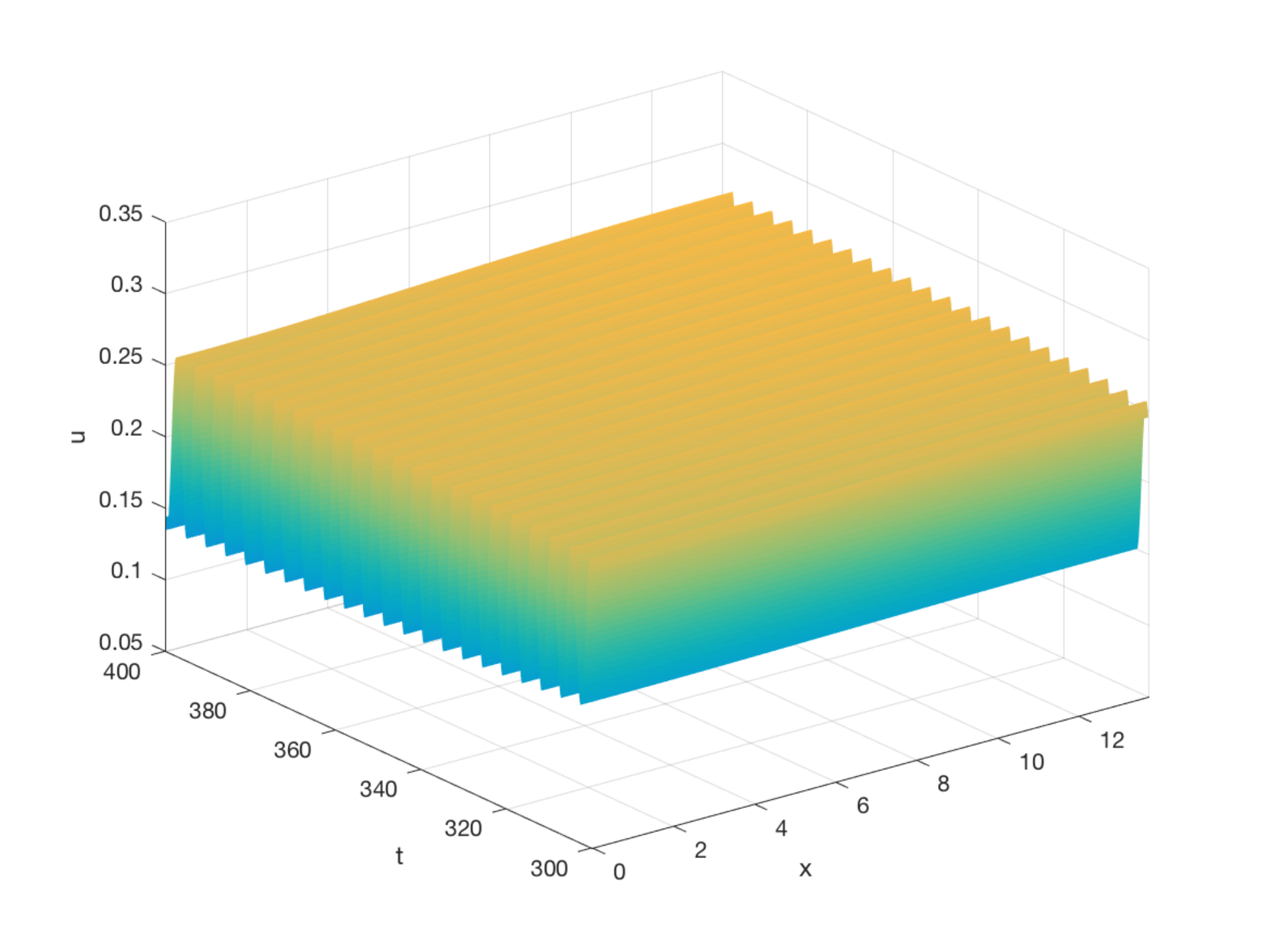}
	\end{minipage}}
	\subfigure[\tiny prey pattern]{\begin{minipage}{0.16\linewidth}
			\centering\includegraphics[scale=0.13]{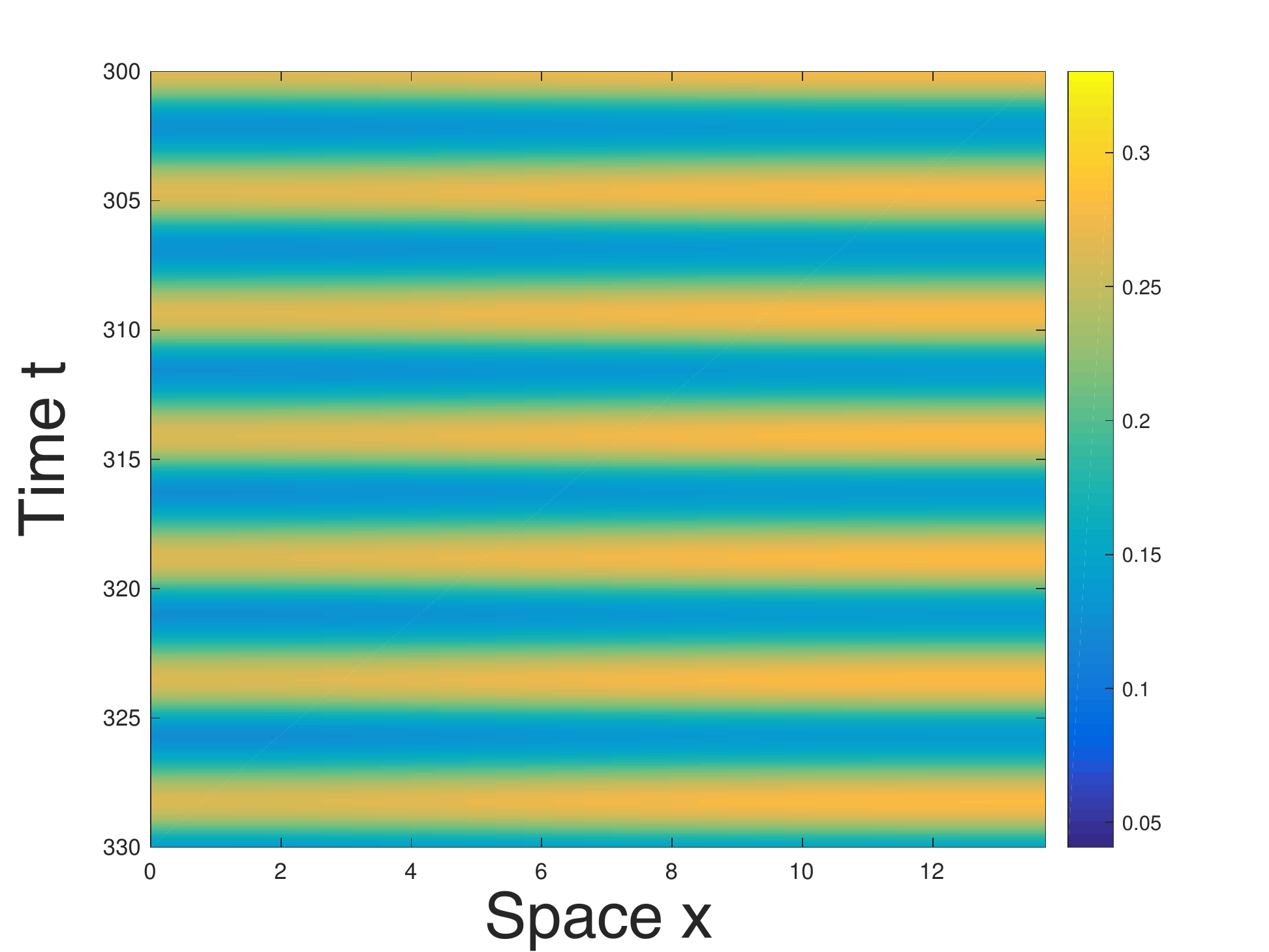}
	\end{minipage}}
	\subfigure[\tiny target pattern]{\begin{minipage}{0.16\linewidth}
			\centering\includegraphics[scale=0.13]{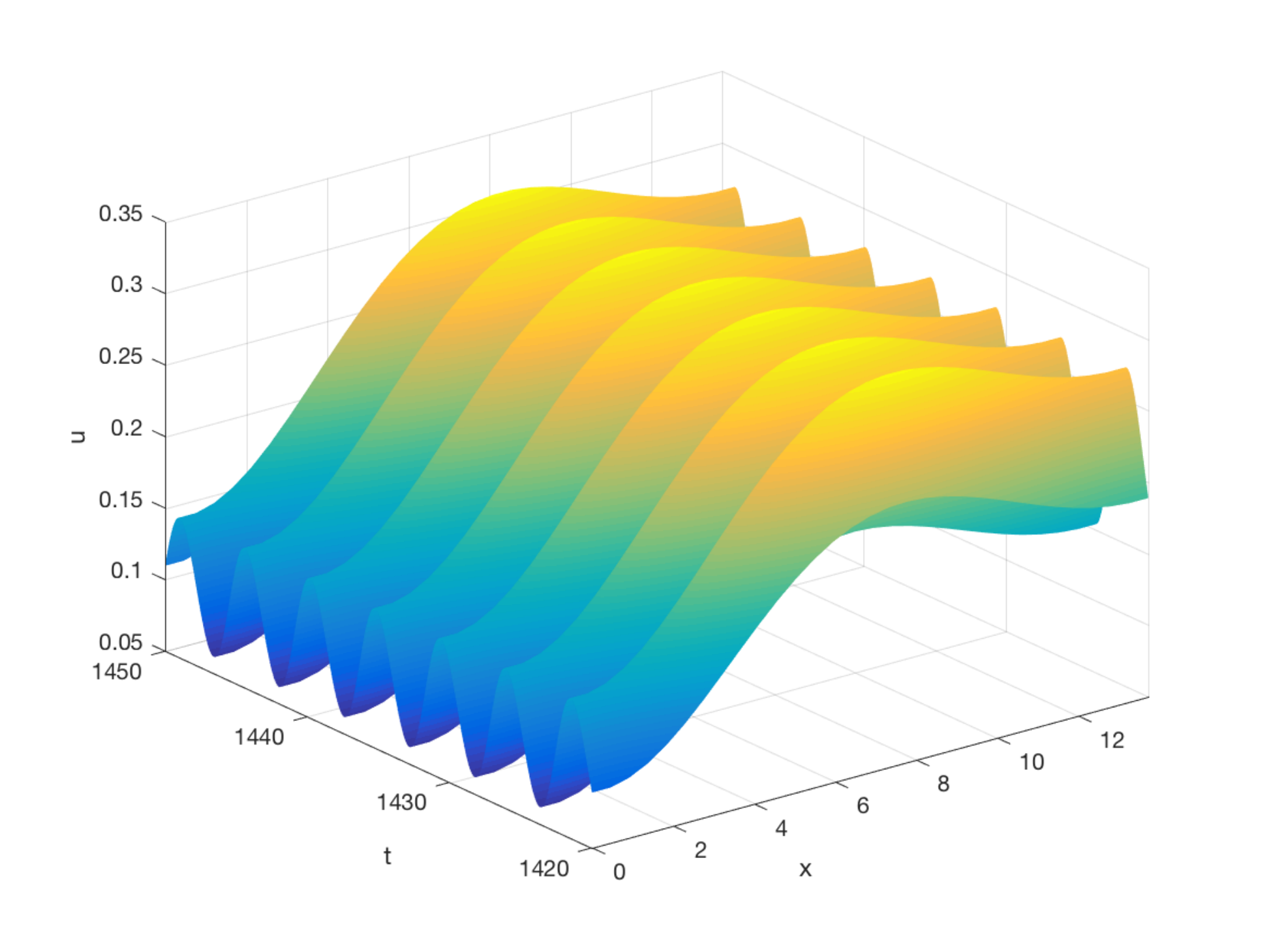}
	\end{minipage}}
	\subfigure[\tiny prey pattern]{\begin{minipage}{0.16\linewidth}
			\centering\includegraphics[scale=0.13]{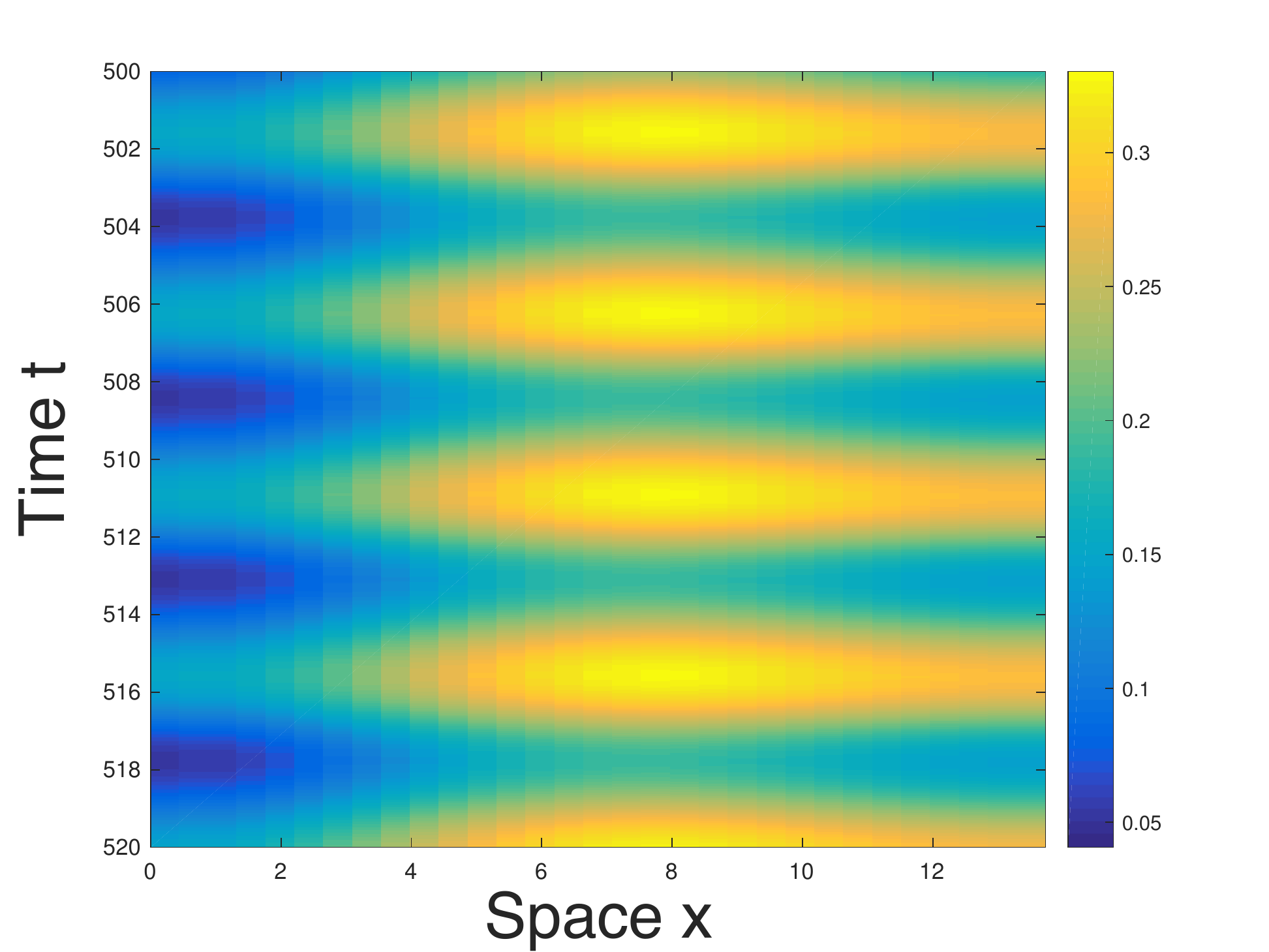}
	\end{minipage}}

	\subfigure[\tiny$v(x,t)$]{\begin{minipage}{0.16\linewidth}
		\centering\includegraphics[scale=0.13]{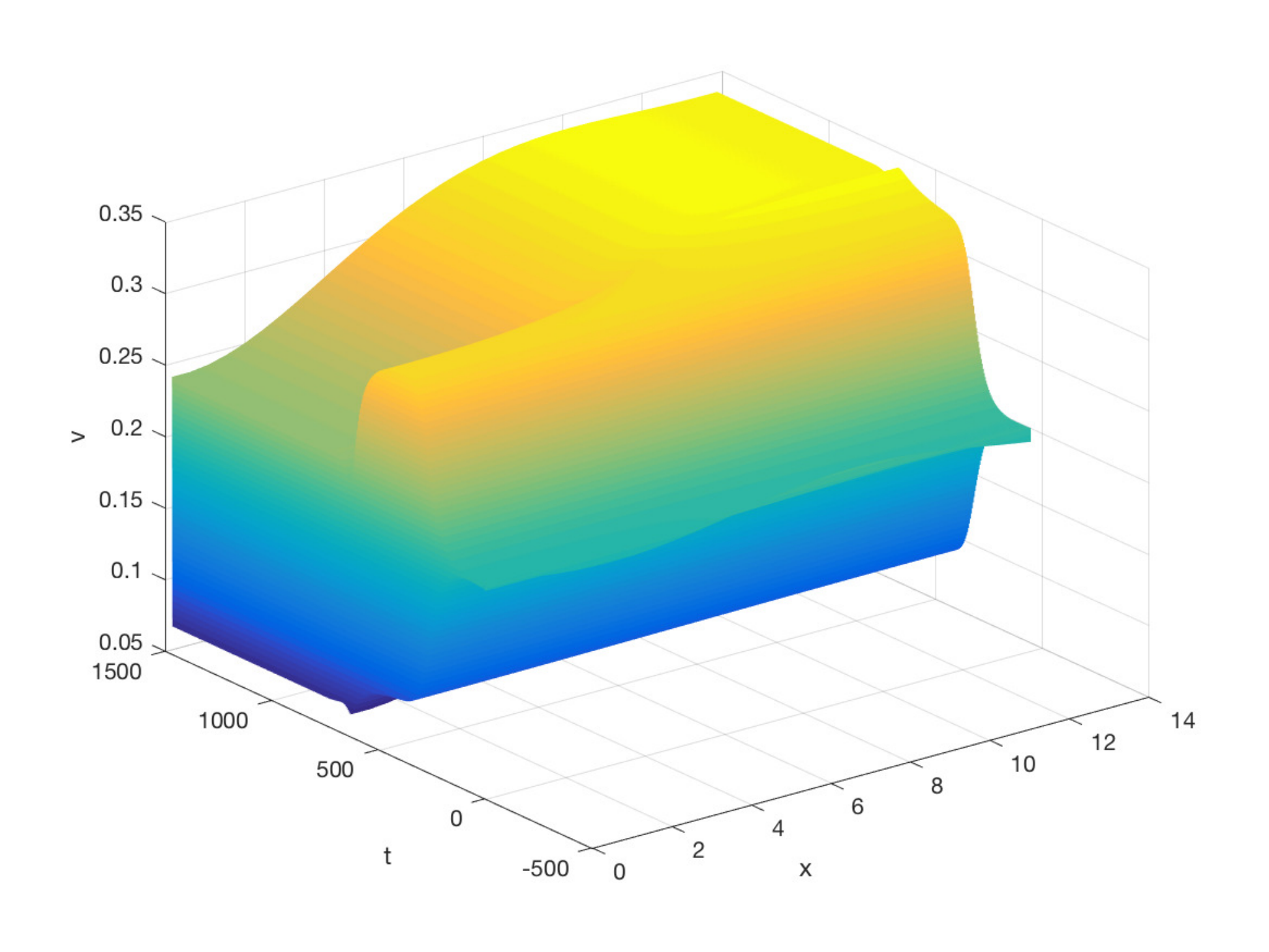}
\end{minipage}}
\subfigure[\tiny predator]{\begin{minipage}{0.16\linewidth}
		\centering\includegraphics[scale=0.13]{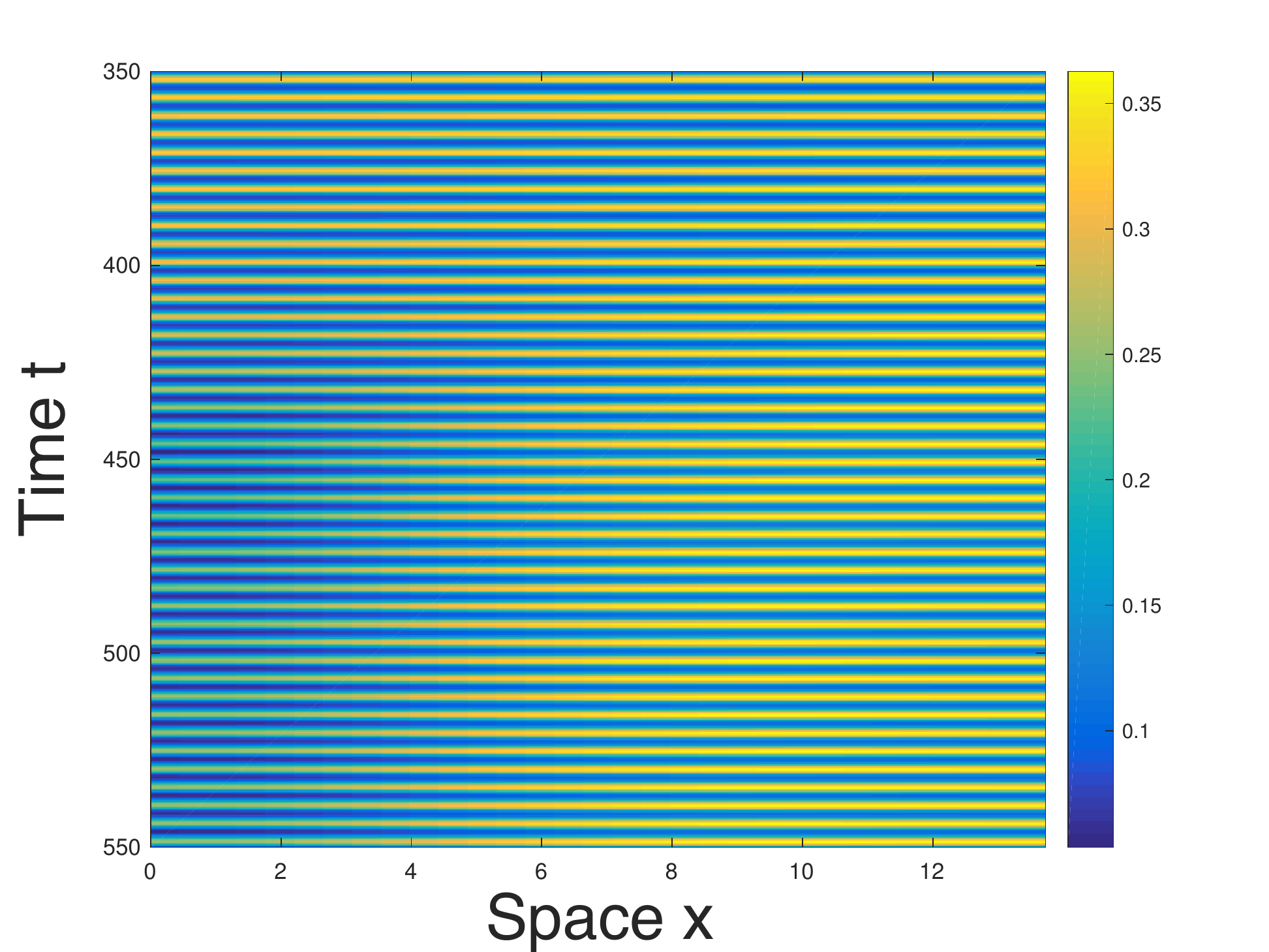}
\end{minipage}}
\subfigure[\tiny transition]{\begin{minipage}{0.16\linewidth}
		\centering\includegraphics[scale=0.13]{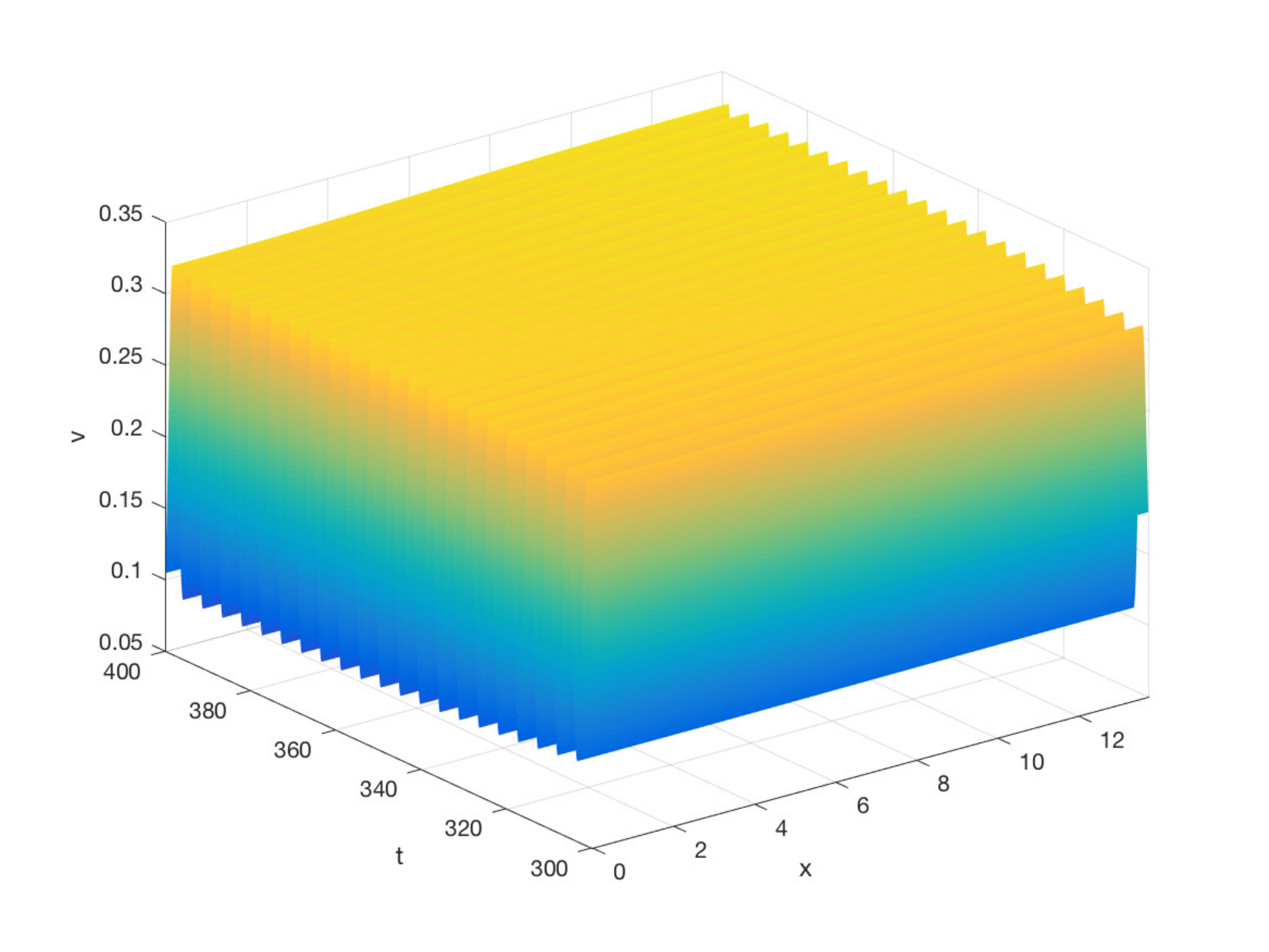}
\end{minipage}}
\subfigure[\tiny predator]{\begin{minipage}{0.16\linewidth}
		\centering\includegraphics[scale=0.13]{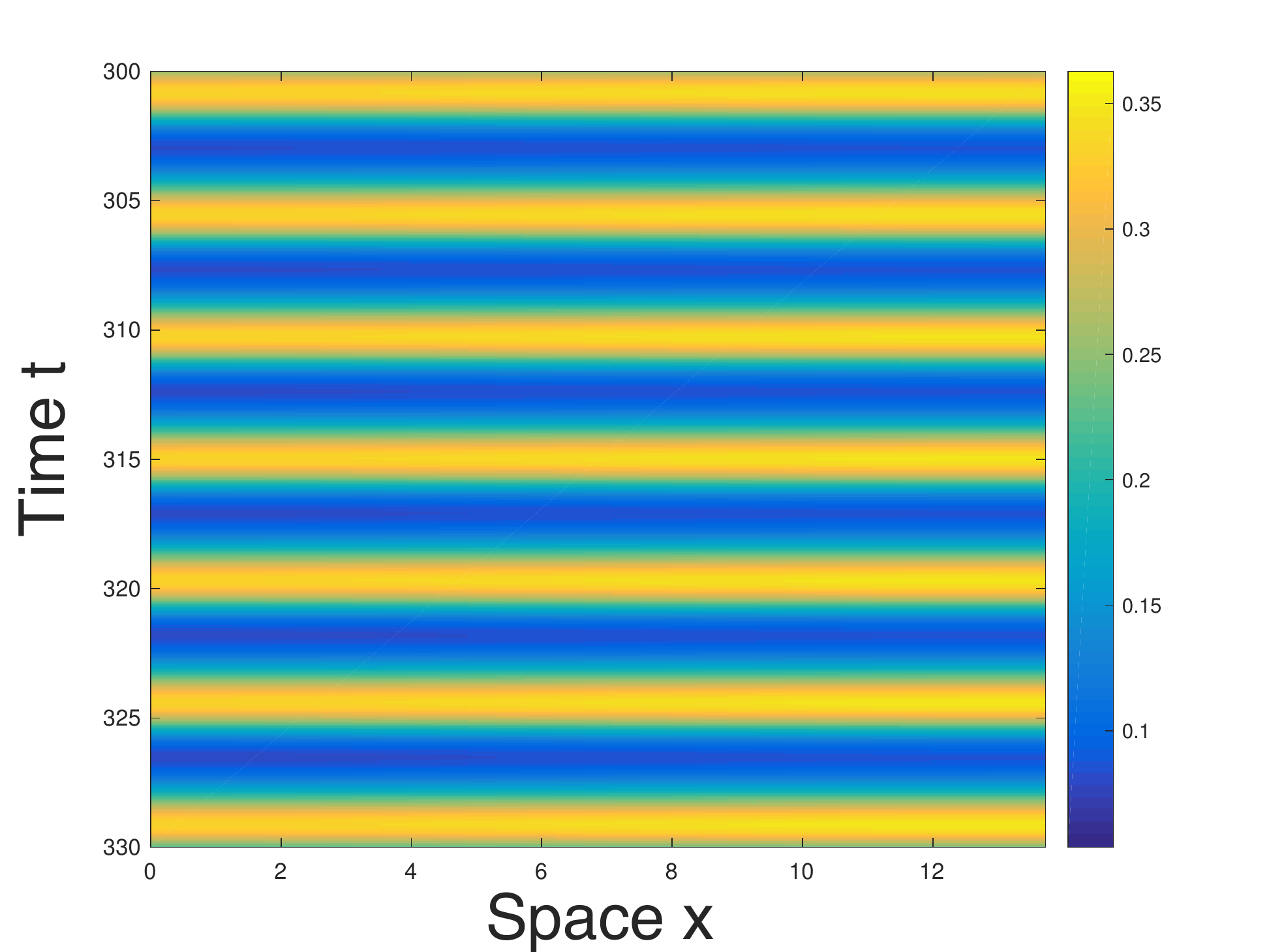}
\end{minipage}}
\subfigure[\tiny target pattern]{\begin{minipage}{0.16\linewidth}
		\centering\includegraphics[scale=0.13]{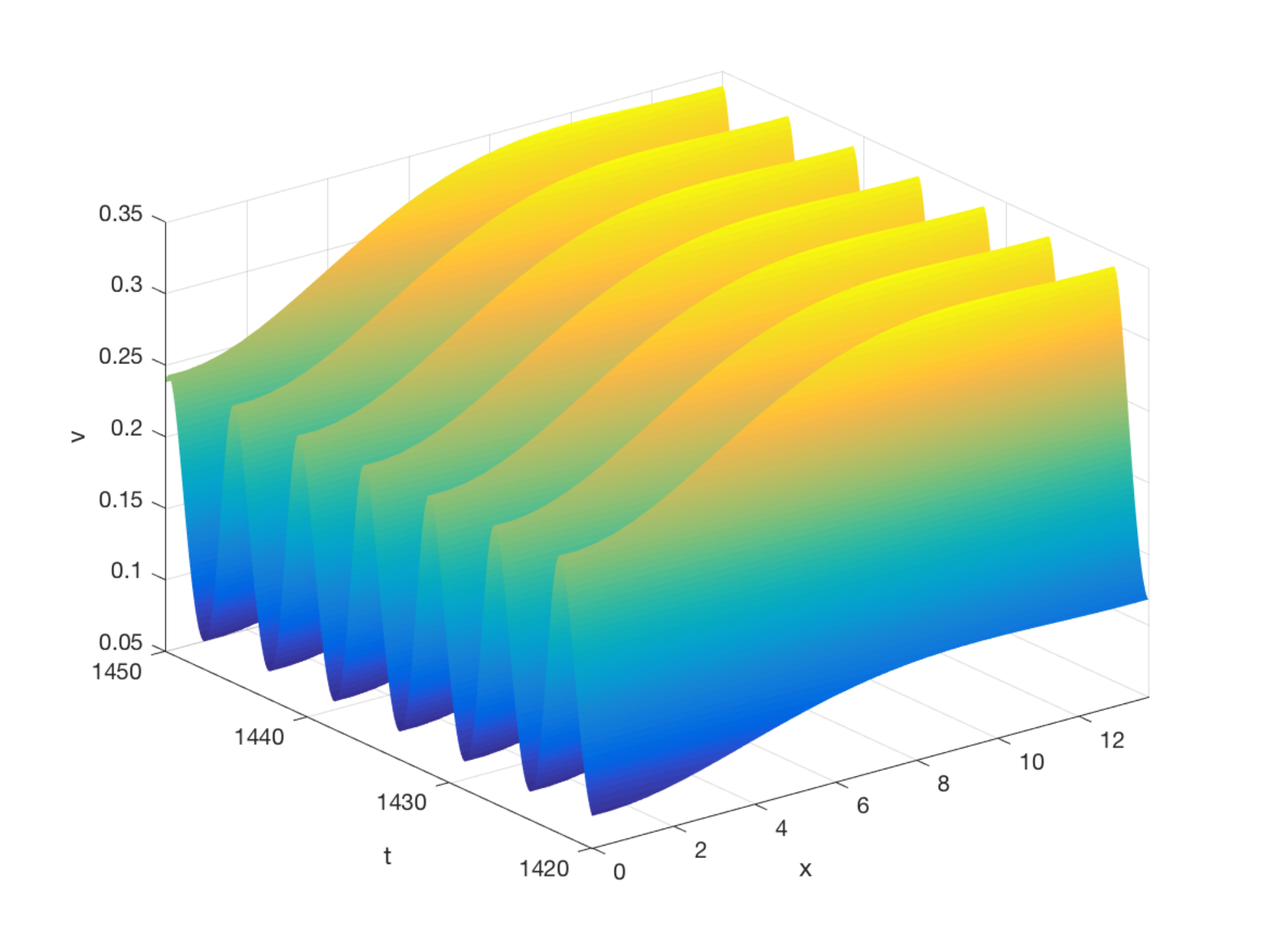}
\end{minipage}}
\subfigure[\tiny predator]{\begin{minipage}{0.16\linewidth}
		\centering\includegraphics[scale=0.13]{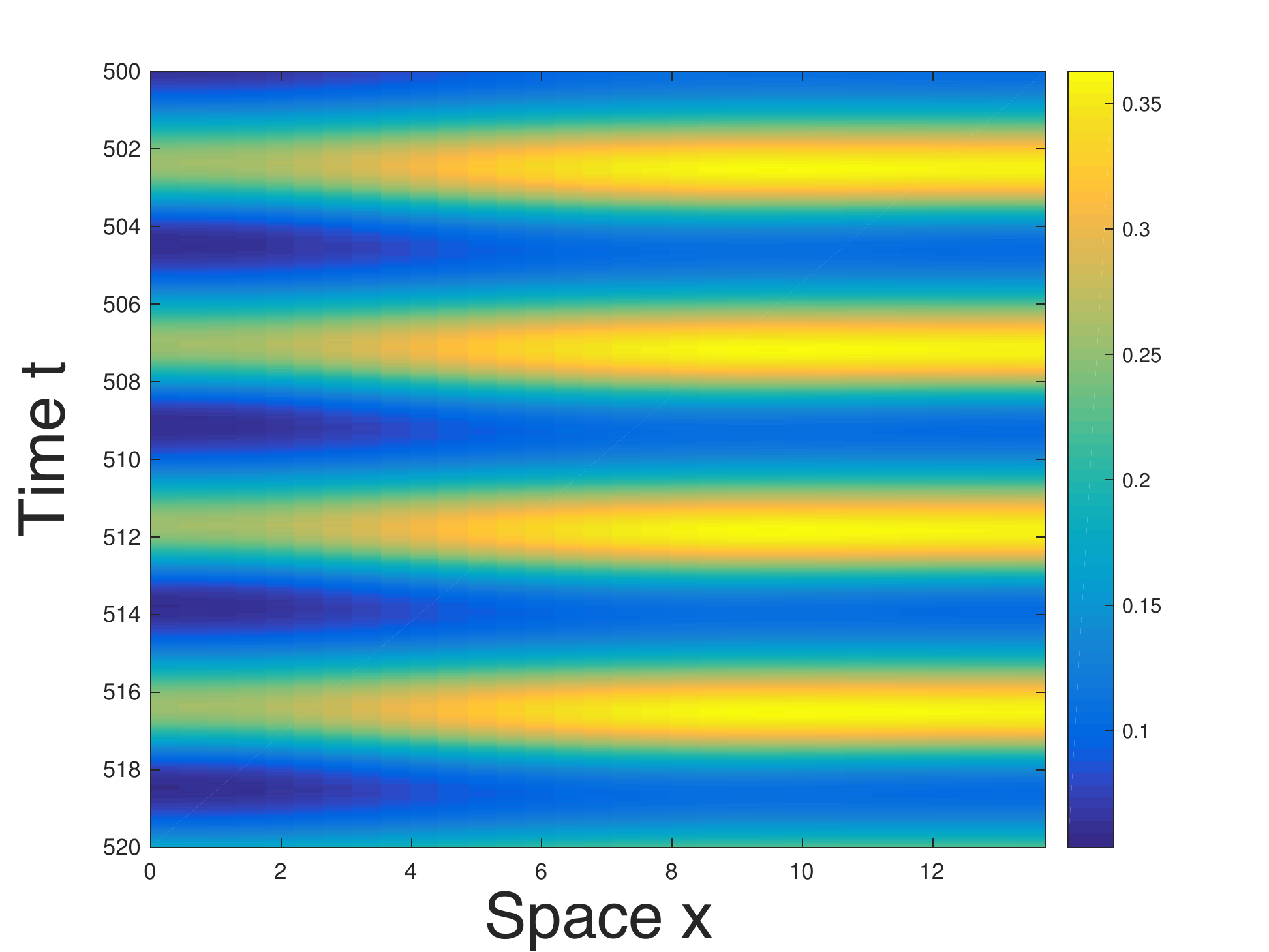}
\end{minipage}}
	\caption{Spatially non-homogeneous periodic solution in $D_5$, with $(\alpha_1,\alpha_2)=(0.07,-0.007)$ and initial functions are $(u_0-0.01\sin 0.5x,u_0-0.01\sin 0.5x)$}\label{fig2D5_2}
\end{figure}

We have done a lot of numerical experiments on the Case \uppercase\expandafter{\romannumeral4}a and observed that when the first Turing critical point (for example, $r=r_*$ in {\bf Group 2}) is relatively close to the second Turing critical point (for example, $r=r_1^T$ in {\bf Group 2}), the dynamics  of the system \eqref{eqA} are similar to those of {\bf Group 2}. That is to say, it's not an accident.
More accurately, { the second Turing bifurcation has no effect on the division of the $(r,\tau)$ parameter plane, it only contributes a superposition function $h_2\cos (\frac{n_I}{l}x)$ to the spatial distribution of the system \eqref{eqA}.}

If we properly adjust the third system parameters (for example, the space size $l$), then the Turing-Hopf bifurcation points {\bf TH1} and {\bf TH3} as shown in Figure \ref{figTH2} will collide into a { Turing-Turing-Hopf} bifurcation point. Therefore, the $(r,\tau)$ plane we have actually studied can be considered as a cross section of the three-dimensional parameter space of the Turing-Turing-Hopf bifurcation. The dynamics observed in our experiments can reflect part of the dynamics caused by Turing-Turing-Hopf bifurcation.
\section{Conclusion}
A rigorously mathematical analysis of the Turing-Hopf bifurcation of the delayed ratio-dependent diffusive Holling-Tanner system is given in this work. The spatiotemporal patterns induced by Turing-Hopf bifurcation and Turing-Turing-Hopf bifurcation are demonstrated through two groups of numerical experiments and theoretical analysis.

When the auxiliary parameters $(a,b,l,d_1,d_2)$  meet the condition $a\geq \frac{(b+1)^2}{2(1-b)}$ {(i.e., for example, that means the predators have a strong ability to consume prey or the inherent growth rate of prey is relatively small)} and $({{\mathbf{A}}}6^{''})$ { (i.e., for example, that means the predator must moves faster than prey, the carrying capacity of the prey should be large and the saturation value of predator should be small)}, the birth ratio $r$ and time delay $\tau$ are taken out as the main parameters to study the spatiotemporal patterns. We claim that the large birth ratio is beneficial to the stability of the system and small birth ratio could lead to the non-uniform distribution of the two populations in space. In addition, large time delay could make the system oscillating as common case.

The Turing-Hopf bifurcation is selected as the main object to study the synergies of the two parameters $(r,\tau)$ to the system \eqref{eqA} by the normal form method, and the complete formula of the normal forms  up to the third order is given near the Turing-Hopf singularity. In both theoretical and numerical experiments, we have proved that the Turing-Hopf bifurcation could generate a wealth of self-organized spatiotemporal patterns. Form the the plane of time and space, the patterns are actually striped and spotted.

More noticeable, we also observed the existence of the spatiotemporal patterns with the form of $\rho\left[\phi_{1}(0)e^{\mathrm{i}\omega_*\tau_* t}+\bar{\phi}_{1}(0)e^{\mathrm{-i}\omega_*\tau_* t}\right]+h_1\cos(\frac{n_T}{l}x)+h_2\cos(\frac{n_I}{l}x)$ in the Holling-Tanner system \eqref{eqA}. Turing-Turing-Hopf bifurcation could be seen as the mechanism to produce them. But these patterns only reflect partial dynamics brought by Turing-Turing-Hopf bifurcation, for a more complete structure, further research is necessary.
\bibliographystyle{plainnat}
\bibliography{mybibfile2.bib}

\end{document}